\documentclass[12pt,amsfonts]{amsart}
\usepackage{amssymb}
\usepackage{mathrsfs}
\usepackage[all,cmtip]{xy}
\usepackage{pstricks}

\usepackage{ulem}

\usepackage{comment}

\numberwithin{equation}{section}

\newtheorem{thm}{Theorem}[section]
\newtheorem{defn}[thm]{Definition}
\newtheorem{prop}[thm]{Proposition}
\newtheorem{res}[thm]{Result}

\newtheorem{lemma}[thm]{Lemma}
\newtheorem{cor}[thm]{Corollary}

\newtheorem{example}[thm]{Example}

\newtheorem{problem}[thm]{Problem}

\newcommand{\PGL}{{\rm PGL}}

\newcommand{\Gr}{{\rm Gr}}

\newcommand{\lt}{{\rm lt}}

\newcommand{\CD}{\xymatrix@R=1pc@C=1pc}
\newcommand{\CDR}{\xymatrix@R=1pc}
\newcommand{\CDC}{\xymatrix@C=1pc}

 \DeclareMathOperator{\Spec}{Spec}
 \DeclareMathOperator{\Proj}{Proj}

\def\cB{{\mathcal B}}

\def\cD{{\mathcal D}}

\def\cE{{\mathcal E}}

\def\cL{{\mathcal L}}

\def\cO{{\mathcal O}}

\def\cZ{{\mathcal Z}}

\def\sF{{\mathscr F}}

\def\sO{{\mathscr O}}

\def\sR{\mathscr{R}}
\def\sV{\mathscr{V}}
\def\tsV{\widetilde{\mathscr{V}}}

\def\fG{\mathfrak{G}}

\def\fL{\mathfrak{L}}

\def\fT{\mathfrak{T}}
\def\fV{\mathfrak{V}}

\def\fd{\mathfrak{d}}
\def\fe{\mathfrak{e}}

\def\fl{\mathfrak{l}}
\def\fr{\mathfrak{r}}

\def\fm{\mathfrak{m}}

\def\fn{\mathfrak{n}}

\def\fr{\mathfrak{r}}
\def\fs{\mathfrak{s}}
\def\ft{\mathfrak{t}}

\def\FF{{\mathbb F}}
\def\NN{{\mathbb N}}
\def\PP{{\mathbb P}}

\def\GG{{\mathbb G}}
\def\GGm{{{\mathbb G}_{\rm m}}}

\def\TT{{\mathbb T}}
\def\ZZ{{\mathbb Z}}

\def\II{{\mathbb I}}

\def\AA{{\mathbb A}}

\def\QQ{{\mathbb Q}}

\def\rU{{\rm U}}
\def\Ga{{\Gamma}}
\def\tGa{{\widetilde{\Gamma}}}

\def\vt{{\varTheta}}

\def\vt{{\vartheta}}

\def\si{{\sigma}}

\def\bL{{\bar L}}

\def\var{{\rm Var}}
\def\Index{{\rm Index}}
\def\invlex{{\rm invlex}}
\def\lex{{\rm lex}}

\def\zero{{=0}}
\def\one{{=1}}

\def\tW{{\widetilde W}}
\def\tJ{{\widetilde J}}

\def\tX{{\widetilde X}}
\def\tY{{\widetilde Y}}
\def\tZ{{\widetilde Z}}

\def\lra{\longrightarrow}

\def\kk{{\bf k}}
\def\ba{{\bf a}}
\def\bp{{\bf p}}

\def\bs{{\bf s}}
\def\bt{{\bf t}}

\def\bx{{\bf x}}
\def\bz{{\bf z}}

\def\bm{{\bf m}}
\def\bn{{\bf n}}

\def\ua{{\underbar {\it a}}}
\def\ub{{\underbar {\it b}}}

\def\ui{{\underbar {\it i}}}

\def\uh{{\underbar {\it h}}}
\def\uk{{\underbar {\it k}}}

\def\um{{\underbar {{\mu}}}}
\def\um{{\underbar {\it m}}}
\def\uu{{\underbar {\it u}}}
\def\uv{{\underbar {\it v}}}
\def\uw{{\underbar {\it w}}}

\def\pl{{\hbox{Pl\"ucker}}}

 \def\2{{\rm I\!I}}

\def\bF{{\bar F}}
\def\-{{\setminus}}

\def\ve{{\varepsilon}}

\def\vp{{\varpi}}
\def\vr{{\varrho}}
\def\hs{{\hslash}}
\def\vi{{\varphi}}

\def\sgn{{\rm sgn}}

\def\rk{{\rm rank \;}}
\def\ori{{\rm ori}}

\def\inc{{\rm inc}}
\def\res{{\rm res}}

\def\mn{{\rm mn}}

\def\pq{{\hbox{\scriptsize ${\rm pre}$-${\rm q}$}}}
\def\q{{\rm q}}

\def\mh{{\hbox{\scriptsize ${\rm m}$-${\rm h}$}}}

\def\sfm{{{\sF}_\um}}

\def\sfmgr{{{\sF}^{\rm rel}_{\um, \Ga}}}
\def\sfmgir{{{\sF}^{\rm irr}_{\um, \Ga}}}

\def\de{\delta}

\def\tsR{{\widetilde{\sR}}}

\def\La{\Lambda}

\def\up{{\Upsilon}}

\def\tV{{\widetilde V}}

\def\tbJ{{\underline {\tilde J}}}

\def\bW{{\underline W}}
\def\bGr{{\underline \Gr}}
\def\tGr{{\widetilde \Gr}}
\def\ud{{\underbar  d}}
\def\tt{{\tilde t}}

\def\whwp{{\widehat\wp}}

\begin{document}

\title{Resolution of Singularities in Arbitrary Characteristic}

\date{}
\author{Yi Hu}
\address{Department of Mathematics, University of Arizona, USA.}

\maketitle

\begin{abstract} 
Let $X$ be an integral affine or projective scheme of finite presentation over a perfect field. 
We prove that $X$ admits a resolution, that is, there exists a smooth
scheme $\tX$ and a  projective birational morphism from $\tX$ onto $X$.   
 \end{abstract}

\maketitle

\bigskip\medskip

\hskip 8.7  cm {\footnotesize \it Regularities are all alike; every}

\hskip 8.7cm {\footnotesize \it   singularity is singular in its own way.}

\bigskip

\tableofcontents

\section{Introduction}


Let $X$  be an  integral affine or projective scheme 
of finite presentation over a perfect  field $\kk$. 
 We say $X$ admits a resolution 
if there exists  a smooth scheme $\tX$ over $\kk$
and a projective birational morphism from $\tX$ onto $X$.

\begin{thm}\label{main:intro}
{\rm (Resolution, Theorems \ref{resolusion-affine} and  \ref{resolusion-proj})}
Let $X$ be an integral affine or projective scheme
of finite presentation over a  perfect field $\kk$. 
Assume further that $X$ is   singular.
Then, $X$  admits a resolution.
\end{thm}

Mn\"ev showed  (\cite{Mnev88}) that every integral {\it singularity type} of finite type
defined over $\ZZ$ appears in some
configuration space of points on the projective plane.
 This result is called Mn\"ev's universality theorem in literature.
Lafforgue (\cite{La99} and  \cite{La03}) strengthened
 and proved the same statement scheme-theoretically. 
 Also, Lee and Vakil (\cite{LV12}) proved the similar  scheme-theoretic statement
 on incidence schemes of  points and lines on the projective plane.
Using Gelfand-MacPerson correspondence, Lafforgue's version of
Mn\"ev's universality theorem is equivalent to the statement that
every integral singularity type of finite type
defined over $\ZZ$ appears in some {\it  matroid Schubert cell} on the Grassmannian 
$\Gr^{3, E}$ of three-dimensional  linear subspaces
of a fixed vector space $E$ of dimension greater than 3.
Every  matroid Schubert cell is an open subset
of a unique closed subscheme of an affine chart of the Grassmannian.
This unique closed subscheme of that affine chart of $\Gr^{3,E}$ is called a $\Ga$-scheme in this article.

We approach Theorem \ref{main:intro} 
via a detour through Mn\"ev's universality theorem by first resolving 
all the aforementioned $\Ga$-schemes that are integral,
hence also, all the  matroid Schubert cells of $\Gr^{3,E}$ that are integral.

Following Lafforgue's presentation of \cite{La03}, 
suppose we have a set of vector spaces, $E_1, \cdots, E_n$ such that 
every $E_\alpha$, $1\le \alpha\le n$,  is of dimension 1 over a field $\kk$
 (or, a free module of rank 1 over $\ZZ$),
 for some positive integer $n>1$.
We let  $$E=E_1 \oplus \ldots \oplus E_n.$$ 

 Then, the Grassmannian $\Gr^{d,E}$, defined by
$$\Gr^{d,E}=\{ \hbox{linear subspaces} \;F \hookrightarrow E \mid \dim F=d\}, $$
is a projective variety defined over $\ZZ$, for  any fixed  integer $1\le d <n$.
(For resolution of singularities, it suffices to focus on $\Gr^{3,E}$;
in this article,  we still consider the general Grassmannian $\Gr^{d,E}$:
see the third paragraph of \S \ref{localization}.)

We have a canonical decomposition
$$\wedge^d E=\bigoplus_{\ui =(i_1<\cdots< i_d) \in \II_{d,n}} E_{i_1}\otimes \cdots \otimes E_{i_d},$$
where $\II_{d,n}$ is the set of all sequences of $d$ distinct integers between 1 and $n$.

This gives rise to the $\pl$ embedding of the Grassmannian by
$$\Gr^{d,E} \hookrightarrow \PP(\wedge^d E)=\{(p_\ui)_{\ui \in \II_{d,n}} \in \GG_m 
\backslash (\wedge^d E \- \{0\} )\},$$
$$F \lra [\wedge^d F],$$
where $\GG_m$ is the multiplicative group.

As a closed subscheme of $\PP(\wedge^d E)$, the Grassmanian  $\Gr^{d,E}$ 
 is defined,  among other relations in general, by 
the $\pl$  ideal $I_\wp$, generated by all $\pl$ relations, 
whose typical member is expressed succinctly, in this article, as
\begin{equation}\label{eq1-intro}
F: \; \sum_{s \in S_F} \sgn (s) p_{\uu_s} p_{\uv_s}
\end{equation}
where $S_F$ is an index set, $\uu_s, \uv_s \in \II_{d,n}$
for any $s \in S_F$, and $\sgn (s)$ is the $\pm$ sign associated with the term $p_{\uu_s} p_{\uv_s}$
(see \eqref{pluckerEq} and \eqref{succinct-pl} for details).

Given the above $\pl$ equation, we introduce
  the projective space $\PP_F$ which comes equipped with the homogeneous coordinates
$[x_{(\uu_s,\uv_s)}]_{s\in S_F}$. 
Then, corresponding to each $\pl$ relation \eqref{eq1-intro}, there is
a linear  homogeneous equation in $\PP_F$, 
called the induced {\it  linearized $\pl$ relation}, 
\begin{equation}\label{eq2-intro}
L_F: \; \sum_{s \in S_F} \sgn (s) x_{(\uu_s,\uv_s)}
\end{equation}
 (see Definition \ref{defn:linear-pl}). We set
 $ \La_F:=\{(\uu_s,\uv_s) \mid s \in S_F\}.$
 
As  any $\Ga$-scheme is a closed subscheme of some affine chart,  we can focus on an affine chart
$\rU_\um =(p_\um \ne 0)$ of the $\pl$ projective space $\PP(\wedge^d E)$ for some fixed $\um \in \II_{d,n}$. 
We can identify the coordinate ring of $\rU_\um$ with the polynomial ring 
$\kk [x_\uu]_{\uu \in \II_{d,n} \- \{\um\}}$. For any $\pl$ relation $F$, we let
$\bF$ be the de-homogenization of $F$ on the chart $\rU_\um$.
Given this chart, we then explicitly describe a set of $\pl$ relations, called
{\it $\um$-primary $\pl$ relations}, listed under a carefully chosen total order $``<_\wp"$,
$$\sfm =\{\bF_1 <_\wp \cdots <_\wp \bF_\Upsilon\},$$
with $\Upsilon= {n \choose d}-1-d(n-d)$, such that together they define
the closed embedding $\rU_\um \cap \Gr^{3,E} \lra \rU_\um$.
Further, on the chart $\rU_\um$,  if we set $p_\um =1$
 and set $x_\uu=p_\uu$ for any $\uu \in \II_{d,n}\-\{\um\}$,
 then any  $\um$-primary relation $\bF\in \sfm$ admits the following de-homogenized expression
\begin{equation}\nonumber
\bF: \; \sgn (s_F) x_{\uu_F} +\sum_{s \in S_F \- \{s_F\}} \sgn (s) x_{\uu_s}x_{\uv_s},
\end{equation}
where $x_{\uu_F}$ is called the leading variable of $\bF$ whose term is called the leading term of $\bF$
and $s_F \in S_F$ is the index
for the leading term.  (See 
\eqref{equ:localized-uu} and \eqref{the-form-LF} for details.)
Correspondingly, the term $\sgn (s_F) x_{(\uu_{s_F},\uv_{s_F})}$ is called the leading term of
the linearized $\pl$ relation $L_F$.

Next, motivated by a parallel construction in \cite{Hu2022},
we introduce the  rational map
\begin{equation}\label{this-theta-intro}
 \xymatrix{
\bar \Theta_{[\up],\Gr}: \rU_\um \cap \Gr^{d,E} \ar @{^{(}->}[r]  & \rU_\um \ar @{-->}[r]  & \prod_{\bF \in \sF_\um} \PP_F   }
\end{equation}
$$
 [x_\uu]_{\uu \in \II_{d,n}} \lra  
\prod_{\bF \in \sF_\um}   [x_\uu x_\uv]_{(\uu,\uv) \in \La_F}
$$   
where $[x_\uu]_{\uu \in \II_{d,n}} $ is the de-homogenized $\pl$ coordinates of a point 
of  $\rU_\um \cap \Gr^{d,E}$.   

We let $\sV_\um$ be the closure of the graph of the rational map $\bar\Theta_{[\up],\Gr}$. Then, 
 we obtain the following diagram 
$$ \xymatrix{
\sV_{\um} \ar[d] \ar @{^{(}->}[r]   \ar[d] \ar @{^{(}->}[r]  &
\sR_\sF:= \rU_\um  \times \prod_{\bF \in \sfm} \PP_F \ar[d] \\
\rU_\um \cap \Gr^{d,E}  \ar @{^{(}->}[r]    & \rU_\um.}
$$

The scheme $\sV_{\um}$ is singular, in general, and is birational to $\rU_\um \cap \Gr^{d,E}$.
(The reader is recommended to read \S \ref{tour} to see the purpose of introducing
the model $\sV_{\um}$.)

As the necessary and crucial steps to achieve our ultimate goal, we are to perform some specific 
sequential embedded blowups for $(\sV_\um \subset \sR_\sF)$.

For  the purpose of applying induction, employed mainly for  proofs, 
we also introduce the following rational map. 

For any positive integer $m$, we set  $[m]:=\{1,\cdots,m\}.$

Then,  for any $k \in [\up]$, we have the rational map
\begin{equation}\label{this-theta[k]-intro}
 \xymatrix{
 \bar\Theta_{[k],\Gr}: \rU_\um \cap \Gr^{d,E} \ar @{^{(}->}[r]  & \rU_\um \ar @{-->}[r]  & 
 \prod_{i \in [k]} \PP_{F_i}   }
\end{equation}
$$
 [x_\uu]_{\uu \in \II_{d,n}} \lra  
\prod_{i \in [k]}   [x_\uu x_\uv]_{(\uu,\uv) \in \La_{F_i}}
$$ 
We let $\sV_{\sF_{[k]}}$ be the closure of the graph of the rational map $\bar\Theta_{[k],\Gr}$. Then, 
 we obtain the following diagram 
$$ \xymatrix{
\sV_{\sF_{[k]}} \ar[d] \ar @{^{(}->}[r]   \ar[d] \ar @{^{(}->}[r]  &
\sR_{\sF_{[k]}}:= \rU_\um  \times  \prod_{i \in [k]} \PP_{F_i}\ar[d] \\
\rU_\um \cap \Gr^{d,E}  \ar @{^{(}->}[r]    & \rU_\um.
}
$$
The scheme $\sV_{\sF_{[k]}}$ is birational to  $\rU_\um \cap \Gr^{d,E}$. 

Set $\sR_{\sF_{[0]}}:=\rU_\um$.  There exists a forgetful map 
$$\sR_{\sF_{[j]}} \lra \sR_{\sF_{[j-1]}},\;\;\; \hbox{ for any $j \in [\up]$}.$$

In the above notations, we have
$$\sV_\um=\sV_{\sF_{[\up]}}, \;\; \sR_{\sF}=\sR_{\sF_{[\up]}}.$$

\begin{prop}  {\rm (Corollary \ref{eq-tA-for-sV})}  The scheme $\sV_\um$, as a closed subscheme of
$\sR_\sF= \rU_\um \times  \prod_{\bF \in \sF_\um} \PP_F$,
is defined by the following relations, for all $\bF \in \sfm$,
\begin{eqnarray} 
B_{F,(s,t)}: \;\; x_{(\uu_s, \uv_s)}x_{\uu_t}x_{ \uv_t}-x_{(\uu_t, \uv_t)}x_{\uu_s}x_{ \uv_s}, \;\; \forall \;\; 
s, t \in S_F \- \{s_F\}, \label{eq-Bres-intro}\\
B_{F, (s_F,s)}: \;\; x_{(\uu_s, \uv_s)}x_{\uu_F} - x_{(\um,\uu_F)}   x_{\uu_s} x_{\uv_s}, \;\;
\forall \;\; s \in S_F \- \{s_F\},  \label{eq-B-intro} \\ 
\cB^\pq,   \;\; \;\; \;\; \;\; \;\; \;\; \;\; \;\; \;\; \;\; \;\; \;\; \;\; \label{eq-hq-intro}\\
L_F: \;\; \sum_{s \in S_F} \sgn (s) x_{(\uu_s,\uv_s)}, 
\label{linear-pl-intro}
\end{eqnarray}
with $\bF$  expressed as 
$\sgn (s_F) x_{\uu_F} +\sum_{s \in S_F \- \{s_F\}} \sgn (s) x_{\uu_s}x_{\uv_s}$,
 where $\cB^\pq$ is the set of binomial equations of pre-quotient type
(see Definition \ref{defn:pre-q}).  
\end{prop}

Our construction of the desired embedded blowups on  $\sV_{\um} \subset \sR_\sF$
 is based upon the set of all binomial relations of \eqref{eq-B-intro}:
$$\cB^\mn_F=\{B_{F,(s_F,s)} \mid  s \in S_F \- \{s_F\}\},\;\;
\cB^\mn=\bigsqcup_{\bF \in \sfm} \cB^\mn_F,$$
and all the linearized $\pl$ relation 
$$L_{\sfm}=\{L_F \mid \bF \in \sfm\}.$$

An element $B_{F,(s_F,s)}$ of  $\cB^\mn$ is called a main binomial relation.
We  also let
$$\cB^\res=\{ B_{F,(s,t)} \mid \bF \in \sfm, \; s, t \in S_F \- \{s_F\}\}.$$
An element $B_{F,(s,t)}$ of  $\cB^\res$ is called a residual binomial relation.
The residual binomial relations or binomial relations of pre-quotient type in $\cB^\pq$
play no roles in the {\it construction} of the aforesaid embedded blowups.

To apply induction, we provide a total order on the set $S_F \- \{s_F\}$ and list it as
$$S_F \- \{s_F\}=\{s_1 < \cdots < s_{\ft_F}\}$$
where $(\ft_F+1)$ is the number of terms in the relation $F$.
This renders us  to write $B_{F,(s_F,s)}$ as $B_{(k\tau)}$ where
$F=F_k$ for some $k \in [\up]$ and $s = s_\tau$ for some $ \tau \in [\ft_{F_k}]$.

We can now synopsize  the process of the embedded blowups for  $(\sV_{\um} \subset \sR_\sF)$.

It is divided into two sequential blowups.
The first is $\vt$-blowups.
The second is constructed by induction on $k \in [\up]$.
For each fixed $k \in [\up]$, it consists of a sequential $\wp$-blowups  and then  a single $\ell$-blowup.

{\it $\bullet$ On $\vt$-sets, $\vt$-centers, and $\vt$-blowups. }

For any primary $\pl$ relation $\bF_k \in \sfm$, we introduce the corresponding $\vt$-set
$\vt_{[k]}=\{x_{\uu_{F_k}}, x_{(\um,\uu_{F_k})} \}$ and the corresponding $\vt$-center 
$Z_{\vt_{[k]}} = X_{\uu_{F_k}} \cap X_{(\um,\uu_{F_k})}$ where
$X_{\uu_{F_k}} = (x_{\uu_{F_k}}=0)$ and
$ X_{(\um,\uu_{F_k})} =(x_{(\um,\uu_{F_k})} =0)$.
 We set $\tsR_{\vt_{[0]}}:=\sR_\sF$. Then, inductively,
 we let  $\tsR_{\vt_{[k]}} \to \tsR_{\vt_{[k-1]}}$ be the blowup of $\tsR_{\vt_{[k-1]}}$
 along (the proper transform of) the $\vt$-center
 $Z_{\vt_{[k]}}$ for all $k \in [\up]$. 
 This gives rise to the sequential $\vt$-blowups
\begin{equation}\label{vt-sequence-intro}
\tsR_{\vt}:=\tsR_{\vt_{[\up]}}  \lra \cdots \lra \tsR_{\vt_{[k]}} \lra \tsR_{\vt_{[k-1]}} \lra \cdots \lra \tsR_{\vt_{[0]}}.
\end{equation}
 Each morphism $\tsR_{\vt_{[k]}} \to \tsR_{\vt_{[k-1]}}$ is a smooth blowup, meaning, 
 the blowup of a smooth scheme along a smooth closed center. For any $k$, we let
 $\tsV_{\vt_{[k]}} \subset \tsR_{\vt_{[k]}} $ be the proper transform of $\sV$ in $\tsR_{\vt_{[k]}}$.
 We set $\tsV_{\vt}:=\tsV_{\vt_{[\up]}}$.
 
{\it $\bullet$ On $\wp$-sets, $\wp$-centers, and $\wp$-blowups
as well as $\ell$-sets, $\ell$-centers, and $\ell$-blowups. }

All these are constructed based upon $\cB^\mn_{F_k}$ and $L_{F_k}$,
 inductively on $k \in [\up]$.

For any main binomial $B_{(k\tau)} \in \cB^\mn_{F_k}$, there exist a finite 
integer $\rho_{(k\tau)}$ depending on $(k\tau)$ and 
a finite integer $\si_{(k\tau)\mu}$ depending on $(k\tau)\mu$ for any $ \mu \in [\rho_{(k\tau)}]$.
We set $\tsR_{(\wp_{(11)}\fr_0)}=\tsR_\vt$.
For each $((k\tau), \mu, h)$ $h \in [\si_{(k\tau)\mu}]$, there exists 
a $\wp$-set $\phi_{(k\tau)\mu h}$ consisting of  two special divisors  on an inductively 
defined scheme $\tsR_{(\wp_{(k\tau)}\fr_{\mu -1})}$; its corresponding $\wp$-center
$Z_{\phi_{(k\tau)\mu h}}$ is the scheme-theoretic intersection of the two divisors.
We let $\cZ_{\wp_k}=\{Z_{\phi_{(k\tau)\mu h}} \mid k \in [\up], \tau \in [\ft_{F_k}], \mu \in [\rho_{(k\tau)}],
h \in [\si_{(k\tau)\mu}]\}$, 
totally ordered lexicographically on the indexes $(k,\tau, \mu, h)$. Then, inductively,
 we let we $\tsR_{(\wp_{(k\tau)}\fr_\mu\fs_{h})} \to \tsR_{(\wp_{(k\tau)}\fr_\mu\fs_{h-1})}$ 
 be the blowup of $\tsR_{(\wp_{(k\tau)}\fr_\mu\fs_{h-1})}$ 
 along (the proper transform of) the $\wp$-center
 $Z_{\phi_{(k\tau)\mu h}}$. 
 This gives rise to the sequential  $\wp$-blowups with respect to $\cB^\mn_{F_k}$
\begin{equation}\label{wp-sequence-intro}
\tsR_{\wp_k}  \to \cdots \to
\tsR_{(\wp_{(k\tau)}\fr_\mu\fs_{h})} \to \tsR_{(\wp_{(k\tau)}\fr_\mu\fs_{h-1})} \to \cdots \to \tsR_{\ell_{k-1}},
\end{equation}
where $\tsR_{\ell_{k-1}}$ is inductively constructed from the previous $\wp$ and $\ell$-blowups, and 
$\tsR_{\wp_k} $ is the end scheme of $\wp$-blowups with respect to $\cB^\mn_{F_k}$.
  For any $(k\tau)\mu h$, we let
 $\tsV_{(\wp_{(k\tau)}\fr_\mu\fs_{h})} \subset \tsR_{(\wp_{(k\tau)}\fr_\mu\fs_{h})} $
  be the proper transform of $\sV$ in $\tsR_{(\wp_{(k\tau)}\fr_\mu\fs_{h})} $.
 We set $\tsV_{\wp_k} \subset \tsR_{\wp_k}$ be the last induced subscheme.
 Every scheme $\tsR_{(\wp_{(k\tau)}\fr_\mu\fs_{h})}$ 
  has a smooth open subset  $\tsR^\circ_{(\wp_{(k\tau)}\fr_\mu\fs_{h})}$ 
  containing  $\tsV_{(\wp_{(k\tau)}\fr_\mu\fs_{h})}$.


Now, we let $D_{L_{F_k}}$ be the divisor of $\sR_\sF$ defined by $(L_{F_k}=0)$;
we let $X_{(\uu_{s_F}, \uv_{s_F})}$ be the divisor of $\sR_\sF$ defined by 
$(x_{(\uu_{s_F}, \uv_{s_F})}=0)$ where $\sgn (s_F) x_{(\uu_{s_F}, \uv_{s_F})}$ is the leading term 
of $L_F$. We then let $D_{\wp_k, L_{F_k}}$ be proper transform of $D_{L_{F_k}}$ 
 and $X_{\wp_k, (\uu_{s_F}, \uv_{s_F})}$ be the proper transform of
 $X_{(\uu_{s_F}, \uv_{s_F})}$ in $\tsR_{\wp_k}$. We let
\begin{equation}\label{ell-sequence-intro}
\tsR_{\ell_k} \lra \tsR_{\wp_k}
\end{equation}
 be the blowup of $\tsR_{\wp_k}$ along the intersection
 $D_{\wp_k, L_{F_k}} \cap X_{\wp_k, (\uu_{s_F}, \uv_{s_F})}$.

We let $\tsV_{\ell_k}$ be the proper transform of $\tsV_{\wp_k}$ in $\tsR_{\ell_k}$. 
The scheme $\tsR_{\ell_k}$ has a smooth open subset   $\tsR^\circ_{\ell_k}$ 
containing $\tsV_{\ell_k}$.

When $k=\up$, we obtain our final schemes
$$\tsV_{\ell_k}\subset \tsR_{\ell_k}.$$

We point out here the $\ell$-blowup with respect to $F_k$ has to immediately follow the $\wp$-blowups
 with respect to $F_k$; the order of $\wp$-blowups  with respect to a fixed $\pl$ relation $F_k$
may be subtle and are carefully chosen.

To study the local structure of $\tsV_\ell \subset \tsR_\ell$, we approach 
it by induction via the sequential blowups \eqref{vt-sequence-intro},
\eqref{wp-sequence-intro},   and \eqref{ell-sequence-intro}.

 Definition \ref{general-standard-chart} introduces the covering
  standard affine charts $\fV$ for any  of the smooth schemes  $\tsR^\circ_{\vt_{[k]}}$,
  $\tsR^\circ_{(\wp_{(k\tau)}\fr_\mu\fs_h)}$,
  and $\tsR^\circ_{\ell_{k}}$ in the above. 
 
 $(\star)$ Proposition \ref{meaning-of-var-vtk} introduces coordinate variables
 for  any standard affine chart $\fV$ of $\tsR_{\vt_{[k]}}$ and provides
 explicit geometric meaning for every coordinate variable.
 
 Proposition  \ref{eq-for-sV-vtk} provides explicit description and properties of the local defining equations of 
the scheme $\tsV_{\vt_{[k]}} \cap \fV$ on any standard affine chart $\fV$ of $\tsR_{\vt_{[k]}}$.
 
 $(\star)$ Proposition \ref{meaning-of-var-wp/ell} introduces coordinate variables
 for  any standard affine chart $\fV$ of $\tsR^\circ_{(\wp_{(k\tau)}\fr_\mu\fs_h)}$
 as well as $\tsR^\circ_{\ell_k}$
 and provides explicit geometric meaning for every coordinate variable.
 
 Proposition  \ref{equas-wp/ell-kmuh}  combined with
 Proposition \ref{meaning-of-var-wp/ell} (9)
 provide explicit description and properties
  of the local defining equations of the scheme  
$\tsV_{\wp_{(k\tau)}\fr_\mu\fs_h}  \cap \fV$ 
or $\tsV_{\ell_k}  \cap \fV$ on any standard affine chart $\fV$ of 
 $\tsR^\circ_{(\wp_{(k\tau)}\fr_\mu\fs_h)}$ or $\tsR^\circ_{\ell_k}$.

To summarize the progress, we depict it in the diagram \eqref{theDiagram} below.

\begin{equation}\label{theDiagram}
 \xymatrix@C-=0.4cm{
  \tsR_{\ell} \ar[r] & \cdots  \ar[r] &  \tsR_{{\hbar}} \ar[r] &  \tsR_{{\hbar}'} \ar[r] &  \cdots  \ar[r] &   \sR_{\sF_{[j]}}  \ar[r] &  \sR_{\sF_{[j-1]}} \cdots \ar[r] &  \rU_\um \\
   \tsR^\circ_{\ell}\ar @{^{(}->} [u]  \ar[r] & \cdots  \ar[r] &  \tsR^\circ_{\hbar}\ar @{^{(}->} [u]  \ar[r] &  \tsR^\circ_{\hbar} \ar @{^{(}->} [u] \ar[r] &  \cdots  \ar[r] &   \sR_{\sF_{[j]}} \ar @{^{(}->} [u]_{=} \ar[r] &  \sR_{\sF_{[j-1]}} \ar @{^{(}->} [u]_{=} \cdots \ar[r] &  \rU_\um \ar @{^{(}->} [u]_{=}\\
    \tsV_{\ell} \ar @{^{(}->} [u]  \ar[r] & \cdots  \ar[r] &  \tsV_{{\hbar}}\ar @{^{(}->} [u]   \ar[r] &  \tsV_{{\hbar}'} \ar @{^{(}->} [u]  \ar[r] &  \cdots    \ar[r] &   \sV_{\sF_{[j]}} \ar @{^{(}->} [u]\ar[r] &  \sV_{\sF_{[j-1]}} \cdots \ar @{^{(}->} [u]  \ar[r] &  \rU_\um \cap \Gr^{ d,E}   \ar @{^{(}->} [u]  \\
   \tZ_{\ell, \Ga} \ar @{^{(}->} [u]  \ar[r] & \cdots  \ar[r] &  \tZ_{{\hbar},\Ga}\ar @{^{(}->} [u]   \ar[r] &  \tZ_{{\hbar}',\Ga} \ar @{^{(}->} [u]  \ar[r] &  \cdots    \ar[r] &   Z_{\sF_{[j]},\Ga} \ar @{^{(}->} [u]\ar[r] &  Z_{\sF_{[j-1])},\Ga} \cdots \ar @{^{(}->} [u]  \ar[r] &  Z_\Ga  \ar @{^{(}->} [u]  \\
    \tZ^\dagger_{\ell, \Ga} \ar @{^{(}->} [u]  \ar[r] & \cdots  \ar[r] &  \tZ^\dagger_{{\hbar},\Ga}\ar @{^{(}->} [u]   \ar[r] &  \tZ^\dagger_{{\hbar}',\Ga} \ar @{^{(}->} [u]  \ar[r] &  \cdots    \ar[r] &   Z^\dagger_{\sF_{[j]},\Ga} \ar @{^{(}->} [u]\ar[r] &  Z^\dagger_{\sF_{[j-1])},\Ga} \cdots \ar @{^{(}->} [u]  \ar[r] &  Z_\Ga, \ar[u]_{=}       }
\end{equation}
 where all vertical uparrows are closed embeddings. 

Thus far, we have obtained the first three rows of the  diagram:  

 $(\star)$ In the first row: each morphism  $\tsR_{{\hbar}} \to  \tsR_{{\hbar}'}$ is 
$\tsR_{\vt_{[k]}} \to \tsR_{\vt_{[k-1]}}$,  or 
$\tsR_{(\wp_{(k\tau)}\fr_\mu\fs_{h})} \to  \tsR_{(\wp_{(k\tau)}\fr_\mu\fs_{h-1})}$, or
  $\tsR_{\ell_{k}} \to  \tsR_{\wp_k}$, and each is a blowup;
every  $\sR_{\sF_{[j]}}  \to  \sR_{\sF_{[j-1]}}$  is a projection, a forgetful map.

 $(\star)$ In the third row: each morphism $\tsV_{{\hbar}} \to  \tsV_{{\hbar}'}$ is 
$\tsV_{\vt_{[k]}} \to \tsV_{\vt_{[k-1]}}$, or 
$\tsV_{(\wp_{(k\tau)}\fr_\mu\fs_{h})} \to  \tsV_{(\wp_{(k\tau)}\fr_\mu\fs_{h-1})}$, or
 $\tsV_{\ell_k} \to  \tsV_{\wp_k}$, 
and  this morphism as well as each $\sV_{\sF_{[j]}}  \to  \sV_{\sF_{[j-1]}}$  
  is surjective, projective, and birational.

$(\star)$ Further, a scheme in the second row is a smooth open subset of the scheme in first row
containing the one in the third row, correspondingly.

To explain the fourth and fifth rows of the  diagram,
we go back to the fixed chart $\rU_\um$. This is the affine space which comes equipped with
the coordinate variables $\var_{\rU_\um}:=\{x_\uu \mid \uu \in \II_{d,n} \- \{\um\}\}$.
Let $\Ga$ be any subset of $\var_{\rU_\um}$
 and let $Z_\Ga$ be the subscheme of $\rU_\um$ defined by the ideal generated by
 all the elements of $\Ga$ together with all the de-homogenized $\um$-primary $\pl$ relations $\bF$
 with $\bF \in \sfm$.   
This is a $\Ga$-scheme mentioned in the beginning of this introduction.
The precise relation between a given  matroid Schubert cell and its corresponding $\Ga$-scheme
is given in \eqref{ud=Ga}.

 Our goal is to resolve the $\Ga$-scheme $Z_\Ga$ when it is integral and singular.

 Let $\Ga$ be a subset $\var_{\rU_\um}$. Assume that 
 $Z_\Ga$ is integral. Then, starting from $Z_\Ga$, step by step,
  via induction within every of the sequential $\vt$-, $\wp$-, and
  $\ell$-blowup, we are able to construct
 the third and fourth rows in the  diagram  \eqref{theDiagram} such that 
 
$(\star)$ every closed subscheme in the fourth row, $Z_{\sF_{[j]},\Ga}$, respectively
$\tZ_{{\hbar}}$, {\it admits explicit local defining equations}
in any standard chart of a smooth open subset,
containing  $\sV_{\sF_{[j]},\Ga}$, respectively,
$\tsV_{{\hbar}}$, of the corresponding scheme in the first row;
 
$(\star)$ every closed subscheme in the fifth row $Z^\dagger_{\sF_{[j]},\Ga}$, respectively,
$\tZ^\dagger_{{\hbar}}$, is an {\it irreducible component} of its
corresponding subscheme $Z_{\sF_{[j]},\Ga}$, respectively, $\tZ_{{\hbar}}$, such that  the induced 
morphism $\hbox{$Z^\dagger_{\sF_{[j]},\Ga} \lra Z_\Ga$, respectively,
$\tZ^\dagger_{{\hbar}}\lra Z_\Ga$}$
is surjective, projective, and birational.
 
$(\star)$   the left-most   $\tZ_{\ell, \Ga} $ is smooth; so is $\tZ^\dagger_{\ell, \Ga}$,
now a connected component of $\tZ_{\ell, \Ga} $.

 The closed subscheme $Z_{\sF_{[j]},\Ga}$, called an $\sF$-transform of $Z_\Ga$,
  is constructed in Lemma \ref{wp-transform-sVk-Ga};
  the closed subscheme $Z_{\vt_{[j]},\Ga}$, called a $\vt$-transform of $Z_\Ga$,
  is constructed in Lemma \ref{vt-transform-k};
 the closed subscheme $\tZ_{(\wp_{(k\tau)}\fr_\mu\fs_{h}),\Ga}$, called a $\wp$-transform of $Z_\Ga$,
 is constructed in Lemma \ref{wp/ell-transform-ktauh};
  the closed subscheme $\tZ_{\ell_k,\Ga}$, called an $\ell$-transform of $Z_\Ga$,
  is also constructed in Lemma \ref{wp/ell-transform-ktauh}.

{\it
In this article, a scheme $X$ is smooth if it is a disjoint union of finitely many 
connected smooth schemes of
possibly various dimensions.}

Our main theorem on the Grassmannian is

\begin{thm}\label{main2:intro} 
{\rm (Theorems \ref{main-thm} and \ref{cor:main})} 
Let $\FF$ be either $\QQ$ or a finite field with $p$ elements where
$p$ is a prime number.
Let $\Ga$ be any subset of $\var_{\rU_\um}$.
Assume that $Z_\Ga$ is integral. 
Let $\tZ_{\ell,\Ga}$ be  the $\ell$-transform of $Z_\Ga$ in $\tsV_{\ell}$.
Then,   $\tZ_{\ell,\Ga}$ is smooth over $\FF$.
In particular, the induced morphism $\tZ^\dagger_{\ell,\Ga} \to Z_\Ga$ is a resolution over $\FF$,
provided that $Z_\Ga$ is singular.
\end{thm}


The proof of Theorem \ref{main2:intro} 
 (Theorems \ref{main-thm} and \ref{cor:main}) is based upon the explicit  description of
 the main binomials and linearized $\pl$ defining equations of $ \tZ_{\ell,\Ga}$
 (Corollary \ref{ell-transform-up})
 and detailed calculation and careful analysis on the Jacobian of these equations (\S \ref{main-statement}).


Theorem \ref{main:intro} is  obtained by applying Theorem \ref{main2:intro},
combining with Lafforgue's version of Mn\"ev's unversality theorem
(Theorems \ref{Mn-La} and \ref{Mn-La-Gr}), provided that $X$ is  defined over $\ZZ$.
For a singular affine or projective variety $X$ over a general perfect field $\kk$, 
we spread it out and deduce that $X/\kk$ admits a resolution as well.
The details are written in \S \ref{global-resolution}.

 
In general, consider any fixed singular integral scheme  $X$. 
By Theorem \ref{main:intro}, $X$ can be covered by finitely many affine open subsets such that
every of these affine open subsets of  $X$ admits a resolution. 
It remains to glue finitely many  such local resolutions to obtain  a global one. 
This is being pursued.


We learned that Hironaka posted a preprint on resolution of singularities 
in positive characteristics \cite{Hironaka17}.

In spite of the current article, the author is not in a position to survey the topics of
resolution of singularities, not even very briefly.
We refer to Koll\'ar's book \cite{Kollar} for an extensive list of references on resolution of singularities.
 There have been some recent progresses since the book \cite{Kollar}:
 risking inadvertently omitting some other's works, let us just mention a few recent ones
\cite{ATW},  \cite{McG}, and \cite{Temkin}.

\medskip
The approach presented in this paper was inspired by \cite{Hu2022}.

The author is grateful to the anonymous reviewers
of \cite{GrassSing} for their very helpful questions
and constructive suggestions, especially for pointing out
the insufficiency of an earlier version of \cite{GrassSing}. In particular, the author would not
have gone this far in a relatively short period of time without
their helpful feedbacks.

He thanks J\'anos Koll\'ar and Chenyang Xu for the suggestion to write 
a summary  section, \S \ref{tour}, to lead the reader a quick tour through the paper.

He thanks Laurent Lafforgue for several very kind suggestions and sharing a general question. 
He  very especially thanks Caucher Birkar,  also 
James McKernan and Ravi Vakil for the invitation to speak in workshop and
seminars, and for helpful correspondences. He thanks
Bingyi Chen for spotting a mistake in the proof of 
Theorem 10.5 of \cite{GrassSing}.

\bigskip 

\centerline {A List of Fixed Notations Used Throughout}
\medskip

\smallskip\noindent $[h]$: the set of all  integers from 1 to $h$, $\{1, 2 \cdots ,h \}.$

\noindent 
$\II_{d,n}$: the set of all sequences of integers $\{(1\le u_1 < \cdots < u_d\le n) \}.$

\noindent
$\PP(\wedge^d E)$: the projective space with $\pl$ coordinates 
$p_\ui, \ui \in \II_{d,n}$. 




\noindent
$I_\wp$: the ideal of $\kk[p_\ui]_{\ui \in \II_{d,n}}$ generated by all $\pl$ relations.

\noindent
$I_\whwp$: the ideal of $\Gr^{d,E}$ in $\PP(\wedge^d E)$.

\noindent
$\rU_\um$: the affine chart of $\PP(\wedge^d E)$ defined by $p_\um \ne 0$ for some fixed $\um \in \II_{d,n}$.

\noindent
$\sfm$: the set of $\um$-primary $\pl$ equations.

 \noindent
$\up:= {n \choose d} -1 - d(n-d)$: the cardinality of $\sfm$;

\noindent
$\fV$:  a standard affine chart of an ambient smooth scheme; 
 
 \noindent 
$\cB^\mn$:  the set of all main binomial relations; 

 \noindent 
$\cB^\res$:  the set of all  residual binomial relations; 

\noindent 
$\cB^\pq$:  the set of all  binomial relations of pre-quotient type;

 \noindent 
$\cB^\q$:  the set of all  binomial relations of quotient type; 

 \noindent 
$\cB$:   $\cB^\mn \sqcup \cB^\res \sqcup \cB^\q$;

\noindent
$L_{\sfm}$: the set of all linearized $\um$-primary $\pl$ equations.



\noindent $\Ga$: a subset of $\rU_\um$.

\noindent $A \- a$: $A\-\{a\}$ where $A$ is a finite set and $a \in A$.

\noindent $|A|$: the cardinality of a finite set $A$.

\noindent $\kk$: a fixed perfect field.

\section{A Quick Tour: the main idea and approach}\label{tour}

{\it This section may be skipped entirely if the reader  prefers to dive into
the main text immediately. However, carefully reading this section first is strongly recommended.}

\medskip
$\bullet$ {\it A detour to $\Gr^{3,E}$ via Mn\"ev's universality.} 

By Mn\"ev's universality, any singularity over $\ZZ$ appears in a matroid Schubert cell of 
the Grassmannian $\Gr^{3,E}$ of three dimensional linear subspaces in a vector space $E$,
up to smooth morphisms.

Consider the $\pl$ embedding $\Gr^{3,E} \subset \PP(\wedge^3 E)$ 
with $\pl$ coordinates $p_{ijk}$.
A  matroid Schubert cell of $\Gr^{3,E}$ is a nonempty intersection of codimension one Schubert cells
of  $\Gr^{3,E}$; it corresponds to  a matroid $\ud$ of rank 3 on the set $[n]$.
Any Schubert divisor is defined by $p_{ijk} =0$ for some $(ijk)$. Thus, 
a  matroid Schubert cell  $\Gr^{3,E}_\ud$ of the matroid $\ud$
 is an open subset of the closed subscheme $\overline{Z}_\Ga$ of
$\Gr^{3,E}$ defined by $\{p_{ijk}=0 \mid p_{ijk} \in \Ga\}$ for some subset $\Ga$ of all $\pl$ variables.
The  matroid Schubert cell  must lie in an affine chart $(p_\um \ne 0)$ for some $\um \in \II_{3,n}$.
Thus,  $\Gr^{3,E}_\ud$  is an open subset of a closed subscheme 
of  $\Gr^{3,E} \cap (p_\um \ne 0)$ of
the following form 
$$Z_\Ga=\{ p_{ijk} =0 \mid p_{ijk} \in \Ga\} \cap \Gr^{3,E} \cap (p_\um \ne 0).$$ 
This is a closed affine subscheme of the affine chart $(p_\um \ne 0)$.
We aim to resolve $Z_\Ga$, hence also the  matroid Schubert cell  $\Gr^{3,E}_\ud$, when both
are integral and singular.

\medskip
 $\bullet$ {\it Minimal set of $\pl$ relations for the chart $(p_\um \ne 0)$.} 
 
 Up to permutation, we may assume that $\um=(123)$ and the chart is
$$\rU_\um:=(p_{123} \ne 0).$$
We write the de-homogenized coordinates of $\rU_\um$ as
$$\{x_{abc} \mid (abc) \in \II_{3,n} \- \{(123) \}.$$

As a closed subscheme of the affine space $\rU_\um,$
$Z_\Ga$ is defined by
$$\{ x_{ijk} =0, \;\; \bF=0 \; \mid \; x_{ijk} \in \Ga\},$$
where $\bF$ rans over all de-homogenized $\pl$ relations.
We need to pin down some explicit $\pl$ relations to form 
a minimal set of generators of $\Gr^{3,E} \cap \rU_\um$.

They are of the following forms:
\begin{eqnarray}
\bF_{(123),1uv}=x_{1uv}-x_{12u}x_{13v} + x_{13u}x_{12v}, \label{rk0-1uv}\\
\bF_{(123),2uv}=x_{2uv}-x_{12u}x_{23v} + x_{23u}x_{12v}, \label{rk0-2uv} \\
\bF_{(123),3uv}=x_{3uv}-x_{13u}x_{23v} + x_{23u}x_{13v} ,\label{rk0-3uv}\\
\bF_{(123),abc}=x_{abc}-x_{12a}x_{3bc} + x_{13a}x_{2bc} -x_{23a}x_{1bc}, \label{rk1-abc}
\end{eqnarray}
where $u < v \in [n]\-\{1,2,3\}$ and $a<b<c \in [n]\-\{1,2,3\}$.
Here, $[n]=\{1,\cdots,n\}$.

In a nutshell, we have the set
\begin{eqnarray}\label{rk0-and-1}
\sfm=\{\bF_{(123),iuv},  \; 1\le i\le 3; \;\; \bF_{(123),abc}\}
\end{eqnarray}
Every relation of $\sfm$ is called $\um$-primary. Here, $\um=(123)$.



\medskip
$\bullet$ {\sl Nicely presented equations of $\Ga$-schemes and arbitrary singularities.}

Hence, as a closed subscheme of the affine space $\rU_\um,$
$Z_\Ga$ is defined by
\begin{eqnarray}\label{normal-form}
Z_\Ga=\{ x_\uu =0, \;\; \bF_{(123),iuv},  \; 1\le i\le 3, \;\; \bF_{(123),abc}  \mid x_\uu \in \Ga\},
\end{eqnarray}
for all $u < v \in [n]\-\{1,2,3\}$ and $a<b<c \in [n]\-\{1,2,3\}$.
The  matroid Schubert cell $Z_\Ga^\circ$ in $Z_\Ga$ is characterized by
$x_\uv \ne 0$ for any $x_\uv \notin \Ga$ (see Proposition \ref{to-Ga} and \eqref{ud=Ga}.)

Upon setting $x_\uu =0$ with $\uu \in \Ga$, we obtain 
the affine coordinate subspace $$\rU_{\um,\Ga} \subset \rU_\um$$ such that $Z_\Ga$,
 as a closed subscheme of
 the affine subspace $\rU_{\um,\Ga}$, is defined by 
 \begin{eqnarray}\label{reduced-normal-form}
\{\bF_{(123),iuv}|_\Ga, \; 1\le i\le 3;  \;\;  \bF_{(123),abc}|_\Ga\},
\end{eqnarray}
where $\bF|_\Ga$ denotes the restriction of $\bF$ to the affine subspace $\rU_{\um,\Ga}$.
These are in general truncated  $\pl$ equations, some of which  may be identically zero.

 {\it One may view \eqref{normal-form} as the normal form of singularities,
and \eqref{reduced-normal-form} as the reduced normal form of singularities.
These are the standardized equations of singularities over $\ZZ$, up to smooth morphisms.
In other words, singularities may arbitrary, but amazingly, their equations can be nicely
presented,  up to smooth morphisms.}

We  {\it do not} analyze singularities of $Z_\Ga$.

But, we make some remarks.
 By the normal form of singularities \eqref{normal-form},
the $\Ga$-scheme $Z_\Ga$ is cut out from the affine chart $\rU_\um(\Gr)$ 
of the Grassmannian $\Gr^{3,E}$
 by the hyperplanes $(x_\uu =0)$ 
for all $x_\uu \in \Ga$. Although $Z_\Ga$ as well as the  matroid Schubert cell $Z_\Ga^\circ$
(see the sentence below \eqref{normal-form})
are nicely described by
$\pl$ variables and $\pl$ relations,  the intersections of these
coordinate hyperplanes with the chart $\rU_\um(\Gr)$ of the Grassmannian 
$\Gr^{3,E}$ are arbitrary, according to
Mn\"ev's universality. We may view $\rU_\um$ (allowing 
$\Gr^{3,E}$ to vary) as a universe that contains arbitrary singularities.
Hence, intuitively, we need to birationally change the universe $``$along these intersections" 
 so that eventually in the new universe, $``$they" re-intersect properly. 




 To achieve this, it is more workable if we can
put all the singularities in a different universe $\sV_\um$, 
birationally modified from the chart $\rU_\um \cap \Gr^{3,E}$,
  so that in the {\it new model $\sV_\um$},  all the terms of 
the above $\pl$ relations can be separated. 
(Years had been passed, or wasted in a way, before we {\it returned} to this correct approach.)

\smallskip
$\bullet$ {\sl Separating the terms of $\pl$ relations.}

Motivated by  \cite{Hu2022}, we establish 
a local model  $\sV_\um$, birational  to the chart $\rU_\um \cap \Gr^{3,E}$,
such  that in a {\it specific} set of defining binomial equations of $\sV_\um$, {\it all the terms} of 
the above $\pl$ relations are separated. 

To explain, we introduce the projective space $\PP_F$ for each and
 every $\pl$ relation $F=\sum_{s \in S_F} \sgn (s) p_{\uu_s}p_{\uv_s}$ with
$[x_{(\uu_s,\uv_s)}]_{s \in S_F}$ as its homogeneous coordinates. 

We then let $\sV_\um$ be the closure of the graph of the rational map $\bar\Theta_{[\up],\Gr}$ of
 \eqref{this-theta-intro} in the case of $\Gr^{3,E}$.
 (This is  motivated by an analogous construction in \cite{Hu2022}.)
By calculating the multi-homogeneous kernel of the homomorphism
\begin{equation}\label{vi-hom} \bar\vi: \kk[(x_\uw);(x_{(\uu,\uv)})] \lra \kk[x_\uw]
\end{equation}
$$x_{(\uu,\uv)} \to x_\uu x_\uv,$$
we determine a set of defining relations of $\sV_\um$ as a closed subscheme
of the smooth ambient scheme $$\sR_\sF:=\rU_\um \times \prod_{\bF \in \sfm} \PP_F.$$ 
These defining relations, among many others, 
include the following binomials
\begin{eqnarray}\label{mainB-tour}
x_{1uv}x_{(12u,13v)} - x_{12u}x_{13v} x_{(123,1uv)}, \; x_{1uv}x_{(13u,12v)}- x_{13u}x_{12v}x_{(123,1uv)}, \\
x_{2uv}x_{(12u,23v)} -x_{12u}x_{23v} x_{(123,2uv)}, \;  x_{2uv}x_{(23u,12v)}-x_{23u}x_{12v} x_{(123,2uv)}, \nonumber \\
x_{3uv}x_{(13u,23v)} -x_{13u}x_{23v}x_{(123,3uv)}, \;  x_{3uv}x_{(23u,13v)} -x_{23u}x_{12v}x_{(123,3uv)},\nonumber\\
 x_{abc}x_{(12a,3bc)}-x_{12a}x_{3bc} x_{(123,abc)},\;
 x_{abc}x_{(13a,2bc)}-x_{13a}x_{2bc} x_{(123,abc)},\nonumber \\
x_{abc}x_{(23a,1bc)} -x_{23a}x_{1bc}x_{(123,abc)}. \nonumber
\end{eqnarray}
We see that  the terms of all the $\um$-primary $\pl$ relations 
of \eqref{rk0-and-1} are separated into the two terms of the above binomials.

 To distinguish, we call $x_\uu$ (e.g., $x_{12u}$)
a $\vp$-variable and $X_\uu=(x_\uu=0)$  a $\vp$-divisor;
 we call $x_{(\uu,\uv)}$ (e.g., $x_{(12u,13v)}$) a $\vr$-variable and
$X_{(\uu,\uv)}=(x_{(\uu,\uv)}=0)$ a $\vr$-divisor.

The defining relations 
also include the linearized $\pl$ relations as in \eqref{eq2-intro}:
$$L_F=\sum_{s \in S_F} \sgn (s) x_{(\uu_s,\uv_s)}, \; \; \forall \; \bF \in \sfm.$$
The set of all  linearized $\pl$ relations is denoted by $L_{\sfm}$.

There are many other {\it extra} defining relations.

 All the $\Ga$-schemes $Z_\Ga$ admit  birational transforms in the singular model
$\sV_\um$. We still {\it do not} analyze the singularities of these transforms. But, we 
make a quick observation: 
when all the terms of  some of the binomials in \eqref{mainB-tour}
vanish at a point of the transform of a $\Ga$-scheme, then a singularity is likely to occur.



Thus, {\it as the first steps,} we would like to $``$remove$"$ all the zero factors from all the terms of 
the binomials in \eqref{mainB-tour}. This also amounts to re-positioning the coordinate hyperplanes 
$(x_\uu=0)$ through blowups 
so that they eventually intersect properly with
the proper transform of the chart $\rU_\um(\Gr) = \rU_\um \cap \Gr^{3,E}$ of the Grassmannian,
for all $x_\uu \in \Ga$.

As it turns out, through years of $``$trial and error$"$, 
$``$removing$"$ all the zero factors from all the  binomial relations of \eqref{mainB-tour}
successfully leads us to the correct path toward our ultimate purpose. 

{\it The geometric intuition behind the above sufficiency is as follows.
 The equations of \eqref{mainB-tour} alone together with $L_{\sfm}$ only define
a reducible closed scheme, in general. The roles of other extra relations (to be discussed soon)
are to  pin down
its main component $\sV_\um$. 
As the process of $``$removing$"$ zero factors goes, a process of some specific blowups, 
all the boundary components are eventually blown out of existence, making  
the proper transforms of \eqref{mainB-tour} together with the linearized
$\pl$ relations  generate the ideal of the final blowup scheme $\tsV_\ell$ of $\sV_\um$,
on all charts.}

We thus call the binomial equations of  \eqref{mainB-tour} the {\it main} binomials.
The set of  main binomials is denoted $\cB^\mn$. 
The set $\cB^\mn$ is equipped with a carefully chosen total ordering 
(see \eqref{indexing-Bmn}).

The defining relations of $\sV_\um$ in $\sR_\sF$ also include many other binomials: we classify them
as {\it residual} binomials  (see Definition \ref{defn:main-res}), 
 binomials {\it of pre-quotient type} (see Definition \ref{defn:pre-q}).
The set of  residual binomials is denoted $\cB^\res$;
the set of  binomials of pre-quotient type is denoted $\cB^\pq$;
Both are finite sets.

Together, the equations in the following sets
$$\cB^\mn, \; \cB^\res, \; \cB^\pq, \; L_{\sfm}$$ 
define the scheme $\sV_\um$ in the smooth ambient scheme $\sR_\sF$.
See Corollary \ref{eq-tA-for-sV}. 

When we focus on an arbitrarily fixed chart $\fV$ of $\sR_\sF$,
binomials  of pre-quotient type of $\cB^\pq$
can be further reduced to binomials {\it of quotient type} whose set
is denoted by $\cB^q_\fV$. See Definition \ref{defn:q} and Proposition \ref{equas-p-k=0}. 
(For the  reason to use the term {\it $``$of quotient type$"$}, see \cite{Hu2022}.)

As mentioned in the introduction, for the purpose of inductive proofs,
we also need the rational map $\bar\Theta_{[k],\Gr}$ 
of  \eqref{this-theta[k]-intro}, 
and we let $\sV_{\sF_{[k]}}$ be the closure of the rational map of $\bar\Theta_{[k],\Gr}$, 
for all $k \in [\up]$.
In this notation, $\sV_\um=\sV_{\sF_{[\up]}}$. We let 
$$\sR_{\sF_{[k]}}:=\rU \times \prod_{i \in [k]} \PP_{F_i}.$$ 
This is a smooth scheme and contains $\sV_{\sF_{[k]}}$ as a closed subscheme.
Further, we have the natural forgetful map
\begin{equation} \label{forgetful-guide}
\sR_{\sF_{[k]}} \lra \sR_{\sF_{[k-1]}}. 
\end{equation}

\medskip
$\bullet$ {\sl  The process of $``$removing$"$ zero factors of main binomials.}

 To remove zero factors of main binomials, we either work with the set of 
 binomials of $\cB^\mn_F$ all together  in the case of  $\vt$-blowups, or work on each 
main binomial  of $\cB^\mn_F$ individually in the case of  $\wp$-blowups. 
Upon completing $\wp$-blowups for the block of relations of $\cB^\mn_F$,
we immediately perform the $\ell$-blowup with respect to the linearized $\pl$ relation $L_F$.

To this end, we need to provide a carefully chosen total order on the set $\sfm$.

We let
$\{\bF_{(123),iuv},  \; 1\le i\le 3\}$ go first, then followed by $\{\bF_{(123),abc}\}$.
Within $\{\bF_{(123),iuv},  \; 1\le i\le 3\}$, we say $\bF_{(123),iuv} < \bF_{(123),ju'v'}$
if $(uv)<(u'v')$  lexicographically, or when  $(uv)=(u'v')$ , $i<j$.
Within $\{\bF_{(123),abc}\}$,  we say $\bF_{(123),abc} < \bF_{(123),a'b'c'}$
if $(abc)<(a'b'c')$  lexicographically. This ordering is compatible with that of
$\cB^\mn$. We also provide an ordering on the set of all $\pl$ variables, compatible
with the above orderings.



The purpose of $``$removing$"$ zero factors is achieved through sequential blowups based
upon factors of main binomials and their proper transforms.

\smallskip
{\it $\star$ On $\vt$-blowups.}  

 From the main binomial equations of \eqref{mainB-tour}, we select the following closed centers
 $$\cZ_\vt: (x_{iuv}=0) \cap (x_{(123, iuv)}=0),  i \in [3]; \;
  (x_{abc}=0) \cap (x_{(123, abc)}=0), a \ne b \ne c \in [n] \- [3].$$
 
  We order the sets $\{(uv)\}$ and  $\{(abc)\}$ lexicographically, respectively;
  we  order  $\{(iuv)\}$, written as  $\{(i, (uv)) \mid i \in [3]\}$, reverse-lexicographically.
  We then let $\{(iuv)\}$ go before  $\{(abc)\}$.
  This way,  the set $\cZ_\vt$ is equipped with a total order induced from the above-mentioned orders on the indexes.

We then blow up $\sR_\sF$ along (the proper transforms of) the centers in $\cZ_\vt$, in the above order.
This gives rise to the sequence \eqref{vt-sequence-intro} in the introduction
$$\tsR_{\vt}:=\tsR_{\vt_{[\up]}}  \lra \cdots \lra \tsR_{\vt_{[k]}} 
\lra \tsR_{\vt_{[k-1]}} \lra \cdots \lra \tsR_{\vt_{[0]}}.$$
Each arrow in this sequence is a smooth blowup.

For any $k \in [\up]$, we let $\tsV_{\vt_{[k]}} \subset  \tsR_{\vt_{[k]}}$ 
be the proper transform of $\sV_\um$ in $\tsR_{\vt_{[k]}}$. We then set
$\tsV_{\vt}=\tsV_{\vt_{[\up]}}$ and $ \tsR_{\vt}= \tsR_{\vt_{[\up]}}$.


Besides removing the zero factors as displayed  in  the centers of $\cZ_\vt$,
 $\vt$-blowups also make the proper transforms of the residual binomial equations become
 dependent on the proper transforms of the main binomial equations on any standard chart.
 Thus, upon completing $\vt$-blowups, we can discard all the residual binomials $\cB^\res$ from
 consideration.
 In addition, it also leads to 
 the conclusion $\tsV_{\vt} \cap X_{\vt, (\um, \uu_k)} = \emptyset$ for all $k \in [\up]$
 where $X_{\vt, (\um, \uu_k)}$ is the proper transform of the $\vr$-divisor 
$X_{(\um, \uu_k)}=(x_{(\um, \uu_k)}=0)$, put it differently, 
the factor $x_{(\um, \uu_k)}$, possibly zero somewhere before the $\vt$-blowup,
now that $``$ zero factor $"$ is  removed upon completing $\vt$-blowups.

\smallskip
{\it $\star$ On $\wp$-blowups.}

Here, we continue the process of $``$removing$"$ zero factors
of the proper transforms of the main binomials. From now on, we focus on each main binomial individually,
starting from the first one.

The first main binomial equation of \eqref{mainB-tour} is 
$$B_{145, 1}:  x_{(124,135)} x_{145} - x_{(123,145)} x_{124}x_{135} .$$
The proper transforms of all the variables of $B_{145}$ may assume zero value on $\tsV_{\vt}$ 
somewhere.
For each and every term of $B_{145, 1}$, we pick a $``$zero$"$ 
 factor to form a pair. For example, $(x_{145}, x_{124})$ is such a pair
of $B_{145, 1}$.
 Such a pair is called a $\wp$-set with respect to $B_{145, 1}$. 
 The common vanishing locus of the variables in a $\wp$-set gives rise to  a $\wp$-center.
Before we can blow up these $\wp$-centers, we need to order them.
The order is somewhat subtle. But, the general rule is that we let $\wp$-sets having $\vr$-variables go last
and those having $\vp$-variable go as the second last. In other words, we first declare $\vr$-variables
are the largest, $\vp$-variables are the second largest, and then compare the $\wp$-sets as pairs
lexicographically. Then, we can  blow up $\tsR_\vt$ along (the proper transforms) of 
 these $\wp$-centers, starting from the smallest one.

We then move on to the next main binomial equation
$$B_{145, 2}:  x_{(125,134)} x_{145} - x_{(123,145)} x_{125}x_{134} .$$
Notice here that the variable $x_{145}$ may become an exceptional variable
or acquires one due to the previous $\wp$-blowups. Hence, $B_{145, 2}$ should have more
$\wp$-sets.
We declare these exceptional variables to be the smallest ones, and wthin them,
we order them by reversing the order of occurrence.
We can then select pairs of variables, one from each term, define $\wp$-centers, 
make an order on them,
and repeat the above.

This way, we complete our $\wp$-blowups with respect to the block of relations of
$\cB^\mn_{F_1}$ and obtain $\tsR_{\wp_1}$. 

Now, consider the linearized $\pl$ relation
$$L_{F_1}=  x_{(123,145)}- x_{(124,135)}+x_{(125,134)}.$$
We can blow up $\tsR_{\wp_1}$ along the proper transform of
the intersection $$(L_{F_1}=0)\cap (x_{(123,145)}=0)$$
to obtain $$\tsR_{\ell_1} \lra \tsR_{\wp_1}.$$
This complete all the desired blowups for the block of equations
$$\{B_{145, 1}, \; B_{145, 2}, \; L_{F_1}\}.$$

We then move on to the next bock of relations
$$\{B_{245, 1}, \; B_{245, 2}, \; L_{F_2}\},$$
and repeat all the above,
and  then,   the next block 
 $\{B_{345,1}, B_{345,2}, L_{F_3}\}$, repeat all the above, and so on. 
 
 This gives rise to the sequential blowups \eqref{wp-sequence-intro}
 and  \eqref{ell-sequence-intro} in the introduction
 $$\tsR_{\ell_k} \lra \tsR_{\wp_k}  \to \cdots \to
\tsR_{(\wp_{(k\tau)}\fr_\mu\fs_{h})} \to \tsR_{(\wp_{(k\tau)}\fr_\mu\fs_{h-1})} \to \cdots \to \tsR_{\ell_{k-1}},$$
coming with the induced blowups
$$\tsV_{\ell_k} \lra \tsV_{\wp_k}  \to \cdots \to
\tsV_{(\wp_{(k\tau)}\fr_\mu\fs_{h})} \to \tsV_{(\wp_{(k\tau)}\fr_\mu\fs_{h-1})} \to \cdots \to \tsV_{\ell_{k-1}}.$$
Each scheme in first sequence above has a smooth open subset
$\tsR^\circ_{\ell_k}$ or $\tsR^\circ_{\wp_k}$ or  $\tsR^\circ_{(\wp_{(k\tau)}\fr_\mu\fs_{h})}$
containing $\tsV_{\ell_k}$ or  $\tsV_{\wp_k}$ or
$\tsV_{(\wp_{(k\tau)}\fr_\mu\fs_{h})}$, respectively.

An intermediate blowup scheme in the above 
is denoted by $\tsR_{(\wp_{(k\tau)}\fr_\mu\fs_h)}$. Here $(k\tau)$ is the index of a main binomial.
As the process of $\wp$-blowups goes on, more and more exceptional 
parameters may be acquired and appear in the proper transform of the later main binomial $B_{(k\tau)}$, 
resulting more pairs of zero factors, hence more corresponding $\wp$-sets
and $\wp$-centers. 
The existence of the index $\fr_\mu$, called {\it round $\mu$}, 
is due to the need to deal with the situation when 
an  exceptional parameter with exponent greater than one is accumulated
in the proper transform of the main binomial $B_{(k\tau)}$ (such a situation does not occur for the
first few main binomials).
The index $\fs_h$, called {\it step $h$}, simply indicates the corresponding step of the blowup.

When the process of $\wp$-blowups terminates, all the main binomials terminate, that is, 
all the variables in the main binomials are invertible along $\tsV_{\wp_k}$. 

Hence, we have achieved
our goal to $``$remove$"$ zero factors of main binomials.

\medskip
$\bullet$ {\sl  $``$Removing$"$ zero factors of the leading terms
of linearized $\pl$ relations.}
\smallskip

 {\it $\star$ On $\ell$-blowups.}
  
  Here, we make some more comments on $\ell$-blowups.
  
For any of the $\pl$ relations of \eqref {rk0-and-1}, either 
$\bF_{(123),iuv}$ for some  $1\le i\le 3$ or  $\bF_{(123),abc}$, we express its linearized $\pl$ relation
as 
$$L_{F_k}: \;\; \sgn(s_{F_k}) x_{(\um, \uu_{F_k})}+\sum_{s \in S_{F_k} \- s_{F_k}} \sgn (s) x_{(\uu_s,\uv_s)}$$
where $\sgn(s_{F_k}) x_{(\um, \uu_{F_k})}$ is the leading term. 
It comes equipped with a divisor $$D_{L_{F_k}}=(L_{F_k} =0)$$ in $\sR_\sF.$ 
We let $D_{\wp_k,L_{F_k}}$ be the proper transform of $D_{L_{F_k}}$ in $\tsR_{\wp_k}$.

 After the process of $\vt$-blowups, the leading $x_{(\um, \uu_F)}$, can become 
 an exceptional variable of the exceptional divisor $E_{\vt_{[k]}}$ created by the corresponding
 $\vt$-blowup with respect to $F_k$ mentioned earlier.  We let $E_{\wp_k, \vt_k}$
 be the proper transform of $E_{\vt_{[k]}}$ in $\tsR_{\wp_k}$. 
 We can then let $$\tsR_{\ell_k} \lra \tsR_{\wp_k}$$
 be the blowup of $\tsR_{\wp_k}$ along the intersection $D_{\wp_k, L_{F_k}}\cap E_{\wp_k, \vt_k}$.
 Then, the blowup
 $\tsR_{\ell_k} \lra \tsR_{\wp_k}$ will remove that $``$ zero factor  $"$ and bring up a 
 variable $y_{(\um, \uu_{F_k})}$ invertible along $\tsV_{\ell_k}$.
 Geometrically, this process separates the two divisors $D_{\wp_k, L_{F_k}}$
 and $E_{\wp_k, \vt_k}$.

We point out here the $\ell$-blowup with respect to $F$ has to immediately follow the $\wp$-blowups
 with respect to $F$; the order of $\wp$-blowups  with respect to a fixed $\pl$ relation $F$,
 as already mentioned,  may be subtle and are carefully chosen. 
 
{\it  In fact, the birational model $\sV_\um$ of the chart $\rU_\um(\Gr)=\rU_\um \cap \Gr^{3,E}$
 of the Grassmannian, also as a blowup of $\rU_\um(\Gr)$, has to be constructed first, as 
 experience has shown. That is to say,  the method of our approach is highly sensitive to
 the order of all these  blowups.}


\smallskip

In the above, the constructions of $\wp$-, and $\ell$-blowups 
are discussed in terms of coordinate variables
 of the proper transforms of the main binomials or lineaized $\pl$ relations on local charts.
 In the main text, the constructions of all these blowups, like $\vt$-blowups, are done globally via induction.


 \smallskip
 From the previous discussions, one sees that the process of $\wp$-blowups
 is highly inefficient. This is not a surprise as we treat {\it all singularities} 
 all together, {\it once and for all.}
To provide a concrete example for the whole process,  $\Gr(2,n)$ would miss some main points;
$\Gr(3,6)$ would be too long to include, and also, perhaps not too helpful as far as 
showing (a resolution of) a singularity is concerned.

\medskip
$\bullet$ {\sl  $\Ga$-schemes and their $\vt$-, $\wp$-, $\ell$-transforms.}

Fix any integral $\Ga$-scheme $Z_\Ga$, 
considered as a closed subscheme of $\rU_\um \cap \Gr^{3,E}$.
Our goal is to resolve $Z_\Ga$ when it is singular.

 As in the introduction,
  we have the  instrumental diagram \eqref{theDiagram2}.
\begin{equation}\label{theDiagram2}
 \xymatrix@C-=0.4cm{
  \tsR_{\ell} \ar[r] & \cdots  \ar[r] &  \tsR_{\hbar} \ar[r] &  \tsR_{\hbar} \ar[r] &  \cdots  \ar[r] &   \sR_{\sF_{[j]}}  \ar[r] &  \sR_{\sF_{[j-1]}} \cdots \ar[r] &  \rU_\um \\
  \tsR^\circ_{\ell}\ar @{^{(}->} [u]  \ar[r] & \cdots  \ar[r] &  \tsR^\circ_{\hbar}\ar @{^{(}->} [u]  \ar[r] &  \tsR^\circ_{\hbar} \ar @{^{(}->} [u] \ar[r] &  \cdots  \ar[r] &   \sR_{\sF_{[j]}} \ar @{^{(}->} [u]_{=} \ar[r] &  \sR_{\sF_{[j-1]}} \ar @{^{(}->} [u]_{=} \cdots \ar[r] &  \rU_\um \ar @{^{(}->} [u]_{=}\\
    \tsV_{\ell} \ar @{^{(}->} [u]  \ar[r] & \cdots  \ar[r] &  \tsV_{\hbar}\ar @{^{(}->} [u]   \ar[r] &  \tsV_{\hbar'} \ar @{^{(}->} [u]  \ar[r] &  \cdots    \ar[r] &   \sV_{\sF_{[j]}} \ar @{^{(}->} [u]\ar[r] &  \sV_{\sF_{[j-1]}} \cdots \ar @{^{(}->} [u]  \ar[r] &  \rU_\um \cap \Gr^{3,E}   \ar @{^{(}->} [u]  \\
   \tZ_{\ell, \Ga} \ar @{^{(}->} [u]  \ar[r] & \cdots  \ar[r] &  \tZ_{\hbar,\Ga}\ar @{^{(}->} [u]   \ar[r] &  \tZ_{\hbar',\Ga} \ar @{^{(}->} [u]  \ar[r] &  \cdots    \ar[r] &   Z_{\sF_{[j]},\Ga} \ar @{^{(}->} [u]\ar[r] &  Z_{\sF_{[j-1])},\Ga} \cdots \ar @{^{(}->} [u]  \ar[r] &  Z_\Ga  \ar @{^{(}->} [u]  \\
    \tZ^\dagger_{\ell, \Ga} \ar @{^{(}->} [u]  \ar[r] & \cdots  \ar[r] &  \tZ^\dagger_{\hbar,\Ga}\ar @{^{(}->} [u]   \ar[r] &  \tZ^\dagger_{\hbar',\Ga} \ar @{^{(}->} [u]  \ar[r] &  \cdots    \ar[r] &   Z^\dagger_{\sF_{[j]},\Ga} \ar @{^{(}->} [u]\ar[r] &  Z^\dagger_{\sF_{[j-1])},\Ga} \cdots \ar @{^{(}->} [u]  \ar[r] &  Z_\Ga.  \ar[u]_{=}       }
\end{equation}

The first three rows follow from the above discussion;
we  only need to explain the  fourth and fifth rows.

  Here,  when  $Z_{\sF_{[j-1])}}$  (resp.  $\tZ_{{\hbar}',\Ga}$) is not contained in the corresponding
  blowup center, $Z_{\sF_{[j])}}$  (resp.  $\tZ_{{\hbar},\Ga}$) is, roughly, obtained
  from  the proper transform
  of $Z_{\sF_{[j-1])}}$  (resp.  $\tZ_{{\hbar}',\Ga}$). 
  When  $Z_{\sF_{[j-1])}}$  (resp.  $\tZ_{{\hbar}',\Ga}$) is contained in the corresponding
  blowup center, then $Z_{\sF_{[j])}}$  (resp.  $\tZ_{{\hbar},\Ga}$) is, roughly,
  obtained from  a canonical rational slice
  of  the total  transform of $Z_{\sF_{[j-1])}}$  (resp.  $\tZ_{{\hbar}',\Ga}$) under the morphism 
   $ \sV_{\sF_{[j]}} \to  \sV_{\sF_{[j-1]}}$ (resp. $\tsV_{\hbar} \to \tsV_{{\hbar}'}$) in the second row.
   Moreover, every $Z_{\sF_{[j])}}$  (resp.  $\tZ_{{\hbar},\Ga}$) admits explicit defining equations 
   over any standard affine chart of the corresponding smooth open subset of the scheme in the first row.
Furthermore, in every case, $Z_{\sF_{[j])}}$  (resp.  $\tZ_{{\hbar},\Ga}$)  contains an irreducible
component  $Z^\dagger_{\sF_{[j]},\Ga}$  (resp. 
   $\tZ^\dagger_{{\hbar},\Ga}$) such that it maps onto $Z_\Ga$ 
   projectively and birationally.

$\bullet$ {\sl Smoothness by Jacobian of 
main binomials  and linearized $\pl$ relations.}

We are now ready to explain the smoothness of 
 $\tZ_{\ell, \Ga}$ when $Z_\Ga$ is integral.
  We first investigate the smoothness of $\tsV_\ell$ which is a special case 
 of  $\tZ_{\ell, \Ga}$ when 
 $\Ga=\emptyset$.

The question is local. So we focus on an affine chart of $\fV$ of $\tsR_\ell$.
Corollary \ref{ell-transform-up} provides  
the local defining equations for $\tZ_{\ell, \Ga}$.

As envisioned, we confirm that the scheme $\tsV_\ell$
 is smooth on the chart $\fV$ by some explicit calculations
and careful analysis on the Jacobian of
{\it the main binomial relations of $\cB^\mn_\fV$ and linearized $\pl$ relations of $L_{\fV, \sfm}$.}
This implies that on the chart $\fV$, 
 the main binomial relations of $\cB^\mn_\fV$ and the linearized $\pl$ relations 
 of $L_{\fV,\sfm}$
together generate the ideal of $\tsV_\ell \cap \fV$. Thus,  as a consequence, 
the binomials of quotient type  $\cB^q_\fV$ can be discarded from consideration, as well.

Then, the similar calculations and analysis on the Jacobian of the induced main binomial relations of
$\cB^\mn_\fV$ and the induced linearized $\pl$ relations of $L_{\fV,\sfm}$ for 
$\tZ_{\ell,\Ga}$ implies that $\tZ_{\ell,\Ga}$  is smooth as well, on all charts $\fV$.
In particular, $\tZ^\dagger_{\ell,\Ga}$, now a connected component of $\tZ_{\ell,\Ga}$,  is  smooth, too.

This implies that $\tZ^\dagger_{\ell,\Ga} \lra Z_\Ga$ is a resolution, if $Z_\Ga$ is singular.

The above are done in \S \ref{main-statement}.

\smallskip
$\bullet$ {\sl Resolution via Mn\"ev universality.} 

Upon reviewing Lafforgue's version of Mn\"ev universality,
we can apply the resolution $\tZ^\dagger_{\ell,\Ga} \lra Z_\Ga$ to obtain 
 a resolution for
any singular affine or projective variety $X$ defined over a prime field. 
For a singular affine or projective algebraic variety $X$ over a general perfect field $\kk$, 
we spread it out and deduce that $X/\kk$ admits a resolution as well.
The details are expanded  in \S \ref{global-resolution}.

\medskip
{ \it  Let $p$ be an arbitrarily fixed prime number.
 Let $\mathbb F$ be either $\QQ$ or a finite field with $p$ elements.
From Section \ref{localization} to Section \ref{main-statement}, 
every scheme considered is defined over $\ZZ$, consequently,
 is defined over $\mathbb F$, and is considered as a scheme over 
 the perfect field $\mathbb F$.}

\section{Primary $\pl$ Relations and De-homogenized  $\pl$-Ideal}\label{localization}

{\it The purpose of this section is to describe a minimal set of $\pl$ relations so that they
generate the $\pl$
ideal  for a given chart.  
The approach of this article depends on these explicit relations. 
The entire section is elementary.} 

Fix a pair of positive integers $n>1$ and $1\le d <n$.
 In this section, we focus on Grassmannians $\Gr^{d,E}$ where
 $E=E_1 \oplus \cdots \oplus E_n$ is as introduced in the introduction.
 
 For application to resolution of singularity, it suffices to consider $\Gr^{3,E}$. However, 
 we choose to work on the general case of $\Gr^{d,E}$ 
 for the following two reasons.  (1) Working on $\Gr^{3,E}$ instead of  $\Gr^{d,E}$
 saves us little space or time: if we focus on 
 \eqref{rk0-and-1} but not the general form $\sum_{s \in S_F} x_{\uu_s}x_{\uv_s}$
 in the construction of $\vt$-, $\wp$-, and $\ell$-blowups,
then  the  proofs of some key propositions would have to
 be somewhat case by case,  less conceptual, and hence may be lengthier.
 However, it is  always good to frequently use the equations 
 of \eqref{rk0-and-1} and \eqref{mainB-tour}  as examples to help to
  understand the notations and the process.
  We caution here that replying only on $\pl$ equations 
  of the form $\bF_{(123),iuv},  \; 1\le i\le 3$ from \eqref{rk0-and-1} 
 (they correspond to $\pl$ equations of $\Gr^{2,E}$)
  might miss some crucial points. 
 (2) As a convenient benefit, the results obtained and proofs provided
 for $\Gr^{d,E}$ here  can be directly cited in the future. 

 All the results of this section are elementary and some might have already been known. 
 Nonetheless,  the development in the current section is  instrumental for our approach.
 Hence, some good details are necessary.

We make a convention. Let $A$ be a finite set and $a \in A$. Then, we write
$$A \- a := A\-\{a\}.$$
Also, we use $|A|$ to denote the cardinality of the set A.

\subsection{$\pl$ relations} $\ $

Fix a pair of positive integers $(n,d)$ with $n >1$ and $1\le d <n$. 
We denote the set $\{1,\cdots, n\}$ by $[n]$.
We let $\II_{d,n}$ be the set of all sequences of distinct integers $\{1\le u_1 < \cdots < u_d\le n \}$.
An element of $\II_{d,n}$ is frequently written as $\uu=(u_1\cdots u_d)$.
We also regard an element of $\II_{d,n}$ as a subset of $d$ distinct integers in $[n]$.
For instance, for any $\uu, \um \in \II_{d,n}$, $\uu \- \um$ takes its set-theoretic meaning.
Also, $u \in [n] \- \uu$  if and only if $u \ne u_i$ for all $1\le i\le d$.

As in the introduction, suppose we have a set of vector spaces, $E_1, \cdots, E_n$ such that 
every $E_\alpha$, $1\le \alpha\le n$,  is of dimension 1 over $\kk$ (or, a free module of rank 1 over $\ZZ$), 
and, we let 
$$E:=E_1 \oplus \ldots \oplus E_n.$$ 

For any fixed  integer $1\le d <n$, the Grassmannian, defined by
$$\Gr^{d,E}=\{ F \hookrightarrow E \mid \dim F=d\}, $$
is a projective variety defined over $\ZZ$.

We have the canonical decomposition
$$\wedge^d E=\bigoplus_{\ui =(i_1,\cdots, i_d)\in \II_{d,n}} E_{i_1}\otimes \cdots \otimes E_{i_d}.$$
This gives rise to the $\pl$ embedding of the Grassmannian:
$$\Gr^{d,E} \hookrightarrow \PP(\wedge^d E)=\{(p_\ui)_{\ui \in \II_{d,n}} \in \GG_m 
\backslash (\wedge^d E \- \{0\} )\},$$
$$F \lra [\wedge^d F],$$
 where $\GGm$ is the multiplicative group.

The group $(\GGm)^n/\GG_m$, where $\GGm$
 is embedded in $(\GGm)^n$ as the diagonal, acts on $\PP(\wedge^d E)$ by
 $${\bf t} \cdot p_{\ui} = t_{i_1} \cdots t_{i_d} p_{\ui}$$
where ${\bf t} = (t_1, \cdots, t_n)$ is (a representative of) an element of $(\GGm)^n/\GG_m$
and $\ui=(i_1, \cdots, i_d)$. This action leaves $\Gr^{d,E}$ invariant.
The $(\GGm)^n/\GG_m$-action on $\Gr^{d,E}$ will only be used in \S \ref{global-resolution}.

The Grassmannian $\Gr^{d,E}$ as a closed subscheme of $\PP(\wedge^d E)$ is
defined by a set of specific quadratic relations, called $\pl$ relations. We describe them below.

For narrative convenience, we will assume that
$p_{u_1\cdots u_d}$ is defined for any sequence of
$d$ distinct integers between 1 and $n$,  not necessarily listed in 
the sequential order of natural numbers,
subject to the relation
\begin{equation}\label{signConvention}
p_{\si(u_1)\cdots \si (u_d)}=\sgn(\si) p_{u_1\cdots u_d}
\end{equation}
for any permutation $\si$ on the set $[n]$, 
where $\sgn(\si)$ denotes the sign of the permutation.
Furthermore, also for convenience, we set 
\begin{equation}\label{zeroConvention}
 p_{\uu} := 0,
\end{equation}
for any $\uu=(u_1\cdots u_d)$ of a set of $d$  integers  in $[n]$ if
$u_i=u_j$ for some $1\le i \ne j\le d$.

Now, for any pair $(\uh, \uk) \in \II_{d-1,n} \times \II_{d+1,n}$ with
$$\uh=\{h_1, \cdots, h_{d-1}\}  \;\;
 \hbox{and} \;\; \uk=\{k_1, \cdots, k_{d+1}\} ,$$
we have the Pl\"ucker relation:
\begin{equation} \label{pluckerEq}
F_{\uh,\uk}= \sum_{\lambda=1}^{d+1} (-1)^{\lambda-1} p_{h_1\cdots h_{d-1} k_\lambda } p_{k_1 \cdots  \overline{k_\lambda} \cdots k_{d+1}},
\end{equation}
where  $``\overline{k_\lambda}"$ means that $k_\lambda$ is deleted from the list.

{\it To make the presentation concise,
we frequently succinctly express a  general $\pl$ relation as
\begin{equation}\label{succinct-pl}
F= \sum_{s \in S_F} \sgn(s) p_{\uu_s}p_{\uv_s},
\end{equation}
 for some index set $S_F$, with $\uu_s, \uv_s \in \II_{d,n}$, where
$ \sgn(s) $ is the $\pm$ sign associated with the quadratic  monomial term $p_{\uu_s}p_{\uv_s}$.
We note here that $\sgn(s)$ depends on 
how every of ${\uu_s}$ and ${\uv_s}$ is presented, per the convention \eqref{signConvention}.
}

\begin{defn}\label{ftF}
Consider any $\pl$ relation $F=F_{\uh,\uk}$ for some pair
 $(\uh, \uk) \in \II_{d-1,n} \times \II_{d+1,n}$.
We let $\ft_{F}+1$ be the number of terms in $F$. We then define the rank of
$F$ to be $\ft_{F}-2$. We denote this number by $\rk (F)$.
\end{defn}
The integer $\ft_{F}$, as defined above, will be frequently used throughout.

\begin{example}\label{exam:(3,6)}
Consider the Grassmannian $\Gr(3,6)$. Then, the $\pl$ relation
$$F_{(16), (3456)}:  p_{163}p_{456} - p_{164}p_{356} + p_{165}p_{346}$$
is of rank zero; the $\pl$ relation
$$F_{(12), (3456)}: p_{123}p_{456} - p_{124}p_{356} + p_{125}p_{346}- p_{126}p_{345}$$
is of rank one. 
\end{example}

Let $\ZZ[p_\ui]_{\ui \in \II_{d,n}}$ be the homogeneous coordinate ring
 of the $\pl$ projective space $\PP(\wedge^d E)$ and $I_\wp \subset \ZZ[p_\ui]_{\ui \in \II_{d,n}}$ 
be the homogeneous ideal generated by all the $\pl$ relations  \eqref{pluckerEq} or
\eqref{succinct-pl}.
We let $I_{\widehat\wp}$ by the homogeneous ideal of  $\Gr^{d,E}$ in $\PP(\wedge^d E)$.
Then $I_{\widehat\wp} \supset I_\wp$, but not equal over $\ZZ$,
or a field of positive characteristic, in general: there are other additional
relations (multivariate $\pl$ relations) for the Grassmannian $\Gr^{d,E}$, thanks to Matt Baker for pointing this out to the author. In characteristic zero, $I_{\widehat\wp} = I_\wp$.


\subsection{Primary $\pl$ equations with respect to a fixed affine chart} $\ $

In this subsection, we focus on a fixed affine chart of the $\pl$ projective space
$\PP(\wedge^d E)$.

Fix any $\um \in \II_{d,n}$. In $\PP(\wedge^d E)$,
we let $$\rU_\um:=(p_\um \equiv 1)$$  stand for the open chart
 defined by $p_\um \ne 0$.  Then,
the affine space $\rU_\um$  comes equipped with the 
coordinate variables $x_\uu=p_\uu/p_\um$ for all $\uu \in \II_{d,n} \- \um$.
In practical calculations, we will simply set $p_\um =1$, whence the notation
$(p_\um \equiv 1)$ for the chart. We let
$$\rU_\um (\Gr) = \rU_\um \cap \Gr^{d, E}$$
be the corresponding induced open chart  of $\Gr^{d, E}$.

The chart $\rU_\um (\Gr)$ is canonically an affine space.
Below, we explicitly describe $$\up:={n \choose d} -1- d(n-d) $$
many specific  $\pl$ relations with respect to the chart $\rU_\um$, called the $\um$-primary 
$\pl$ relations, such that their restrictions to the chart $\rU_\um$ define
$\rU_\um (\Gr)$ as a closed subscheme of the affine space $\rU_\um$.

To this end, we write  $\um=(m_1 \cdots m_d)$. We set 
$$\II^\um_{d,n}=\{\uu \in \II_{d,n} \mid |\uu \- \um| \ge 2\} \subset \II_{d,n}$$
where $\uu$ and  $\um$ are also regarded as subsets of integers, and $|\uu \- \um|$ denotes the cardinality of $\uu \- \um$. 
In  words,  $\uu \in \II^\um_{d,n}$ if and only if
$\uu=(u_1,\cdots, u_d)$ contains at least two elements distinct from elements in  
$\um=(m_1, \cdots, m_d)$. 
It is helpful to write explicitly the set $\II_{d,n} \- \II_{d,n}^\um$:
$$\II_{d,n} \- \II_{d,n}^\um =\{\um\} \cup \{ \{u\} \cup (\um\setminus m_i) \mid \; 
\hbox{for all $ u \in [n]\- \um$ and all $1\le i\le d$} \}, $$
where $u \notin \um$ if and only if $u \ne m_i$ for any $1\le i\le d$.
Then, one calculates and finds
$$|\II^\um_{d,n}|=\up={n \choose d} -1- d(n-d) ,$$
where $|\II^\um_{d,n}|$ denotes the cardinality of $\II^\um_{d,n}$.

Further, let $\ua=(a_1\cdots a_k)$ be a list of some elements of $[n]$, not necessarily mutually distinct,
 for some $k<n$.
We will write $$v \ua=v(a_1\cdots a_k)=(v a_1\cdots a_k)
\;\; \hbox{and} \;\;  \ua  v=(a_1\cdots a_k)v=(a_1\cdots a_k v),$$
each is considered as a list of some elements of $[n]$, 
for any $v \in [n] \- \ua$.

Now, take any element $\uu =(u_1,\cdots, u_d) \in \II_{d,n}^\um$.
We let $u_0$ denote the smallest integer in $\uu \- \um$. 
We then set 
$$ \uh=\uu \setminus u_0 \;\; \hbox{and} \;\;  \uk=(u_0 m_1 \cdots m_d),$$
where $\uu \setminus u_0 =\uu \setminus \{u_0\}$ and $\uu$ is regarded as a set of integers.

 This gives rise
to the $\pl$ relation $F_{\uh,\uk}$, taking of the following form
\begin{equation} \label{keyTrick}
F_{\uh,\uk}=p_{(\uu \setminus u_0)u_0} p_\um -  p_{(\uu \setminus u_0) m_1} p_{u_0 (\um \setminus m_1 )}
+\cdots + (-1)^d p_{(\uu \setminus u_0) m_d} p_{u_0 (\um \setminus m_d)},
\end{equation}
where $\um \setminus m_i = \um \setminus \{m_i\}$ and $\um$ is regarded as a set of integers,
for all $i \in [d]$.

We give a new notation for this particular equation: we denote it by
\begin{equation} \label{keyTrick2}
F_{\um, \uu}=p_{(\uu \setminus u_0)u_0} p_\um + \sum_{i=1}^d (-1)^i  
p_{(\uu \setminus u_0)m_i} p_{u_0 (\um \setminus m_i )},
\end{equation} because it only depends on $\um$ and $\uu \in \II_{d,n}^\um$. 
To simplify the notation, we introduce
$$\uu^r= \uu \setminus u_0, \;\; \widehat{\um_i} = \um\setminus m_i, \;\; \hbox{for all $i \in [d]$.}$$
Then, \eqref{keyTrick2} becomes
\begin{equation} \label{keyTrick4}
F_{\um, \uu}=p_{\um}p_{\uu^r u_0} +
\sum_{i=1}^d (-1)^i p_{ \uu^rm_i } p_{u_0 \widehat{\um_i} }.
\end{equation}

We point out here that $\uu$ and $\uu^r u_0$ may differ by a permutation.

\begin{defn}\label{lt-ltvar}
We call the $\pl$ equation $F_{\um, \uu}$ of \eqref{keyTrick4} a primary $\pl$ equation
for the chart $\rU_\um=(p_\um \equiv 1)$. We also say $F_{\um,\uu}$ is $\um$-primary.
The term $p_\um p_{\uu}$ is called the leading term of $F_{\um, \uu}$.
\end{defn}

(One should not confuse  $F_{\um, \uu}$ with the expression of a general $\pl$ equation
$F_{\uh,\uk}$: we have $(\um, \uu)\in \II_{d,n}^2$ for the former and 
$(\uh,\uk) \in \II_{d-1,n} \times \II_{d+1,n}$ for the latter.)

One sees that the correspondence between $\II_{d,n}^\um$ and 
the set of $\um$-primary $\pl$ equations is a bijection.

\subsection{De-homogenized $\pl$ ideal  with respect to a fixed affine chart} $\ $

Following the previous subsection, we continue to fix an element $\um \in \II_{d,n}$ 
and  focus on the chart $\rU_\um$ of $\PP(\wedge^d E)$.

We will write  $\II_{d,n} \- \um$ for  $\II_{d,n} \- \{\um\}$.

Given any $\uu \in \II^\um_{d,n}$, by \eqref{keyTrick4}, 
it gives rise to the $\um$-primary equation 
$$ F_{\um, \uu}=p_{\um}p_{\uu^r u_0} +
\sum_{i=1}^d (-1)^i p_{\uu^r m_i } p_{ u_0 \widehat{\um_i} }.$$
 If we set $p_\um =1$ and let $x_\uw=p_\uw$, for all $\uw \in \II_{d,n} \- \um$,
 then it  becomes 
\begin{equation}\label{equ:localized-uu}
\bF_{\um, \uu}=x_{\uu^r u_0} +
\sum_{i=1}^d (-1)^i  x_{\uu^r m_i} x_{ u_0 \widehat{\um_i} }.
\end{equation}

\begin{defn}\label{localized-primary} We call the relation \eqref{equ:localized-uu} 
 the de-homogenized (or the localized)
$\um$-primary $\pl$ relation corresponding to $\uu \in \II_{d,n}^\um$.
We call the unique distinguished variable, $x_{\uu}$ 
(which may differ $x_{\uu^r u_0}$ by a sign), the leading variable
of the  de-homogenized $\pl$ relation $\bF_{\um, \uu}$. 
\end{defn}

Throughout  this paper, we often express an $\um$-primary $\pl$ equation $F$ as
\begin{equation}\label{the-form-F}
F=\sum_{s \in S_F} \sgn (s) p_{\uu_s} p_{\uv_s}= \sgn (s_F) p_\um p_{\uu_{s_F}}+
\sum_{s \in S_F\- s_F} \sgn (s) p_{\uu_s} p_{\uv_s}
\end{equation}
where $s_F$ is the index for the leading term of $F$, and $S_F\- s_F:=S_F\-\{s_F\}$.
Then, upon setting $p_\um=1$ and letting $x_\uw=p_\uw$ for all $\uw \in \II_{d,n}\- \um$,
 we can write the corresponding de-homogenized  $\um$-primary $\pl$ equation $\bF$ as
\begin{equation}\label{the-form-LF}
\bF=\sum_{s \in S_F} \sgn (s) x_{\uu_s} x_{\uv_s}= \sgn (s_F) x_{\uu_F}+
\sum_{s \in S_F\- s_F} \sgn (s) x_{\uu_s} x_{\uv_s}
\end{equation}
where  $x_{\uu_F}:=x_{\uu_{s_F}}$ is the leading variable of $\bF$.

\begin{defn}\label{ft-bF}
Let $F$ be an $\um$-primary $\pl$ relation and $\bF$ its 
de-homogenization with respect to the chart $\rU_\um$.
We set $\ft_{\bF}=\ft_F$  and $\rk (\bF) =\rk (F)$.
\end{defn}

For any $\uu \in \II_{d,n} \- \um$,  we let $x_\uu=p_\uu/p_\um$ for all $\uu \in \II_{d,n} \- \um$.
Then, we can identify the coordinate ring of $\rU_\um$ 
with $\kk [x_\uu]_{\uu \in \II_{d,n} \- \um}$.
We let $I_{\wp,\um}$ be the ideal of $\kk [x_\uu]_{\uu \in \II_{d,n} \- \um}$
obtained from the ideal $I_\wp$ be setting $p_\um=1$ and letting 
$x_\uu=p_\uu$ for all $\uu \in \II_{d,n} \- \um$.
The ideal $I_{\wp,\um}$ is the de-homogenization of the 
homogeneous $\pl$ ideal $I_\wp$ on the chart $\rU_\um$.

\begin{defn} For any  $\uu \in \II_{d,n}^\um$, we define 
the $\um$-rank of $\uu$ (resp.  $x_\uu$) to be the rank of its corresponding
primary $\pl$ equation $F_{\um,\uu}$.  If $\uu \in (\II_{d,n} \- \um) \- \II_{d,n}^\um$,
then we set $\rk (\uu)=-1$.
\end{defn}

\begin{prop}\label{primary-generate} 
The affine subspace $\rU_\um (\Gr)=\rU_\um \cap \Gr^{d,E}$ embedded in
the affine space $\rU_\um$ is defined by the relations in
$$\sfm:=\{ \bF_{\um, \uu} \mid \uu \in \II_{d,n}^\um \}.$$
 Consequently, the chart $\rU_\um (\Gr)=\rU_\um \cap \Gr^{d,E}$ comes equipped with
 the set of free variables
 $$\var_{\rU_\um}:=\{x_\uu \mid \uu \in \II_{d,n} \-\{ \um\} \- \II_{d,n}^\um\}$$ and is 
 canonically isomorphic to the affine space 
 with  the above variables as its coordinate variables.
\end{prop}
\begin{proof}
(This proposition is elementary; 
it serves as the initial check of an induction for  some later proposition;
we provide sufficient details for completeness.)

It suffices to observe that for any  $\uu \in \II_{d,n}^\um$, 
its corresponding de-homogenized $\pl$ 
primary $\pl$ equation $\bF_{\um, \uu}$ is equivalent to an expression of
the leading variable $x_\uu$ as a polynomial in  the free variables of
$\var_{\rU_\um}$.  For instance, one can check this by induction on the $\um$-rank, $\rk (\uu)$, 
of $\uu$, as follows.

Suppose $\rk (\uu)= 0$. Then, up to a permutation, we may write
$$\uu= (\um \setminus \{m_i m_j \}) vu$$
where $m_i, m_j \in \um$ for some $1\le i, j,\le d$, and  $ u< v  \notin \um$. Then, we have
\begin{equation}\label{form-of-rk0}
\bF_{\uu, u}:  x_{\uu}  + (-1)^i x_{ \uu^r m_i} x_{u \widehat{m_i}}
+(-1)^j x_{ \uu^r m_j} x_{u \widehat{m_j}} ,
\end{equation}
where $\uu^r =(\um \setminus \{m_i m_j \})v$.
One sees that $x_{ \uu^r m_i}, \;  x_{u \widehat{m_i}} , \;  x_{ \uu^r m_j}$
and $x_{u \widehat{m_j}}$ belong to $\var_{\rU_\um}$. Hence, the statement holds.

Now suppose that $\rk (\uu)>0$.
Using (\ref{equ:localized-uu}),  we have
$$\bF_{\um, \uu}: \; x_{\uu^r u_0}+
\sum_{i=1}^d (-1)^i x_{\uu^r m_i } x_{u_0 \widehat{m_i} }.$$
Note that all variables $x_{u_0 \widehat{m_i} }, i \in [d]$, belong to $\var_{\rU_\um}$.
Note also that $$\rk ({\uu^r m_i }) = \rk (\uu) -1,$$ provided that
 $p_{\uu^r m_i }$ is not identically zero, that is,  it is a well-defined $\pl$ variable
 (see \eqref{zeroConvention}).
 Thus,  applying the inductive assumption, any such $x_{\uu^r m_i }$ is a polynomial in 
 the variables of $\var_{\rU_\um}$.
 Therefore,  $\bF_{\um, \uu}$, 
   is equivalent to an expression of $x_\uu$ as a polynomial in 
 the variables of $\var_{\rU_\um}$.

Let $J$ be the ideal of $\kk [x_\uu]_{\uu \in \II_{d,n}\- \um}$ generated by
 $\{ \bF_{\um, \uu} \mid \uu \in \II_{d,n}^\um \}$ and let
 $V(J)$ the subscheme of $\rU_\um$ defined by $J$.
By the above discussion,  $V(J)$ is canonically isomorphic to
 the affine space of dimension $d(n-d)$ with  
  the variables of $\var_{\rU_\um}$ as its coordinate variables.
Since $\rU_\um(\Gr) \subset V(J)$, we conclude $\rU_\um (\Gr)=V(J)$.
\end{proof}

\begin{defn}\label{basic} We call the variables in
$$\var_{\rU_\um}:=\{x_\uu \mid \uu \in \II_{d,n} \- \um \- \II_{d,n}^\um\}$$
the $\um$-basic $\pl$ variables. When $\um$ is fixed and clear from the context,
we just call them basic variables.
\end{defn}
Only non-basic $\pl$ variables correspond to $\um$-primary $\pl$ equations.

Observe that for all $\pl$ relations $F$, we have $0\le {\rm rank} ( F)\le d-2$.
Hence, for any $0\le r\le d-2$, we let
$$\sF^r_\um = \{\bF_{\um,\uu}  \mid {\rm rank} (F_{\um,\uu}) =r,\; \uu \in \II_{d,n}^\um\}.$$
Then, we have
$$\sfm=\bigcup_{0\le r\le d-2} \sF^r_\um.$$

Then, one observes the following easy but useful fact.

\begin{prop}\label{leadingTerm} Fix  any $0\le r\le d-2$
 any $\uu \in \II_{d,n}^\um$ with ${\rm rank}_\um (F_{\um, \uu})=r$.
 Then,  the leading variable 
$x_\uu$ of $\bF_{\um, \uu}$ does not appear in any relation in
$$\sF^0_\um \cup \cdots \cup \sF^{r-1}_\um \cup (\sF^r_\um \setminus \bF_{\um, \uu}).$$
\end{prop}

To close this subsection, we raise  a concrete question.
 Fix the chart $(p_\um \equiv 1)$.
 In  $\kk[x_\uu]_{\uu \in \II_{d,n}\- \um}$,
according to  Proposition \ref{primary-generate}, 
 every de-homogenized $\pl$ equation
$\bF_{\uh,\uk}$ on the chart $\rU_\um$ can be expressed
as a polynomial in the  de-homogenized primary $\pl$ relations $\bF_{\um, \uu}$ 
with $\uu \in \II_{d,n}^\um$. 
 It may be useful  in practice to find such an expression explicitly for an arbitrary $F_{\uh,\uk}$.
For example, for the case of $\Gr(2,5)$, this can be done as follows.

\begin{example}\label{pl2-5} For $\Gr(2,5)$, we have five $\pl$ relations:
$$F_1= p_{12}p_{34}-p_{13}p_{24} + p_{14}p_{23},\;
F_2= p_{12}p_{35}-p_{13}p_{25} + p_{15}p_{23}, \; $$
$$F_3= p_{12}p_{45}-p_{14}p_{25} + p_{15}p_{24},\;
F_4= p_{13}p_{45}-p_{14}p_{35} + p_{15}p_{34}, \;$$
$$F_5= p_{23}p_{45}-p_{24}p_{35} + p_{25}p_{34}. $$
On the chart $(p_{45} \equiv 1)$,
$F_3, F_4,$ and $ F_5$ are primary. 
One calculates and finds
$$p_{45} F_1 = p_{34} F_3 -p_{24} F_4 + p_{14} F_5, $$ 
$$p_{45} F_2 = p_{35} F_3 -p_{25} F_4 + p_{15} F_5.$$
In addition,  the  Jacobian of the de-homogenized $\pl$ equations of $\bF_3, \bF_4, \bF_5$ with respect to 
all the variables,
$x_{12}, x_{14}, x_{15} ,x_{13}, x_{35} , x_{34},x_{23} ,x_{24} ,x_{25}, $
is given by
$$
\left(
\begin{array}{cccccccccc}
1 & x_{25} & x_{24} & 0  & 0 & 0 & 0 & 0 & 0 \\
0 & 0 & 0 & 1  & x_{14} & x_{15}& 0 & 0 & 0 \\
0& 0 & 0 & 0  & 0 & 0 & 1 & x_{35} & x_{34} \\
\end{array}
\right).
$$
There, one sees visibly  a $(3 \times 3)$  identity minor. 
\end{example}

 



\subsection{Ordering the set of all primary $\pl$ equations}\label{order-eq} $\ $






\begin{defn}\label{gen-order} Let $K$ be any fixed totally ordered finite set, 
with its order denoted by $<$.
Consider any two subsets $\eta \subset K$ and ${ \zeta} \subset K$ with the cardinality $n$
for some positive integer $n$.
We write $\eta=(\eta_1,\cdots,\eta_n)$ 
and  ${\zeta}=(\zeta_1,\cdots,\zeta_n)$ as arrays according to the ordering of $K$.
We say $\eta <_{\lex} { \zeta}$ if the left most nonzero number in the vector $\eta-{\zeta}$ is negative,
or more explicitly,  if we can express
$$\eta=\{t_{1}< \cdots <t_{r-1} <s_r< \cdots \}$$
$${\zeta}=\{t_{1}< \cdots <t_{r-1}<t_r  < \cdots \}$$
such that $s_r< t_r$ for some integer $r \ge 1$.  We call $<_\lex$ the lexicographic order
induced by $(K, <)$.

Likewise, we say $\eta <_{\invlex} { \zeta}$
 if the right most nonzero number in the vector $\eta-{ \zeta}$ is negative,
or more explicitly, if we can express
$$\eta=\{\cdots <s_r< t_{r+1}< \cdots <t_n\}$$
$${\zeta}=\{\cdots <t_r  < t_{r+1}< \cdots <t_n\}$$
such that $s_r< t_r$ for some integer $r \ge 1$. 
We call $<_\invlex$ the reverselexicographic order
induced by $(K, <)$. 
\end{defn}

This definition can be applied to the set
$$\II_{d,n} =\{ (i_1 < i_2 < \cdots <i_d) \; \mid \; 1\le  i_\mu\le n,\; \forall \; 1\le \mu\le d\}$$
for all $d$ and $n$.
Thus, we have equipped the set $\II_{d,n}$ 
with both the  lexicographic ordering $``<_\lex "$ and
the reverse lexicographic ordering $``<_\invlex "$. 

{\it We point out here that neither is the order we used for the set of $\pl$ variables
$\var_{\rU_\um}=\{x_\uu \mid \uu \in \II_{d,n} \- \um\}$, even thought by the obvious bijection
between $\II_{d,n} \- \um$ and $\var_{\rU_\um}$, each provides a total ordering on $\var_{\rU_\um}$. }

\begin{defn}\label{invlex}
Consider any $\uu, \uv \in \II_{d,n} \- \um$. 
We say $$\uu <_\wp \uv$$ if one of the following three holds:
\begin{itemize}
\item ${\rm rank}_\um \; \uu < {\rm rank}_\um \; \uv$;
\item ${\rm rank}_\um \; \uu = {\rm rank}_\um \; \uv$, $\uu \- \um<_\lex \uv \- \um$;
\item ${\rm rank}_\um \; \uu ={\rm rank}_\um \; \uv$, $\uu \- \um = \uv\- \um$, and
$\um \cap \uu  <_\lex \um \cap \uv$.
\end{itemize}
\end{defn}

\begin{defn}\label{cFi-partial-order} Consider any two $\pl$ variables $x_\uu$ and $x_\uv$.
We say 
$$\hbox{$x_\uu <_\wp x_\uv$ if $\uu<_\wp \uv$.}$$

Consider any two distinct primary equations, 
$\bF_{\um,\uu}, \bF_{\um,\uv} \in \sfm$ 
 We say 
$$\bF_{\um,\uu} <_\wp \bF_{\um,\uv}\; \; \hbox{if} \;\;  \uu <_\wp \uv.$$
\end{defn} 

Under the above order, we can write
$$\sfm=\{\bF_1 <_\wp \cdots <_\wp \bF_\up\}.$$

In the sequel, when comparing two $\pl$ variables $x_\uu$ and $x_\uv$
or  two $\um$-primary $\pl$ equations, we exclusively use $<_\wp$.
Thus, throughout, for simplicity, we will simply write $<$ for $<_\wp$.
A confusion is unlikely.

{\it We point out here that $x_\uu <_\wp x_\uv$ is neither lexicographic
nor inverse-lexicographic on the indexes. Indeed,
every non-lexicographic or   non-inverse-lexicographic order,
introduced in this article, is important for our method. Some orders may be subtle.}


For later use, we introduce 

\begin{defn}\label{p-t}
Let $\fT_i$ be a finite set for all $i \in [h]$ for some positive integer $h$. Then,
the order $$\fT_1 < \cdots < \fT_h$$ naturally induces a partial order on the disjoint union
$\sqcup_{i \in [h]} \fT_i $ as follows. Take any $i < j \in [h]$,
$a_i \in \fT_i$, and $a_j \in \fT_j$. Then, we say $a_i < a_j$.
\end{defn}

If every $\fT_i$ comes equipped with a total order for all $i \in [h]$. Then,
in the situation of Definition \ref{p-t}, the disjoint union
$\sqcup_{i \in [h]} \fT_i $ is totally ordered.

\section{A Singular Local Birational Model $\sV$ for $\Gr^{d,E}$}\label{singular-model}

{\it The purpose of this section is to establish a local model $\sV_\um$, birational to
$\Gr^{d,E}$, such that  all
terms of all $\um$-primary $\pl$ equations can be separated in  the
defining main binomial relations of $\sV_\um$ in a smooth ambient scheme $\sR_\sF$.
The construction of $\sV_\um$
is modeled on a chart of the total scheme of the Hilbert family as constructed 
in \cite{Hu2022}.
}

\subsection{The construction  of $\sV \subset \sR_\sF$} $\ $

Consider the fixed affine chart $\rU_\um$ of
the $\pl$ projective space $\PP(\wedge^d E)$.
For any $\bF \in \sfm$, 
written as $F=\sum_{s \in S_F} \sgn (s) p_{\uu_s}p_{\uv_s}$, 
we let $\PP_F$ be the projective space with 
homogeneous coordinates written as $[x_{(\uu_s, \uv_s)}]_{s \in S_F}$. For convenience, we let
\begin{equation}\label{LaF} \La_F=\{(\uu_s, \uv_s) \mid s \in S_F\}.
\end{equation}
This is an index set for  the homogeneous coordinates of the projective
space $\PP_F$.  To avoid duplication, we make a convention: 
\begin{equation}\label{uv=vu}
x_{(\uu_s, \uv_s)}=x_{(\uv_s, \uu_s)}, \; \forall \; s \in S_{F}, \; \forall \; \bF \in \sfm.
\end{equation}
If we write $(\uu_s, \uv_s)$ in the lexicographical order, i.e., we insist $\uu_s <_\lex \uv_s$,
then the ambiguity is automatically avoided. However, the convention is still  useful.

\begin{defn}
We call $x_{(\uu_s, \uv_s)}$ a $\vr$-variable of $\PP_F$, or simply a $\vr$-variable.
To distinguish, we call a $\pl$ variable, $x_\uu$ with $\uu \in \II_{d,n} \- \um$, a $\vp$-variable.
\end{defn}

Fix $k \in [\up]$. We introduce the natural rational map
\begin{equation}\label{theta-k}
 \xymatrix{
\Theta_{[k]}: 
\PP(\wedge^d E) \ar @{-->}[r]  & \prod_{i \in [k]} \PP_{F_i}  } 
\end{equation}
$$
 [p_\uu]_{\uu \in \II_{d,n}} \lra  
\prod_{i \in [k]}  [p_\uu p_\uv]_{(\uu,\uv) \in \La_{F_i}}
$$   
where $[p_\uu]_{\uu \in \II_{d,n}} $ is the homogeneous $\pl$ coordinates of a point in $
\PP(\wedge^d E)$. When restricting $\Theta_{[k]}$ to $\rU_\um$,
it gives rise to 
\begin{equation}\label{bar-theta-k}
 \xymatrix{
\bar\Theta_{[k]}: \rU_\um  \ar @{-->}[r]  & \prod_{i \in [k]} \PP_{F_i}  } 
\end{equation}

We let \begin{equation}\label{tA}
 \xymatrix{
\PP_{\sF_{[k]}}  \ar @{^{(}->}[r]  & \PP(\wedge^d E) \times  \prod_{i \in [k]} \PP_{F_i}  
 }
\end{equation}
be the closure of the graph of the rational map $\Theta_{[k]}$, and
 \begin{equation}\label{bar-tA}
 \xymatrix{
\rU_{\um,\sF_{[k]}}  \ar @{^{(}->}[r]  & \sR_{\sF_{[k]}}:= \rU_\um \times  \prod_{i \in [k]} \PP_{F_i}  
 }
\end{equation}
be the closure of the graph of the rational map $\bar\Theta_{[k]}$.

\begin{defn}
Fix any $ k \in [\up]$.
We let $$R_{[k]}=\kk[p_\uu; x_{(\uv_s, \uu_s)}]_{\uu \in \II_{d,n} , s \in S_{F_i},  i \in [k]}$$
and let $$\bar R_{[k]}=\kk[x_\uu; x_{(\uv_s, \uu_s)}]_{\uu \in \II_{d,n} \- \um, s \in S_{F_i},  i \in [k]}$$
be the de-homogenization of $R_{[k]}$.
A  polynomial $f \in \bar R_{[k]}$ (resp.  $R_{[k]}$) is called multi-homogeneous if it is homogenous 
in $[x_{(\uv_s, \uu_s)}]_{s \in S_{F_i}}$, for every $i \in [k]$
 (resp. and is also homogenous in $[p_\uu]_{\uu \in \II_{d,n}}$). 
 A multi-homogeneous polynomial $f \in \bar R_{[k]}$ (resp.  $R_{[k]}$)  is $\vr$-linear if it is linear 
 in $[x_{(\uv_s, \uu_s)}]_{s \in S_{F_i}}$, 
 whenever it contains some $\vr$-variables of $\PP_{F_i}$,
 for  any $i \in [k]$.
 \end{defn}

We set $R_0:=\kk[p_\uu]_{\uu \in \II_{d,n}}$
and $\bar R_0:=\kk[x_\uu]_{\uu \in \II_{d,n} \- \um}$ 
Then, corresponding to 
the embedding \eqref{tA}, 
there exists a degree two homomorphism
\begin{equation}\label{vik}
\vi_{[k]}: \; R_{[k]}=R_0[x_{(\uv_s, \uu_s)}]_{s \in S_{F_i},  i \in [k]} \lra R_0 , 
\;\;\; x_{(\uu_s,\uv_s)} \to p_{\uu_s} p_{\uv_s}
\end{equation} 
for all  $s \in S_{F_i}, \; i \in [k]$, where $\vi_{[k]}$ restricts the identity on $R_0$.    

We then let
\begin{equation}\label{bar-vik}
\bar\vi_{[k]}: \; \bar R_{[k]}=\bar R_0[x_{(\uv_s, \uu_s)}]_{s \in S_{F_i},  i \in [k]} \lra \bar R_0, \;\;\; x_{(\uu_s,\uv_s)} \to x_{\uu_s} x_{\uv_s} 
\end{equation}  for all  $s \in S_{F_i}, i \in [k]$,
be the de-homogenization of $\vi_{[k]}$ with respect to the chart $\rU_\um=(p_\um \equiv 1)$.
This corresponds to the embedding \eqref{bar-tA}.

We are mainly interested in the case when $k=\up$.
 Hence, we set 
$$R:=R_{[\up]},\;\; \vi:=\vi_{[\up]}, \;\; \bar\vi:=\bar\vi_{[\up]}.$$

We let $\ker^\mh \vi_{[k]}$ (resp. $\ker^\mh \bar\vi_{[k]}$)
denote the set of all  multi-homogeneous polynomials
 in $\ker  \vi_{[k]}$ (resp. $\ker \bar\vi_{[k]}$).


\begin{lemma}\label{defined-by-ker} 
The scheme
$\PP_{\sF_{[k]}}$, as a closed subscheme of 
 $\PP(\wedge^d E) \times  \prod_{i \in [k]} \PP_{F_i}$,
 is defined by $\ker^\mh \vi_{[k]}$.
In particular, the scheme
$\rU_{\um,\sF_{[k]}}$, as a closed subscheme of 
 $\sR_{\sF_{[k]}}= \rU_\um \times  \prod_{i \in [k]} \PP_{F_i}$,
 is defined by $\ker^\mh \bar\vi_{[k]}$. 
\end{lemma}
\begin{proof} This is immediate.
\end{proof}

We need to investigate $\ker^\mh \vi_{[k]}$.

Consider any $f \in \ker^\mh \vi_{[k]}$. We express it as the sum of its monic monomials
(monomials with constant coefficients 1)
$$f= \sum \bm_i.$$
We have $\vi_{[k]} (f)=\sum \vi_{[k]} (\bm_i)=0$ in $R_0$. Thus, the set of 
the monic monomials $\{\bm_i\}$ can be grouped into minimal groups to form partial sums
of $f$ so that {\it the images of elements of each group are 
 identical} and the image of the partial sum of each minimal group equals 0 in $R_0$.
When ch.$\kk=0$, this means each minimal group consists of a pair $(\bm_i, \bm_j)$
and its partial sum
is the difference $\bm_i -\bm_j$.
When ch.$\kk=p>0$ for some prime number $p$, this means each minimal group 
 consists of either  (1): a pair $(\bm_i, \bm_j)$ and $\bm_i -\bm_j$ is a partial sum of $f$;
or (2):   exactly $p$ elements $\bm_{i_1}, \cdots,\bm_{i_p}$
and $\bm_{i_1}+ \cdots + \bm_{i_p}$ is a partial sum of $f$.
But, the relation $\bm_{i_1}+ \cdots + \bm_{i_p}$ is always generated by 
the relations $\bm_{i_a} -\bm_{i_b}$ for all $1\le a, b\le p$.

Thus, regardless of the characteristic of the field $\kk$, it suffices to consider binomials
$\bm -\bm' \in \ker^\mh \vi_{[k]}$.

\begin{example}\label{exam:Bq} Consider $\Gr^{3,E}$. 
Then, the following binomials belong to $\ker^\mh \vi_{[k]}$
for any fixed $k \in [\up]$.

Fix $a,b,c \in [k]$, all being distinct:
\begin{eqnarray}
x_{(12a,13b)}x_{(13a,12c)}x_{(12b,13c)} \; \nonumber \\
-x_{(13a,12b)}x_{(12a,13c)}x_{(13b,12c)}. \label{rk0-0} 
\end{eqnarray}
\begin{eqnarray}
x_{(12a,13b)}x_{(13a,12c)}x_{(12b,23c)} x_{(23b,13c)} \; \nonumber \\
-x_{(13a,12b)}x_{(12a,13c)}x_{(23b,12c)} x_{(13b,23c)}. \label{rk0-0-4}
\end{eqnarray}
Fix $a,b,c, \bar a, \bar b, \bar c \in [k]$, all being distinct:
\begin{eqnarray}
\;\; \;x_{(12a,3bc)}x_{(13a,2\bar b \bar c)}x_{(13 \bar a,2bc)} x_{(12 \bar a,3 \bar b \bar c)} \; \nonumber \\
-x_{(13a,2bc)}x_{(12a,3\bar b \bar c)}x_{(12 \bar a,3bc)} x_{(13 \bar a,2 \bar b \bar c)}.  \label{rk1-1}
\end{eqnarray}
Fix $a,b,c, a', \bar a, \bar b, \bar c \in [k]$, all being distinct:
\begin{eqnarray}
\;\; \; x_{(12a,13a')}x_{(13a,2bc)}x_{(12a',3\bar b \bar c)}x_{(12 \bar a,3bc)} x_{(13 \bar a,2 \bar b \bar c)} 
\; \nonumber \\
- x_{(13a,12a')} x_{(12a,3bc)}x_{(13a',2\bar b \bar c)}x_{(13 \bar a,2bc)} x_{(12 \bar a, 3 \bar b \bar c)}.\label{rk0-1}
\end{eqnarray} 
These binomials are arranged so that one sees visibly the matching for multi-homogeneity.
\end{example}

\begin{lemma} \label{trivialB}
Fix any $i \in [k]$. We have
\begin{equation}\label{tildeBk}
p_{\uu'}p_{\uv'}x_{(\uu,\uv)} - p_\uu p_\uv x_{(\uu',\uv')} \in \ker^\mh \vi_{[k]}.
\end{equation}
where $x_{(\uu,\uv)}, x_{(\uu',\uv')}$ are any two distinct $\vr$-variables of $\PP_{F_i}$.
Likewise, we have
\begin{equation}\label{trivialBk}
x_{\uu'}x_{\uv'}x_{(\uu,\uv)} - x_\uu x_\uv x_{(\uu',\uv')} \in \ker^\mh \bar\vi_{[k]}.
\end{equation}
\end{lemma}
\begin{proof} This is trivial.
\end{proof}

\smallskip

Let $\AA^l$ (resp. $\PP^l$) be the affine (resp. projective)
space of dimension $l$ for some positive integer $l$ with
coordinate variables $(x_1,\cdots, x_l)$ (resp. with
homogeneous coordinates $[x_1,\cdots, x_l]$).
A monomial $\bf m$ is {\it square-free} if 
$x^2$ does not divide $\bf m$ for every coordinate variable $x$ in the affine space.  
A polynomial is square-free if all of its monomials are
square-free.

\smallskip

For any $\bm -\bm' \in \ker^\mh \vi_{[k]}$, we define 
$ \deg_{\vr} (\bm -\bm')$ to be the total degree of $\bm$ (equivalently, $\bm'$)
in $\vr$-variables of $R_{[k]}$.

For any $F=\sum_{s \in S_F} p_{\uu_s} p_{\uv_s}$ with $\bF \in \sfm$ 
and $s \in S_F$, we write $X_s=x_{(\uu_s,\uv_s)}$.

Recall  that we have set $\vi=\vi_{[\up]}: R=R_{[\up]} \to R_0$.

Observe here that  for any nonzero binomial $\bm -\bm' \in \ker^\mh \vi_{[k]}$,
we automatically have $ \deg_\vr (\bm -\bm') > 0$, since $\vi_{[k]}$ restricts
to the identity on $R_0$.

\begin{lemma}\label{ker-phi-k}
Consider a binomial $ \bm -\bm' \in \ker^\mh \vi_{[k]}$ with $ \deg_\vr (\bm -\bm') > 0$. 
 We let $h$ be the maximal common factor of the two monomials $\bm$ and $\bm'$
 in $R_{[k]}$. Then, we have 
$$\hbox{ $\bm=h \prod_{i=1}^{\l} \bm_i$ and $\bm'=h \prod_{i=1}^{\l} \bm'_i$}$$ 
for some positive integer ${\l}$ such that
 for every $i \in [{\l}]$,
$\bm_i -\bm'_i \in \ker^\mh \vi_{[k]}$ and is of  the following form:
\begin{equation}\label{1st-Hq}
\vi (X_1)X_2- \vi(X'_1 )X'_2 ,
 \end{equation}
 where  every of $X_1,X_2,  X'_1,$ and $X'_2$ is a monomial of $R$ in $\vr$-variables only
 (i.e., without $\vp$-variables; here we allow $X_1=X_1'=1$)  such that 
\begin{enumerate}
 \item $X_1X_2-  X'_1 X'_2 \in \ker^\mh \vi$ and  is $\vr$-linear;
 \item $\vi(X_1X_2)$ (equivalently, $\vi(X'_1 X'_2)$) is a square-free monomial;
\item  for any $\bF \in \sfm$ and $s \in S_F$,
suppose $x_{\uu_s} x_{\uv_s} $ divides $\bm$ (resp. $\bm'$), 
then $X_s =x_{(\uu_s,\uv_s)}$  divides   $X_1$  (resp.  $X_1'$) 
 in one of the relations of \eqref{1st-Hq}. 
 \end{enumerate}
\end{lemma}
\begin{proof} 
We prove by induction on $\deg_{\vr} (\bm -\bm')$.

Suppose $\deg_{\vr}(\bm-\bm') =1$. 

Then, we can write $$\bm -\bm' = f x_{(\uu,\uv)} -g  x_{(\uu',\uv')} $$ 
for some $f, g \in R_0$,
and two $\vr$-variables of $\PP_{F_i}$, $x_{(\uu,\uv)}$ and $x_{(\uu',\uv')}$ 
for some $i \in [k]$.  
If $x_{(\uu,\uv)} =x_{(\uu',\uv')}$, then one sees that $f=g$ and $\bm -\bm'=0$.
Hence, we assume that  $x_{(\uu,\uv)} \ne x_{(\uu',\uv')}$. Then, we have
$$f p_{\uu} p_{\uv}=g  p_{\uu'} p_{\uv'}.$$
Because  $x_{(\uu,\uv)}$ and $ x_{(\uu',\uv')}$
are two distinct $\vr$-variables of $\PP_{F_i}$,  one checks from the definition that the two sets
$$\{ p_{\uu}, p_{\uv} \}, \; \{p_{\uu'}, p_{\uv'} \}$$ are disjoint. 
Consequently, 
$$  p_{\uu}p_{\uv} \mid g, \;\; p_{\uu'} p_{\uv'}  \mid f.$$
Write $$ g=g_1 p_{\uu}p_{\uv} , \;\;  f=f_1 p_{\uu'} p_{\uv'} .$$
Then we have 
$$ p_{\uu}p_{\uv} p_{\uu'} p_{\uv'}  (f_1-g_1)=0 \in R_0.$$
Hence, $f_1=g_1$. Then, we have
$$\bm -\bm' =h ( p_{\uu'} p_{\uv'}   x_{(\uu,\uv)} -  p_{\uu}p_{\uv}  x_{(\uu',\uv')}) $$ 
where $h:=f_1=g_1$.
Observe that in such a case, we have that
 $\bm -\bm'$ is generated by the relations of \eqref{tildeBk}, and  it verifies all
 the statements in the lemma.
 
Suppose Lemma \ref{ker-phi-k} holds  for  $\deg_{\vr} < e$ for some positive integer $e >1$.

Consider $\deg_{\vr}(\bm-\bm')=e$.

By the multi-homogeneity of $(\bm -\bm')$, we can write 
\begin{equation}\label{forXs} 
\bm -\bm' = \bn X_s - \bn' X_t 
\end{equation}
such that $X_s$ and  $X_t$ 
are the $\vr$-variables of $\PP_{F_i}$  corresponding to 
some $s, t \in S_{F_i}$ for some $i \in [k]$, and $\bn, \bn' \in R_{[k]}$.

If $s=t$, then $\bm -\bm' = X_s (\bn - \bn') $. Hence, the statement follows from
the inductive assumption applied to $(\bn -\bn') \in  \ker^\mh \vi_{[k]}$
since $\deg_\vr (\bn -\bn')=e-1$.


We suppose now $s \ne t$.
Let $\bar f=\bn x_{\uu_s} x_{\uv_s} - \bn' x_{\uu_t} x_{\uv_t}$.  Then,
 $\bar f \in \ker^\mh \vi_{[k]}$. 
 
First, we suppose $\bar f=0$. 

Then, $x_{\uu_s} x_{\uv_s} \mid \bn'$
 and $x_{\uu_t} x_{\uv_t} \mid \bn$. Hence, we can write
 $$\bn'=x_{\uu_s} x_{\uv_s} \bn'_0, \; \bn=x_{\uu_t} x_{\uv_t} \bn_0.$$
And we have,
$$\bar f= \bn x_{\uu_s} x_{\uv_s} - \bn' x_{\uu_t} x_{\uv_t}
=x_{\uu_s} x_{\uv_s} x_{\uu_t} x_{\uv_t}  (\bn_0 - \bn'_0).$$
 Hence, one sees that $ \bn_0 - \bn'_0 \in \ker^\mh \vi_{[k]}$.
 If $\bn_0 - \bn'_0=0$, then $h=\bn_0=\bn'_0$ is
 the maximal common factor of $\bm$ and $\bm'$, and further,
 $$\bm -\bm' = h( x_{\uu_t} x_{\uv_t}  X_s -  x_{\uu_s} x_{\uv_s} X_t).$$
  In such a case, the statement of the lemma holds.
 
  Hence, we can assume that $0 \ne \bn_0 - \bn'_0 \in \ker^\mh \vi_{[k]}$,
  in particular, this implies that $ \deg_\vr (\bn_0 -\bn'_0) > 0$.
  Then,   we have 
  $$\bm -\bm' = \bn_0 x_{\uu_t} x_{\uv_t}  X_s - \bn'_0 x_{\uu_s} x_{\uv_s} X_t.$$
   Observe here that 
$x_{\uu_t} x_{\uv_t}  X_s -  x_{\uu_s} x_{\uv_s} X_t$ is in the form of
\eqref{1st-Hq}, verifying the conditions (1) - (3).
 Thus, in such a case, the statement of the lemma follows by applying
the inductive assumption  to $(\bn_0 -\bn'_0) \in  \ker^\mh \vi_{[k]}$
since $\deg_\vr (\bn_0 -\bn'_0)=e-1$.

 Next, we suppose $\bar f \ne 0$.  
 
 As $\deg_{\vr} \bar f <e$, by the inductive assumption,  
 we can write
$$\bn x_{\uu_s} x_{\uv_s}= h (x_{\uu_s} x_{\uv_s} \bn_s) \prod_{i=1}^{\l} \bn_i,\;\
\bn' x_{\uu_t} x_{\uv_t}=h  (x_{\uu_t} x_{\uv_t}\bn_t) \prod_{j=1}^{\l} \bn'_{j}$$
for some integer ${\l} \ge 1$, with $\bn_0=x_{\uu_s} x_{\uv_s}\bn_s$
and $\bn_0'=(x_{\uu_t} x_{\uv_t}\bn_t)$
such that  for each $0\le i\le {\l}$, it determines (matches) a unique $0\le i'\le {\l}$
such that $(\bn_i - \bn_{i'}')$ is of the form of \eqref{1st-Hq}  and
verifies all the properties (1) - (3)  of  Lemma \ref{ker-phi-k}.
Consider $\bn_0=x_{\uu_s} x_{\uv_s} \bn_s$. It matches $\bn_{0'}'$. By
the multi-homogeneity of $\bn_0 - \bn_{0'}'$, 
\eqref{1st-Hq} and  (1) of  Lemma \ref{ker-phi-k},
 we can write $\bn_{0'}' = x_{\uu_{t'}} x_{\uv_{t'}}\bn_{t'}$
for some $t' \in S_{F_i}$ and $\bn_{t'} \in R_{[k]}$.
 Therefore, by switching $t$ with $t'$ if $t \ne t'$,
and re-run the above arguments, without loss of generality, we can assume $t'=t$
and $\bn_0=x_{\uu_s} x_{\uv_s}\bn_s$ matches
 $\bn_0'=(x_{\uu_t} x_{\uv_t}\bn_t)$. Further, by re-indexing $\{\bn_j' \mid j \in [{\l}]\}$ if necessary,
we can assume that  $\bn_i$ matches $\bn_i'$ for all $1\le i\le l$.

Now, note that we have
$\bn = h (\bn_s) \prod_{i=1}^{\l} \bn_i, \; \bn' =h  (\bn_t') \prod_{i=1}^{\l} \bn_i$. Hence
$$\bm = h (\bn_s X_s) \prod_{i=1}^{\l} \bn_i, \;\; \bm' =h  (\bn_t' X_t) \prod_{i=1}^{\l} \bn_i.$$
We let $\bm_0=\bn_s X_s$ and  $\bm_i = \bn_i$ for all $i \in [{\l}]$;
 $\bm_0'=(n_t X_t)$ and  $\bm_i= \bn_i'$ for all $i \in [{\l}]$.
Then, one checks directly that
  Lemma \ref{ker-phi-k} holds for $\bm -\bm'$.

This proves the lemma.
\end{proof}

\begin{defn}\label{hatB}
Let $\widehat{\cB}_{[k]}$
be the set of  all binomial relations of \eqref{1st-Hq} that verify Lemma \ref{ker-phi-k} (1) - (3);
 let $\widetilde{\cB}_{[k]}$ 
the de-homogenizations 
with respect to $(p_\um \equiv 1)$ of all binomial relations of $\widehat{\cB}_{[k]}$.
\end{defn}


\begin{cor}\label{cB-generate}
The ideal  $\ker^\mh \vi_{[k]}$ is generated by $\widehat{\cB}_{[k]}$.
 Consequently, the ideal  $\ker^\mh \bar\vi_{[k]}$ is generated by $\widetilde{\cB}_{[k]}$.
\end{cor}
\begin{proof} Take any binomial
$(\bm -\bm') \in \ker^\mh \vi_{[k]}$
 with $\deg_\vr (\bm -\bm') > 0$. We express, by Lemma \ref{ker-phi-k},
$$\bm -\bm'=h (\prod_{i=1}^{\l} \bm_i -  \prod_{i=1}^{\l} \bm_i')$$ 
such that $\bm_i -\bm'_i \in \widehat{\cB}_{[k]}$  for all $i \in [{\l}]$.
Then, we have
$$\bm -\bm'=h (\prod_{i=1}^{{\l}} \bm_i - \bm_{\l}'  \prod_{i=1}^{{\l}-1} \bm_i
+ \bm_{\l}'  \prod_{i=1}^{{\l}-1} \bm_i -\prod_{i=1}^{{\l}} \bm_i')$$ 
$$=h ((\bm_{\l} - \bm_{\l}')  \prod_{i=1}^{{\l}-1} \bm_i
+ \bm_{\l}'  (\prod_{i=1}^{{\l}-1} \bm_i -\prod_{i=1}^{{\l}-1} \bm_i')).$$ 
Thus, by a simple induction on the integer ${\l}$, the corollary follows.
\end{proof}

We now let  $\Theta_{[k],\Gr}$ be the restriction of  $\Theta_{[k]}$ to  $\Gr^{d,E}$: 
\begin{equation}\label{theta-k-Gr}
 \xymatrix{
\Theta_{[k], \Gr}: \Gr^{d,E} \ar @{-->}[r]  & \prod_{i \in [k]} \PP_{F_i}  } 
\end{equation}
$$ [p_\uu]_{\uu \in \II_{d,n}} \lra  
\prod_{i \in [k]}  [p_\uu p_\uv]_{(\uu,\uv) \in \La_{F_i}}.
$$   
We let \begin{equation}\label{tA'-gr}
 \xymatrix{
\Gr_{\sF_{[k]}}  \ar @{^{(}->}[r]  & \Gr^{d,E} \times  \prod_{i \in [k]} \PP_{F_i}  
\ar @{^{(}->}[r]  & \PP(\wedge^d E) \times  \prod_{i \in [k]} \PP_{F_i}   }
\end{equation}
be the closure of the graph of the rational map $\Theta_{[k],\Gr}$.

We let  $\bar\Theta_{[k],\Gr}$ be the restriction of  $\bar\Theta_{[k]}$
(or equivalently, $\Theta_{[k], \Gr}$) to the chart  $\rU_\um(\Gr)=\rU_\um\cap \Gr^{d,E}$: 
\begin{equation}\label{bar-theta-k-Gr}
 \xymatrix{
\bar\Theta_{[k], \Gr}: \rU_\um(\Gr)\ar @{-->}[r]  & \prod_{i \in [k]} \PP_{F_i}  } 
\end{equation}
$$
 [x_\uu]_{\uu \in \II_{d,n}} \lra  
\prod_{i \in [k]}  [x_\uu x_\uv]_{(\uu,\uv) \in \La_{F_i}}.
$$   
We let \begin{equation}\label{bar-tA'-gr}
 \xymatrix{
\sV_{\um, \sF_{[k]}}  \ar @{^{(}->}[r]  & \rU_\um(\Gr) \times  \prod_{i \in [k]} \PP_{F_i}  
\ar @{^{(}->}[r]  & \sR_{\sF_{[k]}}= \rU_\um \times  \prod_{i \in [k]} \PP_{F_i}   }
\end{equation}
be the closure of the graph of the rational map $\bar\Theta_{[k],\Gr}$.

Then, one sees that $\sV_{\um, \sF_{[k]}}$ is the proper transform 
of $\rU_\um(\Gr)$ in $\rU_{\um,\sF_{[k]}}$ under the birational morphism 
$\rU_{\um,\sF_{[k]}} \lra \rU_\um.$

Since we always focus on the fixed chart $\rU_\um$ below, we write 
$\sV_{\sF_{[k]}}=\sV_{\um, \sF_{[k]}}$.

By construction, there exists the natural forgetful map
\begin{equation}\label{forgetfulMap}
\sR_{\sF_{[k]}}  \lra \sR_{\sF_{[k-1]}} \end{equation}  
 and it  induces a birational morphsim 
  \begin{equation}\label{rho-sFk} \rho_{\sF_{[k]}}: \sV_{\sF_{[k]}} \lra \sV_{\sF_{[k-1]}}.
  \end{equation}

  Corresponding to the embedding 
$\Gr_{\sF_{[k]}}  \subset \PP(\wedge^d E) \times  \prod_{i \in [k]} \PP_{F_i}$
of \eqref{tA'-gr}, we have the following homomorphism 
\begin{equation}\label{vi-k-Gr}
\vi_{[k], \Gr}: \; R_{[k]} \; (\lra R_0) \lra R_0/I_\whwp, \;\;\; x_{(\uu_s,\uv_s)} \to p_{\uu_s} p_{\uv_s} 
\end{equation} 
for all  $s \in S_{F_i}, i \in [k]$, where $I_\whwp$ is the homogeneous ideal of $\Gr^{d,E}$.

Corresponding to the embedding 
 $\sV_{\um, \sF_{[k]}} \subset \sR_{\sF_{[k]}}$ of \eqref{bar-tA'-gr}, 
 we have the following homomorphism
   \begin{equation}\label{bar-vi-k-gr}
\bar\vi_{[k],\Gr}: \; \bar R_{[k]} \; (\lra \bar R_0) \lra \bar R_0/\bar I_\wp, \;\;\; x_{(\uu_s,\uv_s)} \to x_{\uu_s} x_{\uv_s} 
\end{equation} 
for all  $s \in S_{F_i}, i \in [k]$, where $\bar I_\wp$ is the de-homogenization of the $\pl$ ideal $I_\wp$.

  \begin{lemma}\label{defined-by-ker-Gr} 
The scheme $\Gr_{\sF_{[k]}}$, as a closed subscheme of 
 $\PP(\wedge^d E) \times  \prod_{i \in [k]} \PP_{F_i}$,
 is defined by $\ker^\mh \vi_{[k], \Gr}$. In particular, 
The scheme $\sV_{\um,\sF_{[k]}}$, as a closed subscheme of 
 $\sR_{\sF_{[k]}}= \rU_\um \times  \prod_{i \in [k]} \PP_{F_i}$,
 is defined by $\ker^\mh \bar\vi_{[k], \Gr}$, where $\bar\vi_{[k]}$
 is the de-homogenization of $\vi_{[k]}$ with respect to $\rU_\um=(p_\um \equiv 1)$.
\end{lemma}
\begin{proof} This is immediate.
\end{proof}

We need to investigate   
$\ker^\mh \bar\vi_{[k], \Gr}$.

 We let $f \in \bar R_{[k]}$ be any multi-homogenous polynomial 
 and $\bar f \in \bar R_{[k]}$ be its de-homogenization with respect to the chart $\rU_\um=(p_\um \equiv 1)$
 such that $\bar\vi_{[k], \Gr}(\bar f)=0$. Then, by \eqref{bar-vi-k-gr}, and \eqref{vi-k-Gr},
 it holds  if and only if $\bar \vi_{[k]} (\bar f) \in \bar I_\wp$  if and only if $\vi_{[k]} (f) \in I_\wp$.
 Thus, we can express $f=\sum_{F \in \sfm} f_F$ such that 
 $\vi_{[k]} (f_F)$ is a multiple of $F$ for all $\bF \in \sfm$. It suffices to consider an
 arbitrarily fixed $\bF \in \sfm$. Hence, we may assume that $f=f_F$ for some arbitrarily fixed $\bF \in \sfm$.
 That is, in such a case,  $\bar\vi_{[k], \Gr} (\bar f)=0$ if and only if
  $$\hbox{$ \vi_{[k]} (f) =h F$  for some $h \in R_{[k]}$.}$$

\begin{defn} \label{defn:linear-pl} Given any $\bF \in \sfm$, written as 
$\bF =\sum_{s \in S_F} \sgn (s) x_{\uu_s}x_{\uv_s}$, we introduce
$$L_F: \;\; \sum_{s \in S_F} \sgn (s) x_{(\uu_s,\uv_s)} .$$ 
This is called the linearized $\pl$ relation with respect to $\bF$ (or $F$). It is a
canonical linear relation on $\PP_F$.
\end{defn}

  Observe here that among all linearized $\pl$ relations, only  
  $L_{F_i}$ with $i \in [k]$ belong to $\bar R_{[k]}$.
    
  \begin{lemma}\label{reduce-to-F-LF}
  Fix any $\bF \in \sfm$.  Let $f \in \bar R_{[k]}$ be any multi-homogenous polynomial such that 
$\vi_{[k]} (f) =h F$  for some $h \in R_0$.
Then, modulo $\ker^\mh \vi_{[k]}$,  either $f \equiv \bn F$  or $f \equiv \bn L_F$ 
for some $\bn \in \bar R_{[k]}$. 
 \end{lemma}
  \begin{proof}  
   First, by writing $h$ as the sum of some monomials, 
   and then expressing $f$ as a sum accordingly, it suffices to consider the case
   when $h$ is a monomial.

 We write $F=\sum_{s\in S_F} \sgn (s) p_{\uu_s}p_{\uv_s}$. Accordingly, we express
  $$f=\sum_{s \in S_F} \sgn (s) f_s$$ such that 
  $$ \hbox{$\vi_{[k]} (f_s) =h p_{\uu_s}p_{\uv_s}$, for all $s \in S_F$.}$$   
  
  We claim for any $s \in S_F$,  either $x_{(\uu_s, \uv_s)} \mid f_s$ or $p_{\uu_s}p_{\uv_s} \mid f_s$.
  
  Assume for some $s \in S_F$, the claim does not hold. Then, (at least) one of the factors of 
  $p_{\uu_s}p_{\uv_s}$ in $\vi_{[k]} (f_s)$ , w.l.o.g., say,  $p_{\uu_s}$  can only come
  from the $\vi_{[k]}$-image of $x_{(\uu_s, \uv_s')}$ for some  $\uv_s' \ne \uv_s$, and hence,
  $p_{\uv_s'}$ must belong to the common factor $h$. Then, we compare $f_s$ with $f_t$
  for any $t \ne s$. Using \eqref{1st-Hq} in Lemma \ref{ker-phi-k}, we see that the $\vi_{[k]}$-image
    $p_{\uu_s}p_{\uv_s'}$ of $x_{(\uu_s, \uv_s')}$ must also belong to $h$.  This makes
    $p_{\uu_s} \mid (\vi_{[k]} (f_s)/h) $ impossible. Hence, the claim holds.

 Now, we first suppose that $f$ does not contain any homogeneous coordinates of  $\PP_F$.
  Then by the claim, we can express $f_s= \bn_s p_{\uu_s}p_{\uv_s}$ for all $s$.
Because $\vi_{[k]} (f_s) = \vi_{[k]} (\bn_s) p_{\uu_s}p_{\uv_s}=h    p_{\uu_s}p_{\uv_s}$,
  it implies that  $\vi_{[k]} (\bn_s)= h$ for all $s \in S_F$.
  Therefore, modulo $\ker^\mh \vi_{[k]}$,  we have
   $$f =\sum_{s \in S_F} \bn_s \; \sgn (s) p_{\uu_s}p_{\uv_s} \equiv \bn F$$ for some $\bn \in R_{[k]}$.
   (For example, one can take $\bn=\bn_s$ for any fixed $s \in S_F$.)

Next, we  suppose that $f$ is nontrivially homogeneous in $\PP_F$, that is, it contains
  some $\vr$-variables of $\PP_F$. Consider any $s \in S_F$.
  Suppose $x_{(\uu_s,\uv_s)}  \mid f_s$, then we 
  can write $f_s=\bn_s   x_{(\uu_s,\uv_s)}$ for some monomial $\bn_s \in R_{[k]}$.
    Suppose $x_{(\uu_s,\uv_s)}  \nmid f_s$.
    Then, by homogeneity, we can write
    $f_s=\bm_s  x_{(\uu_t,\uv_t)}$  for some $t \in S_F$ with $t \ne s$
    and $\bm_s \in R_{[k]}$. Further, by the above claim,
    $f_s=\bm'_s  p_{\uu_s}p_{\uv_s}  x_{(\uu_t,\uv_t)}$.
Then, modulo the relation
 $p_{\uu_s}p_{\uv_s}  x_{(\uu_t,\uv_t)} - p_{\uu_t}p_{\uv_t}  x_{(\uu_s,\uv_s)},$
 we can also express 
 $$ f_s \equiv \bn_s x_{(\uu_s,\uv_s)}, \; \mod (\ker^\mh \vi_{[k]})$$
 where $\bn_s= \bm_s' p_{\uu_t}p_{\uv_t}$.
 Again, because $\vi_{[k]} (f_s) = \vi_{[k]} (\bn_s) p_{\uu_s}p_{\uv_s}=h    p_{\uu_s}p_{\uv_s}$,
  we must have that
  $\vi_{[k]} (\bn_s)= h$ for all $s \in S_F$.   This proves that  
   $ f_s \equiv \bn \;  x_{(\uu_s,\uv_s)}$ for some $\bn \in R_{[k]}$
   for all $s \in S_F$, modulo $\ker^\mh \vi_{[k]}$. 
   Consequently, modulo $\ker^\mh \vi_{[k]}$, we obtain
  $$ f=\sum_{s \in S_F} \sgn (s) f_s\equiv \bn \sum_s \sgn (s) x_{(\uu_s,\uv_s)} = \bn L_F.$$ 



   This proves the lemma.  
   \end{proof}

 \begin{cor}\label{cB-LF-generate}
The ideal  $\ker^\mh \bar\vi_{[k],\Gr}$ is generated by all the 
relations in $\widetilde{\cB}_{[k]}$, $\sfm$, and $\{L_{F_i} \mid i \in [k]\}$.
\end{cor}  
\begin{proof}
This follows from the combination of Corollary \ref{cB-generate} and 
Lemma \ref{reduce-to-F-LF}.
\end{proof}

  \begin{defn}\label{defn:pre-q} 
We let $\cB_{i}$  (resp.  $\cB_{[k]}$) be the set of all binomial relations in
 \eqref{trivialBk} for any fixed $i \in [k]$ (resp. for all $i \in [k]$).
We set $\cB^\pq_{[k]}= \widetilde{\cB}_{[k]} \- \cB_{[k]}$.
An element of $\cB^\pq_{[k]}$ is called a binomial  of pre-quotient type.
\end{defn}

\begin{lemma}\label{equas-for-sVk}
 The scheme $\sV_{\sF_{[k]}}$, as a closed subscheme of
$\sR_{\sF_{[k]}}= \rU_\um \times  \prod_{i=1}^k \PP_{F_i} $,
is defined by the following relations
\begin{eqnarray}
\;\;\;\;\;\;\;\;\;\;\; B_{F_i,(s,t)}: \;\;\;  x_{(\uu_s, \uv_s)}x_{\uu_t} x_{\uv_t}- x_{(\uu_t,\uv_t)}   x_{\uu_s} x_{\uv_s},
\; \forall \;\; s, t \in S_{F_i} \- s_{F_i},    i \in [k], \label{eq-Bres-lek'}\\
B_{F_i,(s_{F_i},s)}: \;\; x_{(\uu_s, \uv_s)}x_{\uu_{F_i}} - x_{(\um,\uu_{F_i})}   x_{\uu_s} x_{\uv_s}, \;\;
\forall \;\; s \in S_{F_i} \- s_{F_i},  \; i \in [k], \label{eq-B-lek'}  \\
\cB^\pq_{[k]},  \;\; \;\; \;\;\;  \;\; \;\; \;\; \;\; \; \;\;\; \;\; \;\; \;\; \;\; \;\; \;\; \;\; \;\; \;\; \;\;\;\; \;\;\;\;\; \label{eq-hq-lek}\\
L_{F_i}: \;\; \sum_{s \in S_{F_i}} \sgn (s) x_{(\uu_s,\uv_s)},   \; i \in [k], \;\;\;\;\; \label{linear-pl-lek'} \\
\bF_j: \;\; \sum_{s \in S_{F_j}} \sgn (s) x_{\uu_s}x_{\uv_s}, \; \;   k < j\le \up.  
\label{linear-pl-gek}
\end{eqnarray}
where $\bF_i$ is 
expressed as $\sgn (s_{F_i}) x_{\uu_{F_i}} +\sum_{s \in S_{F_i} \- s_{F_i}} \sgn (s) x_{\uu_s}x_{\uv_s}$
for every $i \in [k]$.
\end{lemma}
\begin{proof}  
By Lemma \ref{defined-by-ker-Gr} and Corollary \ref{cB-LF-generate},
we have that  $\sV_{\sF_{[k]}}$, as a closed subscheme of
$\sR_{\sF_{[k]}}= \rU_\um \times  \prod_{i=1}^k \PP_{F_i} $,
is defined by
$$\hbox{$\cB_{[k]}, \cB^\pq$,  $\sfm$,  and $L_F$ with  $F=F_i$ for all $i \in [k]$.}$$
Here note that $ \widetilde{\cB}_{[k]}=\cB_{[k]} \sqcup \cB^\pq_{[k]}$
and $\cB_{[k]}$  is precisely made of \eqref{eq-Bres-lek'} and \eqref{eq-B-lek'}.

It suffices to show that under the presence of \eqref{eq-Bres-lek'} and  \eqref{eq-B-lek'}, 
 $\bF_i$  can be reduced to $L_{F_i}$ for all $i \in [k]$.

Fix any $i \in [k]$. Take any $s \in S_{F_i}$. 
Consider the binomial relations of $\cB_i$
  \begin{equation}\label{Buv-s-1st}
  x_{(\uu, \uv)}x_{\uu_s} x_{\uv_s} - x_{\uu} x_{\uv} x_{(\uu_s,\uv_s)}, 
  \end{equation}
for all $(\uu, \uv) \in \La_{F_i}$ (cf. \eqref{LaF}). By multiplying $\sgn (s)$ to  \eqref{Buv-s-1st}
and adding together all the resulted binomials,  
we obtain, \begin{equation}\label{Fi=Li-1st}
  x_{\uu_s} x_{\uv_s} L_{F_i} =x_{(\uu_s,\uv_s)} \bF_i \;, \; \mod (\langle \cB_i \rangle),
 \end{equation}
 where $\langle \cB_i \rangle$ is the ideal generated by the relations in $\cB_i$.
 Since $\PP_{F_i}$ can be covered by affine open charts 
 $(x_{(\uu_s,\uv_s)} \ne 0)$, we conclude that  $\bF_i$ depends on $L_{F_i}$
 and can be discarded for all $i \in [k]$.
\end{proof}

For conciseness, we set the following
$$\sV_{\um}:=\sV_{\um,\sF_{[\up]}}, \; \rU_{\um, \sF}:=\rU_{\um, \sF_{[\up]}},
\; \sR_\sF:=\sR_{\sF_{[\up]}}.$$
Then, we have the following 
diagram 
$$ \xymatrix{
\sV_{\um} \ar[d] \ar @{^{(}->}[r]  & \rU_{\um,\sF} \ar[d] \ar @{^{(}->}[r]  &
\sR_\sF= \rU_\um  \times \prod_{\bF \in \sfm} \PP_F \ar[d] \\
\rU_\um(\Gr)  \ar @{^{(}->}[r]  & \rU_\um  \ar @{=}[r]  & \rU_\um.}$$

 In what follows, we will sometimes write
$\sV$ for $\sV_\um$, as we will exclusively focus on the chart $\rU_\um$, 
throughout, unless otherwise stated.

We also set
$$\cB^\pq=\cB^\pq_{[\up]}.$$

By the case of Lemma \ref{equas-for-sVk} when $k=\up$, we have

\begin{cor}\label{eq-tA-for-sV}  The scheme $\sV_\um$, as a closed subscheme of
$\sR_\sF= \rU_\um \times  \prod_{\bF \in \sF_\um} \PP_F$,
is defined by the following relations
\begin{eqnarray} 
B_{F,(s,t)}: \;\; x_{(\uu_s, \uv_s)}x_{\uu_t}x_{ \uv_t}-x_{(\uu_t, \uv_t)}x_{\uu_s}x_{ \uv_s}, \;\; \forall \;\; 
s, t \in S_F \- s_F  \label{eq-Bres}\\
B_{F,(s_F,s)}: \;\; x_{(\uu_s, \uv_s)}x_{\uu_F} - x_{(\um,\uu_F)}   x_{\uu_s} x_{\uv_s}, \;\;
\forall \;\; s \in S_F \- s_F, \;\; \label{eq-B} \\ 
\cB^\pq,    \;\;\;\;\;\;\;\; \;\; \; \label{eq-hq}\\
L_F: \; \sum_{s \in S_F} \sgn (s) x_{(\uu_s,\uv_s)}, \label{linear-pl} 
\end{eqnarray}
for all $\bF \in \sfm$ with $\bF$ being expressed as
 $\sgn (s_F) x_{\uu_F} +\sum_{s \in S_F \- s_F} \sgn (s) x_{\uu_s}x_{\uv_s}$.
\end{cor}

\begin{defn}\label{defn:main-res} A binomial equation $B_{F,(s_F,s)}$ of \eqref{eq-B}
with $s \in S_F \setminus s_F$ is called a main binomial equation. We let
$$\cB^\mn_F=\{B_{F,(s_F,s)} \mid s \in S_F \setminus s_F  \}
\;\; \and \;\; \cB^\mn=\sqcup_{\bF \in \sfm} \cB^\mn_F.$$
A binomial equation $B_{F,(s,t)}$ of \eqref{eq-Bres}
with $s, t \in S_F \- s_F$ and $s \ne t$ is called a residual binomial equation. We let
$$\cB^\res_F=\{B_{F,(s,t )} \mid s, t \in S_F \setminus s_F  \}
\;\; \and \;\;  \cB^\res=\sqcup_{\bF \in \sfm} \cB^\res_F.$$
Recall that an element of  $\cB^\pq$ is called a binomial relation of pre-quotient type.
\end{defn}


We observe here that
\begin{equation}\label{dim}
\dim (\prod_{\bF \in \sfm} \PP_F) = \sum_{\bF \in \sfm} |S_F\- s_F| 
= \sum_{\bF \in \sfm} |\cB_F^\mn|=|\cB^\mn|,
\end{equation}
where $|K|$ denotes the cardinality of a finite set $K$.

\begin{defn}\label{defn:block}
We let $\fG_F=\cB^\mn \sqcup \{L_F\}$. We call it the block of (a part of)
defining relations. 
\end{defn}
We let
$$\fG=\bigsqcup_{\bF \in \sfm} \fG_F.$$

\subsection{$\vp$-divisors, $\vr$-divisors, and 
$\fL$-divisors of $\sR_\sF$} $\ $

From earlier, we have the set
$\La_F=\{(\uu_s, \uv_s) \mid s \in S_F\}.$
This is an index set for  the homogeneous coordinates of the projective
space $\PP_F$, and  is also an index set for all the variables that appear in 
the linearized $\pl$ equation $\bL_F$ of \eqref{linear-pl}. To be used later, 
 we also set  $$\La_{\sF_\um} =\sqcup_{\bF \in \sF_\um} \La_F.$$

\begin{defn} \label{-divisor}
Consider the scheme $\sR_\sF =\rU_\um \times  \prod_{\bF \in \sfm} \PP_F$.

Recall that the affine chart $\rU_\um$ comes equipped with the coordinate variables 
$\{x_\uu\}_{\uu \in \II_{d,n}\- \um}$.
For any $\uu \in \II_{d,n}\- \um$, we set
$$X_\uu:=(x_\uu =0) \subset \sR_\sF.$$
We call $X_\uu$ the $\pl$ divisor, in short, the $\vp$-divisor,  of $\sR_\sF$ associated with $\uu$.
We let $\cD_\vp$ be the set of all $\vp$-divisors on the scheme $\sR_\sF$.

In addition to the $\vp$-divisors,  the scheme $\sR_\sF$  
comes equipped with the divisors
$$X_{(\uu, \uv)}:=(x_{(\uu, \uv)}=0)$$
for all $(\uu,\uv) \in \La_{\sF_\um}$.
We call $X_{(\uu, \uv)}$ the $\vr$-divisor corresponding to $(\uu, \uv)$.  
We let $\cD_{\vr}$ be the set  of all $\vr$-divisors of $\sR_\sF$. 

Further, the scheme $\sR_\sF$ also 
comes equipped with the divisors
$$D_{L_F} := (L_F=0)$$
for all $\bF \in \sfm$. We call $D_{L_F}$ the $\fL$-divisor corresponding to $F$.
We let $\cD_{\fL}$ be the set of all $\fL$-divisors of $\sR_\sF$.
\end{defn}

\begin{defn}\label{fv-k=0} Fix $k \in [\up]$.
For every $\bF_i \in \sF_\um$ with $i \in [k]$, choose and fix an arbitrary element 
$s_{F_i, o}\in S_{F_i}$.
Then, the scheme $\sR_{\sF_{[k]}}$ is covered by the affine open charts
of the form  $$\rU_\um \times \prod_{i \in [k]}
(x_{(\uu_{s_{F_i, o}}, \uv_{s_{F_i, o}})} \equiv 1)
\subset \sR_{\sF_{[k]}}= \rU_\um \times  \prod_{i \in [k]} \PP_{F_i} .$$ 
We call such an affine open subset a standard chart of
$\sR_{\sF_{[k]}}$, often denoted by $\fV$.
\end{defn}
Fix any standard chart $\fV$ as above. We let 
$$\fV'=\rU_\um \times \prod_{i \in [k-1]}
(x_{(\uu_{s_{F_i, o}}, \uv_{s_{F_i, o}})} \equiv 1)
\subset \sR_{\sF_{[k-1]}}= \rU_\um \times  \prod_{i \in [k-1]} \PP_{F_i} .$$ 
Then, this is a standard chart of $\sR_{\sF_{[k-1]}}$, uniquely determined by $\fV$.
We say $\fV$ lies over $\fV'$. In general, 
suppose $\fV''$ is a standard chart of $\sR_{\sF_{[j]}}$ with $j <k-1$. Via induction,
we say $\fV$ lies over $\fV''$ if $\fV'$ lies over $\fV''$.

Note that the standard chart $\fV$ of
$\sR_{\sF_{[k]}}$ in the above definition is uniquely indexed 
by the set 
\begin{equation}\label{index-sR}
 \La_{\sF_{[k]}}^o=\{(\uu_{s_{F_i,o}},\uv_{s_{F_i,o}}) \in \La_{F_i} \mid i \in [k] \}.
 \end{equation}
Given $ \La_{\sF_{[k]}}^o$, we let 
$$ \La_{\sF_{[k]}}^\star=(\bigcup_{i \in [k]}\La_{F_i}) \- \La_{\sF_{[k]}}^o.$$
We set $\La_\sfm^o:=\La_{\sF_{[\up]}}^o$ and $\La_\sfm^\star:=\La_{\sF_{[\up]}}^\star$.

To be cited as the initial cases of certain inductions later on,
we need the following two propositions.

\begin{prop}\label{meaning-of-var-p-k=0} Consider any standard
 chart $$\fV=\rU_\um \times \prod_{i \in [k]}
(x_{(\uu_{s_{F_i, o}}, \uv_{s_{F_i, o}})} \equiv 1)$$ of $\sR_{\sF_{[k]}}$, 
indexed by  $\La_{\sF_{[k]}}^o$ as above.
It  comes equipped with the set of free variables
$$\var_\fV=\{x_{\fV, \uw}, \; x_{\fV, (\uu,\uv)} \mid \uw \in \II_{d,n} \- \um, \; 
(\uu,\uv) \in \La_{\sF_{[k]}}^\star
\}$$ 
and the de-homogenized linearized $\pl$ relations $L_{\fV, F}$ for all $\bF \in \sfm$
such that on the standard chart $\fV$, we have
\begin{enumerate}
\item the divisor  $X_{ \uw}\cap \fV$ is defined by $(x_{\fV,\uw}=0)$ for every 
$\uw \in \II_{d,n} \setminus \um$;
\item the divisor  $X_{(\uu,\uv)}\cap \fV$ is defined by $(x_{\fV,(\uu,\uv)}=0)$ for every 
$(\uu,\uv) \in \La_{\sF_{[k]}}^\star.$     
\item the divisor  $D_{L_F}\cap \fV$ is defined by $(L_{\fV,F}=0)$ for every 
$\bF \in \sfm$.
\end{enumerate}
\end{prop}
\begin{proof} 
Recall  that $\rU_\um=(p_\um \equiv 1)$.
Then, we let $x_{\fV, \uw}=x_\uw$ for all  $\uw \in \II_{d,n} \- \um$.
Now consider every 
$i \in [k]$.
Upon setting $x_{(\uu_{s_{F_i, 0}}, \uv_{s_{F_i, 0}})} \equiv 1$, we let
 $x_{\fV, (\uu_s,\uv_s)} = x_{(\uu_s,\uv_s)}$ be the de-homogenization of $x_{(\uu_s,\uv_s)}$
 for all  $s  \in S_{F_i} \- s_{F_i,o}$.
From here,  the statement is straightforward to check.
\end{proof}

\begin{prop}\label{equas-fV[k]}  Let the notation be as in
Propsotion \ref{meaning-of-var-p-k=0}. Then,
the scheme $\sV_{\sF_{[k]}} \cap \fV$, as a closed subscheme of $\fV$
is defined by the following relations
\begin{eqnarray}
\;\;\;\;\; x_{\fV,(\uu_s, \uv_s)}x_{\fV,\uu_t} x_{\fV,\uv_t}- x_{\fV, (\uu_t,\uv_t)}   x_{\fV,\uu_s} x_{\fV, \uv_s},
\; \forall \;\; s, t \in S_{F_i} \- s_{F_i},    \; i \in [k], \label{eq-Bres-lek'-2}\\
\;\; \;\; x_{\fV, (\uu_s, \uv_s)}x_{\fV,\uu_{F_i}} - x_{\fV,(\um,\uu_{F_i})}   
x_{\fV,\uu_s} x_{\fV,\uv_s}, \;\;
\forall \;\; s \in S_{F_i} \- s_{F_i},  \; i \in [k], \label{eq-B-lek'-2}  \\
\cB^\pq_{\fV, [k]}, \; \;\; \;\; \;\; \;\; \;\; \;\; \;\; \;\; \;\; \;\; \;\; \;\; \;\; \;\;\; \;\; \;\; \;\; \;\; \;\; \;\; \;\; \;\; \;\; \;\; \;\; \;\; 
\label{eq-hq-lek-2}\\
L_{\fV, F_i}: \;\; \sum_{s \in S_{F_i}} \sgn (s) x_{\fV, (\uu_s,\uv_s)},   \; i \in [k],  \;\; \;\; \;\; \;\;
\label{linear-pl-lek'-2} \\
\bF_{\fV, j}: \;\; \sum_{s \in S_{F_j}} 
\sgn (s) x_{\fV,\uu_s}x_{\fV,\uv_s}, \; \; k < j\le \up.  \label{linear-pl-gek-2}
\end{eqnarray}
where  the equations of $\cB^\pq_{\fV, [k]}$ are the de-homogenizations of the equations
of $\cB^\pq_{[k]}$.
\end{prop}
\begin{proof} This follows directly from Lemma \ref{equas-for-sVk}.
\end{proof}

For any $f \in R$, we let $\deg_\vp f$ be the degree of $f$ considered as
a polynomial in $\vp$-variables only.

\begin{defn}\label{defn:q} Let $f \in \cB^\pq$ be a binomial relation of pre-quotient type. 
We say $f$ is a binomial relation of quotient type
if  $\deg_\vp f =0$, that is, it does not contain any $\vp$-variable.
We let $\cB^\q$ be the set of all binomial relations of quotient type.
 Fix a standard chart $\fV$ as in Definition \ref{fv-k=0}, we let $\cB^\q_{\fV, [k]}
 \subset \cB^\pq_{\fV, [k]}$ be the subset of all
the de-homogenized  binomial relations of quotient type.
\end{defn}

We write 
$$\cB^\q_{\fV}:= \cB^\q_{\fV,[\up]}, \;\; \cB^\pq_{\fV}:= \cB^\pq_{\fV,[\up]}.$$

We let $R_\vr$ be the subring of $R$ consisting of polynomials with  $\vr$-variables only.
Then, binomial relations of quotient type belong to $R_\vr$.

By Lemma \ref{ker-phi-k}, all binomials of $\cB^\q$ and $\cB^\q_\fV$ 
 are $\vr$-linear,  in particular,  
 they are square-free.

\begin{prop}\label{equas-p-k=0} 
Let the notation be as in Proposition \ref{equas-fV[k]} for $k=\up$.  
Then, the scheme $\sV\cap \fV$, as a closed subscheme of
the chart $\fV$ of $\sR_\sF$,  is defined by 
\begin{eqnarray} 
\;\;\;\;\;\;\;\;\;\;\;\;\;\; B_{ \fV,(s,t)}: \;\; x_{\fV,(\uu_s, \uv_s)}x_{\fV,\uu_t}x_{\fV, \uv_t}-x_{\fV,(\uu_t, \uv_t)}x_{\fV,\uu_s}x_{ \fV,\uv_s}, \;\; \forall \;\; s, t \in S_F \- s_F,  \label{eq-Bres-pk=0}\\
B_{ \fV,(s_F,s)}: \;\;\;\;\;\; x_{\fV, (\uu_s, \uv_s)}x_{\fV, \uu_F} - x_{\fV, (\um,\uu_F)}   x_{\fV, \uu_s} x_{\fV, \uv_s}, \;\;
\forall \;\; s \in S_F \- s_F,  \label{eq-B-k=0} \\ 
\cB^\q_\fV,  \;\;\; \;\;\; \;\; \;\; \;\; \;\; \;\; \;\; \;\; \;\; \;\; \;\; \;\; \;\;\;\;\;\;\;\; \label{eq-hq-k=0}\\
L_{\fV, F}: \;\; \sum_{s \in S_F} \sgn (s) x_{\fV,(\uu_s,\uv_s)} \label{linear-pl-k=0}
\end{eqnarray}
for all $\bF \in \sfm$
with $\bF$ being expressed as $\sgn (s_F) x_{\uu_F} +\sum_{s \in S_F \- s_F} \sgn (s) x_{\uu_s}x_{\uv_s}$.
Here, we set 
$$x_{\fV,\um} \equiv 1; \;\; x_{\fV, (\uu_{s_{F, o}}, \uv_{s_{F, o}})} \equiv 1, \;\; \forall \;\; 
\bF \in \sF_\um,$$
Moreover, every binomial $B_\fV \in \cB^\q_\fV$ is linear in $\vr$-variables, in particular, square-free.
\end{prop}
\begin{proof} By Proposition \ref{equas-fV[k]} for $k=\up$,
 the scheme $\sV\cap \fV$, as a closed subscheme of
the chart $\fV$ of $\sR_\sF$,  is defined by relations in \eqref{eq-Bres-pk=0},
\eqref{eq-B-k=0} , \eqref{linear-pl-k=0},  and $\cB^\pq_\fV$.

It remains to reduce $\cB^\pq_\fV$ to $\cB^\q_\fV$.

We claim that any relation $f$ of $\cB^\pq_\fV$ can be reduced to relations of $\cB^\q_\fV$.

We prove it by induction on $\deg_\vp (f)$.

When  $\deg_\vp (f) =0$, the statement holds trivially.

Assume that statement holds for $\deg_\vp< e$ for some $e>0$.

Consider $\deg_\vp (f) = e$. 

By Lemma \ref{ker-phi-k}, we can write 
$$f= x_{\uu_s} x_{\uv_s} \bn_s - x_{\uu_t} x_{\uv_t} \bn_t$$
for some $s, t \in S_F$ and some $\bF \in \sfm$.
Because on the chart $\fV$, we have
$$x_{\uu_s} x_{\uv_s} - x_{\uu_{s_{F,o}}}x_{\uv_{s_{F,o}}} x_{(\uu_s, \uv_s)}, \;
x_{\uu_t} x_{\uv_t} - x_{\uu_{s_{F,o}}}x_{\uv_{s_{F,o}}} x_{(\uu_t, \uv_t)},$$
where $s_{F,o}$ is as in Definition \ref{fv-k=0} with
$x_{(\uu_{s_{F,o}},\uv_{s_{F,o}})}\equiv 1$,
we get
$$f= x_{\uu_{s_{F,o}}}x_{\uv_{s_{F,o}}}(x_{(\uu_s, \uv_s)} \bn_s - x_{(\uu_t, \uv_t)} \bn_t).$$
Observe that $(x_{(\uu_s, \uv_s)} \bn_s - x_{(\uu_t, \uv_t)} \bn_t) \in \cB^\pq_\fV$.
Since  $\deg_\vp (x_{(\uu_s, \uv_s)} \bn_s - x_{(\uu_t, \uv_t)} \bn_t) < e$, 
the statement then follows from the inductive assumption.
\end{proof}

\begin{defn}\label{all-binomials-10} We let $\cB^\mn_\fV$ (respectively, $\cB^\res_\fV$, $\cB^\q_\fV$)
be the set of all binomial relations of
\eqref{eq-B-k=0} (respectively,
\eqref{eq-Bres-pk=0}, \eqref{eq-hq-k=0}).
We call relations of $\cB^\mn_\fV$ (respectively, $\cB^\res_\fV$, $\cB^\q_\fV$)
main (respectively, residual, of quotient type)  binomial on the chart $\fV$.
We let
$$\cB=\cB^\mn \sqcup \cB^\res \sqcup \cB^\q \;\; \; \and \;\;\;
 \cB_\fV=\cB^\mn_{ \fV} \cup \cB^\res_{ \fV} \cup \cB^\q_{ \fV}.$$
We let $L_{\fV, \sfm}$ be the set of all linear equations of
 \eqref{linear-pl-k=0}. We call relations of $L_{\fV, \sfm}$ linearized 
 $\pl$ relations on the chart $\fV$.
\end{defn}


\section{$\vt$-Blowups}\label{vt-blowups} 

{\it  We begin now the process of $``$removing$"$ zero factors of  main binomials
by  sequential blowups. It is divided into two subsequences.
The first are $\vt$-blowups.}
 

To start, it is useful to fix some terminology, used throughout.

\subsection{Some conventions on blowups} \label{blowupConventions} $\ $

Let $X$ be a 
scheme over the base field $\kk$.
When we blow up the scheme $X$ along the ideal (the homogeneous ideal, respectively)
$I=\langle f_0, \cdots, f_m \rangle$,  generated by some elements $f_0, \cdots, f_m$,
we will realize the blowup scheme $\widetilde X$ as the graph of the closure
of the rational map $$f: X \dashrightarrow \PP^m,$$
$$ x \to [f_0(x), \cdots, f_m(x)].$$
Then, upon fixing the generators  $f_0, \cdots, f_m$, we have a natural embedding
\begin{equation}\label{general-blowup} \xymatrix{
\widetilde{X}  \ar @{^{(}->}[r]  & X \times \PP^m.
 }
\end{equation}
We let
\begin{equation}
\pi: \widetilde{X} \lra X 
\end{equation}
be the induced blowup morphism.

We will refer to the projective space $\PP^m$ as the  {\it factor
projective space}  of the blowup corresponding to the generators  $f_0, \cdots, f_m$.
We let $[\xi_0, \cdots, \xi_m]$ be the homogeneous coordinates of
the factor projective space $\PP^m$,  corresponding  to $(f_0, \cdots, f_m)$. 

When $X$ is smooth and the center of the blowup is also smooth, then, the scheme 
$\widetilde X$, as a closed subscheme of $X \times \PP^m$, is defined by the relations
\begin{equation}
f_i \xi_j - f_j \xi_i, \;\; \hbox{for all $0\le i \ne j\le m$}.
\end{equation}

\begin{defn}\label{general-standard-chart}
Suppose that the scheme $X$ is covered by a set $\{\fV' \}$ of open subsets, called
(standard) charts. 

Fix any $0\le i\le m$. We let
\begin{equation} 
\fV=  (\fV' \times (\xi_i \ne 0)) \cap \widetilde{X} .
\end{equation}
We also often express this chart as
 $$\fV= (\fV' \times (\xi_i \equiv 1)) \cap \widetilde{X}.$$
It is an open subset of  $\widetilde{X}$, and will be called a standard chart of $\widetilde{X}$
lying over the (standard) chart $\fV'$ of $X$. Note that every standard chart of $\widetilde{X}$
lies over a unique (standard) chart $\fV'$ of $X$.
Clearly, $\widetilde{X}$ is covered by the finitely many  standard charts.

 In general, we let
 $$\widetilde{X}_k \lra \widetilde{X}_{k-1} \lra  \cdots  \lra \widetilde X_0:=X$$
 be a sequence of blowups such that every blowup $\widetilde{X_j} \to \widetilde{X}_{j-1}$ is
 as in \eqref{general-blowup}, $j \in [k]$.  
 
 Consider any $0\le j < k$.  Let $\fV$ (resp. $\fV''$) be a standard chart of $\widetilde{X}_k$
 (resp. of $\widetilde{X}_j$). Let $\fV'$ be the unique standard chart $\fV'$ of $\widetilde X_{k-1}$
 such that $\fV$ lies over $\fV'$.
 Via induction, we say $\fV$ lies over $\fV''$ if $\fV'$ equals to (when $j=k-1$) or lies over $\fV''$
 (when $j < k-1$).
\end{defn}

 We keep the notation as above. Let $\widetilde{X}  \to X$ be a blowup as
in \eqref{general-blowup}; we let $\fV$ be a standard chart of $\widetilde{X}$, lying over
a unique (standard) chart $\fV'$ of $X$; let
$\pi_{\fV, \fV'}: \fV \lra \fV'$ be the induced projection.

\begin{defn}\label{general-proper-transform-of-variable}  Assume that
the open chart $\fV$ (resp. $\fV'$) 
comes equipped with
a set of free variables  in $\var_\fV$ (resp. $\var_{\fV'}$).
Let $y \in \var_\fV$ (resp. $y' \in \var_{\fV'}$) be a coordinate variable of $\fV$ (resp. $\fV'$).
We say the coordinate variable $y$ is a proper transform of
the coordinate variable $y'$ if the divisor  $(y=0)$ on the chart $\fV$ is the proper transform of  
the divisor $(y'=0)$ on the chart $\fV'$.
\end{defn}

Keep the notation and assumption as in Definition \ref{general-proper-transform-of-variable}.

We assume in addition that the induced blowup morphism 
$$\pi^{-1} (\fV') \lra \fV'$$
corresponds to  the blowup of $\fV'$ along
the coordinate subspace of $\fV'$ defined by
 $$Z=\{y'_0= \cdots =y'_m =0\}$$
 with $\{y'_0, \cdots, y'_m\} \subset \var_{\fV'}$.
As earlier, we let $\PP^m$ be the corresponding factor projective space
with homogeneous coordinates $[\xi_0, \cdots, \xi_m]$, corresponding to $(y'_0, \cdots, y'_m)$.
 
 Without loss of generality, we assume that 
  the standard chart $\fV$ corresponds to $(\xi_0 \equiv 1)$, that is,
 $$\fV = (\fV' \times (\xi_0 \equiv 1)) \cap \widetilde{X}.$$
 Then, we have that $\fV$, as a closed subscheme of $\fV' \times (\xi_0 \equiv 1)$,
 is defined 
 \begin{equation}\label{general-blowup-formulas}
y'_i  - y'_0 \xi_i, \;\; \hbox{for all $i \in [m]$}.
\end{equation}

The following proposition is standard and will be applied throughout. 

\begin{prop}\label{generalmeaning-of-variables} Keep the notation and assumption as above.
In addition, we let $E$ be the exceptional divisor of the blowup $\widetilde{X}  \to X$.

Then, the standard chart $\fV$ comes equipped with a set of free variables 
$$\var_\fV=\{ \zeta, y_1, \cdots, y_m;  y:=y'  \mid y' \in \var_{\fV'} \- \{y_0' , \cdots, y_m'\} \}$$
where $\zeta:=y_0', y_i :=\xi_i, i \in [m]$
such that
\begin{enumerate}
\item $E \cap \fV =(\zeta =0)$; we call $\zeta$ the exceptional variable/parameter of $E$ on $\fV$;
\item $y_i \in \var_\fV$ is a proper transform of  $y'_i \in \var_{\fV'}$ for all  $i \in [m]$;
\item $y \in \var_\fV$ is a proper transform of  
$y' \in \var_\fV$ for all  $y' \in \var_{\fV'} \- \{y_0' , \cdots, y_m'\}$.
\end{enumerate}
\end{prop}
\begin{proof}
It is straightforward from \eqref{general-blowup-formulas}.
\end{proof}
 
Let  $\bf m$ be a monomial in $\var_{\fV}$. Then, for every variable $x \in \var_{\fV}$,
we let $\deg_x {\bf m}$ be the degree of $x$ in $\bf m$. 
 
 \begin{defn}\label{general-proper-transforms} 
 Keep the notation and assumption as in Proposition \ref{generalmeaning-of-variables}.
In addition, we let 
$$\phi=\{y'_0, \cdots, y'_m\} \subset \var_{\fV'}.$$

 Let $B_{\fV'}=T^0_{\fV'}- T^1_{\fV'}$ be 
 a binomial with variables in $\var_{\fV'}$.
We let $$m_{\phi, T^i_{\fV'}} = \sum_{j=0}^m \deg_{y'_j} (T^i_{\fV'}), \;\; i =0, 1, $$ 
$$l_{\phi, B_{\fV'}} = \min \{m_{\phi, T^0_{\fV'}}, m_{\phi, T^1_{\fV'}}\}.$$ 

Applying \eqref{general-blowup-formulas}, we substitute $y'_i$ by $y'_0 \xi_i$, for all $i \in [m]$,
into $B_{\fV'}$ and switch $y_0'$ by $\zeta$ and $\xi_i$ by $y_i$ with $i \in [m]$ to obtain
the pullback $\pi_{\fV,\fV'}^* B_{\fV'}$
where $\pi_{\fV,\fV'}: \fV \lra \fV'$ is the induced projection.
We then let 
\begin{equation}\label{define-proper-t}
B_\fV = (\pi_{\fV,\fV'}^* B_{\fV'}) / \zeta^{l_{\phi, B_{\fV'}}}.
\end{equation}
We call $B_\fV$, a binomial in $\var_\fV$, the proper transform of $B_{\fV'}$.

In general, for any polynomial $f_{\fV'}$ in $\var_{\fV'}$ such that 
$f_{\fV'}$ does not vanish identically along 
$Z= (y'_0= \cdots= y'_m=0)$, we let
$f_\fV = \pi_{\fV,\fV'}^* f_{\fV'}$. This is the pullback, but for convenience, we also call
$f_\fV$ the proper transform of $f_{\fV'}$. 

Moreover, suppose $\zeta$ appears in $B_\fV =(\pi_{\fV,\fV'}^* B_{\fV'}) / \zeta^{l_{\psi, B_{\fV'}}}$ or 
in $f_\fV= \pi_{\fV,\fV'}^* f_{\fV'}$,  and is obtained through the substitution $y'_i$ by $y'_0 \xi_i$
(note here that $\zeta:=y'_0$ and $i$ needs not to be unique), 
then we say that the exceptional parameter $\zeta$
is acquired by $y_i'$. In general, for  sequential blowups, if $\zeta$ is acquired by $y'$ and $y'$ is 
acquired by $y''$, then we also say $\zeta$ is acquired by $y''$.
\end{defn}

\begin{lemma}\label{same-degree}
We keep the same assumption and notation as in Definition \ref{general-proper-transforms}.

We let $T_{\fV', B}$ (resp. $T_{\fV, B}$) be any fixed term of $B_{\fV'}$ (resp.
$B_\fV$).  Consider any $y \in \var_\fV \- \zeta$ and  let 
$y' \in \var_{\fV'}$ be such that $y$ is the proper transform of $y'$. 
Then,  $y^b \mid  T_{\fV, B}$ if and only if 
$y'^b \mid  T_{\fV', B}$ for all integers $b \ge 0$.
\end{lemma}
\begin{proof}
This is clear from \eqref{define-proper-t}.
\end{proof}

\begin{defn}\label{general-termination} 
We keep the same assumption and notation as in Definition \ref{general-proper-transforms}.

Consider an arbitrary binomial  $B_{\fV'}$
(resp. $B_\fV$) with variables in $\var_{\fV'}$ (resp. $\var_\fV$).
Let $\bz' \in \fV'$  (resp. $\bz \in \fV$)  be any fixed closed point of the chart. 
We say $B_{\fV'}$ (resp. $B_\fV$)  terminates at $\bz'$ (resp. $\bz$)
 if (at least) one of the monomial terms of $B_{\fV'}$
(resp. $B_\fV$), say, $T_{\fV', B}$
(resp. $T_{\fV,B}$),
does not vanish at  $\bz'$ (resp. $\bz$).  In such a case, we also say 
$T_{\fV', B}$ (resp. $T_{\fV,B}$) terminates  at $\bz'$ (resp. $\bz$).
 \end{defn}

\subsection{Main binomial equations: revisited} $\ $


Recall that we have  chosen and fix the total order $``<"$ on $\sfm$ and we have listed it as
$$\sfm=\{\bF_1 < \cdots < \bF_\Upsilon\}.$$

Fix  and consider $F_k$ for any $k \in [\up]$.
We express $F_k=\sum_{s \in S_{F_k}} \sgn (s) p_{\uu_s} p_{\uv_s}$.
 Its corresponding linearized $\pl$ equation can
be expressed as $\sum_{s \in S_{F_k}} \sgn(s) x_{(\uu_s,\uv_s)}$, denoted by $L_{F_k}$. 
We let $s_{F_k} \in S_{F_k}$ be the index for the leading term of $F_k$, written as 
$ \sgn (s_{F_k}) p_\um p_{\uu_{F_k}}$.
Correspondingly, the leading term of the de-homogenization $\bF_k$ of $F_k$,
and  the leading term of the linearized $\pl$ equation $L_{F_k}$, 
 are defined to be $\sgn (s_{F_k})x_{\uu_{F_k}}$, and  $\sgn (s_{F_k}) x_{(\um,\uu_{F_k})}$,
 respectively.

We are to provide a total ordering on the set $\cB^\mn$.

Recall that the set of all $\pl$-variables is also totally ordered, compatible with that
on $\sfm$.
(It is neither lexicographical nor inverse-lexicographical.)

Throughout the remaining part of this article, when we use $<_\lex$, we mean
the lexicographical order induced on the power set $2^K$ for a given totally ordered
set $(K,<)$ (see Definition \ref{gen-order}).
 
 \begin{defn}\label{ordering-cBF} First,
 for any fixed $\bF \in \sfm$, we  provide a total ordering on the set $S_F \- s_F$ 
 by induction on $\rk F$ as follows.

 $\bullet$ Suppose $\rk F=0$.  Then, $S_F \- s_F$ consists of two elements $\{s, t\}$.
 We say $s<t$ if $(\uu_s, \uv_s) <_\lex (\uu_t, \uv_t)$ where each pair is listed
 lexicographically according to the order on the set of all $\pl$ variables.
  
 Suppose the set $S_F \- s_F$  is totally order when $\rk F=r -1$ for some $r >0$.
 
 $\bullet$ Suppose $\rk F=r>0$. Then, for any $s\in  S_F \- s_F$,
 one of $p_{\uu_s}$ and $p_{\uv_s}$ is a basic variable, the other is of rank equal to
 $\rk F -1$. Without loss of generality, we suppose $\rk p_{\uu_s}=r-1$ and let
 $F_{\uu_s}$ be its corresponding $\pl$ relation. Then, we say $s <t$ if
 $F_{\uu_s} < F_{\uu_t}$, that is if $\uu_s < \uu_t$.

Let $B' \in \cB^\mn_{F'}$ and $B \in \cB^\mn_{F}$ with $F' \ne F$.
 We say $B'<B$ if $F' < F$.

 Consider two elements in $ \cB^\mn_{F}$.
 By Definition \ref{defn:main-res}, we have
$\cB^\mn_F=\{B_{F,(s_F,s)} \mid s \in S_F \setminus s_F  \}.$
For any two distinct $s, t \in  S_F \- s_F$, we say $B_{F,(s_F,s)}<B_{F,(s_F,t)}$
if $s<t$.

Then, by Definition \ref{p-t}, the above provides a total ordering on the set of all main binomials
$\cB^\mn=\sqcup_{\bF \in \sfm} \cB^\mn_F.$
 \end{defn} 
 This ordering is important for our purpose.
 
 Recall that we have $$\fG=\bigsqcup_{\bF \in \sfm} \fG_F,$$
 where $\fG_F=\{\cB^\mn_F, L_F\}$.
 We can endow a total order on $\fG$ as follows. We say
 $\fG_{F'} < \fG_F$ if $F'< F$. 
This order is compatible with the order on $\sfm$ as well as on $\cB^\mn$.
Put it equivalently, the order on $\cB^\mn$ is induced by
 the order on $\fG$ and the order on each and every $S_F \- s_F$.

As in Definition \ref{ftF},
we let $(\ft_{F_k}+1)$ be the number of terms in ${F_k}$. Then, we can  list $S_{F_k}$ as
$$S_{F_k}=\{s_{F_k}; \; s_1 < \cdots < s_{\ft_{F_k}}\}.$$

Then, by Corollary \ref{eq-tA-for-sV}, the scheme $\sV$ as a closed subscheme of
$\sR_\sF= \rU_\um \times  \prod_{\bF \in \sF_\um} \PP_F $
is defined by the following relations
\begin{eqnarray}
  \cB^\res, \;\;  \cB^\pq,   \\ 
\;\;\;\;\;\;\;\;\; B_{(k\tau)}: \; x_{(\uu_{s_\tau}, \uv_{s_\tau})}x_{\uu_{F_k}} - x_{(\um,\uu_{F_k})}   x_{\uu_{s_\tau}} x_{\uv_{s_\tau}}, \;\;
\forall \;\; {s_\tau} \in S_{F_k} \- s_{F_k},  \; 1\le \tau\le \ft_{F_k}, \label{eq-B-ktau'} \\ 
L_{F_k}: \;\; \sum_{s \in S_{F_k}} \sgn (s) x_{(\uu_s,\uv_s)}, \label{linear-pl-ktau'} \\
\bF_k=\sum_{s \in S_{F_k}} \sgn (s) x_{\uu_s} x_{\uv_s}
\end{eqnarray}
for all $k \in [\up].$

\begin{defn}\label{pm-term} Given any binomial equation
$B_{(k\tau)}$ as in \eqref{eq-B-ktau'},
we let $T^+_{(k\tau)}= \; x_{(\uu_{s_\tau}, \uv_{s_\tau})}x_{\uu_{F_k}}$, called the plus-term of $B_{(k\tau)}$,
and $T^-_{(k\tau)}= x_{(\um,\uu_{F_k})}   x_{\uu_{s_\tau}} x_{\uv_{s_\tau}},$ 
called the minus-term of $B_{(k\tau)}$.
Then, we have $B_{(k\tau)}=T^+_{(k\tau)}-T^-_{(k\tau)}$.
\end{defn}

We do not name any term of a binomial of $\cB^\res \cup \cB^\pq$ a plus-term or a minus-term since
the two terms of such a  binomial  are indistinguishable.  

In addition, we let $\cB^\mn_{F_k}=\{B_{(k\tau)} \mid  \tau \in [\ft_{F_k}]  \}$ for any $k \in [\up]$.
Then, we have
 $$\cB^\mn=\bigsqcup_{k\in [\up]} \cB^\mn_{F_k}=\{B_{(k\tau)} \mid k \in [\up], \;  \tau \in [\ft_{F_k}]  \}.$$ 
We let 
\begin{equation}\label{indexing-Bmn}
\Index_{\cB^\mn}=\{ (k\tau) \mid k \in [\up], \; \tau \in [\ft_{F_k}] \}
\end{equation}
be the index set of $\cB^\mn$. 
Observe that the order $`` < "$ on the set $\cB^\mn$ now coincides with
the lexicographic order on $\Index_{\cB^\mn}$, that is, 
$$B_{(k\tau)} < B_{(k'\tau')} \iff (k,\tau) <_\lex (k',\tau').$$
Further, because $\cB^\mn_F=\{B_{F,(s_F,s)} \mid s \in S_F \setminus s_F  \}$,
we have a natural bijection between
$\cB^\mn$ and $\sqcup_{F} (S_F\- s_F)$. Hence, $\sqcup_{F} (S_F\- s_F)$
admits an induced total order.

\subsection{$\vt$-centers and $\vt$-blowups}\label{vr-centers} $\ $

{\it Besides serving as a part of   the process of $``removing"$   zero factors of
 the main binomial relations, the reason to perform $\vt$-blowups first 
 is  to eliminate all residual binomial relations
by making them dependent on the main binomial relations.  
}

Recall that the scheme $\tsR_{\vt_{[0]}}:=\sR_\sF$ comes equipped with two kinds of divisors:
$\vp$-divisors $X_\uw$ for all $\uw \in \II_{d,n}\- \um$
and $\vr$-divisors $X_{(\uu,\uv)}$ for all $(\uu,\uv) \in \La_\sfm$. 

\begin{defn}\label{defn:vr-centers}
Fix any $\uu \in \II_{d,n}^\um$. 
We let
$$\vt_\uu=(X_\uu, X_{(\um,\uu)}).$$
We call it the  $\vt$-set with respect to $\uu$.
We then call the scheme-theoretic intersection
 $$Z_{\vt_\uu}=X_\uu \cap X_{(\um,\uu)}$$
the $\vt$-center with respect to $\uu$.
\end{defn}

We let
$$\Theta=\{\vt_\uu \mid \uu \in \II_{d,n}^\um\},\;\;\;
\cZ_\Theta=\{Z_{\vt_\uu} \mid \uu \in \II_{d,n}^\um\}.$$
We let $\Theta$, respectively,  $\cZ_\Theta$,
inherit the total order from $\II_{d,n}^\um$. 
Thus,  if we write
$$\II_{d,n}^\um=\{\uu_1 < \cdots <\uu_{\up}\}$$
and also write $\vt_{\uu_k}=\vt_{[k]}$, $Z_{\vt_{\uu_k}}=Z_{\vt_{[k]}}$, then, we can express
$$\cZ_\Theta=\{Z_{\vt_{[1]}} < \cdots < Z_{\vt_{[\up]}} \}.$$

We then blow up $\sR_\sF$ along 
$Z_{\vt_{[k]}},\; k \in [\up]$, in the above order.  More precisely,
we start by setting $\tsR_{\vt_{[0]}}:=\sR_\sF$.  Suppose 
$\tsR_{\vt_{[k-1]}}$ has been constructed for some $k \in [\up]$. We then let
$$\tsR_{\vt_{[k]}} \lra \tsR_{\vt_{[k-1]}}$$
be the blowup of $\tsR_{\vt_{[k-1]}}$ along the proper transform of $Z_{\vt_{[k]}}$,
and we call it  the $\vt$-blowup in ($\vt_{[k]}$).

 The above gives rise to the following sequential $\vt$-blowups
\begin{equation}\label{vt-sequence}
\tsR_{\vt_{[\up]}} \to \cdots \to \tsR_{\vt_{[1]}} \to \tsR_{\vt_{[0]}}:=\sR_\sF,
\end{equation}

Every blowup $\tsR_{\vt_{[j]}} \lra \tsR_{\vt_{[j-1]}}$
comes equipped with an exceptional divisor, denoted by $E_{\vt_{[j]}}$.
Fix $k \in [\up]$. For any $j < k$, we let $E_{\vt_{[k]},j }$ be the proper transform
of $E_{\vt_{[j]}}$ in $\tsR_{\vt_{[k]}}$.  For notational consistency, we set
$E_{\vt_{[k]}}=E_{\vt_{[k]},k }$. We call the divisors $E_{\vt_{[k]},j }$, $j\le k$, the exceptional divisors
on $\tsR_{\vt_{[k]}}$.  

For every 
$\uw \in \II_{d,n} \setminus \um$, we let
$X_{\vt_{[k]}, \uw}$ be the proper transform of $X_\uw$ in $\tsR_{\vt_{[k]}}$,
still called $\vp$-divisor;
for every $(\uu,\uv) \in \La_\sfm$,
we let  $X_{\vt_{[k]}, (\uu,\uv)}\cap \fV$ be 
 the proper transform of $X_{(\uu,\uv)}$ in $\tsR_{\vt_{[k]}}$,
still called $\vr$-divisor.
for every $\bF \in \sfm$,
we let  $D_{\vt_{[k]}, L}$ be the proper transform of $D_{L_F}$ in $\tsR_{\vt_{[k]}}$,
still called the $\fL$-divisor.


\subsection{Properties of $\vt$-blowups}\label{prop-vt-blowups} $\ $

By Definition \ref{general-standard-chart}, the scheme $\tsR_{\vt_{[k]}}$
is covered by a set of standard charts.

\begin{prop}\label{meaning-of-var-vtk}
Consider any standard chart $\fV$ of $\tsR_{\vt_{[k]}}$, 
 lying over a unique chart $ \fV_{[0]}$ of $\tsR_{\vt_{[0]}}=\sR_\sF$.
 We suppose that the chart $ \fV_{[0]}$ is indexed by
$\La_\sfm^o=\{(\uu_{s_{F,o}},\uv_{s_{F,o}}) \mid \bF \in \sfm \}$
(cf. \eqref{index-sR}).
We let $\II_{d,n}^*=\II_{d,n}\- \um$ and $\La_\sfm^\star=\La_\sfm \- \La_\sfm^o$.

Then, the standard chart $\fV$ comes equipped with 
$$\hbox{a subset}\;\; \fe_\fV  \subset \II_{d,n}^* \;\;
 \hbox{and a subset} \;\; \fd_\fV  \subset \La_{\sfm}^\star$$
such that every exceptional divisor 
of $\tsR_{\vt_{[k]}}$
with $E_{\vt_{[k]}} \cap \fV \ne \emptyset$ is 
either labeled by a unique element $\uw \in \fe_\fV$
or labeled by a unique element $(\uu,\uv) \in \fd_\fV$. 
We let $E_{\vt_{[k]}, \uw}$ be the unique exceptional divisor 
on the chart $\fV$ labeled by $\uw \in \fe_\fV$; we call it an $\vp$-exceptional divisor.
We let $E_{\vt_{[k]}, (\uu,\uv)}$ be the unique exceptional divisor 
on the chart $\fV$ labeled by $(\uu,\uv) \in \fd_\fV$;  we call it an $\vr$-exceptional divisor.
(We note here that being $\vp$-exceptional or $\vr$-exceptional is strictly relative to the given
standard chart.)

Further, the standard chart $\fV$  comes equipped with the set of free variables
\begin{equation}\label{variables-vtk} 
\var_{\fV}:=\left\{ \begin{array}{ccccccc}
\ve_{\fV, \uw} , \;\; \de_{\fV, (\uu,\uv) }\\
x_{\fV, \uw} , \;\; x_{\fV, (\uu,\uv)}
\end{array}
  \; \Bigg| \;
\begin{array}{ccccc}
 \uw \in  \fe_\fV,  \;\; (\uu,\uv)  \in \fd_\fV  \\ 
\uw \in  \II_{d,n}^* \- \fe_\fV,  \;\; (\uu, \uv) \in \La_\sfm^\star \-  \fd_\fV  \\
\end{array} \right \}
\end{equation} such that 
on the standard chart $\fV$, we have
\begin{enumerate}
\item the divisor  $X_{\vt_{[k]}, \uw}\cap \fV$ 
 is defined by $(x_{\fV,\uw}=0)$ for every 
$\uw \in \II_{d,n}^* \- \fe_\fV$;
\item the divisor  $X_{\vt_{[k]}, (\uu,\uv)}\cap \fV$ is defined by $(x_{\fV,(\uu,\uv)}=0)$ for every 
$(\uu,\uv) \in \La^\star_\sfm\- \fd_\fV$;
\item the divisor  $D_{\vt_{[k]}, L}\cap \fV$ is defined by $(L_{\fV,F}=0)$ for every
$\bF \in \sfm$ where $L_{\fV, F}$ is the proper transform of $L_F$; 
\item the divisor  $X_{\vt_{[k]}, \uw}\cap \fV$ does not intersect the chart for all $\uw \in \fe_\fV$;
\item the divisor  $X_{\vt_{[k]}, (\uu, \uv)}$ does not intersect the chart for all $(\uu, \uv) \in \fd_\fV$;
\item the $\vp$-exceptional divisor 
$E_{\vt_{[k]}, \uw} \;\! \cap  \fV$  labeled by an element $\uw \in \fe_\fV$
is define by  $(\ve_{\fV,  \uw}=0)$ for all $ \uw \in \fe_\fV$;
\item the $\vr$-exceptional divisor 
$E_{\vt_{[k]},  (\uu, \uv)}\cap \fV$ labeled by  an element $(\uu, \uv) \in \fd_\fV$
is define by  $(\de_{\fV,  (\uu, \uv)}=0)$ for all $ (\uu, \uv) \in \fd_\fV$;
\item  any of the remaining exceptional divisor of $\tsR_{\vt_{[k]}}$
other than those that are labelled by some  $\uw \in \fe_\fV$ or $(\uu,\uv) \in \fd_\fV$ 
 does not intersect the chart.
\end{enumerate}
\end{prop}
\begin{proof}
When $k=0$, we have $\tsR_{\vt_{[0]}}=\sR_\sF$. 
In this case, we set $$ \fe_\fV = \fd_\fV =\emptyset.$$
Then, the statement follows from 
Proposition \ref{meaning-of-var-p-k=0} with $k=\up$.

We now suppose that the statement holds for $\tsR_{\vt_{[k-1]}}$
for some $k \in [\up]$. 

We consider $\tsR_{\vt_{[k]}}$.

As in the statement, we let $\fV$ be a standard chart 
of $\tsR_{\vt_{[k]}}$, lying over a (necessarily unique) 
standard chart  $\fV'$ of $\tsR_{\vt_{[k-1]}}$.

If $(\um,\uu_k) \in  \La_{\sF_{[k]}}^o$
(cf. \eqref{index-sR}), 
then $\fV'$
does not intersect the proper transform of the blowup center $Z_{\vt_k}$ and
$\fV \to \fV'$ is an isomorphism. In this case, we let $\var_\fV=\var_{\fV'}$,
$ \fe_{\fV'}=\fe_\fV, $ and $\fd_{\fV} =\fd_{\fV'}.$ Then, the statements on $\fV'$
carry over to $\fV$.

In what follows, we assume $(\um,\uu_k) \notin  \La_{\sF_{[k]}}^o$.

Consider the embedding $$\tsR_{\vt_{[k]}} \lra  \tsR_{\vt_{[k-1]}} \times \PP_{\vt_{[k]}}$$
where $\PP_{\vt_{[k]}}$ is the factor projective space with homogeneous coordinates $[\xi_0,\xi_1]$
corresponding to $(X_{\uu_k}, X_{(\um,\uu_k)})$.
We let $E_{\vt_{[k]}}$ be the exceptional divisor created by 
the blowup $\tsR_{\vt_{[k]}} \to  \tsR_{\vt_{[k-1]}}$.

First, we consider the case when 
$$\fV = \tsR_{\vt_{[k]}} \cap (\fV' \times (\xi_0 \equiv 1).$$
We let $Z'_{\vt_k}$ be the proper transform of the $\vt$-center $Z_{\vt_k}$
in  $\tsR_{\vt_{[k-1]}}$. Then,
in this case, on the chart $\fV'$, we have
 $$Z'_{\vt_k} \cap \fV' = \{ x_{\fV', \uu_k} = x_{\fV', (\um,\uu_k)}=0\}$$
 where   $x_{\fV', \uu_k}$ (resp. $x_{\fV', (\um,\uu_k)}$) is the proper transform of $x_\uu$ (resp.
 $x_{ (\um,\uu_k)}$)
 on the chart $\fV'$.
 Then, $\fV$ as a closed subset of  $\fV' \times (\xi_0 \equiv 1)$ is defined by
 $$x_{\fV', (\um,\uu_k)} = x_{\fV', \uu_k} \xi_1.$$
 We let 
 $$\fe_\fV= \uu_k \sqcup \fe_{\fV'},\; \fd_\fV=  \fd_{\fV'},\;\; \hbox{and}$$
 $$\ve_{\fV, \uu_k}=x_{\fV', \uu_k}, \;  x_{\fV, (\um,\uu_k)}=\xi_1; \;
 y_\fV = y_{\fV'}, \; \forall \; y_{\fV'} \in \var_{\fV'}\- \{x_{\fV', \uu_k}, x_{\fV', (\um,\uu_k)}\}.$$
 Observe that  
 $E_{\vt_{[k]}} \cap \fV = (\ve_{\fV, \uu_k}=0)$ and
 $x_{\fV, (\um,\uu_k)}=\xi_1$ is the proper transform of  $x_{\fV', (\um,\uu_k)}$.
 By the inductive assumption on the chart $\fV'$, one verifies directly that
 (1) - (7) of the proposition hold (cf. Proposition \ref{generalmeaning-of-variables}).
 
Next, we consider the case when 
$$\fV = \tsR_{\vt_{[k]}} \cap (\fV' \times (\xi_1 \equiv 1).$$
 Then, $\fV$ as a closed subset of  $\fV' \times (\xi_1 \equiv 1)$ is defined by
 $$ x_{\fV', \uu_k} =x_{\fV', (\um,\uu_k)} \xi_0.$$
 We let 
 $$\fe_\fV=  \fe_{\fV'},\; \fd_\fV=  \{(\um,\uu_k)\} \sqcup \fd_{\fV'},\;\; \hbox{and}$$
 $$\de_{\fV,  (\um,\uu_k)}=x_{\fV',  (\um,\uu_k)}, \;  x_{\fV, \uu_k}=\xi_0; \;
 y_\fV = y_{\fV'}, \; \forall \; y_{\fV'} \in \var_{\fV'}\- \{x_{\fV', \uu_k}, x_{\fV', (\um,\uu_k)}\}.$$
 Observe that  
 $E_{\vt_{[k]}} \cap \fV = (\de_{\fV,  (\um,\uu_k)}=0)$ and
 $x_{\fV, \uu_k}=\xi_0$ is the proper transform of  $x_{\fV', \uu_k}$.
 By the inductive assumption on the chart $\fV'$, like in the above case,
 one checks directly that
 (1) - (7) of the proposition hold.

 This proves the proposition.
 \end{proof}

Observe here that $x_{\fV, \uu}$ with $\uu \in \fe_{\fV}$ and
 $x_{\fV, (\uu,\uv)}$ with  $(\uu,\uv) \in \de_{\fV}$ are not variables in $\var_\fV$.
For notational convenience, to be used throughout, we make a convention:  
\begin{equation}\label{conv:=1}
\hbox{$\bullet$ $x_{\fV, \uu} = 1$ if  $\uu \in \fe_{\fV}$; 
\;\; $\bullet$ $x_{\fV, (\uu,\uv)} = 1$ if  $(\uu,\uv) \in \fd_{\fV}$.}
\end{equation}

\smallskip
For any $k \in [\up]$, the $\vt$-blowup in ($\vt_{[k]}$)
gives rise to \begin{equation}\label{tsV-vt-k} \xymatrix{
\tsV_{\vt_{[k]}} \ar[d] \ar @{^{(}->}[r]  &\tsR_{\vt_{[k]}} \ar[d] \\
\sV \ar @{^{(}->}[r]  & \sR_\sF,
}
\end{equation}
where $\tsV_{\vt_{[k]}}$ is the proper transform of $\sV$ 
in  $\tsR_{\vt_{[k]}}$. 

Alternatively, we can set $\tsV_{\vt_{[0]}}:=\sV_\sF$.  Suppose 
$\tsV_{\vt_{[k-1]}}$ has been constructed for some $k \in [\up]$. We then let
 $\tsV_{\vt_{[k]}} \subset \tsR_{\vt_{[k]}}$ be the proper transform of $\tsV_{\vt_{[k-1]}}$.

\begin{defn} Fix any standard chart $\fV$ of $\tsR_{\vt_{[k]}}$ lying over
a unique standard chart $\fV'$ of $\tsR_{\vt_{[k-1]}}$ for any $k \in [\up]$. 
When $k=0$, we let $B_\fV$ and $L_{\fV, F}$ be as in 
 Proposition \ref{equas-p-k=0} for any
$B \in \cB^\mn \cup \cB^\res  \cup \cB^\q$ and $\bF \in \sfm$. Consider any fixed general $k \in [\up]$.
Suppose $B_{\fV'}$ and $L_{\fV', F}$ have been constructed over $\fV'$.
Applying Definition \ref{general-proper-transforms}, we obtain the proper transforms on the chart $\fV$
$$B_{\fV}, \;\; \forall \; B \in \cB^\mn \cup \cB^\res  \cup \cB^\q; \;\; L_{\fV, F}, \;\;
\forall \; \bF \in \sfm.$$ 
\end{defn}

We need the following notations.

Fix any $k \in [\up]$.
We let $\cB^\mn_{< k}$ (resp. $\cB^\res_{< k}$ or $L_{ <k}$)
be the set of all main (resp. residual or linear $\pl$) relations corresponding to
$F<F_k$. Similarly, we let $\cB^\mn_{ > k}$ (resp. $\cB^\res_{ > k}$, $\cL_{ >k}$)
be the set of all main (resp. residual or linear $\pl$) relations corresponding to
$F>F_k$.  Likewise, replacing $<$ by $\le$ or $>$ by $\ge$,
we can introduce 
$\cB^\mn_{\le k}$, $\cB^\res_{\le k}$, and $L_{\le k}$ or
$\cB^\mn_{ \ge k}$,  $\cB^\res_{ \ge k}$, and $\cL_{\ge k}$.
Then, upon restricting the above to a fixed standard chart $\fV$, we obtain
$\cB^\mn_{\fV, < k}$, $\cB^\res_{\fV, < k}$, $L_{\fV, <k}$, etc..

Recall from the above proof, we have 
$$\tsR_{\vt_{[k]}} \subset \tsR_{\vt_{[k-1]}} \times \PP_{\vt_k}$$
where $\PP_{\vt_k}$ be the factor projective space of the blowup
$\tsR_{\vt_{[k]}} \lra \tsR_{\vt_{[k-1]}}$. 
We write $\PP_{\vt_k}=\PP_{[\xi_0,\xi_1]}$ such that 
$[\xi_0,\xi_1]$ corresponds to $(X_{\uu_k}, X_{(\um,\uu_k)})$.

\begin{defn} \label{vp-vr-chart}
Let $\fV'$ be any standard chart on $\tsR_{\vt_{[k-1]}}$. Then, 
we call $$\fV=\tsR_{\vt_{[k]}} \cap (\fV' \times (\xi_0 \equiv 1))$$
a $\vp$-standard chart of $\tsR_{\vt_{[k]}}$; we call $$\fV=\tsR_{\vt_{[k]}} \cap (\fV' \times (\xi_1 \equiv 1))$$
a $\vr$-standard chart of $\tsR_{\vt_{[k]}}$.
\end{defn}

\begin{prop}\label{eq-for-sV-vtk}
We keep the notation and assumptions in Proposition \ref{meaning-of-var-vtk}.

Suppose $(\um,\uu_k) \in \La_{\sF_{[k]}}^o$ or
 $\fV$ is a $\vr$-standard chart.  Then, we have
that the scheme $\tsV_{\vt_{[k]}}  \cap \fV$, as a closed subscheme of
the chart $\fV$ of $\tsR_{\vt_{[k]}} $,  is defined by 
\begin{eqnarray} 
\cB^\q_\fV,  \;\; \cB^\mn_{\fV, < k} ,\;\; \cL_{\fV, <k} , \;\;\;\;\;\;\;\;\;\;  \\
B_{ \fV,(s_{F_k},s)}: \;\;\;\;\;\; x_{\fV, (\uu_s, \uv_s)}  x_{\fV, \uu_k} 
  - \tilde{x}_{\fV, \uu_s} \tilde{x}_{\fV, \uv_s}, \;\;
\forall \;\; s \in S_{F_k} \- s_{F_k},  \label{eq-B-vt-lek=0} \\ 
L_{\fV, F_k}: \;\; \sgn (s_F) \de_{\fV, (\um,\uu_k)} +
\sum_{s \in S_F \- s_F} \sgn (s) x_{\fV,(\uu_s,\uv_s)},   \label{linear-pl-vtk=0} \\
\cB^\mn_{\fV, > k},\;\; \cB^\res_{\fV, > k}, \;\; \cL_{\fV, >k},\;\;\; \;\;\; \;\; \;\; \;\; \;\; \;\; \;\; \;\; \;\; \;\; \;\; \;\; \;\; \label{eq-hq-vtk=0}
\end{eqnarray}
where $\tilde{x}_{\fV, \uu_s}$ and $ \tilde{x}_{\fV, \uv_s}$ are some monomials in $\var_\fV$.

 Suppose  $(\um,\uu_k) \notin \La_{\sF_{[k]}}^o$ and $\fV$ is a $\vp$-standard chart.
Then, we have
that the scheme $\tsV_{\vt_{[k]}}  \cap \fV$, as a closed subscheme of
the chart $\fV$ of $\tsR_{\vt_{[k]}} $,  is defined by 
\begin{eqnarray} 
\cB^\q_\fV,  \;\; \cB^\mn_{\fV, < k} ,\;\; \cL_{\fV, <k}, \;\;\;\;\;\;\;\;\;\;  \\
B_{ \fV,(s_{F_k},s)}: \;\;\;\;\;\; x_{\fV, (\uu_s, \uv_s)} - x_{\fV, (\um,\uu_k)}  
 \tilde{x}_{\fV, \uu_s} \tilde{x}_{\fV, \uv_s}, \;\;
\forall \;\; s \in S_{F_k} \- s_{F_k},  \label{eq-B-vt-lek=0-00} \\ 
L_{\fV, F_k}: \;\; \sgn (s_F) \ve_{\fV, \uu_k} x_{\fV,(\um,\uu_k)}+
\sum_{s \in S_F \- s_F} \sgn (s) x_{\fV,(\uu_s,\uv_s)},   \label{linear-pl-vtk=0-00} \\
\cB^\mn_{\fV, > k},\;\; \cB^\res_{\fV, > k}, \;\; \cL_{\fV, >k},\;\;\; \;\;\; \;\; \;\; \;\; \;\; \;\; \;\; \;\; \;\; \;\; \;\; \;\; \;\; \label{eq-hq-vtk=0-00}
\end{eqnarray}
where $\tilde{x}_{\fV, \uu_s}$ and $ \tilde{x}_{\fV, \uv_s}$ are some monomials in $\var_\fV$.


Moreover, for any binomial 
$B \in \cB^\mn \sqcup \cB^\res_{>k}$, $B_\fV$ is $\vr$-linear and square-free.

Furthermore, consider  an arbitrary binomial $B \in \cB^\q$ and its proper transform $B_{\fV}$ on the chart 
$\fV$. Let $T_{\fV, B}$ be any fixed term of $B_\fV$.
Then, $T_{\fV, B}$  is $\vr$-linear and admits
at most one $\vt$-exceptional parameter in the form of 
$\delta_{(\um, \uu)}$ for some $(\um, \uu) \in \fd_\fV$ 
or $\ve_\uu x_{(\um, \uu)}$ for some $\uu \in \fe_\fV$.
In particular, it is  square-free.
\end{prop}
\begin{proof} We follow the notation as in the proof of 
Proposition \ref{meaning-of-var-vtk}.

When $k=0$,   we have $(\tsV_{\vt_{[0]}} \subset \tsR_{\vt_{[0]}})=  (\sV_\sF \subset \sR_\sF)$. 
Then, the statement follows from
Proposition \ref{equas-p-k=0}.

Suppose that the statement holds for $(\tsV_{\vt_{[k-1]}} \subset \tsR_{\vt_{[k-1]}})$
for some $k \in [\up]$. 

We now consider $(\tsV_{\vt_{[k]}} \subset \tsR_{\vt_{[k]}})$.

As in  the proof of 
Proposition \ref{meaning-of-var-vtk}, we let $\fV$ be a standard chart 
of $\tsR_{\vt_{[k]}}$ lying over a (necessarily unique) 
standard chart  $\fV'$ of $\tsR_{\vt_{[k-1]}}$.
Also, $\fV$ lies over a unique standard chart $\fV_{[0]}$ of 
of $\tsR_{\vt_{[0]}}$. We let $\pi_{\fV, \fV_{[0]}}: \fV \to \fV_{[0]}$ be the induced projection.
 
 To prove the statement about the defining equations of $\tsV_{\vt_{[k]}} \cap \fV$ in $\fV$, 
 by applying the inductive assumption to  $\fV'$,
 it suffices to prove that the proper transform
 of any residual binomial of $F_k$ depends on the main binomials
 on the chart $\fV$.
 
 For that purpose,  we take any two $s, t \in S_{F_k} \- s_{F_k}$ and consider
 the residual binomial  $B_{F_k,(s,t)}$ (cf. \eqref{eq-Bres}).
 
Suppose $(\um,\uu_k) \in \La_{\sF_{[k]}}^o$, hence $x_{\fV, (\um,\uu_k)}\equiv 1$ 
 on the chart $\fV$. In this case, the blowup along (the proper transform of) 
 $Z_{\vt_{[k]}}$ does not affect the chart $\fV'$ of $\tsR_{\vt_{[k-1]}}$. 
 Likewise, suppose $\fV$ is a $\vr$-standard chart. Then, $(\um,\uu_k) \in \fd_\fV$, hence
  $x_{\fV, (\um,\uu_k)}= 1$ 
 on the chart $\fV$ by \eqref{conv:=1}.
In any case, one calculates and finds that we have the following two main binomials 
 $$B_{ \fV,(s_{F_k},s)}: \;\;\;\;\;\; x_{\fV, (\uu_s, \uv_s)}  x_{\fV, \uu_k}   - 
 \tilde x_{\fV, \uu_s} \tilde x_{\fV, \uv_s},$$
 $$B_{ \fV,(s_{F_k},t)}: \;\;\;\;\;\; x_{\fV, (\uu_t, \uv_t)}  x_{\fV, \uu_k}   - \tilde x_{\fV, \uu_t} \tilde x_{\fV, \uv_t},$$
 where  $ \tilde x_{\fV, \uw}=\pi_{\fV, \fV_{[0]}}^* x_{\fV_{[0]},\uw}$ denoted the pullback
 for any $\uw \in \II_{d,n} \- \um$.
  Similarly, one calculates and finds  that 
 we have 
$$ B_{ \fV,(s,t)}: \;\; x_{\fV,(\uu_s, \uv_s)}\tilde x_{\fV,\uu_t} \tilde x_{\fV, \uv_t}-
x_{\fV,(\uu_t, \uv_t)} \tilde x_{\fV,\uu_s} \tilde x_{ \fV,\uv_s}.$$ 
Then, one verifies directly that we have
$$ B_{ \fV,(s,t)}=x_{\fV, (\uu_s, \uv_s)} B_{ \fV,(s_{F_k},t)} -x_{\fV, (\uu_t, \uv_t)} B_{ \fV,(s_{F_k},s)}.$$
This proves  the statement about the defining equations of $\tsV_{\vt_{[k]}} \cap \fV$ in $\fV$.

Moreover,  consider any $B \in \cB^\mn$ with respect to $F_j$.
Observe that  $x_{(\um, \uu_k)}$ uniquely appears in the main binomials 
with respect to $F_k$;  $x_{\uu_k}$ only appears in the main binomials 
with respect to $F_k$ and   the minus terms of certain main binomials 
of $F_{j}$ with $j>k$. It follows that $B_\fV$ is $\vr$-linear and square-free.
 
 Likewise, consider any $B \in \cB^\res$ with respect to $F_j$ with $j >k$.
It is of the form
$$ B_{(s,t)}: \;\; x_{(\uu_s, \uv_s)}x_{\uu_t}  x_{\uv_t}-
x_{(\uu_t, \uv_t)} x_{\uu_s}  x_{ \uv_s}$$ 
for some $s \ne t \in S_{F_j}$.
Observe here that $B$ does not contain any $\vr$-variable of the form $x_{(\um, \uu)}$ and
the $\vp$-variables in $B$ are identical to
those of the minus terms of the corresponding main binomials.
Hence, the same line of the proof above for  main binomials
 implies that $B_\fV$ is $\vr$-linear and square-free.

Further, consider any $B \in \cB^q$. If $B_{\fV'}$ does not contain $x_{\fV',(\um, \uu_k)}$
or $(\um,\uu_k) \in \La_{\sF_{[k]}}^o$,
then the form of $B_{\fV'}$ remains unchanged (except for the meanings of its variables).
Suppose next that $B_{\fV'}$  contains $x_{\fV',(\um, \uu_k)}$ and $(\um,\uu_k) \notin \La_{\sF_{[k]}}^o$.
Note that the proper transform of the $\vt$-center 
$\vt_{[k]}$ on the chart $\fV'$ equals to $(x_{\fV', \uu_k}, x_{\fV', (\um, \uu_k)})$. Thus,
from the chart $\fV'$ to the $\vr$-standard chart $\fV$, we have that  $x_{\fV',(\um, \uu_k)}$ becomes 
$\de_{\fV,(\um, \uu_k)}$ in $B_\fV$. 
 By Lemma \ref{ker-phi-k} (2), applied to the variable $p_\um$ (before 
 de-homogenization), we see that any fixed term $T_B$ of $B$ contains
 at most one $\vr$-variables of the form $x_{(\um, \uu)}$ with
 $\uu \in \II_{d,n}^\um$. Hence, one sees that 
the last statement on  $B_\fV$ holds, in this case. 

Thus,  this proves the  statement of the proposition when 
$(\um,\uu_k) \in  \La_{\sF_{[k]}}^o$ or when $\fV$ is a $\vr$-standard chart.

Next, we consider the case when $(\um,\uu_k) \notin \La_{\sF_{[k]}}^o$
and  $\fV$ is a $\vp$-standard chart.

 Again, to prove the statement about the defining equations of $\tsV_{\vt_{[k]}} \cap \fV$ in $\fV$, 
  it suffices to prove that the proper transform
 of any residual binomial of $F_k$ depends on the main binomials on the chart $\fV$.
 
 To show this, we again take any two $s, t \in S_{F_k} \- s_{F_k}$.
 
 On the chart $\fV$, we have the following two the main binomials 
 $$B_{ \fV,(s_{F_k},s)}: \;\;\;\;\;\; x_{\fV, (\uu_s, \uv_s)}    - x_{\fV, (\um,\uu_k)}
 \tilde x_{\fV, \uu_s} \tilde x_{\fV, \uv_s},$$
 $$B_{ \fV,(s_{F_k},t)}: \;\;\;\;\;\; x_{\fV, (\uu_t, \uv_t)}    - x_{\fV, (\um,\uu_k)}
 \tilde x_{\fV, \uu_t} \tilde x_{\fV, \uv_t}.$$
We also have  the following residual binomial
$$ B_{ \fV,(s,t)}: \;\; x_{\fV,(\uu_s, \uv_s)}\tilde x_{\fV,\uu_t} \tilde x_{\fV, \uv_t}-
x_{\fV,(\uu_t, \uv_t)} \tilde x_{\fV,\uu_s} \tilde x_{ \fV,\uv_s}.$$ 
 Then, we have
 $$B_{ \fV,(s,t)}=  \tilde x_{\fV, \uu_t} \tilde x_{\fV, \uv_t} B_{ \fV,(s_{F_k},s)}
 -\tilde x_{\fV, \uu_s} \tilde x_{\fV, \uv_s} B_{ \fV,(s_{F_k},t)}.$$
 Thus, the  statement of the proposition about the equations of  $\tsV_{\vt_{[k]}}  \cap \fV$ follows.
 
 Next, consider any $B \in \cB^\mn$. The fact
 that $B_\fV$ is $\vr$-linear and 
 square-free follows from the same line of proof in the previous case.
 
Finally, consider any $B \in \cB^q$. If $B_{\fV'}$ does not contain $x_{\fV',(\um, \uu_k)}$,
then the form of $B_{\fV'}$ remains unchanged.
Suppose next that $B_{\fV'}$  contains $x_{\fV',(\um, \uu_k)}$.
Again, the proper transform of the $\vt$-center 
$\vt_{[k]}$ on the chart $\fV'$ equals to $(x_{\fV', \uu_k}, x_{\fV', (\um, \uu_k)})$. Hence,
from the chart $\fV'$ to the $\vp$-standard chart $\fV$, we have that 
 $x_{\fV',(\um, \uu_k)}$ turns into $\ve_{\fV, \uu_k} x_{\fV,(\um, \uu_k)}$  in $B_\fV$.
 Then, again, by applying Lemma \ref{ker-phi-k} (2), applied to the variable $p_\um$ (before 
 de-homogenization), we have that
 any fixed term $T_B$ of $B$ contains
 at most one $\vr$-variables of the form $x_{(\um, \uu)}$ with
 $\uu \in \II_{d,n}^\um$.   Hence, one sees that 
the last statement on  $B_\fV$ holds. 

This completes the proof of the proposition.
 \end{proof}

We need the final case of $\vt$-blowups.

We set $\tsR_{\vt}:=\tsR_{\vt_{[\up]}}$. $\tsV_{\vt}:=\tsV_{\vt_{[\up]}}$.

\begin{cor}\label{eq-for-sV-vr} 
Let the notation be as in Proposition \ref{equas-fV[k]} for $k=\up$.  
Then, the scheme $\tsV_{\vt} \cap \fV$, as a closed subscheme of
the chart $\fV$ of $\tsR_{\vt}=\tsR_{\vt_{[\up]}}$,  is defined by 
\begin{eqnarray} 
\cB^\q_\fV,  \;\; \cB^\mn_{\fV} ,\;\; L_{\fV, \sfm}.
\end{eqnarray}
Further, for any binomial $B_\fV \in \cB^\mn_\fV \cup \cB^\q_\fV$, it is  $\vr$-linear and square-free.
\end{cor}

\begin{cor}\label{no-(um,uu)} Fix any $k \in [\up]$. 
Let $X_{\vt, (\um,\uu_k)}$ be the proper transform of
$X_{(\um,\uu_k)}$ in $\tsR_{\vt}$. Then 
$$\tsV_{\vt} \cap X_{\vt, (\um,\uu_k)} =\emptyset. $$ 
 In particular, $\tsV_{\vt}$ is covered by the standard charts that either lie over  the
chart $(x_{(\um,\uu_k)} \equiv 1)$ of $\sR_\sF$ or
the $\vr$-standard chart of $\tsR_{\vt_{[k]}}$. 
\end{cor}
\begin{proof} Fix any standard chart $\fV$. 

If $\fV$ lies over the chart $(x_{(\um,\uu_k)} \equiv 1)$ of $\sR_\sF$, 
that is, $(\um,\uu_k) \in  \La_{\sF_{[k]}}^o$,
then the first statement $\tsV_{\vt} \cap X_{\vt, (\um,\uu_k)} =\emptyset$
follows from the definition. 

If $\fV$ lies over a $\vr$-standard chart of $\tsR_{\vt_{[k]}}$, 
then the fact that $\tsV_{\vt} \cap X_{\vt, (\um,\uu_k)} =\emptyset$ follows from 
Proposition \ref{meaning-of-var-vtk}  (4);  $x_{(\um,\uu_k)} = 1$ by the convention \eqref{conv:=1}.

Suppose $\fV$ lies over a $\vp$-standard chart 
$\fV_{\vt_{[k]}}$ of $\tsR_{\vt_{[k]}}$. Then, in this case,
we have the following main binomial relation on the chart $\fV_{\vt_{[k]}}$
\begin{equation}\label{invertible-umuu} B_{ \fV,(s_{F_k}, s_{F, o})}: \;\;\;\;\;\; 1   - 
x_{\fV_{\vt_{[k]}}, (\um,\uu_k)}x_{\fV_{\vt_{[k]}}, \uu_{s_{F, o}}} x_{\fV_{\vt_{[k]}}, \uv_{s_{F, o}}}
\end{equation}
because $x_{\fV_{\vt_{[k]}}, (\uu_{s_{F,o}},\uu_{s_{F,o}})} \equiv 1$ 
with $(\uu_{s_{F,o}},\uu_{s_{F,o}}) \in  \La_{\sF_{[k]}}^o$
and $x_{\fV_{\vt_{[k]}}, \uu_k} = 1$ by \eqref{conv:=1}.  
Thus $x_{\fV_{\vt_{[k]}}, (\um,\uu_k)}$ is nowhere vanishing along
  $\tsV_{\vt_{[k]}} \cap \fV$.  This implies the statement.
  \end{proof}

\begin{defn}\label{preferred-chart-vt}
We call any standard chart of $\tsR_\vt$ as described by Corollary \ref{no-(um,uu)} 
a preferred standard chart.
\end{defn}

\section{$\wp$- and $\ell$-Blowups}\label{sect:wp/ell-blowups} 

\subsection{ The initial setup: $\wp_0$-blowups and $\ell_0$-blowup} $\ $

Our initial scheme is  $\tsR_{\ell_{-1}}:=\tsR_\vt$.

We let 
$$\tsR_{\ell_0} \lra \tsR_{\wp_0} \lra \tsR_{\ell_{-1}}:= \tsR_\vt$$
be the trivial blowups along the empty set. 
In the sequel, a blowup is called trivial if it is a blowup along the emptyset.
We make
the identifications
$$\tsR_{\ell_0} = \tsR_{\wp_0} =\tsR_{\ell_{-1}}= \tsR_\vt.$$

For any $k \in [\up]$, we are to construct 
$$\tsR_{\ell_k} \lra \tsR_{\wp_k} \lra \tsR_{\ell_{k-1}}.$$ 
The morphism $\tsR_{\wp_k} \to \tsR_{\ell_{k-1}}$ is decomposed as
a sequential  blowups based on all relations in $\cB^\mn_{F_k}$,
and, $\tsR_{\ell_k} \to \tsR_{\wp_k} $ is a single  blowup
based on $L_{F_k}$.

We do it by induction on the set $[\up]$.

As a part of the initial data on
the scheme $\tsR_{\ell_{-1}}=\tsR_{\vt}$, we have equipped it with the following sets of divisors,

 $\bullet$  The set  $\cD_{\vt, \vp}$ of
  the proper transforms $X_{\vt, \uw}$ of $\vp$-divisors $X_\uw$ for all $\uw \in \II_{d,n}\- \um$.
  These are still called $\vp$-divisors. 

$\bullet$  The set  $\cD_{\vt, \vr}$ of the proper transforms $X_{\vt, (\uu,\uv)}$
of $\vr$-divisors $X_{(\uu,\uv)}$ for all $(\uu,\uv) \in \La_\sfm$.
These are still called $\vr$-divisors. 

 $\bullet$    The set $\cE_{\vt}$ of the proper transforms
 $E_{\vt, k} \subset \tsR_{\vt}$ of the $\vt$-exceptional divisors 
 $E_{\vt_{[k]}} \subset \tsR_{\vt_{[k]}}$ for all $k \in [\up]$.


$\bullet$    The set  $\cD_{\vt, \fL}$ of the proper transforms $D_{\vt, L_F}$
of $\fL$-divisors $D_{L_F}$ for all $\bF \in \sfm$.
These are still called $\fL$-divisors. 

On the intial scheme $\tsR_{\vt}$, we have the set of divisors
$$\cD_\vt=\cD_{\vt, \vr}\sqcup \cD_{\vt, \vp} \sqcup \cE_{\vt}.$$
In addition, we have $\cD_{\vt, \fL}$. The set $\cD_{\vt, \fL}$ will 
not be used until the $\ell$-blowup.

As the further initial data, we need to introduce the instrumental notion: $``$association$"$
with multiplicity  as follows. 

\begin{defn} Consider any main binomial relation $B \in \cB^\mn$ written as
$$B=B_{ \fV,(s_{F_k},s)}=T^+_B-T^-_B: \;\;\;\;\;\; x_{\fV, (\uu_s, \uv_s)}  x_{\fV, \uu_k}   - 
x_{(\um, \uu_k)} x_{\fV, \uu_s} x_{\fV, \uv_s}$$
for some $k \in [\up]$ and $s \in S_{F_k}$, where $\uu_k=\uu_{F_k}$.
Consider any $\vp$-divisor, $\vr$-divisor, or  exceptional divisor $Y$ on $\tsR_{\vt}$.

Let  $Y=X_{\vt,\uu}$ be any $\vp$ divisor for some $\uu \in \II_{d,n}$. 
We set
$$m_{Y, T^+_B}=\left\{
\begin{array}{rcl}
1, & \;\; \hbox{if  $\uu=\uu_k$} \\
0 ,& \;\; \hbox{otherwise.}\\ 
\end{array} \right. $$ 
$$m_{Y, T^-_B}=\left\{
\begin{array}{rcl}
1, & \;\; \hbox{if  $\uu=\uu_s$ or $\uu=\uv_s$,} \\
0 ,& \;\; \hbox{otherwise.} \\ 
\end{array} \right. $$ 

Let  $Y=X_{\vt, (\uu,\uv)}$ be any $\vr$ divisor. We set
$$m_{Y, T^+_B}=\left\{
\begin{array}{rcl}
1, & \;\; \hbox{if  $(\uu,\uv)=(\uu_s,\uv_s)$} \\
0 ,& \;\; \hbox{otherwise.}\\ 
\end{array} \right. $$

Due to Corollary \ref{no-(um,uu)}, we do not  associate
$X_{(\um,\uu_k)}$ with $T^-_B$.
Hence, we set
$$m_{Y, T^-_B}=0.$$


Let  $Y=E_{\vt, j}$ be any exceptional-divisor for some $j \in [\up]$. If $k=j$,
we set  $$m_{Y, T^+_B}=m_{Y, T^-_B}=0.$$
Suppose now $k \ne j$. We  set
$$m_{Y, T^+_B}=0, $$ 
$$m_{Y, T^-_B}=m_{X_{\vt, \uu_j}, T^-_B}. $$ 

We call the number $m_{Y, T^\pm_B}$ the multiplicity of $Y$ associated with the term $T^\pm_B$.
We say $Y$ is associated with $T^\pm_B$ if $m_{Y, T^\pm_B}$ is positive. 
We do not say $Y$ is associated with $T^\pm_B$
if the multiplicity $m_{Y, T^\pm_B}$ is zero.
\end{defn}

\begin{defn} Consider any linearized $\pl$ relation
$$L_F=\sum_{s \in S_F} \sgn (s) x_{(\uu_s, \uv_s)}.$$
for some $F\in \sfm$.  Fix any $s \in S_F$.
Consider any $\vp$-divisor, $\vr$-divisor, or  exceptional-divisor $Y$ on $\tsR_{\vt}$.

Let  $Y=X_{\vt,\uw}$ be any $\vp$ divisor for some $\uw \in \II_{d,n}$. We set
$m_{Y, s}=0.$

Let  $Y=X_{\vt, (\uu,\uv)}$ be any $\vr$ divisor. We set
$$m_{Y, s}=\left\{
\begin{array}{rcl}
1, & \;\; \hbox{if  $(\uu,\uv)=(\uu_s,\uv_s)$} \\
0 ,& \;\; \hbox{otherwise.}\\ 
\end{array} \right. $$

Let  $Y=E_{\vt, j}$ be any exceptional-divisor for some $k \in [\up]$. We let
$$m_{Y, s}=m_{X_{\vt, (\um, \uu_k)}, s} \;. $$

We call the number $m_{Y, s}$ the multiplicity of $Y$ associated with $s \in S_F$.
We say $Y$ is associated with $s$ if $m_{Y,s}$ is positive. We do not say $Y$ is associated with $s$
if the multiplicity $m_{Y,s}$ is zero.
\end{defn}

Now, take any $k \in [\up]$. 

We  suppose that all the blowups
$$\tsR_{\ell_{j}} \lra \tsR_{\wp_{j}} \lra \tsR_{\ell_{j-1}}$$  
have been constructed for all the blocks in 
$$\fG_{ [k-1]}=\bigsqcup_{j \in [k-1]} \fG_j$$
 such that for all $j \in [k-1]$:

$\bullet$ $\tsR_{\wp_{j}}\lra\tsR_{\ell_{j-1}}$ is a sequential $\wp$-blowups
with respect to  the main binomial relations of the block $\cB^\mn_{F_{j}}$.

$\bullet$ $\tsR_{\ell_{j}}\lra\tsR_{\wp_{j}}$ is a single $\ell$-blowup
with respect to  $L_{F_{j}}$.

We  are to construct 
$$\tsR_{\ell_{k}} \lra \tsR_{\wp_{k}} \lra \tsR_{\ell_{k-1}}$$  
in the next subsection.

\subsection{$\wp$-blowups and $\ell$-blowup in $(\fG_k)$}\label{wp/ell-centers} $\ $

\medskip
\centerline{\bf  $\wp$-blowups in $(\fG_k)$}
\smallskip

We proceed by applying  induction on 
the set $$\{(k\tau)\mu \mid k \in [\up], \tau \in [\ft_{F_k}],
\mu \in [\rho_{(k\tau)}]\},$$ ordered lexicographically on $(k,\tau, \mu)$, 
where $\rho_{(k\tau)}$ is a  to-be-defined 
finite positive integer depending on $(k\tau) \in \Index_{\cB^\mn}$
(cf. \eqref{indexing-Bmn}).


The initial case for $\wp$-blowups with respect to
the block $\fG_k$  is $\wp_{(k1)}\fr_0$ and the initial scheme is 
$\sR_{({\wp}_{(k1)}\fr_{0})}:=\tsR_{\ell_{k-1}}$. 
When $k=0$, we get $\sR_{({\wp}_{(11)}\fr_{0})}:=\tsR_{\ell_{-1}}=\tsR_\vt$. 

We suppose that the scheme $\tsR_{({\wp}_{(k\tau)}\fr_{\mu-1})}$ has been constructed and
the following package in $({\wp}_{(k\tau)}\fr_{\mu-1})$ has been introduced
for some  integer $\mu \in [\rho_{(k\tau)}]$, where $1\le \rho_{(k\tau)}\le \infty$ is 
 an integer depending on $(k\tau) \in \Index_{\cB^\mn}$. (It will be proved to be finite.)
 Here, to reconcile notations, we make the convention:
$$({\wp}_{(k\tau)}\fr_0):=({\wp}_{(k(\tau -1))}\fr_{\rho_{(k(\tau-1))}}), \; \forall \;\; 1\le k\le \up, \;
2\le \tau\le \ft_{F_k},$$
$$({\wp}_{(k1)}\fr_0):=({\wp}_{((k-1)\ft_{F_{k-1}})}\fr_{\rho_{((k-1)\ft_{F_{k-1}})}}), \; \forall \;\; 2\le k\le \up, $$
provided that $\rho_{(k(\tau-1))}$ and $\rho_{((k-1)\ft_{F_{k-1}})}$
are (proved to be) finite.

\medskip\noindent
$\bullet$  {\sl The inductive assumption.} 
{\it The scheme $\tsR_{({\wp}_{(k\tau)}\fr_{\mu-1})}$ has been constructed.
It comes equipped with the following.

The set of $\vp$-divisors, 
$$\cD_{({\wp}_{(k\tau)}\fr_{\mu-1}),\vp}: \;\; X_{({\wp}_{(k\tau)}\fr_{\mu-1}), \uw}, \;\; \uw \in \II_{d,n} \- \um.$$

The set of $\vr$-divisors
$$\cD_{({\wp}_{(k\tau)}\fr_{\mu-1}), \vr}: \;\; X_{({\wp}_{(k\tau)}\fr_{\mu-1}), (\uu,\uv)}, \;\; (\uu, \uv) \in \La_\sfm.$$


The set of $ \fL$-divisors 
$$\cD_{({\wp}_{(k\tau)}\fr_{\mu-1}),  \fL}: 
{ D_{({\wp}_{(k\tau)}\fr_{\mu-1}), L_{F_j},} \;\; j \in [k-1].}$$

 The set of the exceptional divisors 
 $$\cE_{({\wp}_{(k\tau)}\fr_{\mu-1})}: \;\; E_{({\wp}_{(k\tau)}\fr_{\mu-1}), (k'\tau') \mu' h'},\;\; 
 \hbox{${ (11)0}\le (k'\tau') \mu' \le (k\tau)(\mu -1), \; h' \in [\si_{(k'\tau') \mu'}]$} $$
for some finite positive integer $\si_{(k'\tau')\mu'}$ depending on $(k'\tau')\mu'$.
We set $\si_{(11)0}=\up$. This counts the number of exceptional divisors on 
$\tsR_{(\wp_1\fr_1\fs_0)}=\tsR_\vt$.


We let $$\cD_{({\wp}_{(k\tau)}\fr_{\mu-1})}=\cD_{({\wp}_{(k\tau)}\fr_{\mu-1}),\vr}
 \sqcup \cD_{({\wp}_{(k\tau)}\fr_{\mu-1}),\vp}
\sqcup \cE_{({\wp}_{(k\tau)}\fr_{\mu-1})}$$ be the set of all  the aforelisted divisors. 
The set $\cD_{({\wp}_{(k\tau)}\fr_{\mu-1}),  \fL}$ will not be used until the $\ell$-blowup.

Fix  any $Y \in \cD_{({\wp}_{(k\tau)}\fr_{\mu-1})}$.
 Consider any $B \in \cB^\q \cup \cB^\mn$ and let $T_B$ be any fixed  term of $B$.
Then, we have that $Y$ is associated with $T_B$ with the multiplicity  $m_{Y,T_B}$, 
a nonnegative integer.  
In the sequel, we say $Y$ is associated with $T_B$  if $m_{Y,T_B}>0$; 
we do not say $Y$ is associated with $T_B$  if $m_{Y,T_B}=0$.

Likewise, for any term of $T_s=\sgn (s) x_{(\uu_s,\uv_s)}$ of $L_F
=\sum_{s \in S_F} \sgn (s) x_{(\uu_s,\uv_s)}$, $Y$ 
is associated with $T_s$ with the multiplicity  $m_{Y,s}$, a nonnegative integer. 
We say $Y$ is associated with $T_s$  if $m_{Y,s}>0$; 
we do not say $Y$ is associated with $T_s$  if $m_{Y,s}=0$.}

\smallskip
We are now  to construct the scheme $\tsR_{({\wp}_{(k\tau)}\fr_\mu)}$.
 The process consists of a finite steps of blowups; the scheme $\tsR_{({\wp}_{(k\tau)}\fr_\mu)}$
is the one obtained in the final step. 




As before, fix any $k \in [\up]$,  we write $\cB^\mn_{F_k}=\{B_{(k\tau)} \mid \tau \in [\ft_{F_k}]\}.$
For every $B_{(k\tau)} \in \cB_{F_k}^\mn$, we have the expression  
$$B_{(k\tau)}=T_{(k\tau)}^+ -T_{(k\tau)}^- =x_{(\uu_s,\uv_s)}x_{\uu_k} -
x_{\uu_s}x_{\uv_s}x_{(\um,\uu_k)} $$  where $ s\in S_{F_k} \- s_{F_k}$ corresponds to $\tau$
and
$x_{\uu_k}$ is the leading variable of $\bF_k$ for some $\uu_k \in \II_{d,n}^\um$.
We can write $s=(k\tau)$, and, use $B_s$ and $B_{(k\tau)}$ interchangeblly. 

\begin{defn}\label{wp-sets-kmu} 
A pre-${\wp}$-set $\phi$ in $({\wp}_{(k\tau)}\fr_\mu)$, written as
$$\phi=\{Y^+, \; Y^- \},$$
 consists of exactly two divisors 
 of the scheme $\tsR_{({\wp}_{(k\tau)}\fr_{\mu-1})}$ such that  
$Y^\pm$ is  associated with $T_{(k\tau)}^\pm$. 
  
  Given the above pre-${\wp}$-set $\phi$, we let 
$$Z_\phi = Y^+ \cap Y^-$$
 be the scheme-theoretic intersection.
The pre-$\wp$-set $\phi$ (resp. $Z_\phi $) is called a $\wp$-set (resp. $\wp$-center) 
in $({\wp}_{(k\tau)}\fr_\mu)$ if $$Z_\phi \cap \tsV_{({\wp}_{(k\tau)}\fr_{\mu-1})} \ne \emptyset.$$
In such a case, we also call $\phi$ (resp. $Z_\phi $) a $\wp_k$-set (resp. $\wp_k$-center). 
\end{defn}
Recall that due to Corollary \ref{no-(um,uu)}, we do not  associate
$X_{(\um,\uu_k)}$ with $T_{(k\tau)}^-$.
Hence $Y^- \ne X_{({\wp}_{(k\tau)}\fr_{\mu-1}), (\um,\uu_k)}$.
Had we associated $X_{(\um,\uu_k)}$ with $T_{(k\tau)}^-$,  the condition,
$Z_\phi \cap \tsV_{({\wp}_{(k\tau)}\fr_{\mu-1})} \ne \emptyset$, would also exclude it.



As there are only finitely many $\vp$-, $\vr$-, and exceptional 
divisors on the scheme $\tsR_{({\wp}_{(k\tau)}\fr_{\mu-1})}$, that is,
the set $\cD_{({\wp}_{(k\tau)}\fr_{\mu-1})}$ is finite,
one sees that there are only finitely many ${\wp}$-sets in $({\wp}_{(k\tau)}\fr_\mu)$. 
We let $\Phi_{{\wp}_{(k\tau)}\fr_\mu}$ be the finite set of all ${\wp}$-sets
in $({\wp}_{(k\tau)}\fr_\mu)$; we let $\cZ_{{\wp}_{(k\tau)}\fr_\mu}$ be the finite set of all 
corresponding ${\wp}$-centers in $({\wp}_{(k\tau)}\fr_\mu)$.
We need  a total ordering on the set
$\Phi_{{\wp}_{(k\tau)}\fr_\mu}$, hence on the set $\cZ_{{\wp}_{(k\tau)}\fr_\mu}$,
to produce a canonical sequential blowups.

\begin{defn}\label{order-phi} Let $\cD^\pm_{({\wp}_{(k\tau)}\fr_{\mu-1})}$ be the set of all 
divisors associated with $T_{(k\tau)}^\pm$. 

We order the set  $\cD^+_{({\wp}_{(k\tau)}\fr_{\mu-1})}$  as follows.
We let $X_{({\wp}_{(k\tau)}\fr_{\mu-1}),(\uu_s,\uv_s)}$ be the largest and
$X_{({\wp}_{(k\tau)}\fr_{\mu-1}),\uu_k}$ be the second largest. The rest are
exceptional divisors. We order them by reversing the order of occurrence
of the exceptional divisors. 
By Definition \ref{p-t}, $\cD^+_{({\wp}_{(k\tau)}\fr_{\mu-1})}$ is totally ordered.

We order the set  $\cD^-_{({\wp}_{(k\tau)}\fr_{\mu-1})}$  as follows.
We let $\cD^-_{({\wp}_{(k\tau)}\fr_{\mu-1}),\vp}$ be the subset 
of $\vp$-divisors with the order 
induced from that  on the set of all $\pl$-variables.
We let $\cE^-_{({\wp}_{(k\tau)}\fr_{\mu-1})}$  be the subset of exceptional divisors
by reversing the order of occurrence.
We then declare
$$
\cE^-_{({\wp}_{(k\tau)}\fr_{\mu-1})}<  \cD^-_{({\wp}_{(k\tau)}\fr_{\mu-1}),\vp}.$$
By Definition \ref{p-t}, $\cD^-_{({\wp}_{(k\tau)}\fr_{\mu-1})}$ is totally ordered.

Now, let $\phi_1, \phi_2 \in \Phi_{{\wp}_{(k\tau)}\fr_\mu}$ be any two distinct elements.
Write $\phi=\{Y_i^+, Y_i^-\}, i=1, 2$.
We $\phi_1 < \phi_2$ if $Y_1^+<Y_2^+$ or $Y_1^+=Y_2^+$ and $Y_1^-<Y_2^-$. 

This endows $\Phi_{{\wp}_{(k\tau)}\fr_\mu}$ a total order $`` < "$.
\end{defn}

Thus, we can list $\Phi_{{\wp}_{(k\tau)}\fr_\mu}$ as
$$\Phi_{ {\wp}_{(k\tau)}\fr_\mu}=\{\phi_{(k\tau)\mu 1} < \cdots < \phi_{(k\tau)\mu \si_{(k\tau)\mu}}\}$$
for some finite positive integer $\si_{(k\tau)\mu}$ depending on $(k\tau)\mu$. We then let the set 
$\cZ_{\wp_{(k\tau)}\fr_\mu}$ of the corresponding $\wp$-centers
inherit the total order from that of $\Phi_{{\wp}_{(k\tau)}\fr_\mu}$. 
Then, we can express
$$\cZ_{\wp_{(k\tau)}\fr_\mu}=\{Z_{\phi_{(k\tau)\mu1}} < \cdots < Z_{\phi_{(k\tau)\mu\si_{(k\tau)\mu}}}\}.$$


We let $\tsR_{({\wp}_{(k\tau)}\fr_\mu \fs_1)} \lra \tsR_{({\wp}_{(k\tau)}\fr_{\mu-1})}$ 
be the blowup of $\tsR_{({\wp}_{(k\tau)}\fr_{\mu-1})}$
along the {$\wp$-}center $Z_{\phi_{(k\tau)\mu1}}$. 
Inductively, we assume that 
$\tsR_{({\wp}_{(k\tau)}\fr_\mu\fs_{(h-1)})}$ has been constructed for some 
$ h \in [\si_{(k\tau)\mu}]$.
We then let  $$\tsR_{({\wp}_{(k\tau)}\fr_\mu\fs_h)} \lra \tsR_{({\wp}_{(k\tau)}\fr_\mu\fs_{h-1})}$$ be the blowup of
$ \tsR_{({\wp}_{(k\tau)}\fr_\mu\fs_{h-1})}$ along (the proper transform of) the  $\wp$-center 
$Z_{\phi_{(k\tau)\mu h}}$. We call it a $\wp$-blowup or a $\wp_k$-blowup.

Here, to reconcile notation, we set 
$$\tsR_{({\wp}_{(k\tau)}\fr_\mu\fs_{0})}:=\tsR_{({\wp}_{(k\tau)}\fr_{\mu-1})}
:=\tsR_{({\wp}_{(k\tau)}\fr_{\mu-1}\fs_{\si_{(k\tau)(\mu-1)}})}.$$

All of these can be summarized as the sequence 
$$\tsR_{({\wp}_{(k\tau)}\fr_\mu)}:= \tsR_{({\wp}_{(k\tau)}\fr_\mu\fs_{\si_{(k\tau)\mu}})}
 \lra \cdots \lra \tsR_{({\wp}_{(k\tau)}\fr_\mu\fs_1)} \lra \tsR_{({\wp}_{(k\tau)}\fr_{\mu-1})}.$$

 Given $h \in [\si_{(k\tau)\mu}]$,
consider the induced morphism
$\tsR_{({\wp}_{(k\tau)}\fr_\mu\fs_h)} \lra 
\tsR_{({\wp}_{(k\tau)}\fr_{\mu-1})}$. 

$\bullet$ We let $X_{({\wp}_{(k\tau)}\fr_\mu\fs_h), \uw}$ be the proper transform
of $X_{(\wp_{(k\tau)}\fr_{\mu-1}), \uw}$ in $\tsR_{({\wp}_{(k\tau)}\fr_\mu\fs_h)}$,
for all $\uw \in \II_{d,n} \- \um$. These are still called $\vp$-divisors.
We denote the set of all $\vp$-divisors
 on  $\tsR_{({\wp}_{(k\tau)}\fr_\mu\fs_h)}$ by  $\cD_{({\wp}_{(k\tau)}\fr_\mu\fs_h),\vp}$.

$\bullet$ We let $X_{({\wp}_{(k\tau)}\fr_\mu\fs_h), (\uu, \uv)}$ be the proper transform
  of the $\vr$-divisor $X_{({\wp}_{(k\tau)}\fr_{\mu-1}), (\uu, \uv)}$ in $\tsR_{({\wp}_{(k\tau)}\fr_\mu\fs_h)}$,
  for all $(\uu, \uv) \in \La_\sfm$. These are still called $\vr$-divisors.
We denote the set of all $\vr$-divisors
 on  $\tsR_{({\wp}_{(k\tau)}\fr_\mu\fs_h)}$ by  $\cD_{({\wp}_{(k\tau)}\fr_\mu\fs_h),\vr}$.
 
 
 $\bullet$ We let $D_{({\wp}_{(k\tau)}\fr_\mu\fs_h), L_F}$ be the proper transform
  of the $\fL$-divisor $D_{({\wp}_{(k\tau)}\fr_{\mu-1}), L_F}$ in $\tsR_{({\wp}_{(k\tau)}\fr_\mu\fs_h)}$,
  for all $\bF \in \sfm$. These are still called $\fL$-divisors.
We denote the set of all $\fL$-divisors
 on  $\tsR_{({\wp}_{(k\tau)}\fr_\mu\fs_h)}$ by  $\cD_{({\wp}_{(k\tau)}\fr_\mu\fs_h),\fL}$.

 $\bullet$ We let
 $E_{({\wp}_{(k\tau)}\fr_\mu\fs_h),  (k'\tau')\mu' h'}$ be the proper transform 
 of $E_{({\wp}_{(k\tau)}\fr_{\mu-1}),  (k'\tau')\mu' h'}$ in $\tsR_{({\wp}_{(k\tau)}\fr_\mu \fs_h)}$,
 for all   $(11)0 \le (k'\tau')\mu'\le (k\tau)(\mu-1)$
 with $h' \in [\si_{(k'\tau')\mu'}]$.
We denote the set of these exceptional divisors on  
  $\tsR_{({\wp}_{(k\tau)}\fr_\mu\fs_h)}$ by $\cE_{({\wp}_{(k\tau)}\fr_\mu \fs_h),{\rm old}}$.
 
  We let  $$\bar\cD_{({\wp}_{(k\tau)}\fr_\mu\fs_h)}= \cD_{({\wp}_{(k\tau)}\fr_\mu\fs_h),\vp}
  \sqcup \cD_{({\wp}_{(k\tau)}\fr_\mu\fs_h),\vr} \sqcup \cE_{({\wp}_{(k\tau)}\fr_\mu\fs_h),{\rm old}}$$ 
  be the set of  all of the aforementioned divisors on $\tsR_{({\wp}_{(k\tau)}\fr_\mu \fs_h)}$.
 The set $\cD_{({\wp}_{(k\tau)}\fr_\mu\fs_h),\fL}$ will not be used until the $\ell$-blowup.

 In addition to the proper transforms of the divisors
  from $\tsR_{({\wp}_{(k\tau)}\fr_{\mu-1})}$, there are
the following {\it new} exceptional divisors.

For any $ h \in [\si_{(k\tau)\mu}]$, we let $E_{({\wp}_{(k\tau)}\fr_\mu\fs_h)}$ be the exceptional divisor of 
the blowup $\tsR_{({\wp}_{(k\tau)}\fr_\mu\fs_h)} \lra \tsR_{({\wp}_{(k\tau)}\fr_\mu\fs_{h-1})}$.
 For any $1\le h'< h\le \si_{(k\tau)\mu}$,
we let $E_{({\wp}_{(k\tau)}\fr_\mu\fs_h), (k\tau)\mu h'}$ 
be the proper transform in $\tsR_{({\wp}_{(k\tau)}\fr_\mu \fs_h)} $
of the exceptional divisor $E_{({\wp}_{(k\tau)}\fr_\mu \fs_{h'})}$. 
To reconcile notation, we also set 
$E_{({\wp}_{(k\tau)}\fr_\mu\fs_h), (k\tau)\mu h}:=E_{({\wp}_{(k\tau)}\fr_\mu \fs_h)}$.
We  set $$\cE_{({\wp}_{(k\tau)}\fr_\mu \fs_h),{\rm new}}=\{ E_{({\wp}_{(k\tau)}\fr_\mu\fs_h), 
(k\tau)\mu h'} \mid 1\le h'\le h\le \si_{(k\tau)\mu} \} .$$
We then order the exceptional divisors of $ \cE_{({\wp}_{(k\tau)}\fr_\mu\fs_h),{\rm new}}$
by  reversing the order of occurrence, that is, 
$E_{({\wp}_{(k\tau)}\fr_\mu\fs_h), (k\tau)\mu h''}\le E_{({\wp}_{(k\tau)}\fr_\mu\fs_h), (k\tau)\mu h'}$
 if $h'' \ge h'$.
 
 We then let 
 $$\cD_{({\wp}_{(k\tau)}\fr_\mu \fs_h)} =\bar\cD_{({\wp}_{(k\tau)}\fr_\mu \fs_h)} 
\sqcup  \cE_{({\wp}_{(k\tau)}\fr_\mu \fs_h),{\rm new}}.$$

Finally, we set
 $\tsR_{({\wp}_{(k\tau)})\fr_\mu}:= \tsR_{({\wp}_{(k\tau)}\fr_\mu\fs_{\si_{(k\tau)\mu}})}$, and let
 $$\cD_{({\wp}_{(k\tau)}\fr_\mu),\vp}=\cD_{({\wp}_{(k\tau)}\fr_\mu\fs_{\si_{(k\tau)\mu}}),\vp}, \;
 \cD_{({\wp}_{(k\tau)}\fr_\mu),\vr}=\cD_{({\wp}_{(k\tau)}\fr_\mu\fs_{\si_{(k\tau)\mu}}),\vr},$$
 $$\cE_{{\wp}_{(k\tau)}\fr_\mu}=\cE_{({\wp}_{(k\tau)}\fr_\mu\fs_{\si_{(k\tau)\mu}}), {\rm old}}  
 \sqcup \cE_{({\wp}_{(k\tau)}\fr_\mu\fs_{\si_{(k\tau)\mu}}), {\rm new}} .$$
 This can be  summarized as 
$$\cD_{({\wp}_{(k\tau)}\fr_\mu)}:=\cD_{({\wp}_{(k\tau)}\fr_\mu),\vp}  
\sqcup \cD_{({\wp}_{(k\tau)}\fr_\mu),\vr}
\sqcup \cE_{({\wp}_{(k\tau)}\fr_\mu)}.$$
 This way,  we have equipped 
  the scheme  $\tsR_{({\wp}_{(k\tau)}\fr_\mu)}$ 
  with the set $\cD_{({\wp}_{(k\tau)}\fr_\mu),\vp}$ of $\vp$-divisors, 
 the set $\cD_{({\wp}_{(k\tau)}\fr_\mu),\vr}$ of $\vr$-divisors, 
 the set  $\cE_{{\wp}_{(k\tau)}\fr_\mu}$ of exceptional divisors,
 together with the set $\cD_{({\wp}_{(k\tau)}\fr_\mu),\fL}$ of $\fL$-divisors
 (which will not be used until the $\ell$-blowup).

Now, we are ready to introduce the notion of $``$association$"$ in  $({\wp}_{(k\tau)}\fr_\mu)$,
as required to carry on the process of induction. 

We do it inductively  on  the set $ [ \si_{(k\tau)\mu}]$. 



\begin{defn} \label{association-vs} 
Fix any $B \in \cB^\q \cup \cB^\mn$. We let $T_B$ be any fixed term of the binomial $B$.
Meanwhile, we also consider any $\bF \in \sfm$ and let $T_s$ be the term of $L_F$ corresponding to
any fixed $s \in S_F$.

We assume that the notion of $``$association$"$ in  
$({\wp}_{(k\tau)}\fr_\mu\fs_{h-1})$ has been introduced. That is, for every divisor 
$Y' \in \cD_{({\wp}_{(k\tau)}\fr_\mu\fs_{h-1})}$, the multiplicities
$m_{Y', T_B}$ and $m_{Y',s}$ have been defined.

Consider an arbitrary divisor $Y \in \cD_{({\wp}_{(k\tau)}\fr_\mu\fs_h)}$.



First, suppose $Y \ne E_{({\wp}_{(k\tau)}\fr_\mu\fs_h)}$.
Then, it is the proper transform of a (unique) divisor $Y' \in \cD_{({\wp}_{(k\tau)}\fr_\mu\fs_{h-1})}$.
In this case, we set 
$$m_{Y, T_B}=m_{Y', T_B},\;\;\; m_{Y,s}=m_{Y',s}.$$

Next, we consider the exceptional $Y=E_{({\wp}_{(k\tau)}\fr_\mu\fs_h)}$.

We let $\phi= \phi_{(k\tau)\mu h}$. 
We have that 
$$\phi=\{ Y^+,  Y^-  \} \subset \cD_{(\wp_{(k\tau)}\fr_{\mu-1})}.$$

For any $B \in \cB^\mn \cup \cB^\q$, we write $B=T_B^0-T_B^1$. We let
$$m_{\phi, T_B^i}=m_{Y^+, T_B^i}+ m_{Y^-, T_B^i}, \; i=0,1,$$ 
 $$l_{\phi, B}=\min \{m_{\phi, T_B^0}, m_{\phi, T_B^1}\}.$$ 
 (For instance, by definition, $l_{\phi, B}>0$ when $B=B_{(k\tau)}$. In general, it can be zero.)
 Then, we let
  $$m_{E_{({\wp}_{(k\tau)}\fr_\mu\fs_h)},T_B^i}= m_{\phi, T_B^i}-  l_{\phi, T_B^i}.$$

Likewise, 
for $s \in S_F$ with $F\in \sfm$,
we let
$$m_{E_{(\wp_{(k\tau)}\fr_\mu\fs_h)},s}=m_{Y^+, s} +  m_{Y^-, s}.$$

We say $Y$ is associated with $T_B$ (resp. $T_s$)  if its multiplicity $m_{Y, T_B}$ 
(resp. $m_{Y, s}$) is positive.
We do not say $Y$ is associated with $T_B$ (resp. $T_s$)   
if its multiplicity $m_{Y,T_B}$ (resp. $m_{Y, s}$) equals to zero.
\end{defn}

When $h = \si_{(k\tau)\mu}$, we obtain all the desired data on 
$\tsR_{({\wp}_{(k\tau)}\fr_\mu)}=\tsR_{({\wp}_{(k\tau)}\fr_\mu)\fs_{\si_{(k\tau)\mu}}}$.

Now, with all the aforedescribed data equipped for $\tsR_{({\wp}_{(k\tau)}\fr_\mu)}$, 
we obtain our inductive package in 
 $({\wp}_{(k\tau)}\fr_\mu)$. This allows us to introduce the set 
 $\Phi_{\wp_{(k\tau)}\fr_{\mu+1}}$ of ${\wp}$-sets 
 and the set  $\cZ_{\wp_{(k\tau)}\fr_{\mu+1}}$ of ${\wp}$-centers
 in  $({\wp}_{(k\tau)}\fr_{\mu+1})$ as in Definition \ref{wp-sets-kmu}, 
 endow a total order on $\Phi_{\wp_{(k\tau)}\fr_{\mu+1}}$
 and  $\cZ_{\wp_{(k\tau)}\fr_{\mu+1}}$ as in 
 Definition \ref{order-phi}, 
 and then  advance to the next 
round of  the $\wp$-blowups.  Here, 
 to reconcile notations, we set
 $$({\wp}_{(k\tau)}\fr_{\rho_{(k\tau)+1}}):=({\wp}_{((k(\tau +1))}\fr_{1}), \;\; 1\le \tau <\ft_{F_k};$$
$$({\wp}_{(k\ft_{F_k})}\fr_{\rho_{(k\ft_{F_k})+1}}):=({\wp}_{((k+1)1)}\fr_{1}), \;\;  1\le k < \up,$$
provided that $\rho_{(k\tau)}$ and $\rho_{(k\ft_{F_k})}$ are (proved to be) finite.

Given any  $({\wp}_{(k\tau)}\fr_\mu \fs_h)$, the ${\wp}$-blowup in (${\wp}_{(k\tau)}\fr_\mu\fs_h$)
gives rise to \begin{equation}\label{tsv-ktauh} \xymatrix{
\tsV_{({\wp}_{(k\tau)}\fr_\mu\fs_h)} \ar[d] \ar @{^{(}->}[r]  & \tsR_{({\wp}_{(k\tau)}\fr_\mu\fs_h)} \ar[d] \\
\sV \ar @{^{(}->}[r]  & \sR_\sF,
}
\end{equation}
where $\tsV_{({\wp}_{(k\tau)}\fr_\mu\fs_h)} $ is the proper transform of $\sV$ 
in  $\tsR_{({\wp}_{(k\tau)}\fr_\mu\fs_h)}$. 

We let $\tsV_{({\wp}_{(k\tau)}\fr_\mu)}=\tsV_{({\wp}_{(k\tau)}\fr_\mu\fs_{\si_{(k\tau)\mu}})}$.

\begin{defn}\label{defn:rhoktau}
Fix any $k \in [\up], \tau \in [\ft_{F_k}]$. Suppose there exists a finite integer $\mu$ such that
for any pre-$\wp$-set $\phi$  
in $({\wp}_{(k\tau)}\fr_{\mu +1})$ (cf. Definition \ref{wp-sets-kmu}), 
we have 
$$Z_\phi  \cap \tsV_{({\wp}_{(k\tau)}\fr_\mu)} = \emptyset.$$
Then, we let $\rho_{(k\tau)}$ be the smallest integer such that the above holds.
Otherwise, we let $\rho_{(k\tau)}=\infty$.
\end{defn}
{\it It will be shown soon 
that $\rho_{(k\tau)}$ is finite for all  $k \in [\up]$,
$\tau \in [\ft_{F_k}]$.}

For later use, we let
\begin{equation}\label{indexing-Phi}
\Phi_k=\{ \phi_{(k\tau)\mu h} \mid                    
\tau \in [\ft_{F_k}],
\mu [\rho_{(k\tau)}],
 h \in [\si_{(k\tau)\mu}]\},  \end{equation}
$$\Index_{\Phi_k}=\{ (k\tau)\mu h \mid     
 \tau \in [\ft_{F_k}],  \mu \in [\rho_{(k\tau)}],
 h \in [\si_{(k\tau)\mu}]\}. $$
 Then, the order of $\wp_k$-blowups coincides with the lexicographical  order 
on $\Index_{\Phi_k}$

Upon proving that $\rho_{(k\tau)}$ is finite  for all $\tau \in [\ft_{F_k}]$,
 we can summarize the process
of ${\wp_k}$-blowups as a single sequence of blowup morphisms:
\begin{equation}\label{grand-sequence-wp}
\tsR_{\wp_k} \lra \cdots \lra 
\tsR_{({\wp}_{(k\tau)}\fr_\mu\fs_h)} \lra \tsR_{({\wp}_{(k\tau)}\fr_\mu\fs_{h-1})}
\lra \cdots \lra \tsR_{({\wp}_{(k1)}\fr_0)}:=\tsR_{\ell_{k-1}},\end{equation}
where $\tsR_{\wp_k}:=\tsR_{({\wp}_{(k \ft_{F_k})}\fr_{\rho_{k \ft_{F_k}}})} 
:=\tsR_{({\wp}_{(k \ft_{F_k})}\fr_{\rho_{k \ft_{F_k}}}\fs_{\si_{(k \ft_{F_k})\rho_{k \ft_{F_k}}}})} $
is the  blowup scheme reached in  the final step 
$({\wp}_{(k \ft_{F_k})}\fr_{\rho_{k \ft_{F_k}}}
\fs_{\si_{(k \ft_{F_k})\rho_{k \ft_{F_k}}}})$ of all $\wp$-blowups in $(\fG_k)$.

Further, the end of all ${\wp}$-blowups in $(\fG_k)$ gives rise to the following induced diagram 
\begin{equation}\label{tsv-final} \xymatrix{
\tsV_{\wp_k} \ar[d] \ar @{^{(}->}[r]  & \tsR_{\wp_k} \ar[d] \\
\sV \ar @{^{(}->}[r]  & \sR_\sF,
}
\end{equation}
where $\tsV_{\wp_k}$ is the proper transform of $\sV$  in  $\tsR_{\wp_k}$.


\bigskip
\centerline{\bf  The $\ell$-blowup in $(\fG_k)$}
\medskip

\begin{defn}\label{ell-set-k} 
Let $D_{\wp_k, L_{F_k}}$ be the $\fL$-divisor defined by
$L_{F_k}$ and $E_{\wp_k, \vt_k}$ be  the proper transform
of the exceptional divisor $E_{\vt, k}$ of $\tsR_\vt$ (or equivalently,
the exceptional divisor  $E_{\vt_{[k]}}$ of $\tsR_{\vt_{[k]}}$) in $\tsR_{\wp_k}$.
We call the set of the two divisors 
$$\chi_k=\{D_{\wp_k, L_{F_k}}, E_{\wp_k, \vt_k} \}$$
the $\ell$-set with respect to $L_{F_k}$ or just $\ell_k$-set.
We let $Z_{\chi_k}$ be the scheme-theoretic intersection
$$Z_{\chi_k} = D_{ \wp_k, L_{F_k}} \cap Y_{ \wp_k, \vt_k}$$
and call it $\ell$-center with respect to $L_{F_k}$ or just $\ell_k$-center.
\end{defn}

We then let
$$\tsR_{\ell_k} \lra \tsR_{\wp_k}$$
be the blowup of $\tsR_{\wp_k}$ along $Z_{\chi_k}$.
The is called the $\ell$-blowup with respect to $L_{F_k}$ or just $\ell_k$-blowup.

Recall that $\tsR_{\wp_k}$ comes equipped with the sets of
$\vr$-, $\vp$-, $\fL$-, and exceptional divisors.
For any such a divisor, we take its proper transform in $\tsR_{\ell_k}$ and let
it inherit its original name.

There is only one new divisor on $\tsR_{\ell_k}$:
we let $E_{\ell_k}$ be the exceptional divisor of the blowup 
$\tsR_{\ell_k} \lra \tsR_{\wp_k}$,  and call it the $\ell_k$-exceptional divisor.
(We comment here that this is not to be confused with the $\fL$-divisor
$D_{\ell_k, L_k}$.)

Now, we are ready to introduce the notion of $``$association$"$ in  $(\ell_k)$,
as required to carry on the process of induction.

\begin{defn} \label{ell-association} 
Fix any binomial $B \in \cB^\mn_G \sqcup B^\q$ with $G>F_k$, written as $B=T_B^+ -T_B^-$.
For any $\bF \in \sfm$, let $T_s$ be the term of $L_F$ corresponding to
some fixed $s \in S_F$.

As assumed, the notion of $``$association$"$ in  
$ (\wp_k)$ has been introduced. That is, for any divisor 
$Y'$ on $\tsR_{\wp_k}$ , the multiplicities
 $m_{Y', T^\pm_B}$, $m_{Y',s}$ have been defined.

Suppose a divisor $Y$ of  $\tsR_{\ell_k}$ is the proper transform 
of $Y'$ of  $\tsR_{\wp_k}$.
We define
$$ m_{Y, T^\pm_{B}} =  m_{Y', T^\pm_{B}}, \;\; m_{Y,s}= m_{Y', s}.$$
We define 
$$  m_{E_{\ell_k}, T^\pm_{B}} =  m_{E_{\wp_k, \vt_k}, T^\pm_{B}}, \;\;
 m_{E_{\ell_k},s}= m_{E_{\wp_k, \vt_k}, s}.$$

We say $Y$ is associated with  $T_B^\pm$ or $T_s$ if its multiplicity  
$m_{Y, T^\pm_B}$ or $m_{Y,s}$  is positive.
\end{defn}

Finally, we have the following diagram \begin{equation}\label{ell_k-schemes} \xymatrix{
\tsV_{\ell_k} \ar[d] \ar @{^{(}->}[r]  & \tsR_{\ell_k} \ar[d] \\
\sV_{\wp_k} \ar @{^{(}->}[r]  & \sR_{\wp_k},  }
\end{equation}
where $\tsV_{\ell_k} $ is the proper transform of $\sV$ 
in  $\tsR_{\ell_k} $.

When $k=\up$, we write $\tsR_{\ell}:= \tsR_{\ell_\up}$
and $\tsV_{\ell}:=\tsV_{\ell_\up}$. We obtain
\begin{equation}\label{final-schemes} \xymatrix{
\tsV_{\ell} \ar[d] \ar @{^{(}->}[r]  & \tsR_{\ell} \ar[d] \\
\sV \ar @{^{(}->}[r]  & \sR.  }
\end{equation}
The schemes $(\tsV_{\ell} \subset \tsR_{\ell})$ are the ones we aim to construct.

\medskip

For the use in induction in the next subsection, we set
$\tsR_{({\wp}_{(k1)}\fr_1\fs_{0})}:=\tsR_{\ell_{k-1}}$.
When $k=0$, we have 
$\tsR_{({\wp}_{(11)}\fr_1\fs_{0})}:=\tsR_{\ell_{-1}}:=\tsR_{\vt}$.

In the next two sections, we will state and prove  certain properties about 
$$\tsV_{\wp_{(k\tau)}\fr_\mu\fs_h} \subset \tsR_{\wp_{(k\tau)}\fr_\mu\fs_h},$$
 using induction on the indexes $(k\tau)\mu h$.
To include $(\tsV_{\ell_k} \subset \tsR_{\ell_k})$ in the statements and proofs,
 we let ${\ell_k}$ correspond to one more step after the last step of $\wp_k$-blowups,
namely, $\si_{(k \ft_{F_k})\rho_{k \ft_{F_k}}}+1$. In full writing, we set
$$\tsR_{({\wp}_{(k \ft_{F_k})}\fr_{\rho_{k \ft_{F_k}}}\fs_{(\si_{(k \ft_{F_k})\rho_{k \ft_{F_k}}}+1)})}
:=\tsR_{\ell_k}, \;\; 
\tsV_{({\wp}_{(k \ft_{F_k})}\fr_{\rho_{k \ft_{F_k}}}\fs_{(\si_{(k \ft_{F_k})\rho_{k \ft_{F_k}}}+1)})}
:=\tsV_{\ell_k}, $$
We also write 
$$\ell_k:=({\wp}_{(k \ft_{F_k})}\fr_{\rho_{k \ft_{F_k}}}
\fs_{(\si_{(k \ft_{F_k})\rho_{k \ft_{F_k}}}+1)}).$$

\subsection{Properties of $\wp$-blowups and $\ell$-blowup in $(\fG_k)$}\label{subs:prop-vs-blowups} 
$\ $

Fix and consider any 
$$(k\tau)\mu h \in
 \Index_{\Phi_{k}} \sqcup \{ ((k \ft_{F_k})\rho_{k \ft_{F_k}}
(\si_{(k \ft_{F_k})\rho_{k \ft_{F_k}}}+1))\}$$
 (cf. \eqref{indexing-Phi} for $\Index_{\Phi_{k}}$).

\begin{prop}\label{meaning-of-var-wp/ell} 
Suppose that  the scheme $\tsR_{({\wp}_{(k\tau)}\fr_\mu\fs_h)}$ has been constructed,
covered by a finite set of open subsets, called standard charts (see Definition \ref{general-standard-chart}).

Consider any standard chart $\fV$ of $\tsR_{({\wp}_{(k\tau)}\fr_\mu\fs_h)}$, 
 lying over a unique chart $\fV'$ of $\tsR_{({\wp}_{(k\tau)}\fr_\mu\fs_{h-1})}$,
 and a unique chart $ \fV_{[0]}$ of $\sR_\sF$.
 We suppose that the chart $ \fV_{[0]}$ is indexed by
$\La_\sfm^o=\{(\uv_{s_{F,o}},\uv_{s_{F,o}}) \mid \bF \in \sfm \}.$  
We set 
$$\hbox{$\II_{d,n}^*=\II_{d,n}\- \um$, $\La_\sfm^\star=\La_\sfm \- \La_\sfm^o$,
and $L_{\sfm, [k]}=\{L_{F_j} \mid j \in [k]\}$.}$$

Then, the scheme $\tsR_{({\wp}_{(k\tau)}\fr_\mu\fs_h)}$ has
a smooth open subset $ \tsR^\circ_{({\wp}_{(k\tau)}\fr_\mu\fs_h)}$
containing the subscheme $\tsV_{({\wp}_{(k\tau)}\fr_\mu\fs_h)}$.
By shrinking the open subsets, if necessary, we can assume that the blowup morphism
sends the open subset 
$ \tsR^\circ_{({\wp}_{(k\tau)}\fr_\mu\fs_h)}$
onto the open subset $\tsR^\circ_{({\wp}_{(k\tau)}\fr_\mu\fs_{h-1})}$.
In the sequel, by a standard chart of $ \tsR^\circ_{({\wp}_{(k\tau)}\fr_\mu\fs_h)}$,
we mean the intersection $\tsR^\circ_{({\wp}_{(k\tau)}\fr_\mu\fs_h)}\cap \fV$
where $\fV$ is a standard chart of $ \tsR_{({\wp}_{(k\tau)}\fr_\mu\fs_h)}$.

The chart $\fV$ of $ \tsR^\circ_{({\wp}_{(k\tau)}\fr_\mu\fs_h)}$ comes equipped with 
$$\hbox{a subset}\;\; \fe_\fV  \subset \II_{d,n}^* , \;\;
 \hbox{a subset} \;\; \fd_\fV  \subset \La_{\sfm}^\star, \;\;
 \hbox{and a subset} \;\;  \fl_\fV \subset  {L_{\sfm,[k]}}
 $$ and also
 $$\hbox{a subset} \;\; \fd_\fV^\lt=\{(\um,\uu_F) \mid L_F \in \fl_\fV\} \subset \fd_\fV$$ 
 such that every exceptional divisor $E$ 
(i.e., not a $\vp$- nor a $\vr$-divisor)  of $\tsR_{{\wp}_{(k\tau)}}$
with $E \cap \fV \ne \emptyset$ is 
either labeled by a unique element $\uw \in \fe_\fV$
 or labeled by a unique element $(\uu,\uv) \in \fd_\fV$ or
 labeled by a unique element $ L \in \fl_\fV$. 
We let $E_{ \ell_k, \uw}$ be the unique exceptional divisor 
on the chart $\fV$ labeled by $\uw \in \fe_\fV$; we call it an $\vp$-exceptional divisor.
We let $E_{\ell_k , (\uu,\uv)}$ be the unique exceptional divisor 
on the chart $\fV$ labeled by $(\uu,\uv) \in \fd_\fV$;  we call it an $\vr$-exceptional divisor.
 We let $E_{\ell_k, L}$ be the unique exceptional divisor 
on the chart $\fV$ labeled by $L \in \fl_\fV$; we call it an $\fl$-exceptional divisor.
(We note here that being $\vp$-exceptional or $\vr$-exceptional or $\fl$-exceptional is strictly relative to the given standard chart. Again, $E_{\ell_k, L}$ is not to be confused with
$D_{\ell_k, L}$, the $\fL$-divisor of $\tsR_{({\wp}_{(k\tau)}\fr_\mu\fs_h)}$.)

Further, the chart $\fV$ admits 
 a set of free variables
\begin{equation}\label{variables-wp/ell} 
\var_{\fV}:=\left\{ \begin{array}{ccccccc}
\ve_{\fV, \uw} , \;\; \de_{\fV, (\uu,\uv)} \\ 
y_{\fV,(\um, \uu_F)}        \\
x_{\fV, \uw} , \;\; x_{\fV, (\uu,\uv)} 
\end{array}
  \; \Bigg| \;
\begin{array}{ccccc}
 \uw \in  \fe_\fV,  \;\; (\uu,\uv)  \in \fd_\fV { \- \fd^\lt_\fV} \\
 (\um, \uu_F) { \in \fd^\lt_\fV}  \\    
\uw \in  \II_{d,n}^* \- \fe_\fV,  \;\; (\uu, \uv) \in \La_\sfm^\star \-  \fd_\fV  
\end{array} \right \}.
\end{equation}
such that { $y_{\fV,(\um, \uu_F)} $ are invertible on the chart}, 
and a set of exceptional variables for $\ell$-exceptional divisors
$$ \var_{\fl_\fV}=\{\de_{\fV,(\um, \uu_F)}  
\mid L_F \in \fl_\fV,
\; i.e., \; (\um, \uu_F) \in \fd^\lt_\fV\}.$$
 Furthermore, we also have a set of free variables
$$\var_\fV^\vee=(\{ y \in \var_\fV \} \- \{y_{(\um,\uu_F)} \mid L_F \in \fl_\fV \})
\sqcup \{ \de_{\fV,(\um, \uu_F)}  
 \mid L_F \in \fl_\fV  \}.$$

 We set
$$ \var_\fV^+=\var_\fV \sqcup \var_{\fl_\fV}.$$
Then, all the relations in $\cB_\fV^\mn,  \cB_\fV^\q,  \{L_{\fV, F} \mid  F \in \sfm\}$
are polynomials in $\var_\fV^+$.

 Moreover,  on the standard chart $\fV$, we have
\begin{enumerate}
\item the divisor  $X_{({\wp}_{(k\tau)}\fr_\mu\fs_h), \uw}\cap \fV$ 
 is defined by $(x_{\fV,\uw}=0)$ for every 
$\uw \in \II_{d,n}^* \- \fe_\fV$;
\item the divisor  $X_{({\wp}_{(k\tau)}\fr_\mu\fs_h), (\uu,\uv)}\cap \fV$ is defined by $(x_{\fV,(\uu,\uv)}=0)$ for every 
$(\uu,\uv) \in \La^\star_\sfm\- \fd_\fV$;
\item the divisor  $X_{({\wp}_{(k\tau)}\fr_\mu\fs_h), \uw}$ does not intersect the chart for all $\uw \in \fe_\fV$;
\item the divisor  $X_{({\wp}_{(k\tau)}\fr_\mu\fs_h), (\uu, \uv)}$ does not intersect the chart for all $ (\uu, \uv) \in \fd_\fV$;
\item the $\vp$-exceptional divisor 
$E_{({\wp}_{(k\tau)}\fr_\mu\fs_h), \uw} \;\! \cap  \fV$  labeled by an element $\uw \in \fe_\fV$
is define by  $(\ve_{\fV,  \uw}=0)$ 
for all $ \uw \in \fe_\fV$;
\item the $\vr$-exceptional divisor 
$E_{({\wp}_{(k\tau)}\fr_\mu\fs_h),  (\uu, \uv)}\cap \fV$ labeled by  an element $(\uu, \uv) \in \fd_\fV$
is define by  $(\de_{\fV,  (\uu, \uv)}=0)$ 
for all $ (\uu, \uv) \in \fd_\fV$;  
 \item the $\fl$-exceptional divisor 
$ E_{{  {  \ell_k }  },  L_F}\cap \fV$ labeled by  an element ${ L_F \in} \fl_\fV$
is define by  $({ \de_{\fV, (\um, \uu_F)}} =0)$ 
for all ${ L_F \in} \fl_\fV$;
\item any of the remaining {\rm exceptional} divisors
of $\tsR_{({\wp}_{(k\tau)}\fr_\mu\fs_h)}$ 
other than those that are labelled by some  $\uw \in \fe_\fV$ or $(\uu,\uv) \in \fd_\fV$ 
or $L \in \fl_\fV$ does not intersect the chart.
 \item Assume $({\wp}_{(k\tau)}\fr_\mu\fs_h)=\ell_k$. 
  
Suppose $\fV$ lies over the standard chart $(x_{(\um,\uu_{F_k})} \equiv 1)$ 
  of $\sR_\sF$. Then, $ L_{F_k} \notin \fl_\fV$, and  the $ \ell_k$-blowup is trivial. 
  
Suppose $\fV$ lies over the $\vr$-standard chart of 
$\tsR_{\vt_{[k]}}$. Then, $ L_{F_k} \in \fl_\fV$, and 
we can choose {\rm those} standard charts $\fV$ of
$\tsR^\circ_{\ell_k}$ such that they cover $\tsV_{\ell_k}$, and on any such chart $\fV$,
 we can express
\begin{equation} \nonumber 
L_{\fV ,F_k} 
=1+ \sgn(s_{F_k}) { y}_{\fV, (\um, \uu_{F_k})}
\end{equation} 
where ${ y}_{\fV, (\um, \uu_{F_k})} \in \var_\fV$ is  an invertible variable  in $\var_\fV$, 
labeled by  $(\um, \uu_{F_k})$.
\end{enumerate}
\end{prop}
\begin{proof} 
 We prove by induction on 
 $(k\tau)\mu h \in \{(11)10 \} \sqcup \Index_{\Phi_k} \sqcup \{\ell_k\}$,
 where $(11)10$ is the smallest element, $\ell_k$ is the largest,
 and the elements of $\Index_{\Phi_k}$ are ordered lexicographically
 (which coincides with the order of $\wp_k$-blowups).

For the initial case, the scheme is $\tsR_{({\wp}_{(11)}\fr_1\fs_0)}=\tsR_{\vt}$.
Then, this proposition is the same as Proposition \ref{meaning-of-var-vtk} with $k=\up$.
In this case, we set $\fl_\fV=\emptyset$. Then, one checks that the proposition holds.

We suppose that the statement holds over $\tsR_{({\wp}_{(k\tau)}\fr_\mu\fs_{h-1})}$
for some $(k\tau)\mu h \in \Index_{\Phi_k}  \sqcup \{\ell_k\}$.
(Recall that for the largest element of $\Index_{\Phi_k}$, if we add 1 to the index of the step, then, by convention, it corresponds to $\ell_k$.)

We now consider $\tsR_{({\wp}_{(k\tau)}\fr_\mu\fs_h)}$.

We have the embedding
$$\xymatrix{
\tsR_{({\wp}_{(k\tau)}\fr_\mu\fs_h)}\ar @{^{(}->}[r]  &
 \tsR_{({\wp}_{(k\tau)}\fr_\mu\fs_{h-1})} \times \PP_{\phi_{(k\tau)\mu h}},  }
$$ where $\PP_{\phi_{(k\tau)\mu h}}$ is the factor projective space. 
We let $\phi'_{(k\tau)\mu h}=\{Y'_0, Y'_1\}$ 
where $Y'_0, Y'_1$ are, respectively,
 the proper transforms in $\tsR_{({\wp}_{(k\tau)}\fr_\mu\fs_{h-1})}$ 
 of the two  divisors of the ${\wp}$-set $\phi_{(k\tau)\mu h}=\{Y^+,Y^-\}$
with $Y^\pm$ being associated with $T^\pm_{(k\tau)}$, or, in the case when
$({\wp}_{(k\tau)}\fr_\mu\fs_{h-1})=\wp_k$, $(Y'_0, Y'_1)=(D_{\wp_k, L_{F_k}}, E_{\wp_k,\vt_k})$.
In addition, we let $[\xi_0, \xi_1]$ 
 be the  homogenous coordinates of $\PP_{\phi_{(k\tau)\mu h}}$
 corresponding to $\{Y'_0, Y'_1\}$.

Let $\fV$ be any standard chart  of $\tsR_{({\wp}_{(k\tau)}\fr_\mu\fs_h)}$ that lies over a
unique standard chart $\fV'$ of  $\tsR_{({\wp}_{(k\tau)}\fr_\mu\fs_{h-1})}$ such that 
$\fV=(\fV' \times (\xi_i \equiv 1)) \cap \tsR_{({\wp}_{(k\tau)}\fr_\mu\fs_h)},$ for $i=0$ or 1. 

By assumption, the chart $\fV'$ comes equipped with
a subset $\fe_{\fV'} \subset \II_{d,n}\- \um$, a subset  $\fd_{\fV'} \subset \La^\star_\sfm$,
and admits a set of free variables
\begin{equation}\label{variables-wp/ell} 
\var_{{\fV'}}:=\left\{ \begin{array}{ccccccc}
\ve_{{\fV'}, \uw} , \;\; \de_{{\fV'}, (\uu,\uv)} \\ 
y_{{\fV'},(\um, \uu_F)}        \\
x_{{\fV'}, \uw} , \;\; x_{{\fV'}, (\uu,\uv)} 
\end{array}
  \; \Bigg| \;
\begin{array}{ccccc}
 \uw \in  \fe_{\fV'},  \;\; (\uu,\uv)  \in \fd_{\fV'} { \- \fd^\lt_{\fV'}} \\
 (\um, \uu_F) { \in \fd^\lt_{\fV'}}  \\    
\uw \in  \II_{d,n}^* \- \fe_{\fV'},  \;\; (\uu, \uv) \in \La_\sfm^\star \-  \fd_{\fV'}  
\end{array} \right \},
\end{equation}
such that { $y_{{\fV'},(\um, \uu_F)} $ are invertible on the chart}, 
a set of exceptional variables for $\ell$-exceptional divisors
$$ \var_{\fl_{\fV'}}=\{\de_{{\fV'},(\um, \uu_F)}                 
\mid L_F \in \fl_{\fV'},
\; i.e., \; (\um, \uu_F) \in \fd^\lt_{\fV'}\},$$
and a set of free variables
$$\var_{\fV'}^\vee=(\{ y \in \var_{\fV'} \} \- \{y_{(\um,\uu_F)} \mid L_F \in \fl_{\fV'} \})
\sqcup \{ \de_{{\fV'},(\um, \uu_F)}                        
\mid L_F \in \fl_{\fV'}  \},$$
all together verifying the properties (1)-(9) as in the proposition.


First, we prove that $\tsR_{({\wp}_{(k\tau)}\fr_\mu\fs_h)}$ is smooth along 
$\tsV_{({\wp}_{(k\tau)}\fr_\mu\fs_h)}$.

We only need to focus on the situation when
$Z'_{\phi_{(k\tau) \mu h}}$
meets $\tsV_{({\wp}_{(k\tau)}\fr_\mu\fs_h)}\cap \fV'$ along a nonempty closed subset.

On the chart $\fV'$, by  assumption, we have
\begin{equation}\label{YY01-wp}
Y'_0 \cap \fV' =(y'_0 =0), \; Y'_1 \cap \fV' =(y'_1 =0), 
 \;\; \hbox{for some $y'_0, y'_1 \in  \var^+_{\fV'}$.} 
 \end{equation} 
 
 Now, because ${ y}_{\fV, (\um, \uu_{F_j})} \in \var_\fV$ is  an invertible variable  in $\var_{\fV'}$
 for all $L_{F_j} \in \fl_{\fV'}$,  we see that 
 we must have $y'_0, y'_1 \in  \var^\vee_{\fV'}$. Then, 
it is immediate that the blowup  center 
on the chart $\fV'$ is smooth, hence,
 $\pi^{-1}(\fV')$, as the blowup of  $ \fV'$ along the smooth center,  is smooth, so is $\fV$.

We now prove the remaining statements for the chart $\fV$ of $\tsR_{({\wp}_{(k\tau)}\fr_\mu\fs_h)}$.

First, we suppose that the proper transform $Z'_{\phi_{(k\tau) \mu h}}$ 
in $\tsR_{({\wp}_{(k\tau)}\fr_\mu\fs_{h-1})}$
of the $\wp$- or the $\ell_k$-center $Z_{\phi_{(k\tau) \mu h}}$ does not
meet the chart $\fV'$.  
(If $(k\tau) \mu h$ corresponds to $\ell_k$, we let $\phi_{(k\tau) \mu h}:=\chi_k$.)
Then, we let $\fV$ inherit all the data from those of $\fV'$, that is,
we set $\fe_{\fV}=\fe_{\fV'}$, $\fd_{\fV}=\fd_{\fV'}$, $\fl_{\fV}=\fl_{\fV'}$, $\var_\fV= \var_{\fV'}$
$\var_{\fl_{\fV'}}= \var_{\fl_{\fV'}}$, and $\var_\fV^+= \var_{\fV'}^+$:
changing the subindex $``\ \fV' \ "$ for all the variables in $\var_{\fV'}^+$
to $``\ \fV \ "$.
As the $\wp$-blowup  along the proper transform of $Z_{({\wp}_{(k\tau)}\fr_\mu\fs_h)}$
does not affect the chart $\fV'$, one sees that 
the statements of the proposition hold for $\fV$.

Next, we suppose that 
$Z'_{\phi_{(k\tau) \mu h}}$
meets the chart $\fV'$ along a nonempty closed subset. 

Then, the chart  $\fV=(\fV' \times (\xi_i \equiv 1)) \cap  \tsR^\circ_{({\wp}_{(k\tau)}\fr_\mu\fs_h)}$ 
of the scheme $ \tsR^\circ_{({\wp}_{(k\tau)}\fr_\mu\fs_{h})}$, as a closed subscheme of  
$\fV' \times (\xi_i \equiv 1),$ is defined by
\begin{equation}\label{blowup-relation:wp/ell}
y'_j = y'_i \xi_j, \;  \hbox{ with $j  \in \{0, 1\} \- i$}.             
\end{equation}

There are { six} possibilities for
$Y'_i \cap \fV'$ according to 
the types of the variable $y_i'$.
Based on every of such possibilities, 
we set 
\begin{equation}\label{proof:de-fv-vskmu}
\left\{ 
\begin{array}{lccr}
\fe_{\fV}=\fe_{\fV'} \sqcup \uw, \; \fd_{\fV}= \fd_{\fV'}, \; \fl_\fV=\fl_{\fV'},
&  \hbox{if} \; y'_i=x_{\fV', \uw}\; \hbox{for some} \;\; \uw \in \II_{d,n}^* \- \fe_{\fV'}\\
\fe_{\fV}=\fe_{\fV'}, \; \fd_{\fV}= \fd_{\fV'} ,\; \fl_\fV=\fl_{\fV'}, &\hbox{if} \;    y'_i=\ve_{\fV', \uw} \; \hbox{for some} \;\; \uw \in  \fe_{\fV'}\\
\fd_{\fV}=\fd_{\fV'}\sqcup (\uu,\uv), \; \fe_{\fV}= \fe_{\fV'}, \; \fl_\fV=\fl_{\fV'},& \hbox{if} \; 
y'_i=x_{\fV', (\uu,\uv)}\; \hbox{for some}  \;\; (\uu,\uv) \in  \La_{d,n}^*  \\
\fd_{\fV}=\fd_{\fV'}, \; \fe_{\fV}= \fe_{\fV'}, \; \fl_\fV=\fl_{\fV'},& \hbox{if} \; 
y'_i=\de_{\fV', (\uu,\uv)}\; \hbox{for some}  \;\; (\uu,\uv) \in  \fd_{\fV'} \-\fd^\lt_{\fV'} \\
\end{array} \right.
\end{equation}
For the fifth possibility, we let
$$ \fl_\fV=\fl_{\fV'}, \; \fd_{\fV}=\fd_{\fV'}, \; \fe_{\fV}= \fe_{\fV'} , \;\; \hbox{if} \; 
y'_i=\de_{\fV', (\um,\uu_F)}\; \hbox{for some}  \;\; L_F \in  \fl_{\fV'} .$$
For the sixth possibility, 
we let 
$$\hbox{$\fl_\fV=\fl_{\fV'} \sqcup L_{F_k}, \; \fd_{\fV}=\fd_{\fV'}, \; \fe_{\fV}= \fe_{\fV'} $, \;\;
if $(y_i'=0)$ defines  $D_{\wp_k, L_{F_k}}\cap \fV'$.}$$
(This last case corresponds the the case of the $\ell$-blowup with respect to $L_{F_k}$.)

Accordingly, we introduce 
\begin{equation}\label{proof:new-ex-fv-vskmu}
\left\{ 
\begin{array}{lccr}
\ve_{\fV, \uw}=y'_i, \; 
&  \hbox{if} \; y'_i=x_{\fV', \uw}\; \hbox{for some} \;\; \uw \in \fe_\fV \- \fe_{\fV'}\\
\ve_{\fV, \uw}=y'_i  , &\hbox{if} \;    y'_i=\ve_{\fV', \uw} \; \hbox{for some} \;\; \uw \in  \fe_{\fV'}=\fe_{\fV}\\
\de_{\fV, (\uu,\uv)}=y'_i,& \hbox{if} \; 
y'_i=x_{\fV', (\uu,\uv)}\; \hbox{for some}  \;\; (\uu,\uv) \in  \La_{d,n}^*  \\
\de_{\fV, (\uu,\uv)}=y'_i, & \hbox{if} \; 
y'_i=\de_{\fV', (\uu,\uv)}\; \hbox{for some}  \;\; (\uu,\uv) \in  \fd_{\fV'} \- \fd^\lt_{\fV'} \\
\end{array} \right.
\end{equation}
$$\hbox{$\de_{\fV, (\um, \uu_F)}=y'_i,$  \; if $y'_i=\de_{\fV', (\um, \uu_F)}$, 
for some  $L_F \in  \fl_{\fV'}$}$$ 
$$\hbox{$\de_{\fV, (\um, \uu_{F_k})}=y_i'$, \;\; if $(y_i'=0)$ defines  $D_{\wp_k, L_{F_k}}\cap \fV'$.}$$
(The last case corresponds  the $\ell$-blowup with respect to $L_{F_k}$.)

This defines the exceptional variable for the blowup.


To introduce the variable corresponding to $j  \in \{0, 1\} \- i$,
we then set
\begin{equation}\label{proof:var-xi-fv-vskmu}
\left\{ 
\begin{array}{lcr}
x_{\fV,\ua}=\xi_j, &   \hbox{if $y'_j=x_{\fV', \ua} $}\\
\ve_{\fV, \ua}=\xi_j, &  \hbox{if $y'_j= \ve_{\fV', \ua}$} \\
x_{\fV, (\ua, \ub)}= \xi_j, & \;\;\;\; \hbox{if $y'_j= x_{\fV', (\ua, \ub)}$} \\
\de_{\fV, (\ua, \ub)}= \xi_j, & \;\;\;\; \hbox{if $y'_j= \de_{\fV', (\ua, \ub)}$}. 
\end{array} \right.
\end{equation}
Also, we let $\de_{\fV, (\um,\uu_F)}=\xi_j,$ 
if $y'_j= \de_{\fV', (\um,\uu_F)}$ for some $L_F \in \fl_{\fV'}$.

Thus, we have introduced $y'_i,  \xi_j \in \var^+_\fV$ where $y'_i$, and  $\xi_j$
are endowed with the new names as in \eqref{proof:new-ex-fv-vskmu}, and in 
\eqref{proof:var-xi-fv-vskmu},  respectively.

Next, we define the set 
$\var_\fV \- \{y'_i, \xi_j\}$ 
to consist of the following variables:
\begin{equation}\label{proof:var-fv-vskmu}
\left\{ 
\begin{array}{lccr}
x_{\fV,\uw}=x_{\fV', \uw}, &  \forall \;\; \uw \in \II_{d,n}^*\- \fe_\fV \;\; \hbox{and $x_{\fV', \uw} \ne y_i',y'_j$}\\
x_{\fV, (\uu, \uv)}=x_{\fV', (\uu, \uv)}, & \;\; \forall \;\; 
(\uu, \uv) \in \La_\sfm^\star \- \fd_{\fV} \;\; \hbox{and $x_{\fV', (\uu,\uv)} \ne y_i', y'_j$} \\
\ve_{\fV, \uw}= \ve_{\fV', \uw}, & \forall \;\; \uw \in \fe_\fV \;\; \hbox{and $\ve_{\fV', \uw} \ne y'_i$},  y'_j \\
\de_{\fV, (\uu, \uv)}= \de_{\fV', (\uu, \uv)}, &  \;\; \forall \;\; (\uu, \uv) \in \fd_{\fV} \;\; \hbox{and $\de_{\fV', (\uu, \uv)} \ne y'_i$}, y'_j \\
y_{\fV, (\uu, \uu_F)}= y_{\fV', (\uu, \uu_F)}, &  \;\; \forall \;\; L_F \in \fl_{\fV}. 
\end{array} \right.
\end{equation}

We let $\var_\fV$ be the set of
 the variables in  \eqref{proof:new-ex-fv-vskmu},
 \eqref{proof:var-xi-fv-vskmu}, and \eqref{proof:var-fv-vskmu}.
Substituting \eqref{blowup-relation:wp/ell},
 one sees that $\var_\fV$ is a set of free variables on the open chart $\fV$.
 This describes \eqref{variables-wp/ell} in the proposition.
 
 We then let $\var_{\fl_\fV}=\{\de_{\fV, (\um,\uu_F) } \mid L_F \in \fl_\fV\},$
 and obtain 
 $$\var_\fV^\vee=(\{ y \in \var_\fV \} \- \{y_{(\um,\uu_F)} \mid L_F \in \fl_\fV \})
\sqcup \{ \de_{\fV,(\um, \uu_F)}  
 \mid L_F \in \fl_\fV  \}.$$
 One sees that $\var_\fV^\vee$ is also a set of free variables on the open chart $\fV$.

 We set $\var_\fV^+=\var_\fV \sqcup \var_{\fl_\fV}.$
By substituting \eqref{blowup-relation:wp/ell} and taking proper transforms,
 one sees that all the relations in $\cB_\fV^\mn,  \cB_\fV^\q, L_{\fV, \sfm}$
are polynomials in $\var_\fV^+$.

 Now, it remains to verity (1)-(9) of the proposition on the chart $\fV$.

First, consider  the unique new exceptional divisor 
$E_{({\wp}_{(k\tau)}\fr_\mu\fs_h)}$
created by the blowup
$\tsR_{({\wp}_{(k\tau)}\fr_\mu\fs_h)} \lra \tsR_{({\wp}_{(k\tau)}\fr_\mu\fs_{h-1})}. $ 
Then, we have
$$ E_{({\wp}_{(k\tau)}\fr_\mu\fs_h)} \cap \fV = (y'_i=0)$$
where $y'_i$ is renamed as in \eqref{proof:new-ex-fv-vskmu} and 
in the  sentence immediately following it.
This way, the new exceptional divisor $E_{({\wp}_{(k\tau)}\fr_\mu\fs_h)}$
 is labelled on the chart $\fV$.
Further, we have that the proper transform of $Y'_i$ in
 $\tsR_{({\wp}_{(k\tau)}\fr_\mu\fs_h)}$ 
 does not meet the chart $\fV$, and
if $Y'_i$ is an exceptional parameter labeled by some element of 
$\fe_{\fV'} \sqcup \fd_{\fV'} \sqcup \fl_{\fV'}$, 
 then, on the chart $\fV$, its proper transform is no longer labelled by that element.
 This verifies the cases of (3)-(7) {\it whenever the statement therein involves
  the newly created exceptional divisor 
 $E_{({\wp}_{(k\tau)}\fr_\mu\fs_h)}$.}
  
For any of the remaining $\vp$-, $\vr$-, 
and exceptional divisors on $\tsR_{({\wp}_{(k\tau)}\fr_\mu\fs_h)}$,
it is the proper transform of a unique corresponding
 $\vp$-, $\vr$-, and exceptional divisor on $\tsR_{({\wp}_{(k\tau)}\fr_\mu\fs_{h-1})}$. 
 Hence,  by applying   the inductive assumption on $\fV'$ accordingly, 
 we conclude that
every of  (1)-(8) of the proposition hold on $\fV$.
 
 (9)

 Assume ${\wp}_{(k\tau)}\fr_\mu\fs_h=\ell_k$.

By the inductive assumption, the statement holds for all $j<k$.
We only need to consider $L_{F_k}$.

Assume $\fV'$ lies over $(x_{(\um,\uu_{F_k})}\equiv 1)$, then 
$Z_{\chi_k} \cap \fV = \emptyset$ because $E_{\wp_k, \vt_k} \cap \fV = \emptyset$, 
where $Z_{\chi_k}$ is  the $  \ell_k$-center. Hence, the statement holds.

Assume $\fV'$ lies over the $\vr$-standard chart of $\tsR_{\vt_{[k]}}$ (hence,
so does $\fV$). Notice that the variable $\de_{\fV', (\um, \uu_F)}$ never appears
in any relation of the blocks $\fG_{\fV', F_j}$ for all $j \le k$, except $L_{\fV', F_k}$.
Hence, by  Proposition \ref{eq-for-sV-vtk} 
  and the inductive assumption, $L_{\fV', F_k}$ takes the following form
$$L_{\fV', F_k}=\sgn(s_{F_k})\de_{\fV', (\um, \uu_F)} 
+ \sum_{s \in  S_{F_k} \- s_{F_k}} \sgn(s) 
\pi_{\fV', \fV_{[0]}}^* x_{\fV', (\uu_s, \uv_s)} .$$

Then, one sees from the definition that the $\ell_k$-blowup ideal on the chart $\fV'$ is
$$\langle L_{\fV', F_k}^\star, \; \de_{\fV', (\um, \uu_F)}  \rangle $$
where $L_{\fV', F_k}^\star=\sum_{s \in  S_{F_k} \- s_{F_k}} \sgn(s) 
\pi_{\fV', \fV_{[0]}}^* x_{\fV', (\uu_s, \uv_s)}$.
We let $\PP_{[\xi_0, \xi_1]}$ be the factor projective space of the $ \ell_k$-blowup
such that $(\xi_0, \xi_1)$ corresponds to $(L_{\fV', F_k}^\star, \; \de_{\fV', (\um, \uu_F)})$.

 First, we consider the chart $(\xi_ 0\equiv 1)$.

Then,   $\fV=(\fV' \times (\xi_0 \equiv 1)) \cap \tsR_{\ell_k}$, 
as a closed subscheme of  
$\fV' \times (\xi_0 \equiv 1),$ is defined by
\begin{equation}\label{proof:x0-chart-ell}
\de_{\fV', (\um, \uu_{F_k})} = L_{\fV', F_k}^\star \cdot \xi_1.
\end{equation}
Notice that on this chart, we have
$$E_{\ell_k} \cap \fV =  (L_{\fV', F_k}^\star  =0)$$
where $E_{\ell_k}$ is the exceptional divisor created by the blowup $\tsR_{\ell_k} \to \tsR_{\wp_k}$.
Upon substituting \eqref{proof:x0-chart-ell}, we obtain
$$L_{\fV', {F_k}}=L_{\fV', F_k}^\star (1 +  \sgn(s_{F_k}) \xi_1),$$ 
and 
\begin{equation}\label{for-isom}
L_{\fV, {F_k}}= 1 + \sgn(s_{F_k}) \xi_1 .
\end{equation}
This implies that $\xi_1$ is invertible along $\tsV_{\ell_k}$.
Thus, if necessary, we can shrink the chart $\fV$ such that it still contains
 $\tsV_{\ell_k}$, and assume that $\xi_1$ is invertible on $\fV$.
Now, as $\xi_1$ is the proper transform 
$y_{\fV, (\um, \uu_{F_k})}$  of $\de_{\fV', (\um, \uu_{F_k})}$,
we can also write
$$L_{\fV, {F_k}}= 1 + \sgn(s_{F_k}) y_{\fV, (\um, \uu_{F_k})} $$ 
with  $y_{\fV, (\um, \uu_{F_k})}$ being invertible on the chart.
Thus, we obtain the desired form of $L_{\fV, F_k}$ as stated in (9).

Finally, observe that by \eqref{proof:x0-chart-ell}, 
as $\xi_1$ is invertible on the chart, we have
$$E_{\ell_k}  \cap \fV=(L_{\fV', F_k}^\star=0)  = (\de_{\fV', (\um, \uu_{F_k})} =0)$$
with $L_{F_k} \in \fl_\fV$.

 \medskip
Next, we  consider 
 the chart $(\xi_1 \equiv 1)$. 
 
 Then, the chart  $\fV=(\fV' \times (\xi_1 \equiv 1)) \cap \tsR_{\ell_k}$ 
of the scheme $\tsR_{\ell_k}$, as a closed subscheme of  
$\fV' \times (\xi_1\equiv 1),$ is defined by
\begin{equation} \label{proof:wp/ell-t0}
L_{\fV', F_k}^\star = \de_{\fV', (\um, \uu_{F_k})} \xi_0.
\end{equation}
By substitution, we obtain 
$$L_{\fV', F_k}=\de_{\fV', (\um, \uu_{F_k})} (\xi_0 +  \sgn(s_{F_k}) )$$
and
\begin{equation} 
\nonumber
L_{\fV, F_k}=\xi_0 +   \sgn(s_{F_k}).
\end{equation}
Then, this implies that $\xi_0$ is invertible along $\fV \cap \tsV_{\ell_k}$.
Hence, again, if necessary, by shrinking the chart $\fV$ 
such that it still contains $\fV \cap \tsV_{\ell_k}$, 
we can assume that $\xi_0$ is invertible on the chart.
Therefore, we can discard the current chart and switch back to the chart lying
over $(\xi_0 \equiv 1)$.  This sends us back to the previous case
where the statement is proved.


By Corollary \ref{no-(um,uu)} and the above discussion,
 the charts chosen in (9) together cover  the scheme $\tsV_{\ell_k}$.

 This completes the proof.
\end{proof}

\begin{defn}\label{preferred-chart-ell}
A standard chart of $\tsR_{\ell_{k}}$ as characterized by
Proposition \ref{meaning-of-var-wp/ell}  (9) 
is called a preferred standard chart.
\end{defn}

\begin{cor}\label{ell-isom}
We let $$\rho_{\ell_k, \wp_k}: \tsV_{\ell_{k}} \lra \tsV_{\wp_{k}}$$
be the morphism
 induced from the blowup morphism $\rho_{\ell_k, \wp_k}: \tsR_{\ell_{k}} \lra \tsR_{\wp_{k}}$.
 Then, $\rho_{\ell_k, \wp_k}$ is an isomorphism.
\end{cor}
\begin{proof} 
We continue  to use the notation of  the proof of Proposition \ref{meaning-of-var-wp/ell}  (9).
We can cover  $\tsV_{\ell_{k}}$ by preferred standard charts.
Let $\fV$ be any such a chart such that $Z_{\chi_k} \cap \fV' \ne \emptyset$.
Then, by \eqref{for-isom}, we have
$$L_{\fV, {F_k}}= 1 + \sgn(s_{F_k}) \xi_1 .$$
Then, $\xi_1=-\sgn(s_{F_k})$. This implies that 
$$\rho_{\ell_k, \wp_k}^{-1}(\tsV_{\wp_{k}}) =  \tsV_{\wp_{k}} \times [1, -\sgn(s_{F_k})].$$
Hence, the morphisms 
$\rho_{\ell_k, \wp_k}$ is an isomorphism.
\end{proof}

We remark here that we obtain our final scheme $\tsV_\ell$ from $\sV$ by 
suquentially blowing up the embient space $\sR_\sF$, and then take the induced
sequential blowups of $\sV$. If we do not
perform  $\ell$-blowups, we will obtain different final schemes and would not suffice
for our purpose, in general.

\subsection{Proper transforms of 
defining relations in $(\wp_{(k\tau)}\fr_\mu\fs_h)$ and in $(\ell_k)$}   $\ $

  Consider any fixed $B \in \cB^\mn  \cup \cB^\q$ and $\bF \in \sfm$.
Suppose $B_{\fV'}$ and $L_{\fV', F}$ have been constructed over $\fV'$.
Applying Definition \ref{general-proper-transforms}, we obtain the proper transforms on the chart $\fV$
$$B_{\fV}, \; B \in \cB^\mn  \cup \cB^\q; \;\; L_{\fV, F}, \; \bF \in \sfm.$$

\begin{defn}\label{termninatingB-wp} 
{\rm  (cf. Definition \ref{general-termination})} 
Consider any main binomial relation $B \in \cB^\mn$. 
Let $\fV$ be a standard chart of $\tsR_{({\wp}_{(k\tau)}\fr_\mu\fs_h)}$ (including $\tsR_{\ell_k}$)
and $\bz \in \fV$ be a closed point.
We say that $B$ terminates at $\bz$ if (at least) one of its two terms 
of $B_{\fV}$ does not vanish at $\bz$.
 We say $B$ terminates on the chart $\fV$ 
 if it terminates at all closed points of $ \tsV_{({\wp}_{(k\tau)}\fr_\mu\fs_h)} \cap \fV$.
 We say $B$ terminates on  $\tsR_{({\wp}_{(k\tau)}\fr_\mu\fs_h)}$ if it terminates on
 all standard charts $\fV$ of  $\tsR_{({\wp}_{(k\tau)}\fr_\mu\fs_h)}$.
 \end{defn}

In the sequel, for any $B =T^+_B - T^-_B \in \cB^\mn$, we express
$B_\fV= T^+_{\fV, B} - T^-_{\fV, B}$. 
If $B=B_{(k\tau)}$ for some $k \in [\up]$ and $\tau \in [\ft_{F_k}]$, we also write 
$$B_\fV= T^+_{\fV, (k\tau)} - T^-_{\fV, (k\tau)}.$$

Below, we follow the notations of Proposition \ref{meaning-of-var-wp/ell} as well as those in its proof. 

In particular, we have that $\fV$ is a standard chart of $\tsR_{({\wp}_{(k\tau)}\fr_\mu\fs_h)}$
(including $\tsR_{\ell_k}$),
 lying over a standard chart $\fV'$ of $\tsR_{({\wp}_{(k\tau)}\fr_\mu\fs_{h-1})}$. We have that
$\phi'_{(k\tau)\mu h}=\{Y'_0, \;Y'_1\}$  is the proper transforms of $\phi_{(k\tau)\mu h}=\{Y^+, Y^-\}$
 in $\tsR_{({\wp}_{(k\tau)}\fr_\mu\fs_{h-1})}$ with $Y^\pm$ being associated with $T^\pm_{(k\tau)}$
 or $\phi_{(k\tau)\mu h}=\chi_k$ in which case $Y_0'$ is the $\fL$-divisor $D_{\wp, L_{F_k}}$.
 Likewise, $Z'_{\phi_{(k\tau)\mu h}}$ is the proper transforms of the ${\wp}$-center
 $Z_{\phi_{(k\tau)\mu h}}$ or is the $\ell_k$-center 
   in $\tsR_{({\wp}_{(k\tau)}\fr_\mu\fs_{h-1})}$.
 Also, assuming that $Z'_{\phi_{(k\tau)\mu h}} \cap \fV'  \ne \emptyset$,
then, as in \eqref{YY01-wp}, we have
 $$Y'_0 \cap \fV' =(y'_0 =0), \; Y'_1 \cap \fV' =(y'_1 =0), 
 \;\; \hbox{with $y'_0, y'_1 \in \var_{\fV'}^+$} $$
Further, we have $\PP_{\phi_{(k\tau)\mu h}}=\PP_{[\xi_0,\xi_1]}$ with 
 the homogeneous coordinates  $[\xi_0,\xi_1]$ corresponding to $(y'_0,y'_1)$.

\begin{prop}\label{equas-wp/ell-kmuh} 
Let the notation be as in Proposition \ref{meaning-of-var-wp/ell} and be as in above.

Let $\fV$ be any standard chart of $\tsR_{(\wp_{(k\tau)}\fr_\mu\fs_h)}$. Then, the scheme 
$\tsV_{({\wp}_{(k\tau)}\fr_\mu\fs_h)}\cap \fV$, as a closed subscheme of the chart $\fV$ 
 is defined by $$\cB_\fV^\mn, \; \cB_\fV^\q, \; L_{\fV, \sfm}.$$

Assume $Z'_{\phi_{(k\tau)\mu h}} \cap \fV'  \ne \emptyset$.
We let $\zeta=\zeta_{\fV, (k\tau)\mu h}$ be the exceptional parameter in $\var_\fV^+$ such that
$$E_{({\wp}_{(k\tau)}\fr_\mu\fs_h)} \cap \fV = (\zeta=0).$$

In the sequel, when $\tsR_{(\wp_{(k\tau)}\fr_\mu\fs_h)}=\tsR_{\ell_k}$, we assume that
$\fV$ is the preferred chart on $\tsR_{\ell_k}$ (cf. Definition \ref{preferred-chart-ell}).

Then, we have that the following hold.
\begin{enumerate}
\item Suppose $\fV= (\fV' \times (\xi_0 \equiv 1)) \cap \tsR_{(\wp_{(k\tau)}\fr_\mu\fs_h)}$.
We let $y_1 \in \var_\fV^+$  be the proper transform of $y_1'$. 
Then, we have
\begin{itemize}
\item[(1a)]   $T^+_{\fV, (k\tau)}$
 is square-free, $y_1 \nmid T^+_{\fV, (k\tau)}$, 
 and $\deg (T^+_{\fV, (k\tau)}) =\deg (T^+_{\fV', (k\tau)})-1$.
Suppose $\deg_{y_1'} T^-_{\fV', (k\tau)}=b$ for some 
integer $b$, positive by definition, then
we have $\deg_{\zeta} T^-_{\fV, (k\tau)}=b-1$.
Consequently,
either $T^-_{\fV, (k\tau)}$ is linear in $y_1$ or else $\zeta \mid T^-_{\fV,  (k\tau)}$.
\item[(1b)]  Let $B \in   \cB^\mn$ with $B > B_{(k\tau)}$.   Then,
$T^+_{\fV, B}$ is square-free and $y_1 \nmid T^+_{\fV, B}$.
Suppose $B \in  \cB^\mn_{F_k}$ and  $y_1 \mid T^-_{\fV, B}$,
 then either $T^-_{\fV, B}$ is linear in $y_1$ or $\zeta \mid T^-_{\fV, B}$.
Suppose $B \notin  \cB^\mn_{F_k}$ and  $y_1 \mid T^-_{\fV, B}$, then $\zeta \mid T^-_{\fV, B}$.
\end{itemize} 
\item Suppose $\fV= (\fV' \times (\xi_1 \equiv 1)) \cap \tsR_{(\wp_{(k\tau)}\fr_\mu\fs_h)}$.
We let $y_0  \in \var_\fV^+$  be the proper transform of $y_0'$. 
 Then, we have
\begin{itemize}
\item[(2a)]  $T^+_{\fV, (k\tau)}$ is square-free,
 $y_0 \nmid T^-_{\fV, (k\tau)}$, and
 $ \deg (T^-_{\fV, (k\tau)}) =\deg (T^-_{\fV', (k\tau)})-1$.
\item[(2b)] Let $B \in   \cB^\mn$ with $B > B_{(k\tau)}$.  
Then, $T^+_{\fV, B}$ is square-free.
Suppose  $B \in   \cB^\mn_{F_k}$, then  $y_0 \nmid T^-_{\fV, B}$.
Suppose $B \notin  \cB^\mn_{F_k}$ and $y_0 \mid T^-_{\fV, B}$, then $\zeta \mid T^-_{\fV, B}$.
\end{itemize}
\item  $\rho_{(k\tau)}< \infty$.   Moreover, 
for every $B \in \cB^\mn$ 
with $B\le B_{(k\tau)}$, we have that $B$ 
terminates on $\tsR_{\wp_k}$. 

\item 
 Consider any fixed term $T_B$ of any given $B \in \cB^\q$. 
 We can assume  $y_i'$ turns into $\zeta$ for some $i \in \{0, 1\}$ and
$y_j$ is the proper transform of  $y_j'$ with $j=\{0,1\}\-\{i\}$ 
 Suppose $y_j \mid T_{\fV, B}$, then either $T_{\fV, B}$ is linear in $y_j$ or 
$\zeta \mid T_{\fV, B}$.
\end{enumerate}      
\end{prop}
\begin{proof} 
We continue to follow  the notation in the proof of Proposition \ref{meaning-of-var-wp/ell}.

We prove the  proposition by applying induction on
$(k\tau) \mu h  \in \{((11)1 0)\}  \sqcup \Index_{\Phi_k}  \sqcup \{\ell_k\}$.


The initial case is $(11)1 0$ with $\tsR_{({\wp}_{(11)}\fr_1 \fs_0)}=\tsR_\vt$. In this case, 
the  statement about defining equations
 of  $\tsR_{(\wp_{(11)}\fr_1 \fs_0)} \cap \fV$ follows from
Proposition \ref{eq-for-sV-vtk} with $k=\up$; the remainder statements (1) - (4) are void.

Assume that the proposition holds  for $({\wp}_{(k\tau)}\fr_\mu \fs_{h-1})$
with $(k\tau)\mu h  \in \Index_{\Phi_k}  \sqcup \{\ell_k\}$.

Consider $({\wp}_{(k\tau)}\fr_\mu \fs_h)$. 

Consider any standard chart $\fV$ of $\tsR_{({\wp}_{(k\tau)}\fr_\mu \fs_{h})}$,
lying over a  standard chart of $\fV'$ of  $\tsR_{{\wp}_{(k\tau)}\fr_\mu \fs_{h-1})}$.
By assumption, all the desired statements of the proposition hold over the chart $\fV'$.



The  statement of the proposition on the defining equations
 of $\tsV_{(\wp_{(k\tau)}\fr_\mu \fs_{h})} \cap \fV$
follows straightforwardly from the inductive assumption.

For the statements of (1)-(4), we structure our proofs as follows.
Because the $\ell_k$-blowup occurs after all $\wp$-blowups are performed,
we will prove (1), (2), and (3) for $\wp$-blowups first, and then, we will return to
prove (1) and (2) for the $\ell_k$-blowup. We prove (4) at the end.

(1)

($1_\wp$) We first consider the case when the blowup is $\wp_k$ blowup (not the $\ell_k$-blowup).

 We may express 
$$B_{(k\tau)}= x_{(\uu_{s_\tau},\uv_{s_\tau})} x_{\uu_k}
 -x_{\uu_{s_\tau}} x_{\uv_{s_\tau}} x_{(\um,\uu_k)}$$
where  $x_{\uu_k}$ is the leading variable of $\bF_k$, and $s_\tau \in S_{F_k} \- s_{F_k}$ corresponds
to $\tau \in [\ft_{F_k}]$. 

(1a)

Observe  that the variables $x_{(\uu_{s_\tau},\uv_{s_\tau})}$ and $x_{\uu_k}$ do not appear in any 
$B \in \cB^\mn_{F_j}$ with $j <k$. Thus, the fact that the plus-term 
$T^+_{\fV, (k\tau)}$ is square-free is immediate if $\tau=1$. 
(This serves as checking the initial case.)
            
For a general $\tau \in [\ft_{F_k}]$, 
 it follows from the inductive assumption on $T^+_{\fV', (k\tau)}$.
The remainder statements follow from straightforward calculations. We omit the obvious details.

(1b). Let $B > B_{(k\tau)}$. 

Suppose $B \in \cB^\mn_{F_k}$.
We  we can write $B=B_{(k\tau')}$ with $\tau' \in [\ft_{F_k}]$ and $\tau' > \tau$.
We can express 
$$B=B_{(k\tau')}= x_{(\uu_{s_{\tau'}},\uv_{s_{\tau'}})} x_{\uu_{F_k}} -
x_{\uu_{s_{\tau'}}} x_{\uv_{s_{\tau'}}} x_{(\um,\uu_{F_k})}.$$
We have $T^+_B=x_{(\uu_{s_{\tau'}},\uv_{s_{\tau'}})} x_{\uu_{F_k}}$,  and it
{\it retains} this form prior to the $\wp$-blowups with respect to binomials of $\cB_{F_k}$
because $x_{\uu_{F_k}}$ and $x_{(\uu_{s_{\tau'}},\uv_{s_{\tau'}})}$ do not appear in any 
relation $\fG_{F_j}$ with $j <k$. 
(Recall here the convention: $x_{\fV, \uu} = 1$ if  $\uu \in \fe_{\fV}$;
 $x_{\fV, (\uu,\uv)} = 1$ if  $(\uu,\uv) \in \fd_{\fV}$.)

Starting the $\wp$-blowups with respect to the first binomial relation 
$B_{(k1)}$ of $\cB_{F_k}$, 
$T^+_B$ can only acquire exceptional parameters through the 
leading variable $x_{\uu_{F_k}}$.  
 From here, one sees directly  that $T^+_{\fV, (k\tau')}$ is square-free.


Now, suppose $y_1 \mid T^-_{\fV, B}$. We can assume $\deg_{y'_1}  (T^-_{\fV', B})=b$ for some 
integer $b>0$. 
Since  $T^+_B$ is square-free, we have two possibilities: 
(1) $\deg_{y_1}  (T^-_{\fV, B})=b$ and $\deg_{\zeta}  (T^-_{\fV, B})=b-1$, if $y_0' \mid  T^+_{\fV, B}$.
(2) $\deg_{y_1}  (T^-_{\fV, B})=b$ and $\deg_{\zeta}  (T^-_{\fV, B})=b$, if $y_0' \nmid  T^+_{\fV, B}$.
Hence, in either case, either $T^-_{\fV, B}$ is linear in $y_1$ when $b=1$ in the first case, 
 or else,   $\zeta \mid T^-_{\fV, B}$ when $b>1$ in the first case or in  any situation of  the second case.

Suppose $B \notin \cB^\mn_{F_k}$.
We  we can write $B=B_{(k'\tau')}$ with $\tau' \in [\ft_{F_{k'}}]$ and $k' > k$.
We can express 
$$B=B_{(k'\tau')}= x_{(\uu_{s_{\tau'}},\uv_{s_{\tau'}})} x_{\uu_{F_{k'}}} -
x_{\uu_{s_{\tau'}}} x_{\uv_{s_{\tau'}}} x_{(\um,\uu_{F_{k'}})}.$$
Since $x_{(\uu_{s_{\tau'}},\uv_{s_{\tau'}})}$ and $x_{\uu_{F_{k'}}}$ do not appear
in any relation in $\fG_{F_j}$ with $j <k'$ and $k <k'$, we see that
$T^+_B=x_{(\uu_{s_{\tau'}},\uv_{s_{\tau'}})} x_{\uu_{F_{k'}}}$
{\it retains} this form under the current blowup, in particular, $T^+_{\fV, B}$ is square-free.  
Furthermore, if
$y_1 \mid T^-_{\fV, B}$, then $\zeta \mid T^-_{\fV, B}$.

This proves ($1_\wp$).

(2)

($2_\wp$) We continue to consider the case when the blowup is $\wp_k$ blowup (not the $\ell_k$-blowup).

(2a).  

The proof of the fact that the plus-term 
$T^+_{\fV, (k\tau)}$ is square-free is totally analogous to the corresponding part of (1a). 
The remainder statements follow from straightforward calculations. We omit the obvious details.

(2b) Let $B > B_{(k\tau)}$. 

 The fact that $T^+_{\fV, (k\tau)}$ is square-free, again, follows from the same line
of arguments as  in the corresponding part of (1b).

If $B \in \cB^\mn_{F_k}$, 
one sees that $y_0' \nmid  T^-_{\fV, B}$.

Suppose $B \in \cB^\mn_{F_k}$. Then, by the same line of argument of 
the correpsonding part of (1b), we again obtain that 
if $y_0 \mid T^-_{\fV, B}$, then $\zeta \mid T^-_{\fV, B}$.

This proves ($2_\wp$).

(3) (The statement is exclusively about $\wp$-blowups.)

From $\tsR_{({\wp}_{(k\tau)}\fr_\mu\fs_{h-1})}$ to $\tsR_{({\wp}_{(k\tau)}\fr_\mu\fs_{h})}$,
over any chart $\fV$ of $\tsR_{({\wp}_{(k\tau)}\fr_\mu\fs_{h})}$, by (1a) and (1b), we have
$$\deg (T^+_{\fV, (k\tau)}) =\deg (T^+_{\fV', (k\tau)})-1$$ or
 $$\deg (T^-_{\fV, (k\tau)}) =\deg (T^-_{\fV', (k\tau)})-1.$$
 Hence, after finitely many steps, over any chart $\fV$, either all variables in
 $B_{\fV, (k\tau)}$ are invertible along the proper transform of $\tsV$,
 or else, one of the two terms of $B_{\fV, (k\tau)}$ must  become a constant.


This implies that the process of $\wp$-blowups in ($\wp_{(k\tau)}$)
must terminate after finitely many rounds.
That is, $\rho_{(k\tau)} < \infty$. The remaining statements follows from $\rho_{(k\tau)} < \infty$.

($1_\ell$) Now we return to consider the $\ell_k$-blowup.

In this case, we have $y_0'$, defining the $\fL$-divisor $D_{\wp_k, L_{F_k}}$ on the chart
$\fV'$, does not appear in any relation $B$ of $\cB^\q$.

(1a) By (3) (already proved), all $B_{(k\tau)}$ of $\cB^\mn_{F_k}$ terminate. Thus,
it follows that $T_{\fV, (k\tau)}^+$ is square-free. The remaining statements are void.

(1b) Then, the same line of aruments  applied to $x_{\uu_{F_{k'}}}$ as in (1a) of 
($1_\wp$) can be reused to obtain the desried statement.

($2_\ell$) We still consider the $\ell_k$-blowup.

Because $\fV$ is a preferred chart (by assumption), the statement is void.

(4) 

($4_\wp$)  We first consider the case when the blowup is $\wp_k$ blowup (not the $\ell_k$-blowup).

  When we begin with the block $\cB^\mn_{F_k}$, we have
 $$B_{(k\tau)}: x_{\fV', (\uu_{s_\tau},\uu_{s_\tau})} x_{\fV', \uu_F} -a_\tau, \; \tau \in [\ft_{F_k}]$$
 where $a_\tau$ are some monomials.
 
 Fix and consider any $B^\q \in \cB^\q$. We can write $B_{\fV'}^\q=T_{\fV',0} -T_{\fV',1}$.

 Consider the $\wp$-blowups with respect to the first relation $B_{(k1)}$.

 First, assume $y_0'= x_{\fV', \uu_F} $. Note that
 the variable $x_{\fV', \uu_F}  $ does not appear in $B^\q$.
  If $y_1'$ does not appear in  $B^\q$, there is nothing to prove.
If $y_1'$ appears in one of the two terms of $B^\q=B_{\fV'}^\q=T_{\fV',0} -T_{\fV',1}$,
  say $T_{\fV', 0}$, then $y_1'$ can  become the exceptional parameter $\zeta$
        or brings $\zeta$ with it into $T_{\fV, 0}$.  
        Hence, in either case, we obtain $\zeta \mid T_0$.
     
  Next,    assume  $y_0'=x_{\fV', (\uu_{s_1},\uu_{s_1})}$.
  Suppose $x_{\fV', (\uu_{s_1},\uu_{s_1})}$ does not appear in $B^\q$,
  then the same line of arguments in the previous case implies the desired statement.
   Suppose  $x_{\fV', (\uu_{s_1},\uu_{s_1})}$ appears in $B^\q$,
   w.l.o.g., say, in $ T_{\fV',0}$. Then it is linear in $T_0$ by  Lemma \ref{ker-phi-k} (2). 
  If $y_1'$ does not appear in  $B^\q$, the statement follows immediately.
   If $y_1'$ also appears in $T_{\fV',0}$, then 
  $\zeta \mid T_{\fV,0}$.  Suppose $y_1'$  appears in $T_{\fV',1}$ with degree $b>0$.
  If $y_1'$ becomes $\zeta$, then $x_{\fV, (\uu_{s_1},\uu_{s_1})}$ 
  is the proper transform of $x_{\fV', (\uu_{s_1},\uu_{s_1})}$, 
  then $B_{\fV}^\q$ remains linear in $x_{\fV, (\uu_{s_1},\uu_{s_1})}$ .
   If $x_{\fV, (\uu_{s_1},\uu_{s_1})}$  becomes $\zeta$, then we have
 $\zeta^{b-1}y_1^b \in T_{\fV',1}$ where $y_1$ is the proper transform of $y_1'$.
   Thus, the statement follows immediately.
  

 Then, we move on to $B_{(k2)}$. Notice that 
 during the previous blowups with respect  $B_{(k1)}$
 the variable       $ x_{\fV', \uu_F}$ does not bring any exceptional
 variable $\ve$ into $B^\q$.  
  Hence, any exceptional variable in the plus term of   $B_{(k2)}$,
  just like   $ x_{\fV', \uu_F}  $ in the previous case, does not appear in $B^\q$ on the chart.       
  When such an exceptional variable belongs to the local blowup center,
   then, by the same arguments for $ x_{\fV', \uu_F}  $ in the previous case,
   we conclude that the desired statement of (4)
      holds when the exceptional variables in the plus term of
   $B_{(k2)}$  belong to the local $\wp$- or $\ell$-set.
        
When  $ x_{\fV', (\uu_{s_2},\uu_{s_2})} $ belongs to the local blowup center,
then the exactly same argument applied to $x_{\fV', (\uu_{s_1},\uu_{s_1})}$
can be reused for      $ x_{\fV', (\uu_{s_2},\uu_{s_2})} $ to obtain (4).

   We can then move to the next and the remaining binomial relations, one by one, 
    repeat  exactly the same argument to obtain the desired statement.

($4_\ell$)  We consider the case when the blowup is  $\ell_k$-blowup.
    
    In this case, note that $y_0'$, locally defining $D_{\wp_k, L_{F_k}}$,
    does not appear in $B \in \cB^\q$. 
    Also, since $\fV$ is a  preferred chart, we have $\zeta=\de_{\fV, (\um, \uu_{F_k)}}$
    with $L_{F_k} \in \fl_\fV$.
    Hence, if $y_1 \mid T_{\fV,B}$, a term of $B_\fV$, then $\zeta \mid T_{\fV, B}$.

    This proves (4).

By induction,    Proposition \ref{equas-wp/ell-kmuh} is proved.  
\end{proof}

\section{$\Ga$-schemes and Their Transforms}\label{Gamma-schemes}

\subsection{$\Ga$-schemes} $\ $

Here, we return to the initial affine chart $\rU_\um \subset \PP(\wedge^d E)$.

\begin{defn}\label{ZGa} 
Let $\Ga$ be an arbitrary subset of $\var_{\rU_\um}=\{x_\uu \mid \uu \in \II_{d,n} \-\um \}$. 
We let $I_\Ga$ be the ideal of $\kk[x_\uu]_{\uu \in \II_{d,n}\-\um}$ generated by all the elements 
$x_{\uu}$ in $\Ga$, 
 and, 
 $$I_{\wp,\Ga}=\langle x_\uu, \; \bF \mid x_\uu \in \Ga, \; \bF \in \sfm \rangle$$
  be the ideal of $\kk[x_\uu]_{\uu \in \II_{d,n}\-\um}$ generated by $I_\Ga$ together with 
all the de-homogenized $\um$-primary $\pl$ relations of $\Gr^{d,E}$. We let
 $Z_\Ga$ $(\subset \Gr^{d,E} \cap \rU_\um)$ be the closed subscheme of 
 the affine space $\rU_\um$ defined by the ideal $I_{\wp,\Ga}$.
The subscheme $Z_\Ga$ is called  the $\Ga$-scheme of $\rU_\um$. 
Note that  $Z_\Ga \ne \emptyset$ since $0 \in Z_\Ga$.
\end{defn}

(Thus, a $\Ga$-scheme is  an intersection of certain Schubert divisors with the chart $\rU_\um$.
 But, in this article, we do not investigate $\Ga$-schemes in any {\it Schubert} way.)
 

Take $\Gamma =\emptyset$. Then, $I_{\wp,\emptyset}$ is the ideal generated by 
all the de-homogenized $\um$-primary $\pl$ relations.
Thus,  $Z_\emptyset =\rU_\um \cap \Gr^{d,E}$.  

Let $\Ga$ be any fixed subset of $\var_{\rU_\um}$.
We let  $\rU_{\um,\Ga}$ be the coordinate  subspace of $\rU_\um$ defined by $I_\Ga$.
That is,
$$\rU_{\um,\Ga}=\{(x_{\uu} =0)_{ x_\uu \in \Ga}\} \subset \rU_{\um}.$$
This is a coordinate subspace of dimension 
${n \choose d}-1 - |\Ga|$ where $|\Ga|$ is the cardinality of $\Ga$. Then,
$Z_\Ga$ is the scheme-theoretic intersection of $\Gr^{d,E}$ with
the coordinate  subspace $\rU_{\um, \Ga}$.
For any $\um$-primary $\pl$ equation $\bF \in \sfm$, we let $\bF|_\Ga$ be the induced 
polynomial obtained from  the de-homogeneous polynomial $\bF$
by setting $x_{\uu} =0$  for all $x_\uu \in \Gamma$. 
Then, $\bF|_\Ga$ becomes a polynomial on the affine subspace $\rU_{\um,\Ga}$.
We point out that $\bF|_\Ga$ can be identically zero on $\rU_{\um,\Ga}$.

\begin{defn}\label{rel-irrel}
 Let $\Ga$ be any fixed subset of $\var_{\rU_\um}$.
 Let ($\bF$) $F$ be any fixed (de-homogenized)  $\um$-primary $\pl$ relation. 
 We say ($\bF$) $F$  is $\Ga$-irrelevant if  every term of 
 $\bF$ belongs to  the ideal $I_\Ga$.   Otherwise, we say ($\bF$) $F$ is $\Ga$-relevant.
 We let $\sfmgr$ be the set of all $\Ga$-relevant de-homogenized $\um$-primary $\pl$ relations.
 We let $\sfmgir$ be the set of all $\Ga$-irrelevant  de-homogenized $\um$-primary $\pl$ relations.
\end{defn}

If  $\bF$ is $\Ga$-irrelevant,  then $\bF|_\Ga$ is identically zero along $\rU_{\um,\Ga}$. 
Indeed,  $\bF$ is $\Ga$-irrelevant if and only if every term of $\bF$ contains a member of $\Ga$.
 The sufficiency direction is clear. To see the necessary direction,
 we suppose a term $x_\uu x_\uv \in I_\Ga$, then as $I_\Ga$ is prime 
 (the coordinate  subspace $\rU_{\um, \Ga}$ is integral), we have
 $x_\uu$ or  $x_\uv \in \Ga$.

\subsection{$\sF$-transforms of  $\Ga$-schemes in $\sV_{\sF_{[k]}}$}
\label{subsection:wp-transform-sfk}   $\ $





In what follows, we keep  notation of Proposition \ref{equas-fV[k]}. 

Recall that for any $\bF \in \sfm$,
$\La_F=\{(\uu_s, \uv_s) \mid s \in S_F\}.$


\begin{lemma}\label{wp-transform-sVk-Ga} 
Fix any  subset $\Ga$ of $\rU_\um$.  Assume that $Z_\Ga$ is integral. 

Consider $F_k \in \sfm$ for any fixed $k \in [\up]$.

Then, we have the following:
\begin{itemize}
\item there exists a closed subscheme $Z_{\sF_{  [k]},\Ga}$ of $\sV_{\sF_{[k]}}$
with an induced morphism  $Z_{\sF_{  [k]},\Ga}
 \to Z_\Ga$;
\item $Z_{\sF_{  [k]},\Ga}$ comes equipped with an irreducible component  
$Z^\dagger_{\sF_{  [k]},\Ga}$ with the induced morphism 
$Z^\dagger_{\sF_{  [k]},\Ga} 
 \to Z_\Ga$;
 \item  for any standard chart $\fV$ of $\sR_{\sF_{[k]}}$ such that
$Z_{\sF_{[k]},\Ga} \cap \fV \ne \emptyset$, there  exists a subset, 
possibly empty,
$$ \tGa^\zero_{\fV} \; \subset \;  \var_\fV. $$
\end{itemize}

Further,  consider any given standard chart $\fV$ of $\sR_{\sF_{[k]}}$ with
$Z_{\sF_{[k]},\Ga} \cap \fV \ne \emptyset$. Then,
the following hold.
\begin{enumerate} 
\item 
The scheme $Z_{\sF_{[k]},\Ga} \cap \fV$, as a closed subscheme of the chart $\fV$,
is defined by the following relations
\begin{eqnarray} 
\;\;\;\;\; y , \; \; \; y \in  \tGa^\zero_\fV ,   \label{Ga-rel-sVk=0} \\
\cB^\pq_{\fV, [k]},    \nonumber \\ 
B_{\fV,(s,t)}: \;\;\;  x_{\fV, (\uu_s, \uv_s)}x_{\fV,\uu_t} x_{\fV,\uv_t} - x_{\fV, (\uu_t,\uv_t)}  
 x_{\fV,\uu_s} x_{\fV,\uv_s},  \;\; s, t \in S_{F_i}, 
  \;  i \in [k], \nonumber\\
L_{\fV, F_i}: \;\; \sum_{s \in S_{F_i}} \sgn (s) x_{\fV, (\uu_s,\uv_s)}, \; \;  
 i \in [k], \nonumber \\
\bF_{\fV,j}: \;\; \sum_{s \in S_{F_j}} \sgn (s) x_{\fV, \uu_s}x_{\fV,\uv_s}, \; \; k < j\le \up.  \nonumber
\end{eqnarray} 
 Further, we take $\tGa^\zero_\fV \subset \var_\fV$
 to be the maximal subset (under inclusion)
among all those subsets that satisfy the above.
\item The induced morphism  $Z^\dagger_{\sF_{  [k]},\Ga}  
\to Z_\Ga$ is birational. 
\item Fix any variable $y=x_{\fV, \uu}$ or $y=x_{\fV, (\uu,\uv)} \in \var_\fV$,
$Z^\dagger_{\sF_{[k]},\Ga} \cap \fV \subset (y=0)$ 
if and only if $Z_{\sF_{[k]},\Ga} \cap \fV \subset (y=0)$.
(We remark here that this property is not used within this lemma, but 
will be used as the initial case of Lemma \ref{vt-transform-k}.)
\end{enumerate}
\end{lemma}
\begin{proof}
We prove the statement by induction on $k$ with $k \in \{0\} \cup [\up]$.

 When $k=0$,  we have $\sR_{\sF_{[0]}}:=\rU_\um$, $\sV_{\sF_{[0]}}:=\rU_\um \cap \Gr^{d,E}$.
 There exists a unique chart $\fV=\rU_\um$. In this case,  we set 
 $$Z_{\sF_{[0]},\Ga}=Z^\dagger_{\sF_{[0]},\Ga}:=Z_\Ga$$
 Further, we let
 $$ \tGa^\zero_\fV=\Ga.$$ 
Then,  the statement holds trivially.

Inductively, we suppose that Lemma \ref{wp-transform-sVk-Ga}
 holds for $\sV_{\sF_{[k-1]}} \subset \sR_{\sF_{[k-1]}} $.

 We now consider $\sV_{\sF_{[k]}} \subset \sR_{\sF_{[k]}}$.
 
Recall from  \eqref{rho-sFk},   we have the natural birational morphsim
$$\rho_{\sF_{[k]}}: \sV_{\sF_{[k]}} \lra \sV_{\sF_{[k-1]}},$$
induced from the forgetful map  $\sR_{\sF_{[k]}} \lra \sR_{\sF_{[k-1]}}$.

First, we suppose $F_k$ is $\Ga$-relevant.

In this case, we set 
\begin{equation}\label{construction-ZkLa} \La^\zero_{F_k, \Ga} 
:=\{ x_{(\uu,\uv)}\in \La_{F_k} \mid \hbox{$x_\uu$ or $x_\uv \in \Ga$}\}.
\end{equation}
(Here, recall the convention of \eqref{uv=vu}: $x_{(\uu,\uv)}= x_{(\uv,\uu)}$.) 

We then let $\rho_{\sF_{[k]}}^{-1}(Z_{\sF_{ [k-1]},\Ga})$ be the scheme-theoretic pre-image and define
$Z_{\sF_{[ k]},\Ga}$ to be
  the scheme-theoretic intersection
\begin{equation}\label{construction-ZkGa}
Z_{\sF_{[ k]},\Ga}=\rho_{\sF_{[k]}}^{-1}(Z_{\sF_{ [k-1]},\Ga}) \cap (x_{(\uu,\uv)} =0 \mid (\uu,\uv) \in  \La^\zero_{F_k, \Ga}),
\end{equation}

Next, because $F_k$ is $\Ga$-relevant
 and $Z^\dagger_{\sF_{ [k-1]},\Ga}$ is birational to $Z_\Ga$,
one checks that  $Z^\dagger_{\sF_{ [k-1]},\Ga}$ 
is not contained in the exceptional locus of
the birational morphism $\rho_{\sF_{[k]}}$. 
Thus, there exists  a Zariski open subset $Z^{\dagger\circ}_{\sF_{ [k-1]},\Ga}$
 of $Z^\dagger_{\sF_{ [k-1]},\Ga}$ such that 
 $$\rho_{\sF_{[k]}}^{-1}(Z^{\dagger\circ}_{\sF_{ [k-1]},\Ga})  \lra Z^{\dagger\circ}_{\sF_{ [k-1]},\Ga}$$
is an isomorphism.

We claim 
\begin{equation}\label{inclusion-dagger}
 \rho_{\sF_{[k]}}^{-1}(Z^{\dagger\circ}_{\sF_{ [k-1]},\Ga})   \subset Z_{\sF_{[ k]},\Ga}=
\rho_{\sF_{[k]}}^{-1}(Z_{\sF_{ [k-1]},\Ga}) \cap (x_{(\uu,\uv)}=0 \mid (\uu,\uv) \in  \La^\zero_{F_k, \Ga}).
\end{equation}
To see this, note that since $\bF_k$ is $\Ga$-relevant, 
there exists a term $x_{\uu_s}x_{\uv_s}$ of $\bF_k$ 
for some $s \in S_{F_k}$ such that it does not vanish
generically along $Z^{\dagger}_{\sF_{ [k-1]},\Ga}$ (which is birational to $Z_\Ga$).
Then, we consider the binomial relation of $\sV_{\sF_{[k]}}$ in $\sR_{\sF_{[k]}}$
  \begin{equation}\label{Buv-s}
  x_{(\uu, \uv)}x_{\uu_s} x_{\uv_s} - x_{\uu} x_{\uv} x_{(\uu_s,\uv_s)},
  \end{equation}
  for any $(\uu, \uv) \in \La_{F_k}$.
 It follows that $x_{(\uu, \uv)}$  vanishes identically 
along $\rho_{\sF_{[k]}}^{-1}(Z^{\dagger\circ}_{\sF_{ [k-1]},\Ga}) \cong Z^{\dagger\circ}_{\sF_{ [k-1]},\Ga}$ 
 if $x_\uu$ or $x_\uv \in \Ga$.  Hence, \eqref{inclusion-dagger} holds.

 We then let $Z^\dagger_{\sF_{ [k]},\Ga}$ be the closure of 
 $\rho_{[k]}^{-1}(Z^{\dagger\circ}_{\sF_{ [k-1]},\Ga})$ in $Z_{\sF_{ [k]},\Ga}$.
  Since $Z^\dagger_{\sF_{ [k]},\Ga}$  is closed in $Z_{\sF_{ [k]},\Ga}$
 and contains the Zariski  open subset $\rho_{[k]}^{-1}(Z^{\dagger\circ}_{\sF_{ [k-1]},\Ga})$
 of $Z_{\sF_{[k]},\Ga}$, it is an irreducible component of  $Z_{\sF_{ [k]},\Ga}$.

Further, consider any standard chart $\fV$ of $\sR_{\sF_{[k]}}$,
lying over a unique standard chart $\fV'$ of $\sR_{\sF_{[k-1]}}$,
such that $Z_{\sF_{[k]},\Ga}\cap \fV \ne \emptyset$.
We set 
\begin{equation}\label{zero-one-fV-sFk} 
\tGa^\zero_\fV= \tGa^\zero_{\fV'} \sqcup 
\{ x_{(\uu,\uv)}\in \La_{F_k} \mid \hbox{$x_\uu$ or $x_\uv \in \Ga$}\}.  
\end{equation}

We are now ready to 
prove Lemma \ref{wp-transform-sVk-Ga} (1), (2) and (3) in the case of $\sR_{\sF_{[k]}}$.

(1).  Note that scheme-theoretically, we have
$$\rho_{\sF_{[k]}}^{-1}(Z_{\sF_{ [k-1]},\Ga}) \cap \fV=
 \pi_{\sF_{[k]}, \sF_{[k-1]}}^{-1}(Z_{\sF_{ [k-1]},\Ga}) \cap \sV_{\sF_{[k]}} \cap \fV$$
 where $\pi_{\sF_{[k]},  \sF_{[k-1]}}: \sR_{\sF_{[k]}} \to \sR_{\sF_{[k-1]}}$ is the projection.
 We can apply Lemma \ref{wp-transform-sVk-Ga} (1)  in the case of $\sR_{\sF_{[k-1]}}$
 to  $Z_{\sF_{ [k-1]},\Ga}$ and $\pi_{\sF_{[k]}, \sF_{[k-1]}}^{-1}(Z_{\sF_{ [k-1]},\Ga})$, 
 apply Proposition \ref{equas-fV[k]}  to $\sV_{\sF_{[k]}} \cap \fV$, and 
 use the construction \eqref{construction-ZkGa} of $Z_{\sF_{[ k]},\Ga}$
 (cf. \eqref{construction-ZkLa} and \eqref{zero-one-fV-sFk}), 
 we then obtain that $Z_{\sF_{[ k]},\Ga} \cap \fV$, as a closed subscheme of $\fV$, is defined by 
$$ y, \;\; y \in  \tGa^\zero_\fV;\;\; \cB^\pq_{[k]}; $$
$$  B_{\fV, (s,t)},\;\; s, t \in S_{F_i} \; \hbox{ with all $i \in [k]$}$$
 $$  L_{\fV,F_i}, \; i \in [k]; \;  \bF_{\fV,j}, \; k <j\le \up.$$ 
 Then, the above implies Lemma \ref{wp-transform-sVk-Ga} (1) in the case of $\sR_{\sF_{[k]}}$.

(2). By construction, we have that the composition
$\tZ^\dagger_{\sF_{[k]},\Ga} \to \tZ^\dagger_{\sF_{[k-1]},\Ga} \to Z_\Ga$
is birational. 
This proves Lemma \ref{wp-transform-sVk-Ga} (2) in the case of $\sR_{\sF_{[k]}}$.

  (3). It suffices to prove that if $Z^\dagger_{\sF_{[ k]},\Ga}\cap \fV \subset (y=0)$,
 then $Z_{\sF_{[ k]},\Ga} \cap \fV \subset (y=0).$
 
 If $y=x_{\fV,\uu}$ ($=x_\uu$, cf. the proof of Proposition \ref{meaning-of-var-p-k=0}), 
 then 
$x_\uu \in \Ga$ because $Z^\dagger_{\sF_{[ k]},\Ga}$ is birational to $Z_\Ga$.
Therefore, $Z_{\sF_{[ k]},\Ga} \cap \fV \subset (y=0)$ by \eqref{Ga-rel-sVk=0},
which holds by (the just proved) Lemma \ref{wp-transform-sVk-Ga} (1) for  $\sR_{\sF_{[k]}}$.

Now assume $y=x_{\fV, (\uu, \uv)}$.
Here, $x_{\fV, (\uu, \uv)}$ is the de-homogenization of $x_{(\uu,\uv)}$
(cf. the proof of Proposition \ref{meaning-of-var-p-k=0}). Below,
upon setting $x_{(\uu_{s_{F_i,o}}, \uv_{s_{F_i,o}})} \equiv 1$ for all $i \in [k]$ (cf. Definition \ref{fv-k=0}),
 we can write $x_{\fV, (\uu, \uv)}=x_{\fV', (\uu, \uv)}=x_{(\uu,\uv)}$.

 Suppose $(\uu, \uv) \in \La_{F_i}$ with $i \in [k-1]$.
By taking the images of $Z^\dagger_{\sF_{[ k]},\Ga}\cap \fV \subset (y=0)$
under $\rho_{\sF_{[k]}}$, we obtain 
$Z^\dagger_{\sF_{[ k-1]},\Ga} \cap \fV' \subset (x_{(\uu, \uv)}=0)$.
Hence, we have
$Z_{\sF_{[k-1]},\Ga} \cap \fV' \subset (x_{(\uu, \uv)}=0)$ by 
Lemma \ref{wp-transform-sVk-Ga} (3)  for 
$ \sR_{\sF_{[k-1]}}$.
Therefore,  $x_{(\uu, \uv)} \in \tGa^\zero_{\fV'}$ by the maximality of the subset $ \tGa^\zero_{\fV'}$.
Then, by  \eqref{zero-one-fV-sFk},
$Z_{\sF_{[ k]},\Ga} \cap \fV \subset   (x_{(\uu, \uv)}=0)$.

Now suppose $(\uu, \uv) \in \La_{F_k}$.
 Consider the relations  $$x_{\fV,\uu} x_{\fV,\uv} -x_{\fV, (\uu, \uv)}x_{\uu_{s_{F_k,o}}}x_{\uv_{s_{F_k,o}}} .$$ 
 Here, we have used $x_{(\uu_{s_{F_k,o}}, \uv_{s_{F_k,o}})} \equiv 1$. 
  Then, we have  $x_{\fV,\uu} x_{\fV,\uv}$ vanishes identically along $Z^\dagger_{\sF_{[ k]},\Ga}$,
  hence,  so does
  one of $x_{\fV, \uu}$ and $x_{\fV, \uv}$, that is, $x_\uu$ or $x_\uv \in \Ga$,
  since $Z^\dagger_{\sF_{[ k]},\Ga}$ (birational to $Z_\Ga$) is integral. In either case, it implies that 
  $Z_{\sF_{[ k]},\Ga} \subset (x_{(\uu, \uv)}=0)$ by \eqref{construction-ZkLa} and  
  \eqref{construction-ZkGa}. 

This proves the lemma when $F_k$ is $\Ga$-relevant.

\smallskip

 Next, we suppose $F_k$ is $\Ga$-irrelevant. 
 
In this case, we have that
$$ ( \rho_{\sF_{[k]}}^{-1}(Z_{\sF_{  [k-1]},\Ga}) ) / (Z_{\sF_{[k-1]},\Ga})$$ 
is defined by the set of
equations of $L_{F_k}$ and $\cB^\pq_{ [k]}$, all regarded as
relations in $\vr$-variables of $F_k$. All these relations are  
 linear in $\vr$-variables of $F_k$, according to Lemma \ref{ker-phi-k}.
 Putting together, we call $\{L_{F_k}, \cB^\pq_{ [k]}\}$ a linear system 
  in $\vr$-variables of $F_k$. 
 
 We can let 
 $\La^{\rm det}_{F_k, \Ga}$ be the subset of $\La_{F_k}$ such that 
  the minor corresponding to  variables
 $$\{x_{\fV,(\uu,\uv)} \mid (\uu,\uv) \in \La^{\rm det}_{F_k,\Ga}\}$$ 
 achieves the maximal rank of the linear system $ \{L_{F_k}, \cB^\pq_{[k]}|\}$,
  regarded as relations in $\vr$-variables of $F_k$,
 at any point of some fixed Zariski open subset $Z_{\sF_{[k-1]},\Ga}^{\dagger\circ}$ of 
 $Z^\dagger_{\sF_{[k-1]},\Ga}$. 

 We then set and plug 
 \begin{equation}\label{a-ne-0}
  x_{(\uu,\uv)} = 0, \; \forall \; (\uu,\uv)  \notin \La^{\rm det}_{F_k,\Ga}
 \end{equation}
   into the linear system $\{L_{F_k}, \cB^\pq_{ [k]}\}$
   to obtain an induced linear system of full rank 
   over $Z_{\sF_{[k-1]},\Ga}^{\dagger\circ}$. This induced linear system can be solved
  over the Zariski open subset $Z^{\dagger\circ}_{\sF_{[k-1]},\Ga}$ such that
  all variables $$\{x_{(\uu,\uv)} \mid (\uu,\uv) \in \La^{\rm det}_{F_k, \Ga}\}$$ 
 are explicitly  determined by the coefficients of the induced linear system.
    
     We then  let 
   \begin{equation}\label{det=0}
   \La^\zero_{F_k, \Ga} \subset \La_{F_k}
   \end{equation}
   be the subset consisting of    $ (\uu,\uv) \notin \La^{\rm det}_{F_k ,\Ga}$ 
and $(\uu,\uv) \in \La^{\rm det}_{F_k ,\Ga}$ 
   such that $x_{(\uu,\uv)} \equiv 0$ over
   $Z^{\dagger\circ}_{\sF_{[k-1]},\Ga}$.
    Observe here that we immediately obtain that  for any $(\uu,\uv) \in \La_{F_k}$,
\begin{equation}\label{who-in-dagger}
\hbox{$x_{(\uu,\uv)}$  vanishes identically over $Z^{\dagger\circ}_{\sF_{[k-1]},\Ga}$
   if and only if  $(\uu,\uv) \in \La^\zero_{F_k,\Ga}$.}
   \end{equation}

We let $Z_{\sF_{[ k]},\Ga}$ be
  the scheme-theoretic intersection
   \begin{equation}\label{ZkGa-irr}
   \rho_{\sF_{[k]}}^{-1}(Z_{\sF_{ [k-1]},\Ga}) \cap ( x_{(\uu,\uv)}=0 , \; (\uu,\uv) \in  \La^\zero_{F_k,\Ga}) . \end{equation}

Now, 
fix and consider any standard chart $\fV$ of $\sR_{\sF_{[k]}}$, lying over a
  standard chart $\fV'$  of $\sR_{\sF_{[k-1]}}$ with $\tZ_{\sF_{[k]},\Ga}\cap \fV \ne \emptyset$, 
  equivalently, $Z_{\sF_{[k-1]}} \cap \fV' \ne \emptyset$. We set 
\begin{eqnarray}\label{zero-one-fV-sFk-irr}
\tGa^\zero_\fV= \tGa^\zero_{\fV'} \sqcup    \{ x_{\fV, (\uu,\uv)} \mid 
(\uu, \uv) \in \La^\zero_{F_k,\Ga}\}. 
\end{eqnarray}

We are now ready to 
prove Lemma \ref{wp-transform-sVk-Ga} (1), (2) and (3) in the case of $\sR_{\sF_{[k]}}$.

  Similar to the proof of Lemma \ref{wp-transform-sVk-Ga} (1) for the previous case when
  $\bF_k$ is $\Ga$-relevant, 
  by Lemma \ref{wp-transform-sVk-Ga} (1)  in the case of $\sR_{\sF_{[k-1]}}$
  applied to $Z_{\sF_{ [k-1]},\Ga}$ and $\rho_{\sF_{[k]}}^{-1}(Z_{\sF_{ [k-1]},\Ga})$,
 applying Proposition \ref{equas-fV[k]} to $\sV_{F_{[k]}} \cap \fV$, and using 
 \eqref{ZkGa-irr} and \eqref{zero-one-fV-sFk-irr}, 
 we obtain that $Z_{\sF_{[ k]},\Ga} \cap \fV$, as a closed subscheme of $\fV$, is defined by 
$$y, \;\; y \in  \tGa^\zero_\fV; \;\; \cB^\pq_{[k]}; $$
$$  B_{\fV, (s,t)},\;\; s, t \in S_{F_i} \; \hbox{ with all $i \in [k]$}$$
$$ L_{\fV,F_i}, \;  \; i \in [k]; \; 
 \bF_{\fV,j}, \; k <j\le \up.$$ 
 Then, the above implies that Lemma \ref{wp-transform-sVk-Ga} (1) holds  on $\sR_{\sF_{[k]}}$.



  Next, by construction, the induced morphism
  $$ \rho_{\sF_{[k]}}^{-1} (Z^{\dagger\circ}_{\sF_{[k]},\Ga}) 
  \cap (x_{(\uu,\uv)}=0 , \; (\uu,\uv) \in  \La^\zero_{F_k,\Ga}) 
   \lra Z^{\dagger\circ}_{\sF_{[k-1]}}$$
  is an isomorphism. We let 
$Z^\dagger_{\sF_{ [k]},\Ga}$ be the closure of 
$$ \rho_{\sF_{[k]}}^{-1} (Z^{\dagger\circ}_{\sF_{[k]},\Ga}) 
  \cap (x_{(\uu,\uv)}=0 , \; (\uu,\uv) \in  \La^\zero_{F_k,\Ga} ) $$ in $Z_{\sF_{[ k]},\Ga}$.
  Then, it is closed in $Z_{\sF_{[ k]},\Ga}$ and contains 
  an open subset of $Z_{\sF_{[ k]},\Ga}$, hence, is an irreducible
  component of $Z_{\sF_{[ k]},\Ga}$.  It follows that the composition
  $$Z^\dagger_{\sF_{ [k]},\Ga}\to Z^{\dagger}_{\sF_{[k-1]}} \to Z_\Ga$$  is birational.
  This proves    Lemma \ref{wp-transform-sVk-Ga} (2) on $\sR_{\sF_{[k]}}$.
  

Finally,  we are to prove  Lemma \ref{wp-transform-sVk-Ga} (3) on $\sR_{\sF_{[k]}}$.
Suppose $Z^\dagger_{\sF_{ [k]},\Ga} \cap \fV \subset (y=0)$ for some $y \in \var_\fV$.
If $y=x_{\fV, \uu}$ or $y=x_{\fV, (\uu, \uv)}$ with  $(\uu, \uv) \in \La_{F_i}$ with $i \in [k-1]$,
then the identical  proof in the previous case carries over here without changes.
We now suppose $Z^\dagger_{\sF_{ [k]},\Ga}\cap \fV \subset (x_{\fV,(\uu,\uv)}=0)$
with $(\uu, \uv) \in \La_{F_k}$, then by \eqref{who-in-dagger},
$(\uu,\uv) \in \La^\zero_{F_k,\Ga}$. Thus,  by \eqref{ZkGa-irr}, 
$Z_{\sF_{ [k]},\Ga} \subset (x_{(\uu,\uv)}=0)$.
This proves    Lemma \ref{wp-transform-sVk-Ga} (3) on $\sR_{\sF_{[k]}}$.

By induction, Lemma \ref{wp-transform-sVk-Ga}  is proved.
\end{proof}

We call $Z_{\sF_{[k]},\Ga}$ the $\sF$-transform of $Z_\Ga$
in $\sV_{\sF_{[k]}}$ for any $k\in [\up]$.

\subsection{$\vt$-transforms of  $\Ga$-schemes in  $\tsV_{\vt_{[k]}}$}
\label{subsection:wp-transform-vtk}  $\ $

We now construct the $\vt$-transform of $Z_\Ga$ in $\tsV_{\vt_{[k]}}
\subset \tsR_{\vt_{[k]}}$. 

\begin{lemma}\label{vt-transform-k} 
 Fix any subset $\Ga$ of $\rU_\um$.  Assume that $Z_\Ga$ is integral. 

Fix any $k \in [\up]$.

Then, we have the following:
\begin{itemize}
\item  there exists a closed subscheme $\tZ_{\vt_{[k]},\Ga}$ of
$\tsV_{\vt_{[k]}}$ with an induced morphism  
$\tZ_{\vt_{[k]},\Ga} 
\to Z_\Ga$;
\item   $\tZ_{\vt_{[k]},\Ga}$ comes equipped with an irreducible component  
$\tZ^\dagger_{\vt_{[k]},\Ga}$ with the induced morphism 
$\tZ^\dagger_{\vt_{[k]},\Ga}  
\to Z_\Ga$;
\item  for any standard chart $\fV$ of $\tsR_{\vt_{[k]}}$ such that
$\tZ_{\vt_{[k]},\Ga} \cap \fV \ne \emptyset$, there are two subsets, possibly empty,
$$ \tGa^\zero_{\fV} \; \subset \;  \var_\fV, \;\;\;
\tGa^\one_{\fV} \; \subset \;  \var_\fV.$$ 
\end{itemize}

Further,  consider any given standard chart $\fV$ of $\tsR_{\vt_{[k]}}$ with
$\tZ_{\vt_{[k]},\Ga} \cap \fV \ne \emptyset$. Then,  the following hold:
\begin{enumerate}
\item the scheme $\tZ_{\vt_{[k]},\Ga} \cap \fV$,
 as a closed subscheme of the chart $\fV$,
is defined by the following relations
\begin{eqnarray} 
\;\;\;\;\; y , \; \; \; y \in  \tGa^\zero_\fV , \label{Ga-rel-wp-ktauh-00}\\
\;\;\;  y -1, \; \; \; y \in  \tGa^\one_\fV,  \nonumber \\
\cB_\fV^\mn, \; \cB^\res_{\fV, >k}, \;  \cB^\q_\fV, \; L_{\fV,{ \sF_\um}};\nonumber
\end{eqnarray}
 further, we take $\tGa^\zero_\fV \subset \var_\fV$
 to be the maximal subset (under inclusion)
among all those subsets that satisfy the above;
\item the induced morphism $\tZ^\dagger_{ \vt_{[k]},\Ga} \to Z_\Ga$ is birational;
\item for any variable $y \in \var_\fV$, $\tZ^\dagger_{\vt_{[k]},\Ga} \cap \fV \subset (y=0)$ if and only 
if  $\tZ_{\vt_{[k]},\Ga} \cap \fV \subset (y=0)$. Consequently, 
 $\tZ^\dagger_{\vt_{[k]},\Ga} \cap \fV \subset \tZ_{\vt_{[k+1]}} \cap \fV$ if and only 
if  $\tZ_{\vt_{[k]},\Ga} \cap \fV \subset \tZ_{\vt_{[k+1]}}\cap \fV$ 
where $\tZ_{\vt_{[k]}}$ is the proper transform of 
 the $\vt$-center $Z_{\vt_{[k+1]}}$ in  $\tsR_{\vt_{[k]}}$. 
\end{enumerate}
\end{lemma}  
\begin{proof} We prove by induction on 
$k \in \{0\} \cup [\up]$. 


The initial case is $k=0$. In this case, we have
 $$\tsR_{\vt_{[0]}}:=\sR_{\sF}, \;\;
 \tsV_{\vt_{[0]}}:=\sV_{\sF}, \;\;
  \tZ_{\vt_{[0]},\Ga}:=Z_{\sF_{ [\up]},\Ga}, \;\;
  \tZ^\dagger_{\vt_{[0]},\Ga}:=Z^\dagger_{\sF_{ [\up]},\Ga}.$$
Then, in this case,  Lemma \ref{vt-transform-k}
 is Lemma  \ref{wp-transform-sVk-Ga} for $k=\up$, where
  we  set $\tGa^\one_\fV=\emptyset$.

We now suppose that Lemma \ref{vt-transform-k} holds
over $\tsR_{\vt_{[k-1]}}$ for some $k\in[\up]$.

We then consider the case of  $\tsR_{\vt_{[k]}}$.

Suppose that $\tZ_{\vt_{[k-1]}, \Ga}$,  or equivalently $\tZ^\dagger_{\vt_{[k-1]},\Ga}$,
by Lemma \ref{vt-transform-k} (3) in the case of $\tsR_{\vt_{[k-1]}}$,
 is not contained in 
$Z'_{\vt_{[k]}}$ where $Z'_{\vt_{[k]}}$ is the proper transform in $\tsR_{\vt_{[k-1]}}$
 of the $\vt$-center $Z_{\vt_{[k]}}$ (of $\tsR_{\vt_{[0]}}$).
We then let $\tZ_{\vt_{[k]},\Ga}$ (respectively, $\tZ^\dagger_{\vt_{[k]},\Ga}$)
 be the proper transform of $\tZ_{\vt_{[k-1]},\Ga}$ (respectively,
  $\tZ^\dagger_{\vt_{[k-1]},\Ga}$) in $\tsV_{\vt_{[k]}}$.
As $\tZ^\dagger_{\vt_{[k]},\Ga}$ is closed in $\tZ_{\vt_{[k]},\Ga}$
and contains a Zariski open subset of $\tZ_{\vt_{[k]},\Ga}$, it is an irreducible
component of $\tZ_{\vt_{[k]},\Ga}$.

Further, consider any standard chart $\fV$ of $\tsR_{\vt_{[k]}}$,
lying over a unique standard chart $\fV'$ of $\tsR_{\vt_{[k-1]}}$,
such that $\tZ_{\vt_{[k]},\Ga}\cap \fV \ne \emptyset$.
We set 
$$\tGa^\zero_\fV=\{y_\fV \mid y_\fV  \hbox{ is the proper transform of some $y_{\fV'} \in \tGa^\zero_{\fV'}$}\};$$
$$\tGa^\one_\fV=\{y_\fV \mid y_\fV  \hbox{ is the proper transform of some $y_{\fV'} \in \tGa^\one_{\fV'}$}\}.$$

We now prove Lemma \ref{vt-transform-k} (1), (2) and (3) in the case of $\tsR_{\vt_{[k]}}$.

 We can apply
 Lemma \ref{vt-transform-k} (1) in the case of $\tsR_{\vt_{[k-1]}}$
 to $\tZ_{\vt_{[k-1]},\Ga}$ to obtain the defining equations of $\tZ_{\vt_{[k-1]},\Ga}\cap \fV'$
as stated in the lemma; we note here that these equations include $\cB^\res_{\fV', \ge k}$.
We then  take the proper transforms of these equations in $\fV'$ to obtain the corresponding
equations in $\fV$,  and then apply (the proof of) 
 Proposition \ref{eq-for-sV-vtk} to reduce $\cB^\res_{\fV, \ge k}$ to $\cB^\res_{\fV, > k}$.
 Because $\tZ_{\vt_{[k]}\Ga}$  is the proper transform of $\tZ_{\vt_{[k-1]},\Ga}$,
this implies Lemma \ref{vt-transform-k} (1)  in the case of $\tsR_{\vt_{[k]}}$.

By construction, we have that the composition
$\tZ^\dagger_{\vt_{[k]},\Ga} \to \tZ^\dagger_{\vt_{[k-1]},\Ga} \to Z_\Ga$
is birational. 
This proves
Lemma \ref{vt-transform-k} (2) in the case of $\tsR_{\vt_{[k]}}$.

To show Lemma \ref{vt-transform-k} (3) in $\tsR_{\vt_{[k]}}$, 
we fix any $y \in \var_\fV$. It suffices to show that
if $\tZ^\dagger_{\vt_{[k]},\Ga}\cap \fV
\subset (y=0)$, then $\tZ_{\vt_{[k]},\Ga}\cap \fV
\subset (y=0)$.  By construction, $y \ne \zeta_{\fV,\vt_{[k]}}$, the exceptional variable
in $\var_\fV$ corresponding to the $\vt$-center $Z_{\vt_{[k]}}$. Hence, $y$ is the proper transform
of some $y' \in \var_{\fV'}$. Then, by taking the images of $\tZ^\dagger_{\vt_{[k]},\Ga}\cap \fV
\subset (y=0)$ under the morphism $\rho_{\vt_{[k]}}: \tsV_{\vt_{[k]}} \to \tsV_{\vt_{[k-1]}}$ 
(which is  induced from the blowup morphism $\pi_{\vt_{[k]}}: \tsR_{\vt_{[k]}} \to \tsR_{\vt_{[k-1]}}$), 
we obtain
$\tZ^\dagger_{\vt_{[k-1]},\Ga}\cap \fV'
\subset (y'=0)$,  hence, $\tZ_{\vt_{[k-1]},\Ga}\cap \fV'
\subset (y'=0)$ by the inductive assumption.
 Then, as $\tZ_{\vt_{[k]},\Ga}$ is the proper transform of $\tZ_{\vt_{[k-1]},\Ga}$,
 we obtain $\tZ_{\vt_{[k]},\Ga}\cap \fV \subset (y=0)$.
 
 The last statement Lemma \ref{vt-transform-k} (3)  follows from the above because
 $\tZ_{\vt_{[k+1]}}\cap \fV=(y_0=y_1=0)$ for some $y_0, y_1 \in \var_\fV$.

\smallskip

We now suppose that $\tZ_{\vt_{[k-1]},\Ga}$, or equivalently $\tZ^\dagger_{\vt_{[k-1]},\Ga}$,
 by  Lemma \ref{vt-transform-k} (3) in $\tsR_{\vt_{[k-1]}}$,
 is contained in the proper transform $Z'_{\vt_{[k]}}$ 
 of the $\vt$-center $Z_{\vt_{[k]}}$. 

Consider  any standard chart $\fV$ of $\tsR_{\vt_{[k]}}$,
lying over a unique standard chart $\fV'$ of $\tsR_{\vt_{[k-1]}}$,
such that $\tZ_{\vt_{[k]},\Ga}\cap \fV \ne \emptyset$.

We let $\vt'_{[k]}$
 be the proper transform in  the chart $\fV'$ of the $\vt$-set ${\vt_{[k]}}$.
Then, $\vt'_{[k]}$ consists of two  variables 
 $$\vt'_{[k]}=\{y'_0,y'_1\} \subset \var_{\fV'}.$$

 We let $\PP^1_{[\xi_0,\xi_1]}$ be the factor projective space for the
$\vt$-blowup $\tsR_{\vt_{[k]}} \to \tsR_{\vt_{[k-1]}}$ with $[\xi_0,\xi_1]$ corresponding to $(y'_0,y'_1)$.
Without loss of generality,  we can assume that the open chart $\fV$ is given by 
  $$(\fV' \times (\xi_0 \equiv 1) ) \cap \tsR_{\vt_{[k]}} \subset \fV' \times \PP^1_{[\xi_0,\xi_1]}.$$

 We let $\zeta_\fV:=\zeta_{\fV, \vt_{[k]}} \in \var_\fV$ be such that 
$E_{\vt_{[k]}} \cap \fV=(\zeta_\fV =0)$ where $E_{\vt_{[k]}}$ is the exceptional divisor
of the blowup $\tsR_{\vt_{[k]}} \to \tsR_{\vt_{[k-1]}}$. 
Note here that according to the proof of Proposition \ref{meaning-of-var-vtk},
the variable $y'_0$ 
corresponds to (or turns into) the exceptional $\zeta_\fV$ on the chart $\fV$.
 We then let $y_1 (=\xi_1) \in \var_\fV$ be 
 the proper transform of $y'_1 \in \var_{\fV'}$ on the chart $\fV$.

  
In addition, we  observe that 
 $$\vt'_{[k]}=\{y'_0,  y'_1\} \subset \tGa^\zero_{\fV'}$$ 
because $\tZ_{\vt_{[k]},\Ga}$ is contained in the proper transform
$Z'_{\vt_{[k]}}$ of the $\vt$-center $Z_{\vt_{[k]}}$.

We set, 
\begin{equation}\label{Gazero-ktah-contained-in-bar-vtk} 
\overline{\Ga}^\zero_\fV= \{\zeta_\fV, \; y_\fV \mid y_\fV  
\hbox{ is the proper transform of some $ y_{\fV'} \in \tGa^\zero_{\fV'} \- \vt'_{[k]}$}\},
\end{equation}
\begin{equation}\label{Gaone-ktah-contained-in-bar-vtk}
 \overline{\Ga}^\one_{\fV}=\{\ y_\fV \mid y_\fV  
\hbox{ is the proper transform of some $y_{\fV'} \in \tGa^\one_{\fV'}$} \}.
\end{equation}

Consider the scheme-theoretic pre-image  
$\rho_{\vt_{[k]}}^{-1}(\tZ_{\vt_{[k-1]},\Ga})$
where $\rho_{\vt_{[k]}}: \tsV_{\vt_{[k]}} \to \tsV_{\vt_{[k-1]}}$ is  induced 
from the blowup morphism $\pi_{\vt_{[k]}}: \tsR_{\vt_{[k]}} \to \tsR_{\vt_{[k-1]}}$.

Note that  scheme-theoretically, we have,
$$\rho_{\vt_{[k]}}^{-1}(\tZ_{\vt_{[k-1]},\Ga})  \cap \fV=
\pi_{\vt_{[k]}}^{-1}(\tZ_{\vt_{[k-1]},\Ga}) \cap \tsV_{\vt_{[k]}} \cap \fV .$$
Applying Lemma \ref{vt-transform-k} (1) in $\tsR_{\vt_{[k-1]}}$ to $\tZ_{\vt_{[k-1]},\Ga}$ 
and $\rho_{\vt_{[k]}}^{-1}(\tZ_{\vt_{[k-1]},\Ga})$, 
and applying Proposition \ref{eq-for-sV-vtk} to 
$\tsV_{\vt_{[k]}} \cap \fV$, we obtain that 
$\rho_{\vt_{[k]}}^{-1}(\tZ_{\vt_{[k-1]},\Ga})  \cap \fV$,
as a closed subscheme of $\fV$,  is defined by 
\begin{equation}\label{vt-pre-image-defined-by}
y_\fV \in \overline{\Ga}^\zero_\fV; \;\;  y_\fV-1, \; y_\fV \in \overline{\Ga}^\one_{\fV};\;\;
\cB^\mn_\fV; \;\; \cB^\res_{\fV, >k}; \;\; \cB^\q_{\fV}; \;\; L_{\fV, \sF_\um}.
\end{equation} 
(Observe here that $\zeta_\fV \in  \overline{\Ga}^\zero_\fV$.)

Thus, by setting $y_\fV=0$ for all $y_\fV \in  \overline{\Ga}^\zero_\fV$ and
 $y_\fV=1$ for all $y_\fV \in \overline{\Ga}^\one_{\fV}$ in 
 $\cB^\mn_\fV, \cB^\res_{\fV, >k}, \cB^\q_{\fV}, L_{\fV, \sF_\um}$ of  the above,
 we obtain 
 \begin{equation}\label{vt-lin-xi-pre}
 \tilde{\cB}^\mn_\fV, \tilde\cB^\res_{\fV,>k}, \tilde\cB^\q_{\fV}, \tilde{L}_{\fV,\sF_\um}.
 \end{equation}
 Note that  for any $\bF \in \sfm$, if $L_{\fV, F}$ contains $y_1$,
 then it contains $\zeta_\fV$, hence $\tilde L_{\fV, F}$ does not contain $y_1$.
 We keep those equations of \eqref{vt-lin-xi-pre} that  contain the variable
 $y_1$ and obtain
 \begin{equation}\label{lin-xi-vtk}
 \hat{\cB}^\mn_\fV, \hat\cB^\res_{\fV,>k}, \hat\cB^\q_{\fV}, 
 \end{equation}
 viewed as a  system of equations in $y_1$.
 By Proposition \ref{eq-for-sV-vtk} (the last two statements), 
one sees that \eqref{lin-xi-vtk} is  a  {\it linear} system of equations in $y_1$.
Furthermore, we have that 
 $$(\rho_{\vt_{[k]}}^{-1}(\tZ_{\vt_{[k-1]},\Ga})\cap \fV)/(\tZ_{\vt_{[k-1]},\Ga}\cap \fV')$$
 is defined by the linear system \eqref{lin-xi-vtk}.

There are the following two cases for \eqref{lin-xi-vtk}:

$(\star a)$  the rank of the linear system  $\eqref{lin-xi-vtk}$ equals one 
over general points of  $\tZ^\dagger_{\vt_{[k-1]},\Ga}$.

$(\star b)$  the rank of the linear system  $\eqref{lin-xi-vtk}$ equals zero
at general points of  $\tZ^\dagger_{\vt_{[k-1]},\Ga}$, hence at all points
of $\tZ^\dagger_{\vt_{[k-1]},\Ga}$.

\smallskip\noindent
{\it Proof of Lemma \ref{vt-transform-k} for $\tsR_{\vt_{[k]}}$ under the condition $(\star a)$.}
\smallskip

By the condition $(\star a)$,
 there exists a Zariski open subset $\tZ^{\dagger\circ}_{\vt_{[k-1]},\Ga}$ 
of $\tZ^\dagger_{\vt_{[k-1]},\Ga}$ such that the rank of the linear system
 \eqref{lin-xi-vtk} equals one at any point of  $\tZ^{\dagger\circ}_{\vt_{[k-1]},\Ga}$. 
By solving $y_1$ from
 the linear system \eqref{lin-xi-vtk} over $\tZ^{\dagger\circ}_{\vt_{[k-1]},\Ga}$, we obtain 
that the induced morphism
$$
\rho_{\vt_{[k]}}^{-1}(\tZ^{\dagger\circ}_{\vt_{[k-1]},\Ga}) 
 \lra \tZ^\circ_{\vt_{[k-1]},\Ga}$$ is an isomorphism. 

First, we suppose  $y_1$ is identically zero along 
$\rho_{\vt_{[k]}}^{-1}(\tZ^{\dagger\circ}_{\vt_{[k-1]},\Ga})$.
We then set,  
\begin{equation}\label{Gazero-ktah-contained-in-a-vtk} 
\tGa^\zero_\fV=\{y_1 \} \cup \overline{\Ga}^\zero_\fV
\end{equation}
where $\overline{\Ga}^\zero_\fV$ is as in \eqref{Gazero-ktah-contained-in-bar-vtk}.
In this case, we let
\begin{equation}\label{ZktahGa-with-xi-vtk}
\tZ_{\vt_{[k]},\Ga}=\rho_{\vt_{[k]}}^{-1}(\tZ_{\vt_{[k-1]},\Ga})\cap D_{y_1}
\end{equation}
scheme-theoretically, where $D_{y_1}$ is the closure of $(y_1=0)$ in $\tsR_{\vt_{[k]}}$.
We remark here that $D_{y_1}$ does not depend on the choice of the chart $\fV$.

Next, suppose  $y_1$ is not identically zero along 
$\rho_{\vt_{[k]}}^{-1}(\tZ^{\dagger\circ}_{\vt_{[k-1]},\Ga})$.
We then set,  
\begin{equation}\label{Gazero-ktah-contained-in-a-vtk'} 
\tGa^\zero_\fV= \overline{\Ga}^\zero_\fV
\end{equation}
where $\overline{\Ga}^\zero_\fV$ is as in \eqref{Gazero-ktah-contained-in-bar-vtk}.
In this case, we let
\begin{equation}\label{ZktahGa-without-xi-vtk}
\tZ_{\vt_{[k]},\Ga}=\rho_{\vt_{[k]}}^{-1}(\tZ_{\vt_{[k-1]},\Ga}).
\end{equation}
We always set (under the condition $(\star a)$)
\begin{equation}\label{Gaone-ktah-contained-in-a-vtk} \tGa^\one_{\fV}= \overline{\Ga}^\one_{\fV}
\end{equation}
where $\overline{\Ga}^\one_{\fV}$ is as in  \eqref{Gaone-ktah-contained-in-bar-vtk}.

In each case, by construction, we have 
$$\rho_{\vt_{[k]}}^{-1}(\tZ^{\dagger\circ}_{\vt_{[k-1]},\Ga})  
\subset \tZ_{\vt_{[k]}\Ga},$$ and we let $\tZ^\dagger_{\vt_{[k]},\Ga}$ be the closure of 
$\rho_{(\vt_{[k]}}^{-1}(\tZ^{\dagger\circ}_{\vt_{[k-1]},\Ga}) $ in $\tZ_{\vt_{[k]},\Ga}$.
It is an irreducible component of $\tZ_{\vt_{[k]},\Ga}$ because
$\tZ^\dagger_{\vt_{[k]},\Ga}$ is closed in $\tZ_{\vt_{[k]},\Ga}$ and contains
the Zariski open subset $\rho_{\vt_{[k]}}^{-1}(\tZ^{\dagger\circ}_{\vt_{[k-1]},\Ga}) $
of $\tZ_{\vt_{[k]},\Ga}$.
Then, we obtain that the composition
 $$\tZ^\dagger_{\vt_{[k]},\Ga} \lra \tZ^\dagger_{\vt_{[k-1]},\Ga}
 \lra Z_\Ga$$ is birational.
 This proves Lemma \ref{vt-transform-k} (2)  over $\tsR_{\vt_{[k]}}$.


In each case of the above (i.e., \eqref{Gazero-ktah-contained-in-a-vtk} and
\eqref{Gazero-ktah-contained-in-a-vtk'}), 
  by the paragraph of \eqref{vt-pre-image-defined-by},
one sees that $\tZ_{\vt_{[k]},\Ga}\cap \fV$,
 as a closed subscheme of $\fV$, is defined by the equations as stated in the Lemma.  
 This proves Lemma \ref{vt-transform-k} (1)  over $\tsR_{\vt_{[k]}}$.

It remains to prove Lemma \ref{vt-transform-k} (3) over $\tsR_{\vt_{[k]}}$.

Fix any $y \in \var_\fV$, it suffices to show that
if $\tZ^\dagger_{\vt_{[k]},\Ga}\cap \fV
\subset (y=0)$, then $\tZ_{(\vt_{[k]},\Ga}\cap \fV
\subset (y=0)$.  
If $y \ne \zeta_\fV,  y_1$, then $y$ is the proper transform of some variable $y' \in \var_{\fV'}$.
Hence, by taking the images under $\rho_{\vt_{[k]}}$, we have $\tZ^\dagger_{\vt_{[k-1]},\Ga}\cap \fV'
\subset (y'=0)$; by Lemma \ref{vt-transform-k} (3) in $\tsR_{\vt_{[k-1]}}$, we obtain
$\tZ_{\vt_{[k-1]},\Ga}\cap \fV'
\subset (y'=0)$,  thus $y' \in \tGa^\zero_{\fV'}$ by the maximality of
the subset $\tGa^\zero_{\fV'}$. Therefore, 
$\tZ_{\vt_{[k]},\Ga}\cap \fV \subset ( y=0)$,  
by (the already-proved) Lemma \ref{vt-transform-k} (1) for $\tsR_{\vt_{[k]}}$
(cf. \eqref{Gazero-ktah-contained-in-bar-vtk} and \eqref{Gazero-ktah-contained-in-a-vtk} 
or \eqref{Gazero-ktah-contained-in-a-vtk'}). 
Next, suppose $y = y_1$ (if it occurs). Then, by construction, 
$\tZ_{\vt_{[k]},\Ga}\cap \fV \subset (y=0)$. 
Finally, we let $y =\zeta_\fV$. 
 Again, by construction, $\tZ_{\vt_{[k]},\Ga}\cap \fV \subset ( \zeta_\fV=0)$. 

As earlier, the last statement Lemma \ref{vt-transform-k} (3)  follows from the above.

\smallskip\noindent
{\it  Proof of Lemma \ref{vt-transform-k} over $\tsR_{\vt_{[k]}}$ under the condition $(\star b)$.}
\smallskip

Under the condition $(\star b)$, we have that $$ 
\rho_{\vt_{[k]}}^{-1}(\tZ^\dagger_{\vt_{[k-1]},\Ga})
 \lra \tZ^\dagger_{\vt_{[k-1]},\Ga}$$
can be canonically identified with the trivial $\PP_{[\xi_0,\xi_1]}$-bundle:
$$\rho_{\vt_{[k]}}^{-1}(\tZ^\dagger_{(\vt_{[k-1]},\Ga}) 
= \tZ^\dagger_{\vt_{[k-1]},\Ga}
\times \PP_{[\xi_0,\xi_1]}.$$
In this case, we define 
$$\tZ_{\vt_{[k]},\Ga}= \rho_{\vt_{[k]}}^{-1}(\tZ_{\vt_{[k-1]},\Ga})\cap ((\xi_0,\xi_1)= (1,1)),$$
 $$\tZ^\dagger_{\vt_{[k]},\Ga}= 
 \rho_{\vt_{[k]}}^{-1}(\tZ^\dagger_{\vt_{[k-1]},\Ga})
 \cap ((\xi_0,\xi_1)= (1,1)),$$
 both scheme-theoretically.
 The induced morphism $\tZ^\dagger_{\vt_{[k]},\Ga} \lra \tZ^\dagger_{\vt_{[k-1]},\Ga}$
 is an isomorphism. 
 Again, one sees that $\tZ^\dagger_{\vt_{[k]},\Ga}$ is 
 an irreducible component of $\tZ_{\vt_{[k]},\Ga}$. Therefore,
 $$\tZ^\dagger_{\vt_{[k]},\Ga} \lra
 \tZ^\dagger_{\vt_{[k-1]},\Ga} \lra  Z_\Ga$$ is  birational.
 This proves Lemma \ref{wp/ell-transform-ktauh} (2) over $\tsR_{\vt_{[k]}}$.

Further, under the condition $(\star b)$, we set 
 \begin{eqnarray}
 \tGa^\zero_\fV= \overline{\Ga}^\zero_\fV,  \;\;\;\;
\tGa^\one_\fV=\{y_1\}\cup \overline{\Ga}^\one_{\fV}. \nonumber \end{eqnarray}
Then, again, by the paragraph of \eqref{vt-pre-image-defined-by},
one sees that $\tZ_{\vt_{[k]},\Ga}\cap \fV$,
 as a closed subscheme of $\fV$, is defined by the equations as stated in the Lemma. 
 This proves Lemma \ref{vt-transform-k} (1) over $\tsR_{\vt_{[k]}}$.

 It remains to prove Lemma \ref{vt-transform-k} (3) in over $\tsR_{\vt_{[k]}}$.

Fix any $y \in \var_\fV$, it suffices to show that
if $\tZ^\dagger_{\vt_{[k]}, \Ga}\cap \fV
\subset (y=0)$, then $\tZ_{\vt_{[k]},\Ga}\cap \fV
\subset (y=0)$.  By construction, $y \ne y_1$.
Then, the corresponding proof of Lemma \ref{vt-transform-k} (3) for $\tsR_{\vt_{[k]}}$
under the condition $(\star a)$ goes through here without change.
 The last statement Lemma \ref{vt-transform-k} (3)  follows from the above.
This proves Lemma \ref{vt-transform-k} (3) in $\tsR_{\vt_{[k]}}$
under the condition $(\star b)$.

This completes the proof of Lemma \ref{vt-transform-k}.
\end{proof}

 We call $\tZ_{\vt_{[k]},\Ga}$ the $\vt$-transform of $Z_\Ga$ 
 in $\tsV_{\vt_{[k]}}$ for any $k \in [\up]$.

  We need the final case of Lemma 
\ref{vt-transform-k}. We set
$$\tZ_{\vt, \Ga}:=\tZ_{\vt_{[\up]},\Ga},
 \;\; \tZ^\dagger_{\vt, \Ga}:=\tZ^\dagger_{\vt_{[\up]},\Ga}.$$

\subsection{$\wp$- and $\ell$-transforms of  
$\Ga$-schemes in   $\tsV_{({\wp}_{(k\tau)}\fr_\mu\fs_h)}$ and in $(\ell_k)$}
\label{subsection:wp-transform-vsk}  $\ $

We now construct the ${\wp}$-transform of $Z_\Ga$ in $\tsV_{({\wp}_{(k\tau)}\fr_\mu\fs_h)}
\subset \tsR_{({\wp}_{(k\tau)}\fr_\mu\fs_h)}$. 
Here, as in Proposition \ref{meaning-of-var-wp/ell},  
we assume that the last of
$\tsR_{({\wp}_{(k\tau)}\fr_\mu\fs_h)}$ is $\tsR_{\ell_k}$.


\begin{lemma}\label{wp/ell-transform-ktauh} 
 Fix any subset $\Ga$ of $\rU_\um$.  Assume that $Z_\Ga$ is integral. 

Consider $(k\tau) \mu h \in  \Index_{\Phi_k} \sqcup \{\ell_k\}$.

Then, we have the following:
\begin{itemize}
\item  there exists a closed subscheme $\tZ_{({\wp}_{(k\tau)}\fr_\mu\fs_h),\Ga}$ of
$\tsV_{({\wp}_{(k\tau)}\fr_\mu\fs_h)}$ with an induced morphism  
$\tZ_{({\wp}_{(k\tau)}\fr_\mu\fs_h),\Ga} 
\to Z_\Ga$;
\item   $\tZ_{({\wp}_{(k\tau)}\fr_\mu\fs_h),\Ga}$ comes equipped with an irreducible component  
$\tZ^\dagger_{({\wp}_{(k\tau)}\fr_\mu\fs_h),\Ga}$ with the induced morphism 
$\tZ^\dagger_{({\wp}_{(k\tau)}\fr_\mu\fs_h),\Ga}  
\to Z_\Ga$;
\item  for any standard chart $\fV$ of $\tsR_{({\wp}_{(k\tau)}\fr_\mu\fs_h)}$ such that
$\tZ_{({\wp}_{(k\tau)}\fr_\mu\fs_h),\Ga} \cap \fV \ne \emptyset$, there come equipped with
two subsets, possibly empty,
$$ \tGa^\zero_{\fV} \; \subset \; \var^\vee_\fV, \;\;\;
\tGa^\one_{\fV} \; \subset \;  \var_\fV.$$ 
\end{itemize}

Further,  consider any given chart $\fV$ of $\tsR_{({\wp}_{(k\tau)}\fr_\mu\fs_h)}$ with
$\tZ_{({\wp}_{(k\tau)}\fr_\mu\fs_h),\Ga} \cap \fV \ne \emptyset$. 
Then,  the following hold:
\begin{enumerate}
\item the scheme $\tZ_{({\wp}_{(k\tau)}\fr_\mu\fs_h),\Ga} \cap \fV$,
 as a closed subscheme of the chart $\fV$,
is defined by the following relations
\begin{eqnarray} 
\;\;\;\;\; y , \; \; \; y \in  \tGa^\zero_\fV , \label{Ga-rel-wp-ktauh}\\
\;\;\;  y -1, \; \; \; y \in  \tGa^\one_\fV,  \nonumber \\
\cB_\fV^\mn, \; \cB^\q_\fV, \; L_{\fV, \sF_\um}; \nonumber
\end{eqnarray}
 further, we take $\tGa^\zero_\fV \subset \var_\fV$
 to be the maximal subset (under inclusion)
among all those subsets that satisfy the above.
\item the induced morphism $\tZ^\dagger_{({\wp}_{(k\tau)}\fr_\mu\fs_h),\Ga} \to Z_\Ga$ is birational; 
\item for any variable $y \in \var_\fV$, $\tZ^\dagger_{({\wp}_{(k\tau)}\fr_\mu\fs_h),\Ga} \cap \fV \subset (y=0)$ if and only 
if  $\tZ_{({\wp}_{(k\tau)}\fr_\mu\fs_h),\Ga} \cap \fV \subset (y=0)$. Consequently, 
 $\tZ^\dagger_{({\wp}_{(k\tau)}\fr_\mu\fs_h),\Ga}\cap \fV \subset \tZ_{\phi_{(k\tau)\mu(h+1)}}\cap \fV$ if and only 
if  $\tZ_{({\wp}_{(k\tau)}\fr_\mu\fs_h),\Ga} \cap \fV \subset \tZ_{\phi_{(k\tau)\mu(h+1)}}\cap \fV$ 
where $\tZ_{\phi_{(k\tau)\mu(h+1)}}$ is the proper transform of the $\wp$-center 
$Z_{\phi_{(k\tau)\mu(h+1)}}$ in  $\tsR_{({\wp}_{(k\tau)}\fr_\mu\fs_h)}$. 
 \end{enumerate}
\end{lemma}  
\begin{proof} We prove by induction on 
 $(k\tau) \mu h \in  \{(11)10)\} \sqcup \Index_{\Phi_k} \sqcup \{\ell_k\}$ (cf. \eqref{indexing-Phi}).


The initial case is $(11)10$. In this case, we have
 $$\tsR_{({\wp}_{(11)}\fr_1\fs_0)}:=\tsR_{\vt}, \;\;
 \tsV_{({\wp}_{(11)}\fr_1\fs_0)}:=\tsV_{\vt} , \;\;
  \tZ_{({\wp}_{(11)}\fr_1\fs_0),\Ga}:=\tZ_{\vt,\Ga}, \;\;
   \tZ^\dagger_{({\wp}_{(11)}\fr_1\fs_0),\Ga}:=\tZ^\dagger_{\vt,\Ga}.$$
Then, in this case,  Lemma \ref{wp/ell-transform-ktauh}
 is Lemma \ref{vt-transform-k} for $k=\up$.

We now suppose that Lemma \ref{wp/ell-transform-ktauh} holds for $(k\tau)\mu(h-1)$
for some $(k\tau)\mu h \in \Index_{\Phi_k}$.

We treat exclusively $\wp$-transforms of $Z_\Ga$ first, in other words, we assume that 
$\tsR_{({\wp}_{(k\tau)}\fr_\mu\fs_{h-1})}\ne \tsR_{\wp_k}$. We treat 
$\ell$-transforms of $Z_\Ga$ in the end.

We then consider the case of $(k\tau)\mu h$.

We let 
\begin{equation}\label{rho-wpktaumuh}
 \rho_{(\wp_{(k\tau)}\fr_\mu\fs_h)}: \tsV_{(\wp_{(k\tau)}\fr_\mu\fs_h)} \lra
\tsV_{(\wp_{(k\tau)}\fr_\mu\fs_{h-1})}\end{equation}
be the morphism induced from
$ \pi_{(\wp_{(k\tau)}\fr_\mu\fs_h)}: \tsR_{(\wp_{(k\tau)}\fr_\mu\fs_h)} \to
\tsR_{(\wp_{(k\tau)}\fr_\mu\fs_{h-1})}$.

Suppose that $\tZ_{({\wp}_{(k\tau)}\fr_\mu\fs_{h-1}),\Ga}$,  or equivalently 
$\tZ^\dagger_{({\wp}_{(k\tau)}\fr_\mu\fs_{h-1}),\Ga}$,
by Lemma \ref{wp/ell-transform-ktauh}  (3) in $({\wp}_{(k\tau)}\fr_\mu\fs_{(h-1)})$,
 is not contained in 
$Z'_{\phi_{(k\tau)\mu h}}$ where $Z'_{\phi_{(k\tau)\mu h}}$ is the proper transform 
in $\tsR_{({\wp}_{(k\tau)}\fr_\mu\fs_{h-1})}$
 of the ${\wp}$-center $Z_{\phi_{(k\tau)\mu h}}$(of $\tsR_{({\wp}_{(k\tau)}\fr_{\mu-1})}$).

We then let $\tZ_{({\wp}_{(k\tau)}\fr_\mu\fs_h),\Ga}$ (resp. $\tZ^\dagger_{({\wp}_{(k\tau)}\fr_\mu\fs_h),\Ga}$) be the
proper transform of $\tZ_{({\wp}_{(k\tau)}\fr_\mu\fs_{h-1}),\Ga}$
 (resp. $\tZ^\dagger_{({\wp}_{(k\tau)}\fr_\mu\fs_{h-1}),\Ga}$)
in $\sV_{({\wp}_{(k\tau)}\fr_\mu\fs_h)}$.
As $\tZ^\dagger_{({\wp}_{(k\tau)}\fr_\mu\fs_h),\Ga}$ is closed in $\tZ_{({\wp}_{(k\tau)}\fr_\mu\fs_h),\Ga}$
and contains a Zariski open subset of $\tZ_{({\wp}_{(k\tau)}\fr_\mu\fs_h),\Ga}$, it is an irreducible
component of $\tZ_{({\wp}_{(k\tau)}\fr_\mu\fs_h),\Ga}$.

Further, consider any standard chart $\fV$ of $\tsR_{({\wp}_{(k\tau)}\fr_\mu\fs_h)}$,
lying over a unique standard chart $\fV'$ of $\tsR_{({\wp}_{(k\tau)}\fr_\mu\fs_{h-1})}$,
such that $\tZ_{({\wp}_{(k\tau)}\fr_\mu\fs_h),\Ga}\cap \fV \ne \emptyset$.
We set 
$$\tGa^\zero_\fV=\{y_\fV \mid y_\fV  
\hbox{ is the proper transform of some $y_{\fV'} \in \tGa^\zero_{\fV'}$}\};$$
$$\tGa^\one_\fV=\{y_\fV \mid y_\fV  
\hbox{ is the proper transform of some $y_{\fV'} \in \tGa^\one_{\fV'}$}\}.$$

We now prove  Lemma \ref{wp/ell-transform-ktauh}  (1), (2) and (3) in $({\wp}_{(k\tau)}\fr_\mu\fs_h)$.

Lemma \ref{wp/ell-transform-ktauh}  (1) in $({\wp}_{(k\tau)}\fr_\mu\fs_h)$ 
follows  from Lemma \ref{wp/ell-transform-ktauh} (1) in $({\wp}_{(k\tau)}\fr_\mu\fs_{h-1})$
because $\tZ_{({\wp}_{(k\tau)}\fr_\mu\fs_h),\Ga}$  is the proper transform of 
$\tZ_{({\wp}_{(k\tau)}\fr_\mu\fs_{h-1}),\Ga}$. 

By construction, we have that 
$\tZ^\dagger_{({\wp}_{(k\tau)}\fr_\mu\fs_h),\Ga} \to \tZ^\dagger_{({\wp}_{(k\tau)}\fr_\mu\fs_{h-1}),\Ga} \to Z_\Ga$
is birational.   
This proves Lemma \ref{wp/ell-transform-ktauh}  (2) in $({\wp}_{(k\tau)}\fr_\mu\fs_h)$.

To show Lemma \ref{wp/ell-transform-ktauh}  (3) in $({\wp}_{(k\tau)}\fr_\mu\fs_h)$, 
we fix any $y \in \var_\fV$. It suffices to show that
if $\tZ^\dagger_{({\wp}_{(k\tau)}\fr_\mu\fs_h),\Ga}\cap \fV
\subset (y=0)$, then $\tZ_{({\wp}_{(k\tau)}\fr_\mu\fs_h),\Ga}\cap \fV
\subset (y=0)$.  By construction, $y \ne \zeta_{\fV,(k\tau)\mu h}$, the exceptional variable
in $\var_\fV$ corresponding to the $\wp$-set $\phi_{(k\tau)\mu h}$. Hence, $y$ is the proper transform
of some $y' \in \var_{\fV'}$. Then, by taking the images under the morphism
$\rho_{(\wp_{(k\tau)}\fr_\mu\fs_h)}$ of \eqref{rho-wpktaumuh}, we obtain
$\tZ^\dagger_{({\wp}_{(k\tau)}\fr_\mu\fs_{h-1}),\Ga}\cap \fV'
\subset (y'=0)$,  hence, $\tZ_{({\wp}_{(k\tau)}\fr_\mu\fs_{h-1}),\Ga}\cap \fV'
\subset (y'=0)$ by the inductive assumption.
 Then,  as $\tZ_{({\wp}_{(k\tau)}\fr_\mu\fs_{h}),\Ga}$ 
 is the proper transform of $\tZ_{({\wp}_{(k\tau)}\fr_\mu\fs_{h-1}),\Ga}$,
 we obtain $\tZ_{({\wp}_{(k\tau)}\fs_{h}),\Ga}\cap \fV \subset (y=0)$.
  The last statement  of Lemma \ref{wp/ell-transform-ktauh}  (3)  follows from the above because
 $\tZ_{\phi_{(k\tau)\mu(h+1)}}\cap \fV=(y_0=y_1=0)$ for some $y_0, y_1 \in \var_\fV$.

\smallskip

We now suppose that $\tZ_{({\wp}_{(k\tau)}\fr_\mu\fs_{h-1}),\Ga}$, or equivalently 
$\tZ^\dagger_{({\wp}_{(k\tau)}\fr_\mu\fs_{h-1}),\Ga}$,
 by  Lemma \ref{wp/ell-transform-ktauh}  (3) in $({\wp}_{(k\tau)}\fr_\mu\fs_{(h-1)})$,
 is contained in the proper transform  $Z'_{\phi_{(k\tau)\mu h}}$ 
 of $Z_{\phi_{(k\tau)\mu h}}$. 

Consider  any standard chart $\fV$ of $\tsR_{({\wp}_{(k\tau)}\fr_\mu\fs_h)}$,
lying over a unique standard chart $\fV'$ of $\tsR_{({\wp}_{(k\tau)}\fr_\mu\fs_{h-1})}$,
such that $\tZ_{({\wp}_{(k\tau)}\fr_\mu\fs_{h-1}),\Ga}\cap \fV' \ne \emptyset$.

We let $\phi'_{(k\tau)\mu h}$
 be the proper transform in  the chart $\fV'$ of the $\wp$-set $\phi_{(k\tau)\mu h}$. Then,  
 $\phi'_{(k\tau)\mu h}$ consists of two  variables such that
 $$\psi'_{(k\tau)\mu h}=\{y'_0,  y'_1\} \subset \var_{\fV'}.$$ 
In addition, we let $\zeta_\fV \in \var_\fV$  be 
 such that  $E_{({\wp}_{(k\tau)}\fr_\mu\fs_{h})} \cap \fV=(\zeta_\fV =0)$ 
 where $E_{({\wp}_{(k\tau)}\fr_\mu\fs_{h})}$
 is the exceptional divisor
of the blowup $\tsR_{({\wp}_{(k\tau)}\fr_\mu\fs_{h})} \to \tsR_{({\wp}_{(k\tau)}\fr_\mu\fs_{h-1})}$. 
  Without loss of generality, we may assume that $y'_0$ 
corresponds to the exceptional variable $\zeta_\fV$ on the chart $\fV$.
 We then let $y_1 \in \var_\fV$ be 
 the proper transform of $y'_1$. 

Now, we observe that 
 $$\phi'_{(k\tau)\mu h}\subset \tGa^\zero_{\fV'}$$
because $\tZ_{({\wp}_{(k\tau)}\fr_\mu\fs_{h-1}),\Ga}$ is contained in the proper transform
$Z'_{\phi_{(k\tau)\mu h}}$.

We set, 
\begin{equation}\label{Gazero-wp-contained-in-bar} 
\overline{\Ga}^\zero_\fV= \{\zeta_\fV, \; y_\fV \mid y_\fV  
\hbox{ is the proper transform of some $ y_{\fV'} \in \tGa^\zero_{\fV'} \-\phi'_{(k\tau)\mu h}$}\},
\end{equation}
\begin{equation}\label{Gaone-wp-contained-in-bar}
 \overline{\Ga}^\one_{\fV}=\{\ y_\fV \mid y_\fV  
\hbox{ is the proper transform of some $y_{\fV'} \in \tGa^\one_{\fV'}$} \}.
\end{equation}

Consider the scheme-theoretic pre-image  
$\rho_{({\wp}_{(k\tau)}\fr_\mu\fs_h)}^{-1}(\tZ_{({\wp}_{(k\tau)}\fr_\mu\fs_{h-1}),\Ga})$.

 Note that scheme-theoretically, we have
$$\rho_{({\wp}_{(k\tau)}\fr_\mu\fs_h)}^{-1}(\tZ_{({\wp}_{(k\tau)}\fr_\mu\fs_{h-1}),\Ga}) \cap \fV=
\pi_{({\wp}_{(k\tau)}\fr_\mu\fs_h)}^{-1}(\tZ_{({\wp}_{(k\tau)}\fr_\mu\fs_{h-1}),\Ga})
 \cap \tsV_{({\wp}_{(k\tau)}\fr_\mu\fs_h)} \cap \fV .$$
Applying Lemma \ref{wp/ell-transform-ktauh} (1)  in $({\wp}_{(k\tau)}\fr_\mu\fs_{h-1})$ to 
$\tZ_{({\wp}_{(k\tau)}\fr_\mu\fs_{h-1}),\Ga}$
and $\rho_{({\wp}_{(k\tau)}\fr_\mu\fs_h)}^{-1}(\tZ_{({\wp}_{(k\tau)}\fr_\mu\fs_{h-1}),\Ga}) $, 
and applying Proposition \ref{equas-wp/ell-kmuh} to 
$\tsV_{({\wp}_{(k\tau)}\fr_\mu\fs_h)} \cap \fV $, we obtain that  the pre-image 
$\rho_{({\wp}_{(k\tau)}\fr_\mu\fs_h)}^{-1}(\tZ_{({\wp}_{(k\tau)}\fr_\mu\fs_{h-1}),\Ga}) \cap \fV$,
as a closed subscheme of $\fV$,  is defined by 
\begin{equation}\label{wp-pre-image-defined-by}
y_\fV \in \overline{\Ga}^\zero_\fV; \;\;  y_\fV-1, \; y_\fV \in \overline{\Ga}^\one_{\fV};\;\;
\cB^\mn_\fV; \;\; \cB^\q_{\fV}; \;\; L_{\fV,  \sF_\um}.
\end{equation} 
(Observe here that $\zeta_\fV \in  \overline{\Ga}^\zero_\fV$.)

Thus, by setting 
$$\hbox{ $y_\fV=0$ for all $y_\fV \in \overline{\Ga}^\zero_\fV$ and
 $y_\fV=1$ for all $y_\fV \in \overline{\Ga}^\one_{\fV}$} $$ in $\cB^\mn_\fV, 
 \cB^\q_{\fV}, L_{\fV,  \sF_\um}$ of the above,
 we obtain 
 \begin{equation}\label{vs-lin-xi-pre} \tilde{\cB}^\mn_\fV, 
  \tilde\cB^\q_{\fV}, \tilde{L}_{\fV,   \sF_\um}.
 \end{equation}
 Note that for any $\bF \in  \sF_\um$,
  if a term of $L_{\fV, F}$ contains $y_1 \in \var_\fV$, 
   then it contains $\zeta_\fV y_1$, hence $\tilde L_{\fV, F}$ does not contain $y_1$.
  We keep those equations of \eqref{vs-lin-xi-pre} such that they contain the variable $y_1 \in \var_\fV$ 
   and obtain
 \begin{equation}\label{vs-lin-xi}
 \hat{\cB}^\mn_\fV, 
 \hat\cB^\q_{\fV}, 
 \end{equation}
 viewed as a  system of equations in $y_1$.
Then, by Proposition \ref{equas-wp/ell-kmuh}  (1), 2), and (4),
\eqref{vs-lin-xi} is  a  {\it linear} system of equations in $y_1$.
(We point out that $y_1$ here can correspond to either $y_0$ or $y_1$ 
as in Proposition \ref{equas-wp/ell-kmuh}.)
Furthermore, one sees that 
 $$(\rho_{({\wp}_{(k\tau)}\fr_\mu\fs_h)}^{-1}(\tZ_{({\wp}_{(k\tau)}\fr_\mu\fs_{h-1}),\Ga})\cap \fV)/
 (\tZ_{({\wp}_{(k\tau)}\fr_\mu\fs_{h-1}),\Ga}\cap \fV')$$
 is defined by the linear system \eqref{vs-lin-xi}.

There are the following two cases for \eqref{vs-lin-xi}:

$(\star a)$  the ranks of the linear system  $\eqref{vs-lin-xi}$ equal one
at  general points of  $\tZ^\dagger_{({\wp}_{(k\tau)}\fr_\mu\fs_{h-1}),\Ga}$.

$(\star b)$  the ranks of the linear system  $\eqref{vs-lin-xi}$ equal zero
at general points of  $\tZ^\dagger_{({\wp}_{(k\tau)}\fr_\mu\fs_{h-1}),\Ga}$, hence at all points
of  $\tZ^\dagger_{({\wp}_{(k\tau)}\fr_\mu\fs_{h-1}),\Ga}$.

\smallskip\noindent
{\it 
Proof of Lemma  \ref{wp/ell-transform-ktauh}   in $({\wp}_{(k\tau)}\fr_\mu \fs_h)$ under the condition $(\star a)$.}
\smallskip

By the condition $(\star a)$,
 there exists a Zariski open subset $\tZ^{\dagger\circ}_{({\wp}_{(k\tau)}\fr_\mu\fs_{h-1}),\Ga}$ 
of $\tZ^\dagger_{({\wp}_{(k\tau)}\fr_\mu\fs_{h-1}),\Ga}$ such that the rank of the linear system
  $\eqref{vs-lin-xi}$ equals one at any point of  $\tZ^{\dagger\circ}_{({\wp}_{(k\tau)}\fr_\mu\fs_{h-1}),\Ga}$. 
By solving  $y_1$ from
 the linear system \eqref{vs-lin-xi} over $\tZ^{\dagger\circ}_{({\wp}_{(k\tau)}\fr_\mu\fs_{h-1}),\Ga}$, we obtain 
that the induced morphism
$$
\rho_{({\wp}_{(k\tau)}\fr_\mu\fs_h)}^{-1}(\tZ^{\dagger\circ}_{({\wp}_{(k\tau)}\fr_\mu\fs_{h-1}),\Ga})  \lra \tZ^\circ_{({\wp}_{(k\tau)}\fr_\mu\fs_{h-1}),\Ga}$$ is an isomorphism. 

Suppose  $y_1$ is identically zero along 
$\rho_{({\wp}_{(k\tau)}\fr_\mu\fs_h)}^{-1}(\tZ^{\dagger\circ}_{({\wp}_{(k\tau)}\fr_\mu\fs_{h-1}),\Ga})$.
We then set,  
\begin{equation}\label{Gazero-ktah-contained-in-a} 
\tGa^\zero_\fV=\{y_1 \} \cup \overline{\Ga}^\zero_\fV
\end{equation}
where $\overline{\Ga}^\zero_\fV$ is as in \eqref{Gazero-wp-contained-in-bar}.
In this case, we let
\begin{equation}\label{ZktahGa-with-xi}
\tZ_{({\wp}_{(k\tau)}\fr_\mu\fs_{h}),\Ga}=\rho_{({\wp}_{(k\tau)}\fr_\mu\fs_h)}^{-1}(\tZ_{({\wp}_{(k\tau)}\fr_\mu\fs_{h-1}),\Ga})\cap D_{y_1}
\end{equation}
scheme-theoretically, where $D_{y_1}$
 is the closure of $(y_1=0)$ in $\tsR_{({\wp}_{(k\tau)}\fr_\mu\fs_h)}$.
We remark here that $D_{y_1}$ does not depend on the choice of the chart $\fV$.

Suppose $y_1$ is not identically zero along 
$\rho_{({\wp}_{(k\tau)}\fr_\mu\fs_h)}^{-1}(\tZ^{\dagger\circ}_{({\wp}_{(k\tau)}\fr_\mu\fs_{h-1}),\Ga})$.
We then set,  
\begin{equation}\label{Gazero-ktah-contained-in-a-2} 
\tGa^\zero_\fV= \overline{\Ga}^\zero_\fV
\end{equation}
where $\overline{\Ga}^\zero_\fV$ is as in \eqref{Gazero-wp-contained-in-bar} .
In this case, we let
\begin{equation}\label{ZktahGa-without-xi}
\tZ_{({\wp}_{(k\tau)}\fr_\mu\fs_{h}),\Ga}=\rho_{({\wp}_{(k\tau)}\fr_\mu\fs_h)}^{-1}(\tZ_{({\wp}_{(k\tau)}\fr_\mu\fs_{h-1}),\Ga}).
\end{equation}
We always set (under the condition $(\star a)$)
\begin{equation}\label{Gaone-ktah-contained-in-a} \tGa^\one_{\fV}= \overline{\Ga}^\one_{\fV}
\end{equation}
where $\overline{\Ga}^\one_{\fV}$ is as in  \eqref{Gaone-wp-contained-in-bar}.

In each case, we have 
$$\rho_{({\wp}_{(k\tau)}\fr_\mu\fs_h)}^{-1}(\tZ^{\dagger\circ}_{({\wp}_{(k\tau)}\fr_\mu\fs_{h-1}),\Ga})  
\subset \tZ_{({\wp}_{(k\tau)}\fr_\mu\fs_h),\Ga},$$ and we let $\tZ^\dagger_{({\wp}_{(k\tau)}\fr_\mu\fs_h),\Ga}$ be the closure of 
$\rho_{({\wp}_{(k\tau)}\fr_\mu\fs_h)}^{-1}(\tZ^{\dagger\circ}_{({\wp}_{(k\tau)}\fr_\mu\fs_{h-1}),\Ga}) $ in $\tZ_{({\wp}_{(k\tau)}\fr_\mu\fs_h),\Ga}$.
It is an irreducible component of $\tZ_{({\wp}_{(k\tau)}\fr_\mu\fs_h),\Ga}$ because
$\tZ^\dagger_{({\wp}_{(k\tau)}\fr_\mu\fs_h),\Ga}$ is closed in $\tZ_{({\wp}_{(k\tau)}\fr_\mu\fs_h),\Ga}$ and contains
the Zariski open subset $\rho_{({\wp}_{(k\tau)}\fr_\mu\fs_h)}^{-1}(\tZ^{\dagger\circ}_{({\wp}_{(k\tau)}\fr_\mu\fs_{h-1}),\Ga}) $
of $\tZ_{({\wp}_{(k\tau)}\fr_\mu\fs_h),\Ga}$.
Then, it follows that the composition
 $$\tZ^\dagger_{({\wp}_{(k\tau)}\fr_\mu\fs_h),\Ga} \lra \tZ^\dagger_{({\wp}_{(k\tau)}\fr_\mu\fs_{h-1}),\Ga}
 \lra Z_\Ga$$ is birational.  
 This proves Lemma \ref{wp/ell-transform-ktauh} (2) in $({\wp}_{(k\tau)}\fs_h)$.


In each case of \eqref{Gazero-ktah-contained-in-a} and \eqref{Gazero-ktah-contained-in-a-2}, 
by the paragraph of \eqref{wp-pre-image-defined-by},
we have that $\tZ_{({\wp}_{(k\tau)}\fr_\mu\fs_h),\Ga}\cap \fV$
 as a closed subscheme of $\fV$ is defined by the equations as stated in the Lemma. 
 This proves Lemma \ref{wp/ell-transform-ktauh} (1) in $({\wp}_{(k\tau)}\fr_\mu\fs_h)$.

It remains to prove Lemma \ref{wp/ell-transform-ktauh} (3) in $({\wp}_{(k\tau)}\fr_\mu\fs_h)$.

Fix any $y \in \var_\fV$, it suffices to show that
if $\tZ^\dagger_{({\wp}_{(k\tau)}\fr_\mu\fs_{h}),\Ga}\cap \fV
\subset (y=0)$, then $\tZ_{({\wp}_{(k\tau)}\fr_\mu\fs_{h}),\Ga}\cap \fV
\subset (y=0)$.  
If $y \ne \zeta_\fV,  y_1$, then $y$ is the proper transform of some variable $y' \in \var_{\fV'}$.
Hence, by taking the images under $\rho_{({\wp}_{(k\tau)}\fr_\mu\fs_h)}$,
 we obtain $\tZ^\dagger_{({\wp}_{(k\tau)}\fr_\mu\fs_{h-1}),\Ga}\cap \fV'
\subset (y'=0)$, and then,  by Lemma \ref{wp/ell-transform-ktauh} (3) in (${\wp}_{(k\tau)}\fr_\mu\fs_{h-1}$),
$\tZ_{({\wp}_{(k\tau)}\fr_\mu\fs_{h-1}),\Ga}\cap \fV'
\subset (y'=0)$,  thus $y' \in \tGa^\zero_{\fV'}$ by the maximality of $\tGa^\zero_{\fV'}$. 
 Therefore, 
$\tZ_{({\wp}_{(k\tau)}\fr_\mu\fs_{h}),\Ga}\cap \fV \subset ( y=0)$,  
by (the already-proved) Lemma \ref{wp/ell-transform-ktauh} (1) in $(\wp_{(k\tau)}\fr_\mu\fs_h)$.
Next, suppose $y = y_1$ (if it occurs). Then, by construction,
 $\tZ_{({\wp}_{(k\tau)}\fs_{h}),\Ga}\cap \fV \subset (y=0)$. 
Finally, we let $y =\zeta_\fV$. 
 Again, by construction, $\tZ_{({\wp}_{(k\tau)}\fr_\mu\fs_{h}),\Ga}\cap \fV \subset (\zeta_\fV=0)$. 
As in the previous case, the last statement  of Lemma \ref{wp/ell-transform-ktauh}  (3)  follows from the above.

\smallskip\noindent
{\it  Proof of Lemma \ref{wp/ell-transform-ktauh} in $({\wp}_{(k\tau)}\fr_\mu\fs_h)$
 under the condition $(\star b)$.}
\smallskip

Under the condition $(\star b)$, we have that $$ 
\rho_{({\wp}_{(k\tau)}\fr_\mu\fs_h)}^{-1}(\tZ^\dagger_{({\wp}_{(k\tau)}\fr_\mu\fs_{h-1}),\Ga})
 \lra \tZ^\dagger_{({\wp}_{(k\tau)}\fr_\mu\fs_{h-1}),\Ga}$$
can be canonically identified with the trivial $\PP_{[\xi_0,\xi_1]}$-bundle:
$$\rho_{({\wp}_{(k\tau)}\fr_\mu\fs_h)}^{-1}(\tZ^\dagger_{({\wp}_{(k\tau)}\fr_\mu\fs_{h-1}),\Ga}) 
= \tZ^\dagger_{({\wp}_{(k\tau)}\fr_\mu\fs_{h-1}),\Ga}
\times \PP_{[\xi_0,\xi_1]}.$$
In this case, we let $\bp=[1,1] \in \PP_{[\xi_0,\xi_1]}$, and define 
$$\tZ_{({\wp}_{(k\tau)}\fr_\mu\fs_h),\Ga} := 
\rho_{({\wp}_{(k\tau)}\fr_\mu\fs_h)}^{-1}(\tZ_{({\wp}_{(k\tau)}\fr_\mu\fs_{h-1}),\Ga}) 
\times_{\tZ_{({\wp}_{(k\tau)}\fr_\mu\fs_{h-1}),\Ga}} \bp$$
 $$\tZ^\dagger_{({\wp}_{(k\tau)}\fr_\mu\fs_h),\Ga}:= 
 \rho_{({\wp}_{(k\tau)}\fr_\mu\fs_h)}^{-1}(\tZ^\dagger_{({\wp}_{(k\tau)}\fr_\mu\fs_{h-1}),\Ga})
 \times_{\tZ^\dagger_{({\wp}_{(k\tau)}\fr_\mu\fs_{h-1}),\Ga}} \bp.$$
 
 The induced morphism $\tZ^\dagger_{({\wp}_{(k\tau)}\fr_\mu\fs_h),\Ga} \lra \tZ^\dagger_{({\wp}_{(k\tau)}\fr_\mu\fs_{h-1}),\Ga}$
 is an isomorphism. 
 Again, one sees that $\tZ^\dagger_{({\wp}_{(k\tau)}\fr_\mu\fs_h),\Ga}$ is 
 an irreducible component of $\tZ_{({\wp}_{(k\tau)}\fr_\mu\fs_h),\Ga}$. Therefore,
 $$\tZ^\dagger_{({\wp}_{(k\tau)}\fr_\mu\fs_h),\Ga} \lra
 \tZ^\dagger_{({\wp}_{(k\tau)}\fr_\mu\fs_{h-1}),\Ga} \lra  Z_\Ga$$ is birational.
 This proves Lemma \ref{wp/ell-transform-ktauh} (2) in $({\wp}_{(k\tau)}\fr_\mu\fs_h)$.

Further, under the condition $(\star b)$, we set 
 \begin{eqnarray}
 \tGa^\zero_\fV= \overline{\Ga}^\zero_\fV,  \;\;\;\;
\tGa^\one_\fV=\{y_1\} \cup \overline{\Ga}^\one_{\fV}. \nonumber \end{eqnarray}

Then, by 
the paragraph of \eqref{wp-pre-image-defined-by},
we have that $\tZ_{({\wp}_{(k\tau)}\fr_\mu\fs_h),\Ga}\cap \fV$,
 as a closed subscheme of $\fV$, is defined by the equations as stated in the Lemma. 
 This proves Lemma \ref{wp/ell-transform-ktauh} (1) in $({\wp}_{(k\tau)}\fr_\mu\fs_h)$.

 It remains to prove Lemma \ref{wp/ell-transform-ktauh} (3) in $({\wp}_{(k\tau)}\fr_\mu\fs_h)$.

Fix any $y \in \var_\fV$, it suffices to show that
if $\tZ^\dagger_{({\wp}_{(k\tau)}\fs_{h}),\Ga}\cap \fV
\subset (y=0)$, then $\tZ_{({\wp}_{(k\tau)}\fs_{h}),\Ga}\cap \fV
\subset (y=0)$.  By construction, $y \ne y_1$.
Then, the corresponding proof of Lemma \ref{wp/ell-transform-ktauh} (3) in $({\wp}_{(k\tau)}\fr_\mu\fs_h)$
under the condition $(\star a)$ goes through here without change. As earlier, 
the last statement  of Lemma \ref{wp/ell-transform-ktauh}  (3)  follows from the above. 
This proves Lemma \ref{wp/ell-transform-ktauh} (3) in $({\wp}_{(k\tau)}\fr_\mu\fs_h)$
 under the condition $(\star b)$.
 
 \medskip
 
 \centerline{$\bullet$ $\ell$-transform}
 
 \medskip
 
 Now, we assume that 
 $\tsR_{({\wp}_{(k\tau)}\fr_\mu\fs_{h-1})}=\tsR_{\wp_k}$
 $\tsR_{({\wp}_{(k\tau)}\fr_\mu\fs_h)}=\tsR_{\ell_k}$.
 By Corollary \ref{ell-isom},
 $\rho_{\ell_k, \wp_k}: \tsR_{\ell_k} \to \tsR_{\wp_k}$ is an isomorphism.
 In this case, we let
 $$ \tZ_{\ell_k, \Ga} =\rho_{\ell_k, \wp_k}^{-1}(\tZ_{\wp_k, \Ga}),$$
 $$ \tZ^\dagger_{\ell_k, \Ga} =\rho_{\ell_k, \wp_k}^{-1}(\tZ^\dagger_{\wp_k, \Ga}).$$

 
 
 


 Consider any standard chart $\fV$ of $\tsR_{\ell_k}$,
lying over a unique standard chart $\fV'$ of $\tsR_{\wp_k}$,
such that $\tZ_{\wp_k,\Ga}\cap \fV' \ne \emptyset$.

First, we suppose that $\tZ_{{ \wp_k},\Ga} \cap \fV'$,  
or equivalently $\tZ^\dagger_{{ \wp_k},\Ga} \cap \fV'$, 
by Lemma \ref{wp/ell-transform-ktauh} (3) in $(\wp_k)$,
 is not contained in   the $\ell$-center $Z_{\chi_k} \cap \fV'$.

We set 
$$\tGa^\zero_\fV=\{y_\fV \mid y_\fV  \hbox{ is the proper transform of some $y_{\fV'} \in \tGa^\zero_{\fV'}$}\};$$
$$\tGa^\one_\fV=\{y_\fV \mid y_\fV  \hbox{ is the proper transform of some $y_{\fV'} \in \tGa^\one_{\fV'}$}\}.$$

Then, Lemma  \ref{wp/ell-transform-ktauh} (1), (2) and (3) follow
from the same  proofs  for the corresponding
cases of $\wp$-blowups.
 We avoid repetation.

 We now suppose that $\tZ_{{\wp_k},\Ga}\cap \fV'$, or equivalently
 $\tZ^\dagger_{{ \wp_k},\Ga} \cap \fV'$,
 by  Lemma \ref{wp/ell-transform-ktauh} (3) in $(\wp_k)$,
 is contained in  $Z_{\chi_k} \cap \fV'$ 
 
This case corresponds to the precious case under the condition $(\star a)$
where $y_1$ there corresponds to $y_{\fV, (\um, \uu_{F_k})}$ here,  and
$y_{\fV, (\um, \uu_{F_k})}$ is not
identically zero along $\tZ_{{ \ell_k},\Ga}$.
So, we follow the proof in that case.

Thus, we set, 
\begin{equation}\label{Gazero-ell-contained-in-bar} 
{\tGa}^\zero_\fV= \{\zeta_\fV, \; y_\fV \mid y_\fV  
\hbox{ is the proper transform of some $ y_{\fV'} \in \tGa^\zero_{\fV'} \-\chi_k$}\},
\end{equation}
\begin{equation}\label{Gaone-ell-contained-in-bar}
{\Ga}^\one_{\fV}= \{\ y_\fV \mid y_\fV  
\hbox{ is the proper transform of some $y_{\fV'} \in \tGa^\one_{\fV'}$} \}.
\end{equation}


Then, following the correponding proofs for the precious case 
{\it under the condition $(\star a)$
where $y_1$ is not identically zero,} 
Lemma  \ref{wp/ell-transform-ktauh} (1), (2) and (3) follow. 
But, we need to point out that here, 
 $ \zeta_\fV=\de_{\fV, (\um,\uu_{F_k})}$ is not a  variable in $\var_\fV$,
 but a free variable in $\var_\fV^\vee$.

Putting all together, this completes the proof of Lemma \ref{wp/ell-transform-ktauh}.    
\end{proof}

 We call $\tZ_{({\wp}_{(k\tau)}\fr_\mu\fs_{h}),\Ga}$ the $\wp$-transform of $Z_\Ga$ 
 in $\tsV_{({\wp}_{(k\tau)}\fr_\mu\fs_{h})}$ for any $(k\tau)\mu h \in \Index_{\Phi_k}$.
 We call $\tZ_{\ell_k}$ he $\wp$-transform of $Z_\Ga$ 
 in $\tsV_{\ell_k}$.

   We need the final case of Lemma 
\ref{wp/ell-transform-ktauh}. We set
$$\tZ_{\ell, \Ga}:=\tZ_{\ell_\up,\Ga}, \;\; \tZ^\dagger_{\ell, \Ga}
:=\tZ^\dagger_{ \ell_\up,\Ga}.$$

\begin{cor} \label{ell-transform-up} 
Fix any standard chart $\fV$ of $\tsR^\circ_\ell$ as described in
Lemma \ref{wp/ell-transform-ktauh}
such that $\tZ_{\ell,\Ga} \cap \fV \ne \emptyset.$
Then, $\tZ_{\ell,\Ga} \cap \fV$,  as a closed subscheme of $\fV$,
is defined by the following relations
\begin{eqnarray} 
\;\;\;\;\; y , \; \; \forall \;\;  y \in   \tGa^\zero_\fV { \subset \var^\vee_\fV}; 
\;\;\;  y -1, \; \; \forall \;\;  y \in  \tGa^\one_\fV; \nonumber \\
\cB_\fV^\mn, \; 
\cB^\q_\fV, \; L_{\fV, \sF_\um}.
\end{eqnarray}
Furthermore, the induced morphism $\tZ^\dagger_{\ell,\Ga} \to Z_\Ga$ is birational.
\end{cor}

\section{The Main Theorem on the Final Scheme $\tsV_\ell$}\label{main-statement}

{\it 
 Let $p$ be an arbitrarily fixed prime number.
 Let $\mathbb F$ be either $\QQ$ or a finite field with $p$ elements.
In this entire section, 
every scheme  is defined over $\ZZ$, consequently,
 is defined over $\mathbb F$, and is considered as a scheme over 
 the perfect field $\mathbb F$.
}

\smallskip

  Take any $(k\tau \mu h) \in \Index_{\Phi_k}$ (cf. \ref{indexing-Phi}). 
  Consider  the $\wp$-blowup at $(k\tau \mu h)$
  $$\pi: \tsR_{(\wp_{(k\tau)}\fr_\mu\fs_h)} \lra \tsR_{(\wp_{(k\tau)}\fr_\mu\fs_{h-1})}.$$
   Fix any chart $\fV$ of  $\tsR^\circ_{(\wp_{(k\tau)}\fr_\mu\fs_h)} $ 
   and we let $\fV'$ be a chart of $\tsR^\circ_{(\wp_{(k\tau)}\fr_\mu\fs_{h-1})}$ 
   such that $\fV$ lies over $\fV'$. 
    Then, we can assume that the induced morphism $\pi^{-1}(\fV') \lra \fV'$
   corresponds to the ideal of the form
   $$\langle y_0', y_1' \rangle,$$
   where $y_0'$ is a variable in $T^+_{\fV', (k\tau)}$, 
   $y_1'$ is a variable in $T^-_{\fV', (k\tau)}$, and $B_{\fV',(k\tau)}=T^+_{\fV', (k\tau)}-T^-_{\fV', (k\tau)}.$
  We let $\PP_{[\xi_0,\xi_1]}$ be the corresponding factor projective space
  such that $(\xi_0,\xi_1)$ corresponds to  $(y_0', y_1')$.
  
  Observe here that the set of variables in $T^\pm_{\fV', (k\tau)}$ corresponds a subset of
  divisors associated to $T^\pm_{(k\tau)}$, hence possesses a naturally induced total order.
  
  Given any point $\bz$ on a chart, we say a variable is blowup-relevant at $\bz$
  if it can appear in the local blowup ideal $\langle y_0', y_1' \rangle$ as above 
  such that the corresponding  blowup center contains $\bz$. 
For example, a variable is not blowup-relevant at $\bz$ if it does not vanish at $\bz$.
   
   \begin{lemma}\label{keepLargest} 
      Let $\langle y_0', y_1' \rangle$ be the local blowup ideal as in the above such that 
   $y_0'$ is the largest blowup-relevant variable at the point $\bz$
    among all the variables of $T^+_{\fV', (k\tau)}$. Then,
    the chart $\fV$ containing the point $\bz$ can be chosen to lie over $(\xi_1 \equiv 1)$ so that
    the proper transform $y_0$ of $y_0'$ belongs to the  term $T^+_{\fV, (k\tau)}$.
   \end{lemma}
   \begin{proof} Using the notation before the statement of the lemma, 
   we can write,
   $$B_{\fV'} = a' y_0' - b' y_1'$$ 
   where $a'$ and $b'$ are some monomials.
   
   Suppose $\fV$  lies over $(\xi_0 \equiv 1)$. Then, by taking proper transforms, we obtain
    $$B_{\fV} = a  - \xi_1 b $$ 
    where $a$ and $b$ are some monomials.
    Because  $y_0'$ is the largest variable in the term $T^+_{\fV, (k\tau)}$, 
    by the order of the $\wp$-blowups, we have that $B_{\fV}$ 
    terminates. Hence, $\xi_1$ is invertible along
    $\tsV_{(\wp_{(k\tau)}\fr_\mu\fs_h)} \cap \fV$. Thus, 
    by shrinking the chart if necessary, we can switch to the chart 
    $(\xi_1 \equiv 1)$.
    
    Now, let $\fV$  lie over $(\xi_1 \equiv 1)$. Then, we obtain
    $$B_{\fV} = a \xi_0 -  b, $$ 
     where $a$ and $b$ are some monomials, and
   $\xi_0=y_0$ is the proper transform of $y_0'$. Hence, the statement follows.
   \end{proof}



We aim to show that the scheme $\tsV_\ell$ is smooth. The question is local.
 In the sequel, we will focus on a fixed closed point $\bz \in \tsV_\ell$ throughout.

Fix any standard chart of $\tsR^\circ_\ell$ containing $\bz$.
We let $\fV_{[0]}$ be the unique standard chart of $\sR_\sF$ 
such that $\fV$ lies over $\fV_{[0]}$ and 
we let $\bz_0 \in \sV_{[0]}$ be the image point of $\bz$.




 \begin{defn}\label{pleasant-v} 
 Consider the ordered set of blocks of relations 
\begin{equation}\label{good-eq-list}\fG=\{\fG_{F_1} < \cdots <\fG_{F_\up}\}.
\end{equation}
 Fix and consider any variable $y \in \var_\fV$ that appears in some relation
 of a block of $\fG$ in the above. 
 We say that $y$ is pleasant if         
  $y$ does not appear in any relation of  any earlier block.
 \end{defn}

  \begin{lemma}\label{max-minor} 
  Fix any closed point $\bz \in \tsV_\ell$. Consider any $\bF \in \sfm $.
  Then, there exists a 
  chart $\fV$ of $\tsR^\circ_\ell$ 
  containing the point $\bz$ such that the Jacobian $J(\fG_{\fV,F})$ 
admits a maximal minor $J^*(\fG_{\fV,F})$ such that
  it is an invertible square matrix at $\bz$,
   and all the variables that are used to compute $J^*(\fG_{\fV,F})$
 are pleasant with respect to the list
$\fG=\{\fG_{F_1} < \cdots <\fG_{F_\up}\}$.  
  \end{lemma}
  \begin{proof} 
   Consider the main binomial relations
 $$B_{\fV_{[0]}, s}:  x_{\fV_{[0]},(\uu_{s}, \uv_{s})}x_{\fV_{[0]}, \uu_F}-
 x_{\fV_{[0]}, (\um, \uu_F)}x_{\fV_{[0]},\uu_{s}} x_{\fV_{[0]},\uv_{s}}$$
 for all $s \in S_F \- s_F$.
We let 
$$S_F^{\vr,\ori}=\{ s \in S_F \- s_F \mid  x_{\fV_{[0]},(\uu_{s}, \uv_{s})} (\bz_0) \ne 0\}$$
and 
$$S_F^{\vr,\inc}=\{ s \in S_F \- s_F \mid  x_{\fV_{[0]},(\uu_{s}, \uv_{s})} (\bz_0) = 0\}$$
The two subsets $S_F^{\vr,\ori}$ and $S_F^{\vr,\ori}$ depend on the point $\bz$.

\medskip\noindent  
  {\it Case $(\alpha)$.  First, we assume 
  $x_{\fV_{[0]}, (\um,\uu_F)}(\bz_0)\ne 0$. }
  
  As  $x_{\fV_{[0]}, (\um, \uu_F)}(\bz_0)\ne 0$, 
  we can assume that $\fV$ lies over $(x_{(\um, \uu_F)} \equiv 1)$.
 
 We can write $$S_F^{\vr,\ori}=\{s_1, \cdots, s_\ell\} \;\; \hbox{and}\;\; S_F^{\vr,\inc}=\{t_1, \cdots, t_q\}$$
for some integers $l$ and $q$ such that $l + q= |S_F|-1$.

Then, on the chart $\fV_{[0]}$, the set $\cB^{\vr,\ori}_F$ consists of the following relations
    \begin{equation}\label{cB-vr-ori-alpha}  
    B_{\fV_{[0]}, s_i}: \;\;  x_{\fV_{[0]},(\uu_{s_i}, \uv_{s_i})}x_{\fV_{[0]}, \uu_F}-
  x_{\fV_{[0]},\uu_{s_i}} x_{\fV_{[0]},\uv_{s_i}} , \;\; 1 \le i \le l.
\end{equation}
By the relation $L_{\fV_{[0]}, F}(\bz_0)=0$, we see that there must exist $s \in S_F\-s_F$ such that
$x_{\fV_{[0]},(\uu_{s}, \uv_{s})}(\bz_0) \ne 0$, thus,
  the set $\cB^{\vr,\ori}_F$ must not be empty. Hence, $l>0$. This observation 
  will be used later.

  The set $\cB^{\vr, \inc}_F$ consists of the following relations
    \begin{equation}\label{cB-vr-inc-alpha}  
    B_{\fV_{[0]}, t_i}: \;\;  x_{\fV_{[0]},(\uu_{t_i}, \uv_{t_i})}x_{\fV_{[0]}, \uu_F}-
  x_{\fV_{[0]},\uu_{t_i}} x_{\fV_{[0]},\uv_{t_i}}, \;\; 1 \le i \le q
\end{equation}
for some integer $q \ge 0$ with $q=0$ when $\cB^{\vr, \inc}_F =\emptyset$.

We treat the relations of \eqref{cB-vr-ori-alpha} first.

First observe that during any of $\wp$-blowups, if a variable $y$ acquires 
an exceptional parameter $\ve$, then we have $\ve <y$ by Definition \ref{order-phi}.

Consider $B_{s_i}$ for any fixed $i \in [l]$.

We suppose $B_s \in \cB^\mn_{F}$ is the smallest binomial relation.

Assume that $\tsR_{(\wp_{(k\tau)}\fr_\mu\fs_h)} \lra \tsR_{(\wp_{(k\tau)}\fr_\mu\fs_{h-1})}$ 
is the last (non-trivial) $\wp$-blowup that makes $B_{s}$
terminate in $\tsR_{(\wp_{(k\tau)}\fr_\mu\fs_h)}$ for some $(k\tau)\mu h \in \Index_{\Phi_k}$
(cf. \eqref{indexing-Phi}). We let $\hs=(\wp_{(k\tau)}\fr_\mu\fs_h)$
and $\hs'=(k\tau)\fr_\mu\fs_{h-1})$.
Then,  over some chart $\fV_{\hs'}$ of $\tsR_{\hs'}$, the last $\wp$-blowup must correspond
to $(x_{\fV_{\hs'},\uu_F}, y_{\fV_{\hs'}})$ when $x_{\fV_{[0]},(\uu_{s}, \uv_{s})}(\bz_0)\ne0$
or $(x_{\fV_{\hs'},(\uu_{s}, \uv_{s})}, y_{\fV_{\hs'}})$ when $x_{\fV_{[0]},(\uu_{s}, \uv_{s})}(\bz_0)=0$,
where $y_{\fV_{\hs'}}$ is a variable in the minus term of $B_{\fV_{\hs'}, s}$.
We apply Lemma \ref{keepLargest} to $B_{\fV_{\hs'}, s}$.
Then,  either we have the $\vp$-variable $x_{\fV_{\hs},\uu_F}$
terminates and belongs to $B_{\fV_\hs, s}$
(e.g., when $x_{\fV_{[0]},(\uu_{s}, \uv_{s})} (\bz_0) \ne 0$),  or,
 $x_{\fV_{\hs'},\uu_F}$  turns into an exceptional-variable
$\ve_{\fV, \uu_F}$ (e..g, in the case when $x_{\fV_{[0]},(\uu_{s}, \uv_{s})} (\bz_0) = 0$).
Further, $x_{\fV_{\hs},\uu_F}$ or $\ve_{\fV_\hs, \uu_F}$
also  appears in all the remaining binomials that are larger than $B_s$ (in this special
case, it is just all  the remaining binomials since $B_s$ is assumed to be smallest; 
we term it this way so that the same line of arguments can be reused later).

During the $\wp$-blowups with respect to $B_s$, after the variable $x_{\uu_F}$
terminates or becomes exceptional, any further $\wp$-blowup must correspond
 $(x_{\fV_{\hs'},(\uu_{s}, \uv_{s})}, y_{\fV_{\hs'}})$ when $x_{\fV_{[0]},(\uu_{s}, \uv_{s})}(\bz_0)=0$.
 But, such a $\wp$-blowup will not affect {\it the plus term} of $B_{\fV_\hs, t}$ with $t \ne s$. 

In any case, $x_{\fV_{\hs},\uu_F}$ or $\ve_{\fV, \uu_F}$, remains to be second largest variable,
second only the $\vr$-variables
in $B_{\fV_\hs, s_i}$ with $s_i >s$. 
 
 We then move on to the second smallest binomial relation of $\cB^\mn_F$ 
 and repeat all the above arguments, until 
it is the turn to start the process of $\wp$-blowups with respect to $B_{s_i}$.

Then,  because $x_{\fV_{\hs},\uu_F}$ or $\ve_{\fV, \uu_F}$, 
 remains to be the largest blowup-relevant variable in $B_{\fV_\hs, s_i}$
 (since $x_{\fV_{[0]},(\uu_{s_i}, \uv_{s_i})}(\bz_0)\ne0$), we can then
 apply Lemma \ref{keepLargest} to  $B_{\fV_\hs, s_i}$ to obtain that 
 that there exists a chart $\fV$
containing the point $\bz$  such that we have
\begin{equation}\label{cB-ori}  \nonumber 
    B_{\fV, s_i}: \;\; a_ix_{\fV_\vt,(\uu_{s_i}, \uv_{s_i})} y_{\fV,\uu_F} - c_i
\end{equation}
for some monomial $a_i$ and $c_i$,
where $y_{\fV,\uu_F}$ is either the $\vp$-variable $x_{\fV,\uu_F}$
or the proper transform of an exceptional-variable $\ve_{\uu_F}$. 


Hence,  by shrinking the charts if necessary, we conclude that 
 that there exists a chart $\fV$
containing the point $\bz$  such that we have
\begin{equation}\label{cB-vr-ori} 
   B_{\fV, s_i}: \;\; a_i x_{\fV_\vt,(\uu_{s_i}, \uv_{s_i})} y_{\fV,\uu_F} - c_i, \;\; \hbox{for all} \;\; i \in [l]
\end{equation}
where $y_{\fV,\uu_F}$ is either the $\vp$-variable $x_{\fV,\uu_F}$
or the proper transform of an exceptional-variable $\ve_{ \uu_F}$ such that all of these
relations  terminate at $\bz$.

Now consider $B_{\fV_\vt, t_i}$ with $i \in [q]$.


Because $x_{\fV_{[0]},(\uu_{t_i}, \uv_{t_i})}(\bz_0)=0$ and 
$x_{\fV_{[0]},(\uu_{t_i}, \uv_{t_i})}$ is the largest
variable in the plus term of $B_{\fV_{[0]}, t_i}$, 
we can  apply Lemma \ref{keepLargest} directly to  $B_{t_i}$ for all $i \in [q]$ to obtain that 
 there exists a chart $\fV$
containing the point $\bz$  such that we have
$$B_{\fV, t_i}: \;\;  b_i x_{\fV, (\uu_{t_i}, \uv_{t_i})} - d_i , i \in [q].$$
where $b_i$ and $d_i$ are some monomials for all $i \in [q]$.
 
Put all together, shrinking the charts if necessary, we conclude that 
there exists a chart $\fV$ of $\tsR^\circ_\ell$, containing the point $\bz$ such that we have
\begin{eqnarray}\label{cB-vrChart-final}  
   B_{\fV, s_i}: \;\; a_i x_{\fV,(\uu_{s_i}, \uv_{s_i})} y_{\fV,\uu_F} - c_i, \;\; i \in [l]  \\
  B_{\fV, t_i}: \;\;  b_i x_{\fV,(\uu_{t_i}, \uv_{t_i})} - d_i , \;\; i \in [q]. \;\;\;\;\;\;\;\;\nonumber
\end{eqnarray}

Further, because $x_{\fV_{[0]}, (\um,\uu_F)}(\bz_0)\ne 0$,
  the $\ell$-blowups do not affect the (unique) chart $\fV_\vt$ 
  of $\tsR_\vt$ which $\fV$ lies over. Hence, we have
\begin{equation}\label{for-1column}
L_{\fV, F} =\sgn(s_F) + \sum_{i=1}^l \sgn (s_i) x_{\fV, (\uu_{s_i},\uv_{s_i})}
  + \sum_{i=1}^q \sgn (s_i) e_i x_{\fV, (\uu_{t_i},\uv_{t_i})}
  \end{equation}
  where $e_i$ are monomials in exceptional variables such that $e_i(\bz)=0$,
  for all $i \in [q]$.

 {\it As the chart $\fV$ is fixed and is clear from the context, in the sequel, to save space,
we will selectively drop some subindex $``\ \fV \ "$. For instance, we may write
$y_{\uu_F}$ for $y_{\fV,\uu_F}$, $x_{(\uu_{s_1},\uv_{s_1})}$ for $x_{\fV, (\uu_{s_1},\uv_{s_1})}$, etc.
A confusion is unlikely.
}

\smallskip


We  introduce the following maximal minor of the Jacobian 
  $J(\cB^\mn_{\fV,F}|_{\tGa_\fV}, L_{\fV,F}|_{\tGa_\fV})$
$$J^*(\cB^\mn_{\fV,F}|_{\tGa_\fV}, L_{\fV,F}|_{\tGa_\fV})= {{\partial(B_{\fV, s_1}|_{\tGa_\fV} \cdots B_{\fV, s_{\l}}|_{\tGa_\fV}, B_{\fV, t_1}|_{\tGa_\fV} \cdots B_{\fV, t_q}|_{\tGa_\fV},
L_{\fV,F}|_{\tGa_\fV})} \over {{\partial(y_{\uu_F},
x_{(\uu_{s_1},\uv_{s_1})} \cdots x_{(\uu_{s_{\l}},\uv_{s_{\l}})},
x_{(\uu_{t_1},\uv_{t_1})} \cdots x_{(\uu_{t_q},\uv_{t_q})}
 )}}} .$$


Then, one calculates and finds that at the point $\bz$, it is equal to
\begin{eqnarray} \nonumber
{\footnotesize
\left(
\begin{array}{cccccccccc}
a_1x_{(\uu_{s_1}, \uv_{s_1})} & a_1y_{\uu_F}   & \cdots & 0 & 0  & \cdots &0 \\
\vdots \\
a_l x_{(\uu_{s_{\l}}, \uv_{s_{\l}})} & 0 &  \cdots & a_l y_{\uu_F} & 0 &  \cdots & 0\\
* & 0 & \cdots & 0 & b_1 & \cdots & 0 \\
\vdots \\
* & 0 & \cdots & 0 & 0 & \cdots & b_q \\
0 & \sgn (s_1)&  \cdots  & \sgn (s_{\l}) & 0 & \cdots & 0
\end{array}
\right) (\bz).
}
\end{eqnarray}
Recall here that we have $l >0$. 

We can use the last $q$ columns to cancel the entries marked $``* "$ in the first column without
affecting the remaining entries.
  
  Then, multiplying the first column by $-y_{\uu_F}$ ($\ne 0$ at $\bz$), we obtain
\begin{eqnarray} \nonumber
{\footnotesize
\left(
\begin{array}{cccccccccc}
-a_1x_{(\uu_{s_1}, \uv_{s_1})}y_{\uu_F} & a_1y_{\uu_F}   & \cdots & 0 & 0  & \cdots &0 \\
\vdots \\
- a_l x_{(\uu_{s_{\l}}, \uv_{s_{\l}})} y_{\uu_F}& 0 &  \cdots & a_l y_{\uu_F} & 0 &  \cdots & 0\\
0 & 0 & \cdots & 0 & b_1  & \cdots & 0 \\
\vdots \\
0 & 0 & \cdots & 0 & 0 & \cdots & b_q   \\
0 & \sgn (s_1)&  \cdots  & \sgn (s_{\l}) & 0 & \cdots & 0
\end{array}
\right) (\bz).
}
\end{eqnarray}
Multiplying the $(i+1)$-th column by $x_{(\uu_{s_i}, \uv_{s_i})}$ and adding it to the first column 
for all $1\le i\le  {\l}$, 
we obtain
\begin{eqnarray} \nonumber
{\footnotesize
\left(
\begin{array}{cccccccccc}
0 & a_1y_{\uu_F}   & \cdots & 0 & 0  & \cdots &0 \\
\vdots \\
0 & 0 &  \cdots & a_l y_{\uu_F} & 0 &  \cdots & 0\\
0 & 0 & \cdots & 0 & b_1   & \cdots & 0 \\
\vdots \\
0 & 0 & \cdots & 0 & 0 & \cdots & b_q  \\
\sum_{i=1}^l \sgn (s_i) x_{(\uu_{s_i}, \uv_{s_i})}  & \sgn (s_1)&  \cdots  & \sgn (s_{\l}) & 0 & \cdots & 0
\end{array}
\right) (\bz).
}
\end{eqnarray}
But, at the point $\bz$, by \eqref{for-1column}, we have
$$\sum_{i=1}^l \sgn (s_i) x_{(\uu_{s_i}, \uv_{s_i})} (\bz) = - \sgn (s_F) \ne 0.$$
Thus, we conclude that 
$J^*(\cB^\ori_{\fV,F}|_{\tGa_\fV}, L_{\fV,F}|_{\tGa_\fV})$ is a square matrix of full rank at $\bz$,
and one sees that all the variables used to compute it are pleasant with respect to
the list \eqref{good-eq-list}. More precisely, $x_{(\uu_{s_i}, \uv_{s_i})}$ and 
 $x_{(\uu_{t_j}, \uv_{t_j})}$ are pleasant because they uniquely appear in the block
 $\fG_{F}$. The variable $y_{\fV, \uu_F}$ does not appear in the block $\fG_{\fV, F'}$ with $F' <F$,
 because all the relations of $\fG_{F'}$, 
 terminate before $\wp$- and $\ell$-blowups
 with respect to the relations of the block of $\fG_F$ are performed.



 \medskip\noindent  
  {\it Case $(\beta)$.  Next, we assume 
  $x_{\fV_{[0]}, (\um,\uu_F)}(\bz_0)= 0$. }

  As  $x_{\fV_{[0]}, (\um, \uu_F)}(\bz_0)= 0$, 
we can assume  $\fV$ lies over $(x_{(\uv_{s_0}, \uv_{s_0})} \equiv 1)$ for some 
$s_0 \in S_F \- s_F$.

In this case, we can write 
 $$S_F^{\vr,\ori}=\{s_0, s_1, \cdots, s_\ell\} \;\; \hbox{and}\;\; S_F^{\vr,\inc}=\{t_1, \cdots, t_q\}$$
for some integers $l$ and $q$ such that $l + q= |S_F|-2$.

Then, on the chart $\fV_{[0]}$, the set $\cB^{\vr,\ori}_F$ consists of the following relations
    \begin{eqnarray}\label{cB-vr-ori-beta}  
     B_{\fV_{[0]}, s_0}: \;\;  x_{\fV_{[0]},\uu_F} -
  x_{\fV_{[0]}, (\um, \uu_F)}x_{\fV_{[0]},\uu_{s_0}}  x_{\fV_{[0]},\uv_{s_0}} \;\;\;\;\; \\
 \;\;\;\;   B_{\fV_{[0]}, s_i}: \;\;  x_{\fV_{[0]},(\uu_{s_i}, \uv_{s_i})}x_{\fV_{[0]}, \uu_F}-
x_{\fV_{[0]}, (\um, \uu_F)}  x_{\fV_{[0]},\uu_{s_i}} x_{\fV_{[0]},\uv_{s_i}} , \;\; 1 \le i \le l.
\end{eqnarray}

  The set $\cB^{\vr, \inc}_F$ consists of the following relations
    \begin{equation}\label{cB-vr-inc-beta}  
    B_{\fV_{[0]}, t_i}: \;\;  x_{\fV_{[0]},(\uu_{t_i}, \uv_{t_i})}x_{\fV_{[0]}, \uu_F}-
 x_{\fV_{[0]}, (\um, \uu_F)} x_{\fV_{[0]},\uu_{t_i}} x_{\fV_{[0]},\uv_{t_i}}, \;\; 1 \le i \le q
\end{equation}
for some integer $q \ge 0$ with $q=0$ when $\cB^{\vr, \inc}_F =\emptyset$.


 By Corollary \ref{no-(um,uu)}, we can assume that $\fV$ lies over 
 a preferred chart, that is, in this case, the $\vr$-chart with respect to $F$. 
 Then, by Proposition \ref{eq-for-sV-vtk},  we have
  \begin{eqnarray}\label{cB-vrChart}   
     B_{\fV_\vt, s_0}: \;\;  x_{\fV_\vt,\uu_F} -
 \tilde x_{\fV_\vt,\uu_{s_0}} \tilde x_{\fV_\vt,\uv_{s_0}} \;\;\;\;\; \\
 B_{\fV_\vt, s_i}:  x_{\fV_\vt,(\uu_{s_i}, \uv_{s_i})} x_{\fV_\vt,\uu_F}-
 \tilde x_{\fV_\vt,\uu_{s_i}} \tilde x_{\fV_\vt,\uv_{s_i}} , 1 \le i \le l=|S_F| -2 \nonumber \\
 B_{\fV_\vt, t_i}:  x_{\fV_\vt,(\uu_{t_i}, \uv_{t_i})} x_{\fV_\vt,\uu_F}-
 \tilde x_{\fV_\vt,\uu_{t_i}} \tilde x_{\fV_\vt,\uv_{t_i}} ,  \;\;  i \in [q],
\end{eqnarray}
\begin{equation}\label{vrChart-LF-for-lt-inc}
 L_{\fV_\vt, F} =\sgn(s_F) \de_{\fV_\vt, (\um, \uu_F)}
+ \sum_{s \in S_F\- s_F} \sgn (s) x_{\fV_\vt, (\uu_{s},\uv_{s})}
  \end{equation}
where $\fV_\vt$ is the unique chart of $\tsR_\vt$ that $\fV$ lies over.

We treat the relation $B_{\fV_\vt, s_0}$ first. 

Notice that
$x_{\fV_\vt,\uu_F}$ is the largest variable in the plus-term of $B_{\fV_\vt, s_0}$. 
If $B_{s_0}$ is the smallest in $\cB^\mn_F$, then we can  apply
Lemma \ref{keepLargest} directly to $B_{\fV_\vt, s_0}$
 and conclude that there exists a chart $\fV$ containing $\bz$
such that we have
$$B_{\fV, s_0}: \;\; a_0 x_{\fV,\uu_F} - c_0$$
for some monomial $a_0$ and $c_0$.

Suppose $B_{s_0}$ is not the smallest. Then by the same lines of 
 arguments applied for $B_{\fV, s_i}$ with $i \in [l]$ as in {\it Case $(\alpha)$},
we can  obtain   that there exists a chart $\fV$
containing the point $\bz$  such that we have
\begin{equation}\label{cB-vt-ori-beta}  \nonumber 
     B_{\fV, s_0}: \;\; a_0 y_{\fV,\uu_F} - c_0
\end{equation}
where $y_{\fV,\uu_F}$ is either the $\vp$-variable $x_{\fV,\uu_F}$
or the proper transform of an exceptional-variable $\ve_{\fV', \uu_F}$. 

Now consider $B_{\fV_\vt, s_i}$ with $ i \in [l]\}$.
because $x_{\fV_{[0]},(\uu_{s_i}, \uv_{s_i})}(\bz_0) \ne 0$, one sees that 
there exists a chart $\fV$ containing $\bz$
such that we have
$$B_{\fV, s_i}: \;\; a_i x_{\fV,(\uu_{s_i}, \uv_{s_i})} - c_i, \;\; i \in [l]$$
for some monomial $a_i$ and $c_i$.

Next, consider $B_{\fV_\vt, t_i}$ with $i \in [q]$.

Because $x_{\fV_\vt,(\uu_{t_i}, \uv_{t_i})}$ is the largest blowup-relevant 
variable in the plus term of $B_{\fV_\vt, t_i}$, 
we can  apply Lemma \ref{keepLargest} to  $B_{t_i}$ for all $i \in [q]$ to obtain that 
 there exists a chart $\fV$
containing the point $\bz$  such that we have
$$B_{\fV, t_i}: \;\; b_i x_{\fV,(\uu_{t_i}, \uv_{t_i})} - d_i , i \in [q]$$
for some monomial $b_i$ and $d_i$.

Put all together, shrinking the charts if necessary, we conclude that 
there exists a chart $\fV$ containing the point $\bz$ such that we have
\begin{eqnarray}\label{cB-beta-all}  
B_{\fV, s_0}: \;\; a_0 y_{\fV,\uu_F} - c_0        \;\;\;\;\;\;\;\;\;\;\;\; \;\;\;\;\;\;\;\;\;\;\;\;\;\;\;           \\ 
  B_{\fV, s_i}: \;\; a_i x_{\fV,(\uu_{s_i}, \uv_{s_i})} - c_i, \;\; i \in [l]  \;\;\;\;\;\;\;\;\; \nonumber \\
 B_{\fV, t_i}: \;\;  b_i x_{\fV,(\uu_{t_i}, \uv_{t_i})} - d_i , \;\; i \in [q]. \;\;\;\;\;\;\;\;  \nonumber
\end{eqnarray}




Furthermore, by Proposition \ref{meaning-of-var-wp/ell} (9),  
 we can choose the chart $\fV$ such that
 $$L_{\fV, F}= 1 +\sgn (s_F) y_{\fV, (\um,\uu_F)}$$
  where $y_{\fV, (\um,\uu_F)}$ is the variable for the proper transform of
the divisor $E_{\ell, \vt_k}$ and is pleasant with respect to the list \eqref{good-eq-list}.

Now, we  introduce the following maximal minor of the Jacobian $J(\fG_{\fV,F})$                     
$$J^*(\fG_{\fV,F}|_{\tGa_\fV})= {{\partial(B_{\fV, s_0}|_{\tGa_\fV},
B_{\fV, s_1}|_{\tGa_\fV} \cdots B_{\fV, s_{\l}}|_{\tGa_\fV}), B_{\fV, t_1}|_{\tGa_\fV} \cdots 
B_{\fV, t_q}|_{\tGa_\fV}, L_{\fV, F}|_{\tGa_\fV})}
 \over {{\partial(  y_{ \uu_F},
x_{(\uu_{s_1},\uv_{s_1})} \cdots x_{(\uu_{s_{\l}},\uv_{s_l})}}},
x_{(\uu_{t_1},\uv_{t_1})} \cdots x_{(\uu_{t_q},\uv_{t_q})},
 y_{\fV, (\um,\uu_F)})} $$
Then, one calculates and finds that at the point $\bz$, it is equal to
\begin{eqnarray} \nonumber
\left(
\begin{array}{cccccccccc}
 a_0 & 0 & \cdots   & 0  & 0   & \cdots & 0 & 0\\
*  &  a_1  & \cdots &   0  & 0 & \cdots & 0 & 0\\
\vdots \\
*  & 0 & \cdots & a_l   & 0   & \cdots & 0 & 0       \\
*  & 0&  \cdots & 0 & b_1 & \cdots  &0     & 0\\
\vdots \\
*  &  0 & \cdots & 0 & 0 & \cdots & b_l         & 0     \\
 0 & * & \cdots & *&  * & \cdots  & * &   \sgn (s_F) 
\end{array}
\right) (\bz).
\end{eqnarray}
Thus, we conclude that 
$J^*(\cB^\mn_{\fV,F}|_{\tGa_\fV})$ 
is a square matrix of full rank at $\bz$, and all the variables that are used to compute it are pleasant.
  
  This proves the lemma.
   \end{proof}

   \begin{defn}\label{disjoint-smooth} 
A scheme $X$ is smooth if it is a disjoint union of finitely many connected smooth schemes of
possibly  various dimensions.
\end{defn}

\begin{thm} \label{main-thm}
 Let $\Ga$ be any subset $\var_{\rU_\um}$.
Assume that $Z_\Ga$ is integral.
Let $\tZ_{\ell,\Ga}$ be  the $\ell$-transform of $Z_\Ga$ in $ \tsV_{\ell}$.
Then,  $\tZ_{\ell,\Ga}$ is smooth over  $\Spec \mathbb F$.
Consequently,  $\tZ^\dagger_{\ell,\Ga}$ is smooth over  $\Spec \mathbb F$.
 
 In particular,  when $\Ga=\emptyset$, we obtain that
$\tsV_\ell$ is smooth  over $\Spec \mathbb F$.
\end{thm}
\begin{proof}
Let $\Ga$ be any subset $\var_{\rU_\um}$. Assume that $Z_\Ga$ is integral.

We let $\tZ_{\ell,\Ga}$ be  the $\ell$-transform of $Z_\Ga$ in $ \tsV_{\ell}$.
(As mentioned earlier, $Z_\Ga$ and $\tZ_{\ell,\Ga}$ are considered as $\FF$-schemes.)
Recall that  $\tZ_{\ell,\emptyset}=\tsV_\ell$ when $\Ga=\emptyset$.

Fix any closed point $\bz \in  \tZ_{\ell,\Ga} \subset \tsV_\ell.$
We let $\fV$ be a standard chart containing the point $\bz$ 
as chosen in Lemma \ref{max-minor}. In the sequel, we such a chart
a preferred chart for the point $\bz$.

By Corollary \ref{ell-transform-up},
the scheme $\tZ_{\ell, \Ga} \cap \fV$, if nonempty, 
 as a closed subscheme of
the chart $\fV$ of $\tsR_\ell$,  is defined by 
\begin{eqnarray} 
 y, \; \; y \in  \tGa^\zero_\fV; \;\;\;\; y-1, \;\;y \in \tGa^\one_\fV; \;\;\; \label{eq-for-Ga} \\ 
\cB^\mn_\fV, \; \cB^q_\fV, \; L_{\fV, \sfm}. \;\;\;  \nonumber
 \end{eqnarray}
We let $$\tGa_\fV=\tGa^\zero_\fV \sqcup \tGa^\one_\fV.$$
By setting $y=0$ for all $y \in \tGa^\zero_\fV $ and 
$y=1$ for all $y \in \tGa^\one_\fV $, we obtain a smooth open subset $\fV_\Ga$ of $\fV$:
 $$\fV_\Ga =\{ y =0, \; y \in  \tGa^\zero_\fV;  \;  y=1, \; y \in \tGa^\one_\fV\} \subset \fV.$$
The open susbet $\fV_\Ga$ comes equipped with the set of free variables
 $$\{ y \mid y \in \var_\fV \-  \tGa_\fV\}.$$     
 For any polynomial $f \in \kk[y]_{y \in \var_\fV}$, we let $f|_{\tGa_\fV}$ be obtained from
 $f$ by setting all variables in  $\tGa^\zero_\fV$ to be  0
 and setting  all variables in  $\tGa^\zero_\fV$ to be 1.
 This way,  $f|_{\tGa_\fV}$
 becomes a polynomial over $\fV_\Ga$.
  For any subset $P$ of polynomials over $\fV$, we let
$P|_{\tGa_\fV}=\{f|_{\tGa_\fV} \mid f \in P\}.$ This way, we have
$\cB^\mn_\fV|_{\tGa_\fV}, \cB^\q_\fV|_{\tGa_\fV}$,  etc.
 
  Then, $\tZ_{\ell, \Ga}\cap \fV$ 
  can be identified with the closed subscheme of $\fV_\Ga$ defined by
 \begin{eqnarray} 
\cB^\mn_\fV|_{\tGa_\fV}, \; \cB^q_\fV|_{\tGa_\fV}, \; L_{\fV, \sfm}|_{\tGa_\fV}. \;\;\; \label{eq-for-Ga-reduced}
 \end{eqnarray}

Now, we introduce the following maximal minor of the Jacobian $J(\fG_\fV|_{\tGa_\fV})$
 \begin{equation}\label{the-grand-matrix}  
J^*(\fG_\fV|_{\tGa_\fV})=\left(
\begin{array}{cccccccccc}
 J^*(\fG_{\fV,F_1}|_{\tGa_\fV})  & 0 & 0& \cdots & 0  \\
 * & J^*(\fG_{\fV,F_2}|_{\tGa_\fV}) &  0 & \cdots & 0  \\
\vdots &    \\
 * &  * & * & \cdots & J^*(\fG_{\fV, F_\up}|_{\tGa_\fV}) \\
 \end{array}
\right).
\end{equation}
By Lemmas \ref{max-minor},
all the blocks along diagonal are invertible at $\bz$; the entries in the upper right blocks
are due the fact that the variables used to compute the diagonal blocks are all pleasant. 
Therefore,  \eqref{the-grand-matrix}  is a square matrix of full rank at the point $\bz$.
We need to point out here that the terminating variables that we use to compute diagonal blocks
as in Lemma \ref{max-minor} can not belong to $\tGa_\fV^{=1}$ because 
when the variables of $\tGa_\fV^{=1}$ are introduced, the corresponding main binomial relation
must not terminate by consrtuction; they obviously do not belong to $\tGa_\fV^{=0}$.

\smallskip
 Now, we begin to prove that $\tZ_{\ell,\Ga}$ is smooth.

 First, we consider the case when $\Ga=\emptyset$. In this case, we have
  $Z_\emptyset=\rU_\um \cap \Gr^{d,E}$ and
$\tZ_{\ell,\emptyset }=\tsV_\ell$. 

As earlier, we fix and consider an arbitrary closed point $\bz \in \fV \subset \tsV_\ell$
 where $\fV$ is a preferred standard chart of $\tsR_\ell$.

We let $J:=J(\cB^\q, \cB^\mn_\fV, L_{\fV,\sfm})$ be the full Jacobian
 of all the defining equations of $\tsV_\ell \cap \fV$ in $\fV$.
We let  $J^*:=J^*(\fG_\fV)$ 
be the matrix of \eqref{the-grand-matrix} in the case of $\Ga=\emptyset$.
at the given point $\bz \in \tsV_\ell$.  (The maximal minor $J^*$ depends on the point $\bz$.)
Let $T_\bz (\tsV_\ell)$ be the Zariski tangent space of
$\tsV_\ell$ at $\bz$. Then,
we have
$$\dim T_\bz (\tsV_\ell)= \dim \tsR_\ell- \rk J(\bz)\le \dim \tsR_\ell- \rk J^*(\bz)$$
$$= \dim \rU_\um + |\cB^\mn|  - (|\cB^\mn| + \up)
= \dim \rU_\um  - \up = \dim \tsV_\ell,$$
where  $\dim \tsR_\ell=\dim \rU_\um + |\cB^\mn|$ by \eqref{dim} and
$\rk J^*(\bz) =|\cB^\mn| + \up$ by \eqref{the-grand-matrix}.
Hence, $\dim T_\bz (\tsV_\ell)  = \dim \tsV_\ell,$
thus, $\tsV_\ell$ is smooth at $\bz$. Therefore, $\tsV_\ell$ is smooth.

Consequently,  one sees that on any preferred standard chart $\fV$ of
the  scheme $\tsR_\ell$, all the relations of $\cB^q_\fV$
 must lie in the ideal generated by relations of $\cB^\mn_\fV$ and $L_{\fV, \sfm}$,  
 thus,   can be discarded from the chart $\fV$.

 Now, we return to a general  subset $\Ga$ of $\var_{\rU_\um}$
 as stated in the theorem.
 
Again, we fix and consider any closed point $\bz \in \fV \subset \tZ_{\ell,\Ga}$
 where $\fV$ is a preferred standard chart of $\tsR^\circ_\ell$ for the point.

  By the previous paragraph (immediately after proving that $\tsV_\ell$ is smooth),
  over any preferred standard chart $\fV$ of $\tsR_\ell$
  with $\tZ_{\ell,\Ga} \cap \fV \ne \emptyset$, we can discard 
   $\cB^\q_\fV|_{\tGa_\fV}$
   from the defining equations
 of $\tZ_{\ell,\Ga} \cap \fV$ and focus only on
 the equations of $\cB^\mn_\fV|_{\tGa_\fV}$ and 
 $L_{\fV,  \sF_\um}|_{\tGa_\fV}$.
 In other words, 
  $\tZ_{\ell,\Ga} \cap \fV$, if nonempty, as a closed subcheme of $\fV_\Ga$
  (which depends on both $\Ga$ and the point $\bz$), is defined by the equations in 
 $$\cB^\mn_\fV|_{\tGa_\fV}, \;\; L_{\fV,  \sF_\um}|_{\tGa_\fV}.$$

 Then, by  \eqref{the-grand-matrix}, the rank of the full Jacobian of 
 $\cB^\mn_\fV|_{\tGa_\fV}$ and $L_{\fV, \sF_\um}|_{\tGa_\fV}$
 equals to the number of the above defining equations at the closed point $\bz$ 
 of $\tZ_{\ell,\Ga} \cap \fV$.
 Hence, $\tZ_{\ell,\Ga}$ is smooth at $\bz$, thus,  so is $\tZ_{\ell,\Ga}$.

This proves the theorem.
\end{proof}

Let $X$ be an integral scheme.  We say $X$ admits a resolution  if there exists  a smooth
scheme $\tX$ and a projective  birational morphism
from $\tX$ onto $X$.

\begin{thm}\label{cor:main} 
Let $\Ga$ be any subset $\var_{\rU_\um}$.
Assume that $Z_\Ga$ is integral. Then, the morphism
 $\tZ^\dagger_{\ell, \Ga} \to Z_{\Ga}$ can be decomposed as
$$\tZ^\dagger_{\vr, \Ga} \to \cdots 
\to \tZ^\dagger_{\hs,\Ga}  \to \tZ^\dagger_{\hs',\Ga} \to \cdots \to
Z^\dagger_{\sF_{[j]},\Ga}  \to Z^\dagger_{\sF_{ [j-1]},\Ga} \to \cdots \to Z_\Ga$$
such that every morphism $\tZ^\dagger_{\hs,\Ga}  \to \tZ^\dagger_{\hs',\Ga}$
in the sequence is $\tZ^\dagger_{\vt_{[k]},\Ga}  \to \tZ^\dagger_{\vt_{ [k-1]},\Ga}$ for some $k \in [\up]$, or
$ \tZ^\dagger_{(\wp_{(k\tau)}\fr_\mu\fs_{h}),\Ga}
 \to \tZ^\dagger_{(\wp_{(k\tau)}\fr_\mu \fs_{h-1}),\Ga}$ for some $(k\tau) \mu h \in \Index_{\Phi_k}$, or
$ \tZ^\dagger_{\ell_k,\Ga} \to \tZ^\dagger_{\wp_k} $ for some $k \in [\up]$.
Further, every morphism in the sequence is surjective, projective, and  birational.  
In particular, $\tZ^\dagger_{\ell, \Ga} \to Z_\Ga$ 
is a resolution if $Z_\Ga$ is singular.
\end{thm}
\begin{proof} The smoothness of $\tZ^\dagger_{\ell, \Ga}$ follows from Theorem \ref{main-thm};
the decomposition of $\tZ^\dagger_{\ell, \Ga} \to Z_{\Ga}$ follows from
Lemmas \ref{wp-transform-sVk-Ga},  \ref{vt-transform-k},
 \ref{wp/ell-transform-ktauh}.
\end{proof}

\section{
Resolution of Singularity}\label{global-resolution} 

\subsection{Lafforgue's version of Mn\"ev's universality}\label{universality} $\ $

We first  review
Lafforgue's presentation of \cite{La03} on Mn\"ev's universality theorem.

As before, suppose we have a set of vector spaces, 
$E_1, \cdots, E_n$ such that 
$E_\alpha$ is of dimension 1 (or, a free module of rank 1 over $\ZZ$).
 We let 
$$E_I = \bigoplus_{\alpha \in I} E_\alpha, \;\; \forall \; I \subset [n],$$
$$E:=E_{[n]}=E_1 \oplus \ldots \oplus E_n.$$ 
(Lafforgues \cite{La03} considers the  more general case by allowing $E_\alpha$ to be
of any  finite dimension.)

For any fixed  integer $1\le d <n$, the Grassmannian
$$\Gr^{d,E}=\{ F \hookrightarrow E \mid \dim F=d\}$$
 decomposes into a disjoint union of locally closed strata
$$\Gr^{d,E}_\ud=\{ F \hookrightarrow E \mid \dim (F\cap E_I)=d_I,  \;\; \forall \; I \subset [n] \}$$
indexed by the family  $\ud=(d_I)_{I \subset [n]}$ of nonnegative integers $d_I \in \NN$ verifying

$\bullet$ $d_\emptyset=0, d_{[n]}=d$,

$\bullet$ $d_I +d_J\le d_{I\cup J} + d_{I \cap J}$, for all $I, J \subset [n]$.

The family $\ud$ is called a matroid of rank $d$ on the set $[n]$.
The stratum $\Gr^{d,E}_\ud$ is called a  matroid Schubert cell.

The Grassmannian $\Gr^{d,E}$ comes equipped with the (lattice) polytope
$$\Delta^{d,n} =\{ (x_1, \cdots, x_n) \in {\mathbb R}^n \mid 0 \le x_\alpha\le 1, \;
\forall \; \alpha; \; x_1 +\cdots + x_n = d\}.$$
For any $\ui=(i_1,\cdots,i_d) \in \II_{d,n}$, we let $\bx_\ui = (x_1, \cdots, x_n)$ be defined by
\begin{equation}\label{eta-L-2}
\left\{ 
\begin{array}{lcr}
x_i=1, &  \hbox{if $i \in \ui$,} \\
x_i=0, & \hbox{otherwise}.
\end{array} \right.
\end{equation}
It is known that  $\Delta^{d,n} \cap \NN^n =\{\bx_\ui \mid \ui \in \II_{d,n}\}$ and
it consists of precisely the vertices of the polytope $\Delta^{d,n}$.

Then, the matroid $\ud=(d_I)_{I \subset [n]}$ above defines the 
following subpolytope of $\Delta^{d,n}$
$$\Delta^{d,n}_\ud =\{ (x_1, \cdots, x_n) \in \Delta^{d,n} \mid  \sum_{\alpha \in I} x_\alpha \ge d_I, \; \forall \; I \subset [n]\}.$$
This is called the matroid subpolytope of $\Delta^{d,n}$ corresponding to $\ud$.

Recall that we have a canonical decomposition
$$\wedge^d E=\bigoplus_{\ui \in \II_{d,n}} E_{i_1}\otimes \cdots \otimes E_{i_d}$$
and it gives rise to the $\pl$ embedding of the Grassmannian
$$\Gr^{d,E} \hookrightarrow \PP(\wedge^d E)=\{(p_\ui)_{\ui \in \II_{d,n}} \in \GG_m 
\backslash (\wedge^d E  \- \{0\} )\}.$$

\begin{prop}\label{to-Ga} {\rm (Proposition, p4, \cite{La03})} 
Let $\ud$ be any matroid of rank $d$ on the set $[n]$ as considered above.
Then, in the Grassmannian 
$$\Gr^{d,E} \hookrightarrow \PP(\wedge^d E)=\{(p_\ui)_{\ui \in \II_{d,n}} \in \GG_m 
\backslash (\wedge^d E \- \{0\} )\},$$
the  matroid Schubert cell $\Gr^{d,E}_\ud$, as a locally closed subscheme, is defined by
$$p_\ui = 0, \;\;\; \forall \; \bx_\ui \notin \Delta^{d,n}_\ud ,$$ 
$$p_\ui \ne 0, \;\;\; \forall \; \bx_\ui \in \Delta^{d,n}_\ud . $$   
\end{prop}

Let $\ud=(d_I)_{I \subset [n]}$ be a matroid of rank $d$ on the set $[n]$ as above. 
Assume that $\ud_{[n]\setminus \{\alpha\}} =d-1$ for all $1\le \alpha\le n$.
Then, the configuration space $C^{d,n}_\ud$ defined by the matroid $\ud$ is the classifying scheme
of families of $n$ points
$$P_1, \cdots, P_n$$
on the projective space $\PP^{d-1}$ such that for any nonempty subset $I \subset [n]$,
the projective subspace $P_I$ of $\PP^{d-1}$ generated by the points $P_\alpha, \alpha \in I$, is
of dimension 
$$\dim P_I = d-1 -\ud_I.$$

\begin{thm}\label{Mn-La} {\rm (Mn\"ev, Theorem I. 14, \cite{La03})}
Let $X$ be an affine scheme of finite type over $\Spec \ZZ$.
Then, there exists a matroid $\ud$ of rank 3 on
the set $[n]$ such that $\PGL_3$ acts freely on the configuration space $C^{3,n}_\ud$.
Further, there exists a positive integer $r$ and
 an open subset $U \subset X \times \AA^r$ projecting surjectively
onto $X$ such that $U$ is isomorphic to the quotient space
${\underline C}^{3,n}_\ud :=C^{3,n}_\ud/\PGL_3$.
\end{thm}

\begin{thm}\label{GM} {\rm (Gelfand, MacPherson, Theorem I. 11, \cite{La03})}
Let $\ud$ be any matroid of rank $d$ on the set $[n]$ as considered above.
Then, the action of $\PGL_{d-1}$ on $C^{d,n}_\ud$ is free if and only if
$\dim_{\mathbb R} \Delta^{d,n}_\ud =n-1$. Similarly, 
 the action of $\GG_m^n/\GG_m$ on $\Gr^{d,n}_\ud$ is free if and only if
$\dim_{\mathbb R} \Delta^{d,n}_\ud =n-1$. 
Further, when $\dim_{\mathbb R} \Delta^{d,n}_\ud =n-1$,  the quotient
$C^{d,n}_\ud/\PGL_{d-1}$ can be canonically identified with  the quotient
$\Gr^{d,E}_\ud/(\GG_m^n/\GG_m )$.
\end{thm}


By the above correspondence, we have the following equivalent version of Theorem \ref{Mn-La}.

\begin{thm}\label{Mn-La-Gr} {\rm (Mn\"ev, Theorem I. 14, \cite{La03})}
Let $X$ be an affine scheme of finite type over $\Spec \ZZ$.
Then, there exists a matroid $\ud$ of rank 3 on
the set $[n]$ such that $(\GG_m^n/\GG_m )$ acts freely on the  matroid Schubert cell $\Gr^{3,E}_\ud$.
Further, there exists a positive integer $r$ and an open subset $U \subset X \times \AA^r$ projecting 
onto $X$ such that $U$ is isomorphic to the quotient space
$\bGr^{3,E}_\ud:=\Gr^{3,E}_\ud/(\GG_m^n/\GG_m)$.
\end{thm}

\subsection{Global resolution: the affine case}

\begin{thm}\label{resolusion-affine} {\rm (Resulution: Affine Case)}
Let $X$ be 
an affine scheme of finite presentation over a 
 perfect field $\kk$. 
Assume further that $X$ is integral and 
singular. Then,  $X$ admits a resolution, 
that is, there exists a smooth scheme $\tX$ and a 
projective birational  morphism from $\tX$ onto $X$.
\end{thm}
\begin{proof} 

 First, we assume that $X$ is defined over $\Spec \ZZ$.

 We apply Theorem \ref{Mn-La-Gr} to $X$ and follow  the notations in Theorem \ref{Mn-La-Gr}.

We identify $U \subset X \times \AA^r$ with the quotient space
$\bGr^{3,E}_\ud=\Gr^{3,E}_\ud/(\GG_m^n/\GG_m)$.

Consider the quotient map 
$$\pi: \Gr^{3,E}_\ud \lra \bGr^{3,E}_\ud=\Gr^{3,E}_\ud/(\GG_m^n/\GG_m).$$
We have the diagram
\begin{equation}\label{la-di} \xymatrix{
\Gr^{3,E}_\ud  \ar^{\pi \;\;\;\;\;\;\;\;\;\;\;\;\;\;\;\;\;\;\;\;\;\;\;\;\;\;}[r] 
& \bGr^{3,E}_\ud=\Gr^{3,E}_\ud/(\GG_m^n/\GG_m) 
\cong U  \ar @{^{(}->}[r]  & X \times \AA^r  \ar[d]   \\
& & X.
} \end{equation}

We can apply Proposition \ref{to-Ga} to the   matroid Schubert cell 
$\Gr^{3,E}_\ud$.  

Since $\Delta^{3,n}_\ud \ne \emptyset$, there exists $\um \in \II_{3,n}$ such that
$\bx_\um \in \Delta^{3,n}_\ud$.
We define
\begin{equation}\label{ud=Ga}
\Ga:=\Ga_\ud =\{ \ui \in \II_{3,n} \mid \bx_{\ui} \notin \Delta^{3,n}_\ud \}.
\end{equation}

Then, we have that 
$$\xymatrix{ 
\Gr^{3,E}_\ud \ar @{^{(}->}[r]  &  Z_\Ga \ar @{^{(}->}[r]  & \rU_\um } $$
and $\Gr^{3,E}_\ud$ is an open subset of the $\Ga$-scheme $Z_\Ga \subset \rU_\um$. 
As $X$ is integral (by assumption), one sees that $Z_\Ga$ is integral.


We then let $$\vp_\Ga: \tZ^\dagger_{\ell, \Ga} \lra Z_\Ga$$ be as in
 Theorem \ref{cor:main}. This is a resolution.
We set $\widetilde\tGr^{3,E}_\ud=\vp_\Ga^{-1}(\Gr^{3,E}_\ud)$, scheme-theoretically.
Then $$\varpi|_{\tGr^{3,E}_\ud}:  \widetilde\tGr^{3,E}_\ud \lra \Gr^{3,E}_\ud$$ is a resolution.

We now aim to produce a resolution of the scheme
$\bGr^{3,E}_\ud \cong U \subset X \times \AA^r$.
Recall that we have  the quotient map 
$\pi: \Gr^{3,E}_\ud \lra \bGr^{3,E}_\ud \cong U \subset X \times \AA^r$.
In what follows, we will identify  $\bGr^{3,E}_\ud$ with the open subset $U$ of $X \times \AA^r$.

As the morphism $\vp_\Ga: \tZ^\dagger_{\ell, \Ga} \lra Z_\Ga$ is projective, 
so is the restricted morphism 
$\varpi|_{\tGr^{3,E}_\ud}:  \tGr^{3,E}_\ud \to \Gr^{3,E}_\ud$.
Hence, by Theorem 7.17 of \cite{Hartshorne},  we can assume   that 
$\tGr^{3,E}_\ud \to \Gr^{3,E}_\ud$ is the blowup of $\Gr^{3,E}_\ud$
along an ideal sheaf $\tJ$  on $\Gr^{3,E}_\ud$.

As $X$ is affine, we assume that $X$ is a closed affine subscheme of $\AA^m$ for some
integer $m$.  Thus, $X \times \AA^r$ is 
a closed affine subscheme of $\AA^m \times \AA^r$.
We let $\bx=(x_1, \cdots, x_m)$ be the affine
coordinates of $\AA^m$ (the first factor of $\AA^m \times \AA^r$)
and $\bt=(t_1, \cdots, t_r)$ be the affine
coordinates of $\AA^r$ (the second factor of $\AA^m \times \AA^r$).

Now, observe  that the quotient map $$\pi: 
\Gr^{3,E}_\ud \lra \bGr^{3,E}_\ud (\cong U)$$ is 
 a principal $(\GG_m^n/\GG_m)$-bundle, and is
\'etale  locally trivial.  As  any \'etale  locally trivial principal $(\GG_m^n/\GG_m)$-bundle
is Zariski locally trivial (that is, $(\GG_m^n/\GG_m)$ is special in the sense of Serre),
we can over $\bGr^{3,E}_\ud (\cong U)$ by a finite set $\{O\}$ of open subsets
 such that for any open subset $O$ in the cover, we have a trivialization
\begin{equation}\label{trivialization-O}
\Gr^{3,E}_\ud|_{O} \cong O \times (\GG_m^n/\GG_m).
\end{equation}
Further, we take a split and let
 $$(\GG_m^n/\GG_m) \cong \GG_m^{n-1}=\Spec \FF[s_1^\pm, \cdots, s_{n-1}^\pm].$$
Then, we can realize $\Gr^{3,E}_\ud|_{O} \cong O \times (\GG_m^n/\GG_m)$
as a locally closed subset of 
$$\Spec \FF[x_1, \cdots, x_m, t_1, \cdots, t_r, s_1^\pm, \cdots, s_{n-1}^\pm].$$
In what follows, we will write ${\bf s}=(s_1, \cdots, s_{n-1})$.
For any $\ba=(a_1, \cdots, a_{n-1}) \in \ZZ^{n-1}$, we write
$\bs^\ba=s_1^{a_1}\cdots s_{n-1}^{a_{n-1}}$.

Over the open subset $\Gr^{3,E}_\ud|_{O}$ and using the trivialization \eqref{trivialization-O}, 
we can suppose that the ideal $\tJ|_O$ ($\tJ$ restricted to $\Gr^{3,E}_\ud|_{O}$)
is generated by 
$$g_1(\bx, \bt, \bs), \cdots, g_k(\bx, \bt, \bs) \in \tJ|_{O}
\subset \FF[\bx, \bt, s_1^\pm, \cdots, s_{n-1}^\pm],$$ modulo the ideal of 
(the closure of) $\Gr^{3,E}_\ud|_{O}$ in $ \FF[\bx, \bt, s_1^\pm, \cdots, s_{n-1}^\pm],$
for some positive integer $k$.


By the construction of the isomorphism $\bGr^{3,E}_\ud \cong U \subset X \times \AA^r$
(see the proof of Theorem I. 14, \cite{La03}), the variables $\bt$ of $\AA^r$ correspond to
the choices of some auxiliary free points on the projective plane. Then, by the construction of the
resolution $\tZ^\dagger_{\ell, \Ga} \lra Z_\Ga$, we conclude that for every $i \in[k]$,
$g_i(\bx, \bt, \bs)= g_i(\bx)$ is independent of  the variables $\bt$ and $\bs$.
 
  In particular, we can let $\tbJ|_O$ be the ideal of
 $(O \subset U \subset) \;X \times \AA^r$ generated by $g_i(\bx), 1\le i\le k$, modulo the ideal
 of $O$. 
 
 We define the following
 \begin{itemize}
\item We let $J$ be the ideal of $\FF[\bx]$ generated by  
 $g_i(\bx)$, $1\le i\le k$, modulo the ideal of $X$. We then let
 $$\tX \to X$$ be the blowup of $X$ along $J$.
 \item We let $\tbJ$ be the ideal sheaf of  $\bGr^{3,E}_\ud$ corresponding to the closure
 of $V(\tbJ|_O)$, where $V(\tbJ|_O)$ is the closed subscheme of $O$ corresponding to $\tbJ|_O$. We let 
 $$\widetilde\bGr^{3,E}_\ud \to \bGr^{3,E}_\ud$$ be the blowup of $\bGr^{3,E}_\ud$ along $\tbJ$. 
 \item  We let $\tJ$ be the ideal sheaf of  $\Gr^{3,E}_\ud$ corresponding to the closure
 of $V(\tJ|_O)$,  where $V(\tJ|_O)$ is the closed subscheme of 
 $\Gr^{3,E}_\ud|_{O}$ corresponding to $\tJ|_O$. Note here that 
 $\tGr^{3,E}_\ud \to \Gr^{3,E}_\ud$ is the blowup of $\Gr^{3,E}_\ud$ along $\tJ$.
  \end{itemize}

 A priori, the ideals $\tJ$ and $\tbJ$, and their corresponding close subscheme $V(\tJ)$ and $V(\tbJ)$,
 are constructed from the open subset  $\Gr^{3,E}_\ud|_{O}$.
 But, because $\Gr^{3,E}_\ud$ and  $\bGr^{3,E}_\ud$ are irreducible, 
 it is straightforward to check that  $V(\tJ)$ (resp. the ideal sheaf $\tJ$)
$V(\tbJ)$ (resp. the ideal sheaf $\tbJ$) 
 do not depend on the choice of the open subset $O$ nor on
 the trivialization  $\Gr^{3,E}_\ud|_{O} \cong O \times (\GG_m^n/\GG_m)$.
Of course, for any open subset $O$ in the cover $\{O\}$ of $\bGr^{3,E}_\ud$,
the ideals $\tJ$ and  $\tbJ$  admit
  similar descriptions as above.

The above discussion implies the following. We have the  cartesian diagram
 \begin{equation}\label{fiber-square} \xymatrix{ 
\tGr^{3,E}_\ud =\Proj_{\Gr^{3,E}_\ud} \oplus_{d \ge 0} (\tilde J)^d \ar[d]  \ar[r]& 
\Gr^{3,E}_\ud   \ar[d] \\
 \widetilde\bGr^{3,E}_\ud =\Proj_{\bGr^{3,E}_\ud} \oplus_{d \ge 0} (\tbJ)^d \ar[r] \ar[d]& \bGr^{3,E}_\ud \cong U \subset X \times \AA^r \ar[d]\\
  \tX =\Proj_{X} \oplus_{d \ge 0} J^d \ar[r]^{\;\;\;\;\;\; \rho} & X ,
  } \end{equation}
  such that
  \begin{itemize}
\item $\tGr^{3,E}_\ud \to \widetilde\bGr^{3,E}_\ud$ is the $(\GG_m^n/\GG_m)$-fiber bundle,
obtained from the pullback of the $(\GG_m^n/\GG_m)$-fiber bundle
$\Gr^{3,E}_\ud \to \bGr^{3,E}_\ud$;
\item 
 $\widetilde\bGr^{3,E}_\ud \lra \tX$ is a smooth morphism with typical fiber
 isomorphic to an open subset of $\AA^r$.
 \end{itemize}
 Here, we say a morphism $f:Y \to S$ between two schemes $Y$ and $S$ is a $F$-fiber bundle 
 for some fixed scheme $F$ if
 $S$ can be covered by an open subset $\{O\}$ such that
 $f^{-1}(O)$ is isomorphism to $O \times F$.
 
 Because  $\tGr^{3,E}_\ud$ is smooth, by 
 the first fiber bundle, $\widetilde\bGr^{3,E}_\ud$ is smooth.
 Then by the second smooth morphism, we conclude that $\tX$ is smooth.
 Therefore, the morphism
$$\rho: \tX \lra X$$ is a resolution over 
the prime field $\FF$, provided that $X$ is defined over $\ZZ$.

Now, we consider the general case when the affine scheme 
$X/\kk$ is of finite presentation over a perfect field $\kk$.

   The field $\kk$ is an extension of its unique prime (minimal) subfield ${\mathbb F}'$. This 
   unique prime  subfield ${\mathbb F}'$ is isomorphic to 
   $\QQ$ when $\kk$ has characteristic zero or isomorphic to 
   ${\mathbb F}_p$ when $\kk$ has the characteristic $p>0$,
   that is, ${\mathbb F}' \cong \FF$.
 
 We suppose that $X/\kk$ is defined by a finite set of polynomials
  $g_1,\cdots, g_m$ in $\kk[x_1,\cdots, x_n]$ for some positive integers $m$ and $n$.
  Let $R'=\FF'[\hbox{coefficients of $g_1,\cdots, g_m$}]$. When $\FF'$ has the characteristic zero,
  we let $R=R'$; when $\FF'$ has the characteristic $p>0$, we let $R$ be the subring of $\kk$
  generated by $R'$ and $p^i$-th roots of elements of $R'$ for all $i >0$.
  Then the same description of $X$ over $\kk$ also makes sense as the description of 
  a scheme $Y/\FF'$ over $B=\Spec R$ (cf. \cite{EGAIV}, Theorems 8.8.2 and 8.10.5).
 Let $K$ be the fraction field of $R$. This is a perfect subfield of $\kk$.
Then, by the above, we have  a dominant morphism $f: Y \to B$ of finite presentation over $B$,
  such that $X/K$ is isomorphic to the generic fiber $Y_K$ of the morphism $f$.

   
   
   Now, as $Y$ is affine and defined over $\FF'$, we can 
 take a resolution $\tY \to Y$ over $\FF'$. Then, we consider the induced dominant morphism
 $\tilde f: \tY \to B$. 
 Since $\tY$ is  smooth,  we have that the generic fiber $\tX/K$ of $\tilde f$ is regular as well.
 Because $K$ is perfect, we have that $\tX/K$  is smooth.  Then by the scalar extension
 $K \subset \kk$, we obtain that $\tX/\kk$ is smooth.
 Clearly, the induced morphism $\tX/\kk \to  X/\kk$ is projective, birational,
 and surjective,  hence is a resolution,
 as desired.

This implies  Theorem \ref{resolusion-affine}.
\end{proof}

\subsection{Global resolution: the projective case}

\begin{thm}\label{resolusion-proj} {\rm (Resulution: Projective Case)}
Let $X$ be 
a projective scheme of finite presentation over any fixed  
 perfect field $\kk$. 
Assume further that $X$ is integral and 
singular. Then,  $X$ admits a resolution, 
that is, there exists a smooth scheme $\tX$ and a 
projective birational  morphism from $\tX$ onto $X$.
\end{thm}
\begin{proof} 
We continue to follow the idea and notation of the proof of Theorem \ref{resolusion-affine}.

We first assume that $X$ is defined over $\ZZ$.

Take and fix a projective embedding of $X \subset \PP^m$.
Let $C_X$ be the affine cone of $X$ defined by the above embedding,
and, let $C_X^0 = C_X \- \{0\}$. Then, $C_X^0$ is a $\GG_m$-bundle over $X$, 
locally trivial in Zariski topology.

The affine cone $C_X$ is a closed affine subscheme of the affine space $\AA^{m+1}$. 
We let  $\bx=(x_0, \cdots, x_m)$ (resp. $[\bx]=[x_0, \cdots, x_m]$) be the
affine coordinates of $\AA^{m+1}$ (resp. the homogeneous coordinates of 
the projective space $\PP^m$).

 As in the proof of Theorem \ref{resolusion-affine}, by applying the similar argument to
the affine scheme $C_X$, we obtain  the following diagram
\begin{equation}\label{la-di} \xymatrix{
\tGr^{3,E}_\ud \ar[r] & \Gr^{3,E}_\ud  \ar^{\pi \;\;\;\;\;\;\;\;\;\;\;\;\;\;\;\;\;\;\;\;\;\;\;\;\;\;}[r] 
& \bGr^{3,E}_\ud=\Gr^{3,E}_\ud/(\GG_m^n/\GG_m) 
\cong U  \ar @{^{(}->}[r]  & C_X \times \AA^r  \ar[d]   \\
& & C_X^0 \ar @{^{(}->}[r]  \ar[d] & C_X \ar @{-->} [d] \\
& & X \ar^{=}[r] & X \subset \PP^m,
} \end{equation}
where $\tGr^{3,E}_\ud \to \Gr^{3,E}_\ud$ is induced from 
the resolution $\tZ^\dagger_{\ell, \Ga} \lra Z_\Ga$ for some $\Ga$.

Using $C_X$ to take the role of the affine scheme $X$ as in the proof of 
Theorem \ref{resolusion-affine}, we can keep and follow the notations used in that proof.
In particular, using the trivialization of \eqref{trivialization-O}
$$\Gr^{3,E}_\ud|_{O} \cong O \times (\GG_m^n/\GG_m),$$
we  can realize $\Gr^{3,E}_\ud|_{O} \cong O \times (\GG_m^n/\GG_m)$
as a locally closed subset of 
$$\Spec \FF[x_1, \cdots, x_m, t_1, \cdots, t_r, s_1^\pm, \cdots, s_{n-1}^\pm].$$

We can assume that the $\tGr^{3,E}_\ud|_{O} \to \Gr^{3,E}_\ud|_{O}$ is
the blowup of $ \Gr^{3,E}_\ud|_{O}$ along an ideal   $\tilde J|_O$,
generated by $g_1, \cdots, g_k \in 
\subset \FF[\bx, \bt, s_1^\pm, \cdots, s_{n-1}^\pm]$, modulo the ideal of 
(the closure of) $\Gr^{3,E}_\ud|_{O}$ in $ \FF[\bx, \bt, s_1^\pm, \cdots, s_{n-1}^\pm],$
for some positive integer $k$.

 Again, as in the proof of Theorem \ref{resolusion-affine},
by the construction of the isomorphism $\bGr^{3,E}_\ud \cong U \subset C_X \times \AA^r$
(see the proof of Theorem I. 14, \cite{La03}), 
and by the construction of the 
resolution $\tZ^\dagger_{\ell, \Ga} \lra Z_\Ga$, we conclude that for every $i \in [k]$,
$g_i(\bx, \bt, \bs)= g_i(\bx)$ is free of  the variables $\bt$ and $\bs$, and furthermore,
 it is also homogeneous in $\bx$.  In particular, we can let $\tbJ|_O$ be the (affine) ideal of
 $(O \subset U \subset) \; C_X \times \AA^r$ generated by $g_i(\bx), 1\le i\le k$, modulo the ideal
 of $O$. 
 
 We define the following
 \begin{itemize}
\item We let $J$ be the homogeneous ideal of $\FF[\bx]$ generated by  
 $g_i(\bx)$, modulo the ideal of $X$. We then let
 $$\tX \to X$$ be the blowup of $X$ along $J$.
 \item  We let $J_{\rm aff}$ be the affine ideal of $\FF[\bx]$ generated by  
 $g_i(\bx)$, modulo the ideal of $C_X$.  We then let
 $$\widetilde{C_X} \to C_X$$ be the blowup of $C_X$ along $J_{\rm aff}$. 
 \item We let $\tbJ$ be the ideal sheaf of  $\bGr^{3,E}_\ud$ corresponding to the closure
 of $V(\tbJ|_O)$. We then let 
 $$\widetilde\bGr^{3,E}_\ud \to \bGr^{3,E}_\ud$$ be the blowup of $\bGr^{3,E}_\ud$ along $\tbJ$. 
 \item  We let $\tJ$ be the ideal sheaf of  $\Gr^{3,E}_\ud$ corresponding to the closure
 of $V(\tJ|_O)$.  Note here that 
 $\tGr^{3,E}_\ud \to \Gr^{3,E}_\ud$ is the blowup of $\Gr^{3,E}_\ud$ along $\tJ$.
  \end{itemize}
As in the proof of Theorem \ref{resolusion-affine}, the above definitions do not depend on the choice
of the open subset $O$ nor on the trivialization  \eqref{trivialization-O}. 

Then, the above discussions imply the following. We have the diagram

 \begin{equation}\label{fiber-square} \xymatrix{ 
& \tGr^{3,E}_\ud =\Proj_{\Gr^{3,E}_\ud} \oplus_{d \ge 0} (\tilde J)^d \ar[d]  \ar[r] & 
\Gr^{3,E}_\ud   \ar[d]  & \\
&  \widetilde\bGr^{3,E}_\ud =\Proj_{\bGr^{3,E}_\ud} \oplus_{d \ge 0} (\tbJ)^d \ar[r] \ar[d]& \bGr^{3,E}_\ud \cong U \subset C_X \times \AA^r \ar[d] & \\
  \widetilde{C_X^0} \ar[r] \ar[d] & \widetilde{C_X} =\Proj_{C_X} \oplus_{d \ge 0} J_{\rm aff}^d \ar[r]\ar @{-->}[d] & C_X \ar @{-->}[d] & C_X^0 \ar[l] \ar[d]  \\ 
\tX \ar[r]^{=\;\;\;\;\;\;\;\;\;\;\;\;\;\;\;\;\;\;} & \tX =\Proj_{X} \oplus_{d \ge 0}J^d \ar[r] & X & X\ar[l]_{=}
  } \end{equation}
such that 
\begin{itemize}
\item $\tGr^{3,E}_\ud \to \widetilde\bGr^{3,E}_\ud$ is the $(\GG_m^n/\GG_m)$-fiber bundle,
obtained from the pullback of the $(\GG_m^n/\GG_m)$-fiber bundle
$\Gr^{3,E}_\ud \to \bGr^{3,E}_\ud$;
\item 
 $\widetilde\bGr^{3,E}_\ud \lra \widetilde{C_X}$ is a smooth morphism with typical fiber
 isomorphic to an open subset of $\AA^r$;
 \item $\widetilde{C_X^0} \lra \tX$ is the $\GG_m$-fiber bundle,
obtained from the pullback of the $\GG_m$-fiber bundle $C_X^0 \lra X$.
 \end{itemize}
 
 Because $\tGr^{3,E}_\ud$ is smooth, by the first fiber-bundle,
 we obtain that $\widetilde\bGr^{3,E}_\ud$ is smooth; by the second-fiber bundle, we see that
 $\widetilde{C_X}$ and hence its open subset $\widetilde{C_X^0}$ are smooth;
 by the third fiber bundle, we conclude that  $\tX$ is smooth over any prime field $\FF$.
 
 Thus,  the morphism $\tX \lra X$ is a resolution over $\FF$.

 
 Now, we consider the general case when the affine scheme 
$X/\kk$ is of finite presentation over a perfect field $\kk$.

   The field $\kk$ is an extension of its unique prime (minimal) subfield ${\mathbb F}'$. This 
   unique prime  subfield ${\mathbb F}'$ is isomorphic to 
   $\QQ$ when $\kk$ has characteristic zero or isomorphic to 
   ${\mathbb F}_p$ when $\kk$ has the characteristic $p>0$,
   that is, ${\mathbb F}' \cong \FF$.
 
Now, using the similar arguments as in the end of the proof of Theorem \ref{resolusion-affine},
 there exist a subring $R$ of $\kk$ such that its fraction field $K$ is a perfect subfield of $\kk$,
 an integral scheme $Y/\FF'$,
 and a dominant morphism $f: Y \to B$ of finite presentation over $B=\Spec R$,
  such that $X/K$ is isomorphic to the generic fiber $Y_K$ of the morphism $f$.
 By taking a projective closure of $B$ and the corresponding projective
 closure of $Y/B$, we may assume that $B$ and $Y$ are projective.

   
   
   Now, as $Y$ is projective and defined over $\FF'$, we can 
 take a resolution $\tY \to Y$ over $\FF'$.
Then, we consider the induced dominant morphism
 $\tilde f: \tY \to B$. 
 Since $\tY/\FF'$ is  smooth,  we have that the generic fiber $\tX/K=Y_K$ of $\tilde f$ is regular as well.
 Because $K$ is perfect, we have that $\tX/K$  is smooth.  
 Hence, by the scalar extension $K \subset \kk$, $\tX/\kk$  is smooth.
 Clearly, the induced morphism $\tX/\kk \to  X/\kk$ is projective, birational,
 and surjective,  hence is a resolution,
 as desired.

This proves  Theorem \ref{resolusion-proj}.
\end{proof}

(Of course Theorem \ref{resolusion-affine} is a direct consequence of Theorem \ref{resolusion-proj},
and the proofs of the two theorems are largely parallel. But, we find the above organization makes
our idea and proofs more tranparent.)

When the base field $\kk$ has characteristic zero,
the above two theorems are  well known from Hironaka's resolution \cite{Hironaka64}. 
When the base field $\kk$ has positive characteristic, Abhyankar \cite{Abh} proved resolution of singularities for algebraic threefolds in characteristic greater than 6.
(One may consult \cite{Zariski2} for the case when $\kk$ is not perfect.)

 In this article, we approach resolution of singularity 
by performing  blowups of  (a chart of)  $\Gr^{3,E}$. 
It is convincible that certain parallel blowups exist for $(\PP^2)^n$ that can also lead to achieve
 resolution of singularity  (\cite{Hu15b}). 
(Indeed, when the author began to work on resolution of singularity, he tried both approaches and switched between the two for quite a while 
before settling down on the current approach via  Grassmannians.)

\section{ Geometric Resolution}

Prior to de Jong's geometric approach \cite{deJong96}, resolutions of varieties in general dimensions
are essentially done by finding good algorithms.
In such an approach, one isolates a set of bounded invariants and prove that after certain finite steps,
such invariants improve strictly. As the invariants are bounded, the algorithm terminates.
These approaches are nicely presented in Kollar's book \cite{Kollar}.

According to \cite{Vakil}, many moduli spaces or deformation spaces exhibit arbitrary singularities.
In other words, all singularities exist {\it geometrically}.
Since singularities exist for geometric reasons, one would wonder whether
there should be {\it geometric ways} to resolve them, avoiding pure algorithms on polynomials. Being  philosophically optimistic,
the author believes that every singular moduli admits a resolution, in a specific relative sense,
 such that the resolution itself is also a moduli. 
 
In other words, it would be desirable if the following problem can be answered in some positive ways.

\begin{problem}\label{program}
For any singular moduli space $\mathfrak M$, find another moduli space $\widetilde{\mathfrak M}$
that only modifies the boundary objects  of  $\mathfrak M$ 
 such that every irreducible component of $\widetilde{\mathfrak M}$, endowed with the reduced stack structure,
  is smooth, and
all such irreducible components meet transversally.
\end{problem}

Here, an object of $\widetilde{\mathfrak M}$ should be obtained from the corresponding object of
$\mathfrak M$ by adding certain extra data. The extra data should reduce the automorphisms
of the original object, and ideally, should remove  all removable obstructions.

See Conjectures 5.4 and 5.5 of \cite{Hu17} for  somewhat more precise formulations.


\end{document}